\documentclass{svjour3}                     
\journalname{}
\usepackage[
    backend=bibtex,
    style=alphabetic, 
    ]{biblatex}
\usepackage{multirow} 
\usepackage{amsmath,amsxtra,amsfonts,amscd,amssymb,bm,bbm,mathrsfs}
\usepackage{nicefrac}       
\usepackage{extpfeil}
\usepackage{mathtools}
\usepackage{booktabs}
\usepackage{arydshln}
\usepackage[table]{xcolor} 
\usepackage{colortbl}
\usepackage{float}
\usepackage{hyperref}
\hypersetup{colorlinks=true,linkcolor=blue,citecolor=blue}
\usepackage{makecell}
\usepackage{lipsum}
\usepackage{graphicx}
\usepackage{epstopdf}
\usepackage{algorithmic}
\usepackage{tcolorbox}

\ifpdf
  \DeclareGraphicsExtensions{.eps,.pdf,.png,.jpg}
\else
  \DeclareGraphicsExtensions{.eps}
\fi
\newtheorem{assumption}{Assumption}

\makeatletter
\@addtoreset{equation}{section}
\makeatother

\usepackage{booktabs}
\usepackage{arydshln}
\usepackage{multirow}

\usepackage{tikz}
\usetikzlibrary{positioning, arrows.meta,calc,decorations.text,bending}
\usetikzlibrary{arrows.meta, cd}
\usetikzlibrary{matrix,shadows,positioning,decorations.pathreplacing,decorations.markings}
\usetikzlibrary{matrix,shapes.arrows, shapes.geometric}
\usetikzlibrary{patterns}
\usepackage{standalone}

\tcbuselibrary{breakable, skins}
\newtcolorbox{smallObjectiveBox}[1][]{
  breakable,
  enhanced,
  skin=enhanced,  
  colback=gray!5,
  colframe=black,
  boxrule=0.pt,       
  arc=2mm,
  left=1pt, right=1pt, top=1pt, bottom=1pt,
  before upper=\everymath{\displaystyle},
  frame style={
	draw=black,
    line width=0.6pt,
    dash pattern=on 3pt off 2pt
  },
  #1
}

\usepackage{bm}
\usepackage{enumitem}
\usepackage[normalem]{ulem}

\smartqed  

\newcommand{\innerp}[1]{\langle {#1} \rangle }
\newcommand{\norm}[1]{\left\| {#1} \right\| }

\newcommand{\zkh}[1]{\left[ {#1} \right]}
\newcommand{\hkh}[1]{\left\{ {#1} \right\}}
\newcommand{\kh}[1]{\left( {#1} \right)}

\DeclareMathOperator*{\argmin}{arg\,min}

\DeclareMathOperator{\ddiag}{diag}
\DeclareMathOperator{\Diag}{Diag}
\DeclareMathOperator{\dist}{dist}
\DeclareMathOperator{\rank}{rank}

\DeclareMathOperator{\vecspan}{span}

\DeclareMathOperator{\ima}{im}

\DeclareMathOperator{\range}{range}
\DeclareMathOperator{\kernel}{ker}

\newcommand{\st}{\mathrm{s.\,t.}}

\newcommand{\projection}{\mathcal{P}}
\newcommand{\manifold}{\mathcal{M}}

\newcommand{\danifold}{\mathcal{D}}
\newcommand{\aanifold}{\mathcal{A}}
\newcommand{\lanifold}{\mathcal{L}}
\newcommand{\wanifold}{\mathcal{W}}

\newcommand{\qanifold}{\mathcal{Q}}
\newcommand{\janifold}{\mathcal{J}}
\newcommand{\canifold}{\mathcal{C}}
\newcommand{\uanifold}{\mathcal{U}}
\newcommand{\vanifold}{\mathcal{V}}
\newcommand{\ranifold}{\mathcal{R}}
\newcommand{\eanifold}{\mathcal{E}}

\newcommand{\kanifold}{\mathcal{K}}
\newcommand{\xanifold}{\mathcal{X}}

\newcommand{\nanifold}{\mathcal{N}}
\newcommand{\sanifold}{\mathcal{S}}
\newcommand{\hanifold}{\mathcal{H}}
\newcommand{\zanifold}{\mathcal{Z}}

\newcommand{\hyperboloid}{\mathbb{H}_{m-1}}
\newcommand{\matrixhyperboloid}{\mathbb{H}^n_{m-1}}
\newcommand{\boundedrank}{\mathcal{M}_{\le r}}
\newcommand{\lowrank}{\mathcal{M}_{s}}
\newcommand{\fixedrank}{\mathcal{M}_{r}}

\newcommand{\stiefel}{\mathrm{St}}

\newcommand{\grassmann}{\mathrm{Gr}}

\newcommand{\trace}{\mathrm{tr}}
\newcommand{\orth}{\mathcal{O}}
\newcommand{\oblique}{\mathrm{Ob}}
\newcommand{\Fsphere}{\mathrm{S_F}}
\newcommand{\affine}{\mathrm{Aff}}

\newcommand{\sdp}{\mathcal{S}}
\newcommand{\boundedranksdp}{\sdp^+_{\le r}}
\newcommand{\boundedranks}{\sym_{\le r}(n)}
\newcommand{\manifoldLR}{\manifold_{\mathrm{LR}}}
\newcommand{\phiLR}{\phi_{\mathrm{LR}}}
\newcommand{\manifoldBM}{\manifold_{\mathrm{BM}}}
\newcommand{\phiBM}{\phi_{\mathrm{BM}}}
\newcommand{\manifolddesing}{\manifold_{\mathrm{desing}}}
\newcommand{\phidesing}{\phi_{\mathrm{desing}}}

\newcommand{\bmanifold}{\overline{\mathcal{M}}}
\newcommand{\bkanifold}{\overline{\mathcal{K}}}
\newcommand{\bhanifold}{\overline{\mathcal{H}}}
\newcommand{\buanifold}{\overline{\mathcal{U}}}

\newcommand{\bq}{\bar{q}}

\newcommand{\secfunda}{\mathrm{II}}

\newcommand{\diff}{\mathrm{D}}
\newcommand{\codiff}{\mathrm{D}^*}
\newcommand{\tangent}{\mathrm{T}}
\newcommand{\normal}{\mathrm{N}}

\newcommand{\tanL}{\mathrm{\bf L}}
\newcommand{\tanQ}{\mathrm{\bf Q}}

\newcommand{\qf}{\operatorname{orth}}
\newcommand{\Fnormal}{\hat{\mathrm{N}}}
\newcommand{\Lnormal}{\mathrm{N}}

\newcommand{\Btangent}{\mathrm{T}}
\newcommand{\tangenttwo}{\mathrm{T}^2}

\newcommand{\twototwo}{{``$2\!\Rightarrow\!2$''}}

\newcommand{\graph}{\mathrm{gph}}

\newcommand{\upr}{\bar{r}}
\newcommand{\downr}{{\underline{r}}}
\newcommand{\sigmar}{\sigma_{r+1}}

\newcommand{\tensorspace}{\mathbb{R}^{n_1\times n_2\times\cdots\times n_d}}
\newcommand{\vecr}{\mathbf{r}}

\newcommand{\boundedtucker}{\mathcal{M}^{\mathrm{tc}}_{\le \vecr}}
\newcommand{\boundedtt}{\mathcal{M}^{\mathrm{tt}}_{\le \vecr}}
\newcommand{\boundedht}{\mathcal{M}^{\mathrm{ht}}_{\le \vecr}}

\newcommand{\mtt}{{\mathrm{tt}}}
\newcommand{\mtc}{{\mathrm{tc}}}
\newcommand{\mht}{{\mathrm{ht}}}
\newcommand{\boundedtensor}{\manifold_{\le \vecr}^\mht}

\newcommand{\ranktt}{\mathrm{rank}_{\mathrm{tt}}}
\newcommand{\ranktc}{\mathrm{rank}_{\mathrm{tc}}}
\newcommand{\rankht}{\mathrm{rank}_{\mathrm{ht}}}

\newcommand{\tensorize}{\mathrm{ten}}

\newcommand{\tensorizeht}{\mathrm{ten}^{\mathrm{ht}}}

\newcommand{\tensY}{\mathbf{Y}}
\newcommand{\tensX}{\mathbf{X}}

\newcommand{\tenseta}{\bm{\eta}}
\newcommand{\tenszeta}{\bm{\zeta}}

\newcommand{\myfig}{{Fig.\,}}

\newcommand{\frob}{\mathrm{F}}

\newcommand{\sksym}{\mathrm{Skew}}

\newcommand{\sym}{\mathcal{S}}

\newcommand{\mbR}{\mathbb{R}}
\newcommand{\R}{\mathbb{R}}
\newcommand{\mbN}{\mathbb{N}}
\newcommand{\mbRmn}{\mathbb{R}^{m\times n}}

\newcommand{\bU}{\bar{U}}
\newcommand{\bV}{\bar{V}}
\newcommand{\bareta}{\bar{\eta}}
\newcommand{\barzeta}{\bar{\zeta}}

\newcommand{\lineB}{\overline{\neighbor}}

\newcommand{\closeconv}{\overline{\mathrm{conv}}}

\newcommand{\clique}{CLIQUE}
\newcommand{\verifysec}{VERSOC}

\newcommand{\neighbor}{\mathcal{B}}

\definecolor{comblue}{RGB}{2,0,255}

\definecolor{blue}{rgb}{0, 0.451, 1}
\definecolor{myblue}{rgb}{0,0,.5}
\definecolor{bblue}{rgb}{0,0,.85}
\definecolor{mygreen}{rgb}{0, 0.62, 0.38}
\definecolor{red}{rgb}{0.796, 0.059, 0.063}
\definecolor{myred}{rgb}{.5,0,0}
\definecolor{LightGray}{rgb}{0.98, 0.98, 0.98}
\definecolor{Gray}{gray}{0.85}

\definecolor{setcolor}{cmyk}{0, 0, 0, 1}  
\definecolor{tan1color}{RGB}{70, 160, 220}  
\definecolor{tan2color}{RGB}{220, 100, 100}  

\definecolor{myLightBlueBorder}{RGB}{70, 160, 220} 
\definecolor{myLightBlueFill}{cmyk}{0.05, 0, 0, 0}  

\definecolor{myLightRedBorder}{RGB}{220, 100, 100} 
\definecolor{myLightRedFill}{cmyk}{0, 0.08, 0.08, 0}

\definecolor{colFocusAlpha}{RGB}{0, 80, 180}   
\definecolor{colFocusBeta}{RGB}{180, 30, 30}    
\definecolor{colNormal}{RGB}{0, 0, 0}           

\begin{document}

\title{Variational analysis of determinantal varieties\thanks{This work was supported by the National Key R\&D Program of China (grant 2023YFA1009300). BG and YY were supported by the National Natural Science Foundation of China (grant No. 12288201).}}

\titlerunning{Variational analysis of determinantal varieties}

\author{Yan Yang  \and Bin Gao \and Ya-xiang Yuan}


\institute{
Yan Yang \at
Academy of Mathematics and Systems Science, Chinese Academy of Sciences, and the University of Chinese Academy of Sciences, Beijing, China
\\
\email{yangyan@amss.ac.cn} 
\and
Bin Gao \and Ya-xiang Yuan \at
State Key Laboratory of Mathematical Sciences, Academy of Mathematics and Systems Science, Chinese Academy of Sciences, Beijing, China \\
\email{\{gaobin,yyx\}@lsec.cc.ac.cn; }
}

\date{Received: date / Accepted: date}

\maketitle

\begin{abstract}
Determinantal varieties---the sets of bounded-rank matrices or tensors---have attracted growing interest in low-rank optimization. The tangent cone to low-rank sets is widely studied and underpins a range of geometric methods. The second-order geometry, which encodes curvature information, is more intricate. In this work, we develop a unified framework to derive explicit formulas for both first- and second-order tangent sets to various low-rank sets, including low-rank matrices, tensors, symmetric matrices, and positive semidefinite matrices. The framework also accommodates the intersection of a low-rank set and another set satisfying some assumptions, thereby yielding a tangent intersection rule. Through the lens of tangent sets,  we establish a necessary and sufficient condition under which a nonsmooth problem and its smooth parameterization share equivalent second-order stationary points. As applications, we exploit tangent sets to characterize optimality conditions for low-rank optimization and prove that verifying second-order optimality is NP-hard. Moreover, for low-rank semidefinite programs, the Burer--Monteiro parameterization maps fourth-order stationary points of the parameterized problem to second-order stationary points of the original problem. In a separate line of analysis, we investigate variational geometry of the graph of the normal cone to matrix varieties, deriving the explicit Bouligand tangent cone, Fr\'echet and Mordukhovich normal cones to the graph. These results are further applied to develop optimality conditions for low-rank bilevel programs.

\end{abstract}

\hypersetup{linkcolor=black}
\tableofcontents
\hypersetup{linkcolor=blue}

\section{Introduction}\label{sec:intro}
The low-rank structure of matrix data is widely exploited in various applications to improve memory and computation efficiency \cite{markovsky2008systemide,udell2019big}, which leads to the consideration of the set of bounded-rank matrices:
\begin{equation}\label{eq:boundedrank}
    \boundedrank := \hkh{X \in \mathbb{R}^{m \times n} \mid \rank(X) \leq r},
\end{equation}
where $r\le \min\{m,n\}$ is the rank parameter. In fact, $\boundedrank$ is a real algebraic variety, commonly referred to as the \emph{determinantal variety} \cite{harris2013algebraic}, and the concepts can be extended to tensor determinantal varieties~\cite{kutschan2018tangentTT,gao2025lowranktucker}. 

First-order variational results on $\boundedrank$ \cite{luke2013Mordukhovichcone,cason2013iterative,hosseini2019MordukhovichClarke,olikier2022continuity}, including the characterization of its tangent and normal cones, have contributed to low-rank optimization by allowing the derivation of first-order optimality conditions and supporting the development of geometric methods \cite{schneider2015Lojaconvergence,gao2022rankadap,olikier2023RFDR,olikier2025PGD,olikier2025gauss}. Furthermore, the intersection $\boundedrank\cap\hanifold$ with an additional structure $\hanifold$ (e.g., an affine set, the Frobenius sphere, the oblique manifold, or the positive semidefinite cone) has attracted growing interest, which, however, was previously treated on a case-by-case basis for different $\hanifold$ in the literature~\cite{cason2013iterative,tam2017sparsesdp,li2020jotaspectral,li2023normalboundedaffine,yang2025spacedecouple}.

Moreover, the second-order geometry of $\boundedrank$ remains unclear, mainly due to two obstacles: 1) second-order analysis in essence captures curvature information \cite{bonnans2000perturbationanalysis,gfrerer2017robinsonstability}, thereby appearing more involved than the first-order counterpart; 2) the determinantal variety is both nonconvex and nonsmooth~\cite{olikier2022continuity,levin2023remedy}, hindering the direct use of standard tools in variational analysis.

The following two problems underpin the importance of variational analysis of $\boundedrank$.

\paragraph{\emph{\textbf{Optimization problem over bounded-rank matrices.}}} By imposing a low-rank constraint (possibly together with an additional constraint $\hanifold$) on the matrix variable, the following formulation has been successfully used in a multitude of applications~\cite{markovsky2008systemide,zhu2022learningmarkov},
\begin{equation}\label{eq:boundedrankopt}
\begin{aligned}
    \min_{X\in\mbR^{m\times n}}\ \ & f(X)\\[-1.5mm]
    \mathrm{s.\,t.}\ \ \ \ & X\in\boundedrank\cap\hanifold,
    \end{aligned}
\end{equation}
where $f:\mbRmn\to\mbR$ is twice continuously differentiable. When $\hanifold=\mbRmn$, the problem reduces to minimizing a function over bounded-rank matrices. The first-order optimality conditions are well understood \cite{cason2013iterative,hosseini2019MordukhovichClarke,levin2023remedy}, whereas the second-order counterparts remain ambiguous, since the second-order geometry of $\boundedrank$, specifically its second-order tangent set (see section~\ref{sec:background_vari}), has yet to be fully explored. When the data exhibits an additional structure, the constraint evolves into $X\in\boundedrank\cap\hanifold$ with a nontrivial $\hanifold$~\cite{cason2013iterative,tam2017sparsesdp,li2020jotaspectral,li2023normalboundedaffine,yang2025spacedecouple}. More specifically, when $\hanifold$ enforces the semidefinite constraint, problem~\eqref{eq:boundedrankopt} gives rise to a line of low-rank semidefinite programming problems~\cite{boumal2016BMSDP,wang2023decomposition,tang2024feasibleSDP,levin2025effect}. The coupled structure renders the geometry of the feasible region more intricate, and the analysis of $\boundedrank\cap\hanifold$ was developed upon the specific structure of $\hanifold$.

\paragraph{\emph{\textbf{Bilevel programming with low-rank structure.}}}
We consider a bilevel programming problem where the lower-level task seeks a low-rank solution, modeling applications across various fields \cite{shaban2019truncated,grangier2023LLMshift,shen2025seal,zangrando2025debora} (see section~\ref{sec:bilevel_app}):
\begin{equation}\label{eq:lowrank_BiO}\tag{LRBP}
    \begin{aligned}
    \min_{x\in\mbR^q,X^*\in\mbR^{m\times n}}&\ \lanifold(x,X^*)\\[-1.5mm]
    \mathrm{s.\,t.}\ \ \ \ \ \ &\ G(x)\le 0,
        \\
        &\ X^*\in \argmin_{X\in\mbRmn}\ \ F(x,X),
        \\
        &\ \ \ \ \ \ \ \ \ \ \ \ \ \ \mathrm{s.\,t.}\ \ \ X\in\boundedrank.
    \end{aligned}
\end{equation}
Generally, finding a global minimizer of a function subject to the bounded-rank constraint is NP-hard \cite{gillis2011NPlowrank}. Nevertheless, existing literature~\cite{schneider2015Lojaconvergence,levin2023remedy,jia2023convergencePGDKL,aragon2024coderivativeNewton} is able to compute an \emph{M-stationary point}, at which the antigradient belongs to the Mordukhovich normal cone of $\boundedrank$. Therefore, it is reasonable to introduce the following relaxation for \eqref{eq:lowrank_BiO}, by replacing the lower-level global optimality with the M-stationarity,
\begin{equation}\label{eq:lowrank_BiO_M1}
\begin{aligned}
    \min_{x\in\mbR^q,X\in\mbR^{m\times n}}&\ \lanifold(x,X)\\[-1.5mm]
    \mathrm{s.\,t.}\ \ \ \ \ \,&\ G(x)\le 0,
        \\
        &\ -\nabla_X F(x,X)\in\normal_{\boundedrank}(X),
    \end{aligned}
\end{equation}
where $\nabla_X$ denotes the partial gradient with respect to $X$. Since the Mordukhovich normal cone mapping $\normal_{\boundedrank}(\cdot)$ arises in the constraints, its \emph{coderivative}---a generalized subdifferential~\cite{mordukhovich2006variationalI}---would be involved in deriving the first-order optimality condition for \eqref{eq:lowrank_BiO_M1}. Note that the normal cone mapping corresponds to the subdifferential of the indicator function associated with the set $\boundedrank$, and thus the coderivative naturally enters the scope of the second-order variational analysis of $\boundedrank$.

\medskip

Corresponding to the above two examples, we are concerned with variational analysis of determinantal varieties from two aspects: the first- and second-order tangent sets to an array of low-rank sets, and the coderivative of the normal cone mapping $\normal_{\boundedrank}(\cdot)$. Next, we provide an overview of the existing literature and main challenges.

\subsection{Related work and main challenges}\label{sec:relatedwork}
We begin by summarizing the first-order variational results on $\boundedrank$, which enlighten a range of geometric algorithms for low-rank optimization. The term first-order tangent set refers to the Bouligand tangent cone.

\paragraph{\emph{\textbf{First-order variational analysis of the determinantal variety.}}} Research on the geometry of $\boundedrank$, especially the associated tangent and normal cones, has flourished in low-rank optimization. Typical characterizations include the Mordukhovich normal cone \cite{luke2013Mordukhovichcone}, the Bouligand tangent cone \cite{cason2013iterative,schneider2015Lojaconvergence,olikier2025fourtangentproof}, and the Clarke tangent and normal cones \cite{hosseini2019MordukhovichClarke,li2019optimalitylowrank}. The optimality conditions are derived: the projection of the antigradient onto the corresponding tangent cone vanishes. Furthermore, Olikier and Absil \cite{olikier2022continuity} investigated the continuity of the cone mappings, which underlies the so-called \emph{apocalypse} phenomenon observed in optimization problems over the bounded-rank matrices~\cite{levin2023remedy}. 

The developed geometry has given rise to numerous algorithms for low-rank optimization, i.e., problem~\eqref{eq:boundedrankopt} with $\hanifold=\mbRmn$. One class builds on the projected gradient descent framework \cite{schneider2015Lojaconvergence,olikier2022P2GDR,olikier2025PGD}, where each iteration proceeds by taking a descent step followed by a projection onto the feasible region $\boundedrank$. Another line of work embraces a retraction-free approach \cite{schneider2015Lojaconvergence,olikier2023RFDR,olikier2024ERFDR}: instead of performing projections, one adopts search directions from the so-called \emph{restricted tangent cone} to~$\boundedrank$ (see \cite{olikier2023RFDR}) and iterates along straight lines. The third class leverages a smooth parameterization of $\boundedrank$ by constructing a manifold~$\bmanifold$ and a mapping~$\phi$ such that $\phi(\bmanifold) = \boundedrank$, thereby implementing optimization algorithms over the smooth manifold instead of the nonsmooth determinantal variety \cite{khrulkov2018desingularization,rebjock2024boundedrank,levin2023remedy,levin2025effect,olikier2025gauss}. 

When we introduce an additional constraint set $\hanifold$ to study the coupled region $\boundedrank\cap\hanifold$, the geometry becomes more complicated. Specifically, for the case $m=n$ and $\hanifold =\sym(n)$, the set of $n\times n$ symmetric matrices, Tam et al. \cite{tam2017sparsesdp} established the Mordukhovich normal cone to $\boundedrank \cap \hanifold$. Subsequently, Li et al. \cite{li2020jotaspectral} provided formulations for the Fr\'{e}chet normal cones when $\hanifold$ represents the intersection of $\sym(n)$ with the closed unit Frobenius ball, the symmetric box, or the spectrahedron. Moreover, Levin et al. \cite{levin2025effect} explicitly computed the Bouligand tangent cone to $\boundedranksdp(n):=\boundedrank\cap\sym^+(n)$, where $\sym^+(n)$ denotes the closed convex cone of all positive semidefinite matrices in $\sym(n)$. The above advancements require the matrix to be square and symmetric. When $m\neq n$ breaks this symmetry, more challenges arise, and new techniques tailored to the geometry of the coupled region are needed. A closed-form expression of the Bouligand tangent cone to the intersection of $\boundedrank$ and the Frobenius sphere was developed in~\cite{cason2013iterative}. Recently, Li and Luo \cite{li2023normalboundedaffine} characterized the Fr\'{e}chet normal cone to $\boundedrank\cap\hanifold$ with $\hanifold$ as an affine manifold. In addition, Yang et al.~\cite{yang2025spacedecouple} derived the Bouligand tangent cone and the Fr\'{e}chet normal cone to $\boundedrank\cap\hanifold$ for $\hanifold=\{X\in\mbRmn\mid h(X)=0\}$ with a differentiable and \emph{orthogonally invariant} mapping $h$, i.e., $h(X)=h(XQ)$ for all orthogonal matrices~$Q$.

\bigskip

Notice that second-order results of the determinantal variety remain limited. Therefore, we then revisit relevant developments in second-order variational analysis for general sets; details are referred to \cite{rockafellar2009variationalanalysis,mordukhovich2024secondorder} and the references therein. More specifically, regarding certain structured sets, we discuss some existing techniques for deriving the associated second-order tangents set and coderivatives of normal cone mappings.

\paragraph{\emph{\textbf{Second-order tangent set.}}} Analogous to how the second-order derivative of a mapping refines the linear approximation into the quadratic one, the second-order tangent set to a given set provides a more accurate local approximation than the Bouligand tangent cone \cite{bonnans2000perturbationanalysis,chen2019exactSOC}. Therefore, it serves as an important tool for analyzing optimality conditions \cite{bonnans1999parabolicset,gfrerer2022secondnonconvex}, metric subregularity \cite{gfrerer2011subregularity}, and system stability \cite{gfrerer2017robinsonstability}. More relevant to this work, Levin et al. \cite{levin2025effect} tackled optimization problems over a nonsmooth set (e.g., the determinantal variety) by studying the parameterization technique that recasts the problem on a smooth manifold. Specifically, the second-order tangent set to the manifold is exploited to establish the equivalence between the smooth and nonsmooth problems.

There are two classes of sets to which the second-order tangent set is well understood. The first class considers sets of the form $\xanifold = \{X \in \mathbb{R}^q \mid h(X) \in \canifold\}$, where $h$ is a smooth mapping and $\canifold$ is a closed set. A characterization of the second-order tangent set to $\xanifold$ was given in \cite[Proposition 13.13]{rockafellar2009variationalanalysis}. As a special case, when $\xanifold$ is a smooth manifold, an interpretation through the lens of smooth curves was provided in \cite{levin2025effect}. However, the characterization relies on constraint qualifications such as Robinson's condition~\cite{bonnans2000perturbationanalysis,ruszczynski2006nonlinearopt}, which imply the Clarke regularity of $\xanifold$, whereas $\boundedrank$ is not Clarke regular at rank-deficient points~\cite{hosseini2019MordukhovichClarke}. The second class consists of sets of the form $\xanifold = \{X \in \mathbb{R}^q \mid h(X) \le 0\}$, where $h$ is convex. Under the Slater condition, Bonnans et al. computed the second-order tangent set by relating it to the second-order subdifferential of $h$ \cite[Proposition 2.1]{bonnans1999parabolicset}. Applying this theory, the second-order tangent sets to the convex cone $\sym^+(n)$ and to the \emph{second-order cone}\footnote{The second-order cone is defined as $\xanifold := \{(X_1, X_2) \in \mathbb{R} \times \mathbb{R}^q : \norm{X_2}_2 \le X_1 \}$.} (SEC) were characterized in \cite[Example 3.40]{bonnans2000perturbationanalysis} and \cite[Lemma 27]{bonnans2005perturbation}, respectively.

However, when the set $\xanifold$ of interest is both nonsmooth and nonconvex, the analysis should resort to the specific structure of $\xanifold$. An example of $\xanifold$ is the \emph{SEC complementarity set}. In this case, Chen et al. \cite{chen2019exactSOC} noticed that the metric projection operator onto the SEC is well-defined and admits second-order directional derivatives, based on which they gave the exact formula for the second-order tangent set to $\xanifold$.

\paragraph{\emph{\textbf{Coderivative of normal cone mappings.}}} Given a set $\xanifold$, the coderivative of the associated normal cone mapping, denoted by $\codiff\Lnormal_\xanifold$, can be treated as the second-order subdifferential of the indicator function of $\xanifold$ \cite{mordukhovich2015secondvarconic}. The concept of coderivative plays a pivotal role in investigating the stability and sensitivity of variational systems \cite{poliquin1998tilt,dontchev2009implicitsolutionmappings}, regularity properties of set-valued mappings \cite{dontchev2009implicitsolutionmappings}, and optimality conditions of bilevel programming problems \cite{ding2014SDCMPCC,dempe2018optimality}. As pointed out in \cite{chieu2017coderivativeweakcondition}, computing explicitly the coderivative of a given set-valued mapping is generally a demanding task.

When the considered $\xanifold$ is a polyhedral convex set, Dontchev and Rockafellar \cite{dontchev1996polyhedralconvex} characterized the associated $\codiff\Lnormal_\xanifold$; and then a line of works~\cite{gfrerer2015weakest,gfrerer2016computationgeneralizedderivatives,chieu2017coderivativeweakcondition} extended the results to the case of $\xanifold=\{X\in\mbR^q\mid h(X)\in\canifold\}$ with a twice continuously differentiable mapping $h$ and a polyhedral convex set $\canifold$ satisfying some qualification conditions. Additionally, for $\xanifold=\sym^+(n)$ as a closed convex cone, the directional derivative of the projection operator onto $\sym^+(n)$ was exploited to obtain the explicit formula of $\codiff\Lnormal_{\sym^+(n)}$ \cite{ding2014SDCMPCC,wu2014SDCMPCC}. This identification facilitates the derivation of optimality conditions for bilevel programming problems where $\sym^+(n)$ appears as a constraint set in the lower-level problem \cite{dempe2018optimality}.

\bigskip

In summary, the first-order geometry of low-rank sets $\boundedrank \cap \hanifold$ has been treated in the literature only on a case-by-case basis for different choices of $\hanifold$---there is currently no framework that both unifies existing results and guides new developments. In addition, the second-order analysis of $\boundedrank$ is even more challenging. Essentially, the determinantal variety is nonconvex and nonsmooth: the recent work \cite{olikier2022continuity} showed that the first-order geometric objects associated with $\boundedrank$---such as the Mordukhovich normal cone mapping---are discontinuous, further underscoring the difficulty of developing second-order theory. Moreover, the lack of regularity precludes the straightforward employment of existing techniques to low-rank sets. For instance, in the cases of the aforementioned SEC and $\sym^+(n)$, analyzing the directional derivative of the associated projection operator is able to identify the second-order tangent set or the coderivative of the normal cone mapping. However, for the determinantal variety, the projection $\projection_{\boundedrank}(X)$ becomes set-valued when the nonzero $r$-th and $(r+1)$-th largest singular values of $X$ coincide. This non-uniqueness breaks the differentiability of the projection operator, thereby restricting the translation of the existing projection-based techniques.

\subsection{Contributions}
In this paper, we investigate variational properties of low-rank sets. The contributions unfold along two lines: I. a unified framework for developing first- and second-order tangent sets to low-rank sets; II. the geometry of the graph of the normal cone mapping induced by the determinantal variety, which facilitates the characterization of the associated coderivatives; see the roadmap in \myfig\ref{fig:twocontributions}.

\begin{figure}[htbp]
\begin{minipage}{1\textwidth}
\hspace{-3.7mm}
\begin{tikzpicture}
	\coordinate (O) at (0,0);
	\coordinate (M1) at (-3.1,2);
	\coordinate (Tangent) at (-3.1,-0.331);

	\coordinate (M2) at (3.2,1.55);
	\coordinate (M3) at (0,-4.);

	\draw[-,thick,color=cyan!85!green!80!blue,dashed,rounded corners,fill=myLightBlueFill] ($(M1)+(-2.96,-0.6)$) rectangle ($(M1)+(3.1,0.65)$);
	\node at (M1) {
		\begin{minipage}{4.5cm}
            \begin{small}
            \centering
            Problem on intersection of sets
			\vspace{-1.5mm}
			\begin{align*}
				\min\ \ f(X)\ \ \ \st\ \ X\in\manifold\cap{\kanifold}
			\end{align*}   
            \end{small}
		\end{minipage} 
	};
	\node at ([xshift=0cm, yshift=10mm]M1) {I. {Tangent sets and optimization}};

	\draw[-,thick,color=cyan!85!green!80!blue,dashed,rounded corners,fill=myLightBlueFill] ($(Tangent)+(-2.96,-0.85)$) rectangle ($(Tangent)+(3.1,0.75)$);
	\node at ($(Tangent)+(0.14cm,0.13cm)$) {
		\begin{minipage}{0.55\textwidth}
            \vspace{3mm}
			\begin{small}
            \phantom{Table}Theorem~\ref{the:cal_tangentsets}: tangent sets to $\manifold$
            \\
            \phantom{Table}Theorem~\ref{the:expressions_McapK}: intersection rule for $\manifold\cap\kanifold$
            \\[1mm]
            \phantom{Theorem}Table~\ref{tab:tangentsets}: applications to low-rank sets
			\end{small}
		\end{minipage} 
	};

	\coordinate (App1) at ([xshift=-1.55cm, yshift=-2.4cm]Tangent);
	\draw[-,thick,color=cyan!85!green!80!blue,dashed,rounded corners,fill=myLightBlueFill] ($(App1)+(-1.4,-1.)$) rectangle ($(App1)+(1.12,0.45)$);
	\node[align=center] at ($(App1)+(-0.15cm,-0.23cm)$) {
        {\small Second-order}
        \\
        {\small optimality (SOC)}
	};

	\coordinate (App2) at ([xshift=1.4cm, yshift=-2.4cm]Tangent);
	\draw[-,thick,color=cyan!85!green!80!blue,dashed,rounded corners,fill=myLightBlueFill] ($(App2)+(-1.6,-1.)$) rectangle ($(App2)+(1.7,0.45)$);
	\node[align=center] at ($(App2)+(0.05cm,-0.24cm)$) {
        {\small Sufficient and necessary}
        \\
        {\small condition for \twototwo}
	};

	\coordinate (NP) at ([xshift=-1.55cm, yshift=-5cm]Tangent);
	\draw[-,thick,color=cyan!85!green!80!blue,dashed,rounded corners,fill=myLightBlueFill] ($(NP)+(-1.4,-1.)$) rectangle ($(NP)+(1.12,0.45)$);
	\node[align=center] at ($(NP)+(-0.2cm,-0.23cm)$) {
        {\small NP-hardness of}
        \\
        {\small verifying SOC}
	};

	\coordinate (find2) at ([xshift=1.4cm, yshift=-5cm]Tangent);
	\draw[-,thick,color=cyan!85!green!80!blue,dashed,rounded corners,fill=myLightBlueFill] ($(find2)+(-1.6,-1.)$) rectangle ($(find2)+(1.7,0.45)$);
	\node[align=center] at ($(find2)+(0.1cm,-0.24cm)$) {
        {\small Parameterizations yield }
        \\
        {\small SOC on $\boundedrank$}
        \\
        {\small only at rank-$r$ points}
	};

	\draw[->,thick] ($(M1)+(0,-0.7)$) -- ($(Tangent)+(0,0.84)$);
    
    \draw[->,thick] 
        ($(Tangent)+(-0.4,-0.93)$) 
        -- ($(App1)+(0.7,0.55)$) 
        node[midway, left, xshift=-1pt] {Propositions~\ref{pro:2order_opt_H}-\ref{pro:sec_opt_Mr}};
        
	\draw[->,thick] ($(Tangent)+(0.4,-0.93)$) -- ($(App2)+(-0.7,0.55)$) node[midway, right, xshift=1pt] {Theorem~\ref{the:2to2}};
    
	\draw[->,thick] ($(App1)+(0.7,-1.1)$) -- ($(NP)+(0.7,0.55)$) node[midway, left, xshift=-1pt] {Theorems~\ref{the:np}-\ref{the:np_approx}};

	\draw[->,thick] ($(App2)+(-0.7,-1.1)$) -- ($(find2)+(-0.7,0.55)$) node[midway, right, xshift=0.1pt] {Propositions~\ref{pro:LR_2to2}-\ref{pro:desing_2to2}};

	\coordinate (BM) at ([xshift=0.07cm, yshift=-7.35cm]Tangent);
	\draw[-,thick,color=cyan!85!green!80!blue,dashed,rounded corners,fill=myLightBlueFill] ($(BM)+(-3.03,-0.48)$) rectangle ($(BM)+(3.03,0.48)$);
	\node[align=center] at (BM) {
        \small Burer--Monteiro yields SOC
        \\
        \small via fourth-order stationarity
	};
	\draw[->,thick] ($(find2)+(0,-1.1)$) -- ($(BM)+(1.33,0.52)$) node[midway, right, xshift=1pt] {Theorem~\ref{the:F4to2}};

	\draw[-,thick,color=red!70!red,dashed,rounded corners,fill=myLightRedFill] ($(M2)+(-2.7,-1.5)$) rectangle ($(M2)+(2.7,1.1)$);
	\node at ($(M2)+(0,-0.2)$) {
		\begin{minipage}[c][1.5cm][c]{5cm}\centering
            \begin{small}
                Low-rank bilevel programming problem~\eqref{eq:lowrank_BiO}
            \end{small}
            
            \vspace{1mm}
            
            {\large $\Downarrow$}
            
            \vspace{1mm}
            
            \begin{small}
                Relaxed version involving $\graph\Lnormal_{\boundedrank}$ 
                \\
                as a constraint~\eqref{eq:lowrank_BiO_M}
            \end{small}
		\end{minipage} 
	};
	\node at ([xshift=-0.15cm, yshift=14.5mm]M2) {II. Geometry of $\graph\normal_{\!\boundedrank}$\,and bilevel program};

	\coordinate (NN) at (3.2,-1.5228);
	\draw[-,thick,color=red!70!red,dashed,rounded corners,fill=myLightRedFill] ($(NN)+(-2.7,-0.48)$) rectangle ($(NN)+(2.7,0.48)$);

    \node[align=center] at (NN) {
        {\scriptsize Bouligand tangent cone to ${\graph\Lnormal_{\boundedrank}}$}
	};

    \coordinate (Fnormal) at (3.2,-3.5756);
    \draw[-,thick,color=red!70!red,dashed,rounded corners,fill=myLightRedFill] ($(Fnormal)+(-2.7,-0.48)$) rectangle ($(Fnormal)+(2.7,0.48)$);
    \node[align=center] at (Fnormal) {
    {\scriptsize Fr\'echet normal cone to ${\graph\Lnormal_{\boundedrank}}$}
	};

    \coordinate (Mnormal) at (3.2,-5.6284);
    \draw[-,thick,color=red!70!red,dashed,rounded corners,fill=myLightRedFill] ($(Mnormal)+(-2.7,-0.48)$) rectangle ($(Mnormal)+(2.7,0.48)$);
    \node[align=center] at (Mnormal) {
    {\scriptsize Mordukhovich normal cone to ${\graph\Lnormal_{\boundedrank}}$}
	};

	\coordinate (BiO) at (3.2,-7.681);
	\draw[-,thick,color=red!70!red,dashed,rounded corners,fill=myLightRedFill] ($(BiO)+(-2.7,-0.48)$) rectangle ($(BiO)+(2.7,0.48)$);
	\node[align=center] at (BiO) {
        {\scriptsize Optimality conditions for~\eqref{eq:lowrank_BiO_M}}
	};

	\draw[->,thick] ($(M2)+(0,-1.6)$) -- ($(NN)+(0,0.58)$) node[midway, right, xshift=1pt] {Theorem~\ref{the:Btangent_graph}};
	\draw[->,thick] ($(NN)+(0,-0.58)$) -- ($(Fnormal)+(0,0.58)$) node[midway, right, xshift=1pt] {Corollaries~\ref{cor:Fnormal_graph_1}-\ref{cor:Fnormal_graph_2}};
	\draw[->,thick] ($(Fnormal)+(0,-0.58)$) -- ($(Mnormal)+(0,0.58)$) node[midway, right, xshift=1pt] {Theorem~\ref{the:Lnormal_graph}};
	\draw[->,thick] ($(Mnormal)+(0,-0.58)$) -- ($(BiO)+(0,0.58)$) node[midway, right, xshift=1pt] {Proposition~\ref{pro:M-stationary}};
	
\end{tikzpicture}
\end{minipage}
\caption{Roadmap of the contributions. Part~I: sections~\ref{sec:2tangentset}-\ref{sec:PSD_BM}; Part~II: sections~\ref{sec:Geometryofgraph}-\ref{sec:LRBP}.}
\label{fig:twocontributions}
\end{figure}
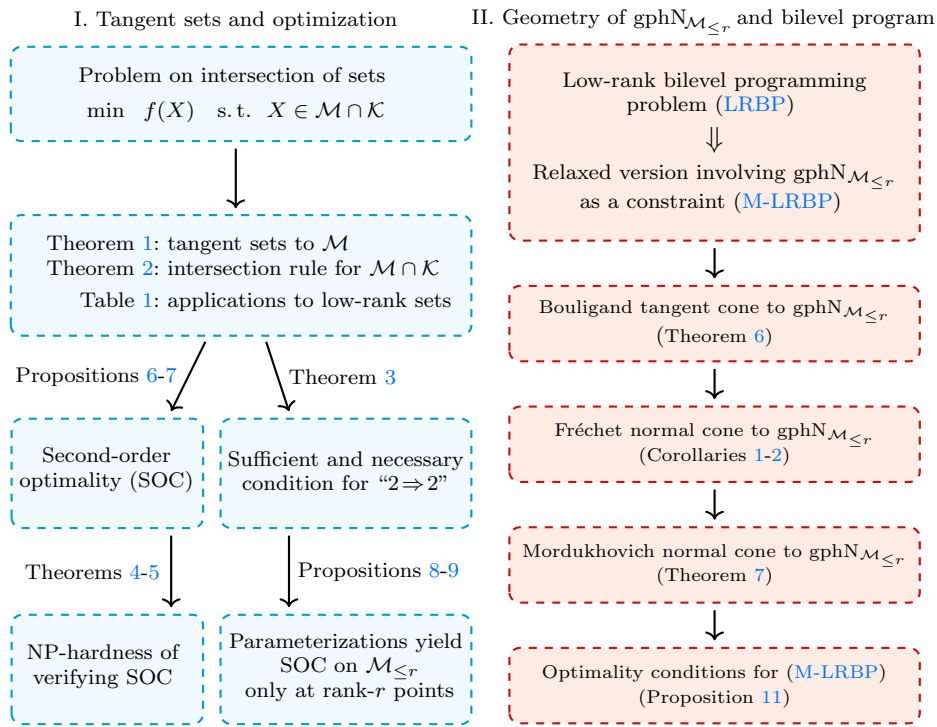

\medskip

\noindent\textbf{I. Tangent sets and optimization.} The main goal of this part is to identify the tangent sets to low-rank sets, thereby providing the optimality analysis for rank-constrained optimization problems.

We endeavor to relate the tangent sets of the determinantal variety to appropriate (generalized) differentials of mappings that capture the low-rank structure. To this end, we consider the singular value mappings which are non-negative and are ordered in a non-increasing fashion: $\sigma_1(X)\ge\sigma_2(X)\ge\cdots\ge\sigma_{\min\{m,n\}}(X)\ge 0$. Subsequently, $\boundedrank$ can be treated as the zeros of the $(r+1)$-th singular value mapping $\sigmar$:
\begin{equation*}
    \boundedrank=\{X\in\mbRmn\mid\sigma_{r+1}(X)=0\}.
\end{equation*}
Specifically, the mapping $\sigmar$ comes with two benefits: 1) it is Lipschitz continuous by Weyl's inequality~\cite{weyl1912inequality}; 2) it satisfies the error bound condition according to the truncated SVD, namely, $\dist(X,\boundedrank)\le (\min\{m,n\}-r)^{1/2} \sigmar(X)$ for any $X\in\mbRmn$. We generalize these two conditions into Assumption~\ref{assu:errorbound}, under which Theorem~\ref{the:cal_tangentsets} is established: given a general set $\manifold$ realized as the zeros of a mapping $c$, the first- and second-order tangent sets to $\manifold$ coincide with the zeros of the first- and second-order directional differentials of $c$, respectively. Applying the developed results to $\manifold=\boundedrank$ and substituting the directional differentials of $\sigmar$ given in~\cite{zhang2013secondordersingular}, we recover the first-order geometry of $\boundedrank$~\cite{schneider2015Lojaconvergence}, and compute explicitly the second-order tangent set to $\boundedrank$ in Proposition~\ref{pro:tangenttwo_mr}. 

Interestingly, the perspectives can be extended to broader scenarios: we view low-rank sets of tensors, symmetric matrices, and positive semidefinite (PSD) matrices as the zeros of singular or eigenvalue mappings (see~\eqref{eq:tensor_defined_sigma}, \eqref{eq:characterization_Sr}, and~\eqref{eq:characterization_S+r}), which allows us to invoke Theorem~\ref{the:cal_tangentsets} to obtain the associated tangent sets.

Moreover, when an additional structured set $\kanifold$ is imposed, we consider the geometry of the intersection $\manifold\cap\kanifold$, and employ the idea of \emph{smooth parameterization} \cite{levin2025effect}, which introduces a smooth manifold $\bmanifold$ embedded in another Euclidean space $\wanifold$ and a smooth mapping $\phi$ with $\phi(\bmanifold)=\manifold$; see~\myfig\ref{fig:diagram} for an illustration. We pull back $\kanifold$ through $\phi$ to obtain $\bkanifold:=\phi^{-1}(\kanifold)$, and then take into account the intersection $\bmanifold\cap\bkanifold$ in the ``auxiliary space'' $\wanifold$. Theorem~\ref{the:expressions_McapK} reveals that, under Assumption~\ref{assu:MintersecK}, some benign properties of the lift $\bmanifold\cap\bkanifold$ can be inherited by $\manifold\cap\kanifold$, therefore decoupling the computation of the tangent sets to $\manifold\cap\kanifold$ into the intersection of those to each component,
\begin{equation*}
\begin{aligned}
    \Btangent_{\manifold\cap\kanifold}(X) &= \Btangent_{\manifold}(X)\cap\Btangent_{\kanifold}(X),
    \\
    \tangenttwo_{\manifold\cap\kanifold}(X;\eta) &= \tangenttwo_{\manifold}(X;\eta)\cap\tangenttwo_{\kanifold}(X;\eta)\ \ \text{for any}\ \eta\in\Btangent_{\manifold\cap\kanifold}(X),
\end{aligned}
\end{equation*}
whereas only one-sided inclusions like~\eqref{eq:cone_oneside} hold in general. Theorem~\ref{the:expressions_McapK} serves as an extension of Theorem~\ref{the:cal_tangentsets} to the intersection of sets.

Applications of Theorems~\ref{the:cal_tangentsets} and~\ref{the:expressions_McapK} to the (structured) low-rank sets---realized with different choices of $\manifold$ and $\kanifold$---are summarized in Table~\ref{tab:tangentsets}. Notably, the Bouligand tangent cones to hierarchical Tucker varieties and to $\boundedrank\cap\hanifold$ with hyperbolic $\hanifold$ are new results, and all the second-order tangent sets reported in the table are also novel, to our knowledge.

Indeed, first- and second-order tangent sets play a crucial role in bridging optimization landscapes. Specifically, as demonstrated in~\myfig\ref{fig:diagram}, the smooth parameterization $(\bmanifold,\phi)$ reformulates the original nonsmooth problem~\eqref{eq:P} as a smooth optimization problem~\eqref{eq:PM}. In light of~\cite{levin2025effect}, Theorem~\ref{the:2to2} exploits the structures of tangent sets to $\bmanifold$ and $\manifold$, and provides a sufficient and necessary condition to characterize when second-order stationary points of~\eqref{eq:PM} map to those of~\eqref{eq:P}, which partially answers a question raised in~\cite[\S 6]{levin2025effect}.

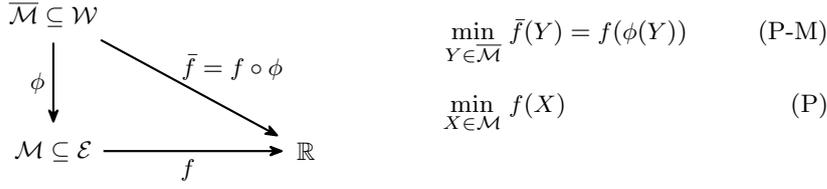
\begin{figure}[htbp]
    \newcommand{\ratio}{0.5}
    \newcommand{\customrectangle}[7]{
            \draw[densely dashed, draw=#6, fill=#7] 
                (#1-#3/2, #2-#4/2) rectangle 
                (#1+#3/2, #2+#4/2); 
            \node at (#1, #2) {#5}; 
        }
    \begin{minipage}{1\textwidth}
    \begin{center}
    \begin{tikzpicture}
        \newcommand{\nodegap}{1.3}
        \tikzset{
        node distance=\nodegap cm,
        post/.style={->,shorten >=2pt,shorten <=2pt,>={Stealth[round]},thick},
        space/.style={
          draw=none, 
          fill=none, 
          inner sep=0pt,
          minimum size=6mm
        },
        hookarrow/.style={{Hooks[left]}->}
        }
        
        \node[space] (Y) {\normalsize $\bmanifold\subseteq\wanifold$};
        \node[space] (X) [below=1.2cm of Y] {\normalsize $\manifold\subseteq \eanifold$};
        \node[space] (R) [right=2.5cm of X] {\normalsize $\mbR$};

        \node[rectangle, text width=6cm, align=left, right=3.5cm of Y, yshift=-0.6cm] (PM) {\normalsize
            \begin{align}
                \min_{Y\in\bmanifold}&\ \bar{f}(Y) = f(\phi(Y))   \tag{P-M} \label{eq:PM}
                \\[2mm]
                \min_{X\in\manifold}&\ f(X)   \tag{P}   \label{eq:P}
            \end{align}
        };

        \draw[post] (Y) to node[midway, left] {\normalsize $\phi$} (X);
        \draw[post] (Y) to node[midway, above, xshift=16pt] {\normalsize $\bar{f}=f\circ \phi$} (R);
        \draw[->,shorten <=4pt,>={Stealth[round]},thick] (X) to node[midway, below] {\normalsize $f$} (R);

    \end{tikzpicture}
    \end{center}
    \end{minipage}
    \caption{Illustration of the smooth parameterization, where $\manifold$ is a possibly nonsmooth set, $\bmanifold$ is a smooth manifold, and $\phi$ is a smooth mapping between  two Euclidean spaces $\wanifold$ and $\eanifold$.}
    \label{fig:diagram}
\end{figure}

The developed framework is applied to low-rank optimization in section~\ref{sec:SecondorderstationaryMr}. Specifically, substituting the characterizations of tangent sets to the low-rank sets, we derive in Proposition~\ref{pro:2order_opt_H} the first- and second-order optimality conditions for the low-rank problem~\eqref{eq:boundedrankopt}. Building on these results, we reveal in Theorems~\ref{the:np} and~\ref{the:np_approx} that verifying second-order optimality for low-rank optimization is NP-hard in general. Nevertheless,  it is still possible to achieve second-order stationarity on $\boundedrank$ at rank-$r$ points, by adopting specific smooth parameterizations; see Propositions~\ref{pro:LR_2to2} and~\ref{pro:desing_2to2}. Therefore, the NP-hardness identified in Theorems~\ref{the:np} and~\ref{the:np_approx} essentially stems from the singular points on the determinantal variety.

Finally, for low-rank semidefinite programs with affine constraints~\eqref{eq:Fproblem}, studied in section~\ref{sec:PSD_BM}, we characterize second-order stationarity (Proposition~\ref{pro:Fopt}) and show that the Burer--Monteiro parameterization $R\mapsto RR^\top$ maps second-order stationary points of the parameterized problem~\eqref{eq:BM} to those of~\eqref{eq:Fproblem} exactly at rank-$r$ points (Proposition~\ref{pro:BM_2to2}). Moreover, fourth-order stationary points of~\eqref{eq:BM} map to second-order stationary points of~\eqref{eq:Fproblem} at every rank (Theorem~\ref{the:F4to2}), which answers a question of Habib, Musat and Boumal~\cite{habib2026sphere}.

\begin{table*}[htbp]
    \setlength{\tabcolsep}{5.5pt}
    \setlength{\extrarowheight}{1ex} 
    \centering
    {
    \caption{Bouligand tangent cones (first-order) and second-order tangent sets to  low-rank sets.}
    \label{tab:tangentsets}
    \vspace{0.1cm}
        \begin{tabular}{@{\hspace{2.5em}}c@{\hspace{0.5em}}clll}
            \toprule
            \multicolumn{2}{c}{ Set } &  {Format} & {First-order} & {Second-order} \\
            \midrule
            {$\manifold_{\phantom{\le r}}$} & \phantom{\eqref{eq:boundedrank}} &  {Assumption~\ref{assu:errorbound}} & {Theorem~\ref{the:cal_tangentsets}} & {Theorem~\ref{the:cal_tangentsets}} \\
            $\boundedrank$ & \eqref{eq:boundedrank}  & matrix & \cite{cason2013iterative,schneider2015Lojaconvergence} & Proposition~\ref{pro:tangenttwo_mr} \\
            $\boundedht$ & \eqref{eq:httensor}  & hierarchical Tucker & Proposition~\ref{pro:tangenttovarieties}  & Proposition~\ref{pro:tangenttovarieties} \\
            $\boundedtucker$ & \eqref{eq:tctensor} & Tucker & \cite{gao2025lowranktucker}  & Proposition~\ref{pro:tangenttovarieties}  \\
            $\boundedtt$ & \eqref{eq:tttensor} & tensor train & \cite{kutschan2018tangentTT}  & Proposition~\ref{pro:tangenttovarieties} \\
            $\sym_{\le r}(n)$ & \eqref{eq:define_Srn} & symmetric matrix & \cite{li2020jotaspectral} & Proposition~\ref{pro:tangenttoS} \\
            $\boundedranksdp(n)$ & \eqref{eq:boundedranksdp} & PSD matrix & \cite{levin2025effect} & Proposition~\ref{pro:tangenttoS+r} \\
            \midrule
            \multicolumn{2}{c}{Intersection of sets} & {Structured set} & {First-order} & {Second-order} \\
            \midrule
            \multicolumn{2}{c}{{\ \ \ $\manifold\cap\kanifold$}} &  Assumption~\ref{assu:MintersecK} & Theorem~\ref{the:expressions_McapK} & Theorem~\ref{the:expressions_McapK} \\
            \multicolumn{2}{c}{{$\boundedrank\cap\hanifold$}} & $\hanifold=\{X\in\mbRmn\mid\aanifold(X)=b\}$ & \cite{li2023normalboundedaffine} & Appendix~\ref{app:Haffine} \\
            \multicolumn{2}{c}{{$\boundedrank\cap\hanifold$}} & $\hanifold$ is orthogonally invariant~\eqref{eq:orth_invariant_h} & \cite{yang2025spacedecouple} & Appendix~\ref{app:Horthogonalinvariant} \\
            \multicolumn{2}{c}{{$\boundedrank\cap\hanifold$}} & $\hanifold$ is hyperbolic~\eqref{eq:hyperbolicmatrixset} & Appendix~\ref{app:hyperbolic} & Appendix~\ref{app:hyperbolic} \\
            \multicolumn{2}{c}{\!\!\!\!\!$\sdp_{\le r}(n)\cap\uanifold$} & $\uanifold=\{X\in\sym(n)\mid\|X\|^2_\frob=1\}$ & Appendix~\ref{app:SU} & Appendix~\ref{app:SU} \\      
            \multicolumn{2}{c}{\!\!\!\!\!$\boundedranksdp(n)\cap\uanifold$} & $\uanifold=\{X\in\sym(n)\mid\aanifold(X)=b\}$ & \cite{levin2025effect} & Appendix~\ref{app:S+U} \\
            \bottomrule	
        \end{tabular}
    }
\end{table*}

\medskip

\noindent\textbf{II. Geometry of $\bm \graph\normal_{\boundedrank}$\,and bilevel program.} We investigate, in section~\ref{sec:Geometryofgraph}, the geometry of the graph of the Mordukhovich normal cone mapping associated with the determinantal variety $\boundedrank$, along the right side of \myfig\ref{fig:twocontributions}. In detail, the Bouligand tangent cone to $\graph\Lnormal_{\boundedrank}$ is characterized in Theorem~\ref{the:Btangent_graph}, and then we take the polar operation to obtain the Fr\'echet normal cone in Corollaries~\ref{cor:Fnormal_graph_1}-\ref{cor:Fnormal_graph_2}. Consequently, in Theorem~\ref{the:Lnormal_graph}, the Mordukhovich normal cone to $\graph\Lnormal_{\boundedrank}$ is identified as the outer limit of the developed Fr\'echet normal cone, which induces the calculation of the coderivative of the Mordukhovich normal cone mapping $X\mapsto\normal_{\boundedrank}(X)$.

The geometry of $\graph\normal_{\boundedrank}$ is important in analyzing bilevel programs with low-rank constraints at the lower level. Specifically, relaxing the bilevel problem~\eqref{eq:lowrank_BiO}, we propose a formulation~\eqref{eq:lowrank_BiO_M} equivalent to~\eqref{eq:lowrank_BiO_M1}, which involves $\graph\normal_{\boundedrank}$ in the constraints. Consequently, Proposition~\ref{pro:M-stationary} applies the results of Theorem~\ref{the:Lnormal_graph} to give an optimality condition for problem~\eqref{eq:lowrank_BiO_M}.

\subsection{Organization}
Section~\ref{sec:prelimi} reviews some notation and preliminaries from variational analysis. In section~\ref{sec:2tangentset}, we present a rule to compute the first- and second-order tangent sets to a general set, and then apply the results to  $\boundedrank$. Section~\ref{sec:extent_structured} extends the analysis to  structured low-rank sets, including sets of low-rank matrices, tensors, symmetric matrices, PSD matrices, and intersections of sets; details are organized in appendices. Section~\ref{sec:tangent_bridge_land} unveils that the concept of tangent sets  bridges optimization landscapes under smooth parameterization. The developed framework is then applied to low-rank optimization in section~\ref{sec:SecondorderstationaryMr} and to low-rank semidefinite programs with the Burer--Monteiro parameterization in section~\ref{sec:PSD_BM}. In section~\ref{sec:Geometryofgraph}, we investigate the geometry of the graph of the normal cone to the determinantal variety, which is applied to give an optimality condition in low-rank bilevel programs; see section~\ref{sec:LRBP}. Finally, we draw the conclusion in section~\ref{sec:conclusion}.

\section{Notation and preliminaries}\label{sec:prelimi}
This section summarizes the notation and reviews some preliminaries in variational analysis. For further background and references, see \cite{bonnans2000perturbationanalysis,rockafellar2009variationalanalysis}.

\subsection{Notation} 
Let $\stiefel(n, p)=\{X \in \mbR^{n \times p}\mid X^{\top} X=I_p\}$ be {the} Stiefel manifold, $\orth(n)=\stiefel(n,n)$ be the orthogonal group, $\lowrank=\{X\in \mbR^{m\times n}\mid \rank(X)=s\}$ be the set of fixed-rank matrices, and $\sksym(n)=\{\Omega\in\mbR^{n\times n}\mid \Omega^\top=-\Omega\}$ be the set of skew-symmetric matrices. We use $\Diag(x)$ to denote the diagonal matrix with  entries given by the vector $x$, and $\ddiag(X)$ to denote the vector extracting the diagonal from a  matrix $X$. Given a smooth manifold $\xanifold$, $\tangent_{\xanifold}(X)$ denotes the tangent space at $X\in\xanifold$. Given a mapping $F:\xanifold_1 \to \xanifold_2$ between two manifolds, $\diff F_{X}: \tangent_{\xanifold_1}(X) \to \tangent_{\xanifold_2}(F(X))$ denotes the differential of $F$ at $X\in\xanifold_1$. The standard inner product in an Euclidean space is  $\innerp{X_1,X_2}:=\trace(X_1^\top X_2)$. The distance from $Y$ to $\xanifold$ is defined by $\dist(Y,\xanifold)=\inf_{X\in\xanifold}\|X-Y\|$. Let $\projection_\xanifold $ denote the projection onto the set $\xanifold$.  Given a matrix $V\in\stiefel(n,r)$, $V_{\bot}$ is an orthogonal complement of it in the sense of $[V\ V_\bot]\in\orth(n)$. Throughout the paper, whenever the rank of a matrix $X\in\mathbb{R}^{m\times n}$ is explicitly specified---e.g., $\rank(X)=s$ or $X\in\lowrank$---and an SVD of $X$ is invoked, we use the following conventions: for the \emph{compact SVD}, $X= U\varSigma V^\top$ with $U\in\stiefel(m,s)$, $\varSigma\in\mbR^{s\times s}$, and $V\in\stiefel(n,s)$; for the \emph{full SVD}, $X=\bar{U}[\bar{\varSigma}\ 0]\bar{V}^\top$ with $\bar U\in\orth(m)$, $\bar{\varSigma}\in\mbR^{m\times m}$, and $\bar V\in\orth(n)$ (when $m\le n$). Based on the compact SVD, the Moore--Penrose inverse of $X$ is $X^\dagger=V\varSigma^{-1} U^\top$. Additionally, we use $\sigma_i(X)$ to denote the $i$-th {largest} singular value of $X\in\mbRmn$; when $X$ is symmetric, i.e., $X\in\sym(n)$,  we use $\lambda_i(X)$ to denote its $i$-th {largest} eigenvalue. For convenience, we set $\sigma_i \equiv 0$ for all $i > \min\{m, n\}$ and $\lambda_i \equiv 0$ for all $i>n$. Matrices or vectors are denoted by usual roman letters (e.g., $X$, $\eta$), while higher-order tensors are written in boldface (e.g., $\tensX$, $\tenseta$).

\subsection{Background in variational analysis}\label{sec:background_vari}
Let $\xanifold$ be a nonempty and closed subset of a finite-dimensional Euclidean space $\eanifold$. The \emph{Bouligand tangent cone} to $\xanifold$ at a point $X\in\xanifold$ is
\begin{equation}\label{eq:tangentcone}
    \begin{aligned}
        \tangent_\xanifold(X) = \left\{ \eta\in\eanifold\mid \exists\, t_i\downarrow 0,\,\text{such that}\,\dist(X+t_i\eta,\xanifold)=o(t_i)\right\}.
    \end{aligned}
\end{equation}
The \emph{second-order tangent set} to $\xanifold$ at $X$ in the direction $\eta\in\tangent_\xanifold(X)$ is defined by
\begin{equation}\label{eq:second_tangentcone}
    \begin{aligned}
        \tangenttwo_\xanifold(X;\eta) = \{ \zeta\in\eanifold\mid \exists\, t_i\downarrow 0,\,\text{such that}\, \dist(X+t_i\eta+\frac{1}{2}t_i^2\zeta,\xanifold)=o(t_i^2)\}.      
    \end{aligned}
\end{equation}
In fact, the tangent cone $\tangent_{\xanifold}(X)$ provides a linear approximation of the set $\xanifold$ around $X$, whereas the second-order tangent set further captures curvature information and thus facilitates a more precise local approximation, as illustrated by the example in \myfig\ref{fig:tangent_comparison}.
\begin{figure}[htbp]
    \centering
    \begin{tikzpicture}[scale=3, >=stealth]
        \draw[->, gray!50] (-1.4, 0) -- (1.4, 0) node[above, black] {};
        \draw[->, gray!50] (0, -0.2) -- (0, 0.6) node[left, black] {};
        \node[below left] at (0,0) {$X$};

        \draw[tan1color, thick, dashed] (-1.1, 0) -- (1.1, 0);
        
        \node[tan1color, below, font=\small, align=left] at (1.1, -0.05) {
            {First-order approximation}
        };
        
        \draw[tan2color, ultra thick, dashed, domain=-1.05:1.05, samples=100, opacity=1] 
            plot (\x, {0.4 * (\x*\x)});
            
        \node[tan2color, above, font=\small, align=right] at (-0.8, {0.5*(-0.6)*(-0.6) + 0.3}) {
           {Second-order approximation}
        };

        \draw[setcolor, thick, domain=-1.2:0.85, samples=100] 
            plot (\x, {0.4 * (\x*\x + \x*\x*\x)});
            
        \node[setcolor, left, font=\small] at (0.7, {0.5*(0.6*0.6 + 0.6*0.6*0.6)+0.1}) {
            $\mathcal{X}$
        };

    \fill[black] (0,0) circle (0.6pt); 

    \end{tikzpicture}
    \caption{Illustration of approximations to the set $\mathcal{X} = \{(x,y) \in \mathbb{R}^2 \mid y = x^2 + x^3\}$ around $X=(0,0)$. The blue dashed line corresponds to the first-order path $\gamma_1(t)=X+t\eta$ with $\eta=(1,0)\in \mathrm{T}_{\mathcal{X}}(X) = \mathbb{R} \times \{0\}$. The red parabola represents the second-order path $\gamma_2(t)=X + t\eta + \frac{1}{2}t^2\zeta$ with $\zeta=(0,2)\in \mathrm{T}^2_{\mathcal{X}}(X;\eta) = \{(\zeta_1, 2) \mid \zeta_1 \in \mathbb{R}\}$, capturing the curvature of $\mathcal{X}$ and yielding a quadratic approximation of $\xanifold$.}
    \label{fig:tangent_comparison}
\end{figure}
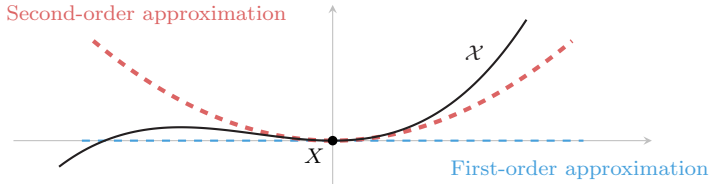

Taking the polar operation on $\tangent_\xanifold(X)$ yields the \emph{Fr\'echet normal cone} (also known as the regular normal cone), 
\begin{equation*}
\Fnormal_\xanifold(X):=\kh{\tangent_\xanifold(X)}^\circ = \hkh{ Y\in\eanifold \mid \langle Y, \eta\rangle \le 0, \ \text{for all}\ \eta \in \Btangent_\xanifold(X)}.
\end{equation*}
The \emph{Mordukhovich normal cone} can be obtained as the outer limit of $\Fnormal_{\xanifold}$, i.e.,
\begin{equation}\label{eq:liminting_normalcone}
\begin{aligned}
    \Lnormal_\xanifold(X) := \{Y\in\eanifold\mid\exists\,X_k\rightarrow X,\  Y_k\rightarrow Y\ \text{such that}\ X_k\in\xanifold,\ Y_k\in\Fnormal_{\xanifold}(X_k)\},
\end{aligned}
\end{equation}
which is also referred to as the limiting normal cone. When $\xanifold$ is a smooth manifold, the two normal cones coincide with the normal space.

We revisit the basics for tangent and normal sets to the union (or intersection) of  sets; see~\cite{rockafellar2009variationalanalysis,lee2012manifolds}. Given $X\in\bigcup_{i=1}^d{\xanifold_i}$ with finitely many closed sets $\xanifold_1,\dots,\xanifold_d$, it holds that
\begin{equation}\label{eq:tangentcupXi}
    \begin{aligned}
        \tangent_{\bigcup\nolimits_{i=1}^d{\xanifold_i}}(X) &=  \bigcup\nolimits_{i=1}^d \Btangent_{\xanifold_i}(X),
        \\
        \tangenttwo_{\bigcup\nolimits_{i=1}^d{\xanifold_i}}(X;\eta) &=  \bigcup\nolimits_{i=1}^d \tangenttwo_{\xanifold_i}(X;\eta)\ \ \text{for any}\ \eta\in\tangent_{\bigcup\nolimits_{i=1}^d{\xanifold_i}}(X),
    \end{aligned}
\end{equation}
where we denote $\Btangent_{\xanifold_i}(X)=\emptyset$ and $\tangenttwo_{\xanifold_i}(X;\eta)=\emptyset$ if $X\notin\xanifold_i$. Moreover, given $X\in \xanifold_1\cap \xanifold_2$, {it has}
\begin{equation}\label{eq:cone_oneside}
    \Btangent_{\xanifold_1\cap \xanifold_2}(X) \subseteq  \Btangent_{\xanifold_1}(X) \cap \Btangent_{\xanifold_2}(X)\quad \text{and} \quad\ 
    \normal_{\xanifold_1\cap \xanifold_2}(X) \supseteq \normal_{\xanifold_1}(X) + \normal_{\xanifold_2}(X).
\end{equation}
Specifically, if both $\xanifold_1$ and $\xanifold_2$ are smooth manifolds and intersect \emph{transversally}, i.e., for any $X\in\xanifold_1\cap\xanifold_2$, $\tangent_{\xanifold_1}(X) + \tangent_{\xanifold_2}(X) = \eanifold$,  or equivalently, $\ \normal_{\xanifold_1}(X)\cap \normal_{\xanifold_2}(X) = \{0\}$, then $\xanifold_1\cap \xanifold_2$ is also a smooth manifold with
\begin{equation}\label{eq:cone_tranverse}
    \Btangent_{\xanifold_1\cap \xanifold_2}(X) = \Btangent_{\xanifold_1}(X) \cap \Btangent_{\xanifold_2}(X)\quad \text{and} \quad\ 
    \normal_{\xanifold_1\cap \xanifold_2}(X) = \normal_{\xanifold_1}(X) + \normal_{\xanifold_2}(X).
\end{equation}

The attention then turns to directional derivatives of a mapping $h:\eanifold_1\to\eanifold_2$. Specifically, we say that $h$ is directionally differentiable at $X\in\eanifold_1$ in a direction $\eta\in\eanifold_1$ if the following limit exists,
\begin{equation*}
    h^\prime(X;\eta):=\lim_{t\downarrow 0}\frac{h(X+t\eta)-h(X)}{t}.
\end{equation*}
If $h$ is directionally differentiable at $X$ in any direction $\eta\in\eanifold_1$, then $h$ is termed \emph{directionally differentiable} at $X$. Additionally, its parabolic second-order directional derivative is defined by
\begin{equation*}
    h^{\prime\prime}(X;\eta,\zeta):=\lim_{t\downarrow 0}\frac{h(X+t\eta+\frac{1}{2}t^2\zeta)-h(X)-th^\prime(X;\eta)}{\frac{1}{2}t^2},
\end{equation*}
provided that the above limit exists. When $h$ is locally Lipschitz around $X$ and its first- and second-order directional derivatives exist, the following two expansions hold, respectively:
\begin{align}
    &h\kh{X+t\eta+o(t)} = h(X) + th^\prime(X;\eta)+o(t),    \label{eq:first_order_exp}
    \\
    &h(X+t\eta+\frac{1}{2}t^2\zeta+o(t^2)) = h(X)+th^\prime(X;\eta)+ \frac{1}{2}t^2h^{\prime\prime}(X;\eta,\zeta) + o(t^2). \label{eq:second_order_exp}
\end{align}

\subsection{Variational geometry of low-rank sets}
The first-order geometry of the low-rank sets is well developed; see~\cite{luke2013Mordukhovichcone,vandereycken2013lowrankcompletion,schneider2015Lojaconvergence}. As a fixed-rank layer of {$\boundedrank=\bigcup_{s=0}^{s=r}\lowrank$}, $\lowrank$ is indeed an analytic manifold. Given $X\in\lowrank$ with the singular value decomposition $X=U\varSigma V^\top$, the tangent and normal spaces are outlined below, 
\begin{align}
     & \tangent_{\lowrank}(X) = \hkh{[U\ U_{\bot}]\left[ \begin{matrix}
        	W_1&		W_2\\
        	W_3&		0\\
        \end{matrix} \right][V\ V_{\bot}]^\top \left|\,\begin{array}{l}
        W_1\in\mbR^{s\times s},\, W_2\in\mbR^{s\times (n-s)},
        \\
        W_3\in\mbR^{(m-s)\times s}
        \end{array}\right.
        },  \label{eq:Tcone_lowrank}
    \\[1mm]
    &\normal_{\lowrank}(X) = \hkh{U_\bot JV_\bot^\top\mid J\in\mbR^{(m-s)\times (n-s)}}. \label{eq:Ncone_lowrank}
\end{align}
The tangent and normal cones to $\boundedrank$ at $X\in\lowrank$ are given as follows~\cite{cason2013iterative,schneider2015Lojaconvergence,hosseini2019MordukhovichClarke},
\begin{align}
    &\Btangent_{\boundedrank}(X) = \tangent_{\lowrank}(X) + \hkh{R\in\normal_{\lowrank}(X)\mid\rank(R)\le r-s}, \label{eq:Btangent_cone_boundedrank}
    \\
    &\Fnormal_{\boundedrank}(X) = \begin{cases}
        \normal_{\lowrank}(X),\,\,\quad \ \mathrm{if}\,\,s=r,\\
    \left\{ 0 \right\}, \,\,\quad\quad\quad \ \,  \mathrm{if}\,\,s<r,
    \end{cases} \label{eq:Fnormal_cone_boundedrank}
    \\
    &\Lnormal_{\boundedrank}(X) = \hkh{R\in\normal_{\lowrank}(X)\mid\rank(R)\le \min\{m,n\}-r}.
    \label{eq:Lnormal_cone_boundedrank}
\end{align}
{Let $P_W := WW^\top$ and $P_{W_\bot}:=I-P_W$ for any $W\in\stiefel(n,s)$.} The projection of $E\in\mbR^{m\times n}$ onto $\tangent_{\boundedrank}(X)$ is given by
\begin{equation*}
    \projection_{\tangent_{\boundedrank}(X)}(E) = P_UEP_V + P_UEP_{V_\bot} + P_{U_\bot}EP_V + \projection_{\manifold_{r-s}} \kh{P_{U_\bot}EP_{V_\bot}}.
\end{equation*}

\section{First- and second-order tangent sets}\label{sec:2tangentset}
As introduced in section~\ref{sec:relatedwork}, the first- and second-order tangent sets play a key role in developing optimality conditions for constrained problems \cite{bonnans2000perturbationanalysis,chen2019exactSOC}. Specifically, regarding the determinantal variety $\boundedrank$, its first-order geometry is well understood~\cite{luke2013Mordukhovichcone,vandereycken2013lowrankcompletion,schneider2015Lojaconvergence}, while the second-order counterpart remains unclear. Additionally, we note that the rank function defining $\boundedrank$ in \eqref{eq:boundedrank} is discontinuous, through which implementing the second-order variational analysis is not straightforward. To circumvent the challenges, we turn to the following perspective,
\begin{equation}\label{eq:sigmarboundedrank}
    \boundedrank=\{X\in\mbRmn\mid\sigma_{r+1}(X)=0\}.
\end{equation}
This serves as another characterization of $\boundedrank$, since the singular value mappings are non-negative and are ordered in a non-increasing fashion---$\sigma_1(X)\ge\sigma_2(X)\ge\cdots\ge\sigma_{\min\{m,n\}}(X)\ge 0$---implying that $\sigma_{r+1}(X)=0$ {if and only if} $\rank(X)\le r$.

Although the mapping $\sigmar$ still exhibits both nonsmoothness and nonconvexity, we notice that
\begin{itemize}
    \item[1)] it is locally Lipschitz, since $|\sigmar(X)-\sigmar(X+\Delta)|\le \norm{\Delta}_2$ for any $\Delta\in\mbRmn$ by Weyl's inequality \cite{weyl1912inequality}; and
    \item[2)] according to the truncated SVD, $\sigmar$ can control the distance of points to the set $\boundedrank$, i.e., $\dist(X,\boundedrank)\le (\min\{m,n\}-r)^{1/2} \sigmar(X)$ for any $X\in\mbRmn$.
\end{itemize}
Motivated by the two observations, we distill the ideas into a rule for computing tangent sets to a general set, as presented in section~\ref{sec:generaltengentset}. Subsequently, the developed results, together with several properties of singular value mappings reviewed in section~\ref{sec:directionalderisigma}, are applied to the determinantal variety in section~\ref{sec:tangenttoMr}.

\subsection{Tangent sets to a general set}\label{sec:generaltengentset}
Given an Euclidean space $\mbR^q$ and two mappings $c_1:\mbR^{q}\to\mbR^{n_1}$ and $c_2:\mbR^{q}\to\mbR^{n_2}$, define the set $\manifold\subseteq\mbR^q$ by
\[\manifold:=\{\tilde{X}\in\mbR^{q}\mid\ c_1(\tilde{X})=0,\ c_2(\tilde{X})\le 0\},\]
where the relations ``$=$" and ``$\le$" are understood component-wise. Around a point $X\in\manifold$, we introduce a regularity assumption as follows.

\begin{assumption}\label{assu:errorbound}
There exists a neighborhood $\neighbor$ of $X\in\manifold$ and a constant $\rho>0$ satisfying the following two conditions.
\begin{itemize}
\item[\emph{(i)}] Both $c_1$ and $c_2$ are Lipschitz in $\neighbor$;
\item[\emph{(ii)}] For all $\tilde X\in \neighbor$, it holds that $\dist(\tilde{X},\manifold)\le\rho\|(c_1(\tilde{X}),\,[c_2(\tilde{X})]_+)\|_2$, where $[\,\cdot\,]_+$ is a component-wise operation and maps each entry $e$ to $\max\{e,0\}$.
\end{itemize}
\end{assumption}

Under the above assumption, we can bridge the tangent sets to $\manifold$ with the directional derivatives of $c_1$ and $c_2$.
\begin{theorem}[Computation of tangent sets]\label{the:cal_tangentsets}
Suppose that $\manifold$ satisfies Assumption~\ref{assu:errorbound} at $X\in\manifold$, and define the index set $I_0(X):=\{j\in\{1,\dots,n_2\}\mid c_2(X)_j=0\}$.
\begin{enumerate}
\item[\emph{(i)}] \emph{(First-order)} If $c_1$ and $c_2$ are directionally differentiable at $X$, then
\[
\tangent_{\manifold}(X)=\big\{\eta\in\mbR^q\,|\ c_1^\prime(X;\eta)=0,\ c_2^\prime(X;\eta)_j\le 0\ \ \text{for all}\ j\in I_0(X)\big\}.
\]

\item[\emph{(ii)}] \emph{(Second-order)} If, in addition, $c_1$ and $c_2$ admit parabolic second-order directional derivatives at $X$ for every direction pair $(\eta,\zeta)$, then for any $\eta\in\tangent_{\manifold}(X)$,
\[
\tangenttwo_{\manifold}(X;\eta)
=\Big\{\zeta\in\mbR^q\,|\ c_1^{\prime\prime}(X;\eta,\zeta)=0,\ c_{2}^{\prime\prime}(X;\eta,\zeta)_j\le 0\ \ \text{for all}\ j\in I_1(X;\eta)\Big\},
\]
where $I_1(X;\eta):=\{j\in I_0(X)\,|\ c_2^{\prime}(X;\eta)_j=0\}$.
\end{enumerate}
\end{theorem}
\begin{proof}
    (i) If $\eta\in\tangent_\manifold(X)$, there exist $t_i\downarrow 0$ and $\eta_i\to\eta$ such that $X+t_i\eta_i\in\manifold$. Therefore, we have $0=c_1(X+t_i\eta_i)-c_1(X)=t_ic_1^\prime(X;\eta)+o(t_i)$ and $0\ge c_2(X+t_i\eta_i)_j-c_2(X)_j=t_ic_2^\prime(X;\eta)_j+o(t_i)$ for $j\in I_0(X)$, by the local Lipschitzness of $c_1,c_2$. Dividing the (in)equalities by $t_i$ and letting $t_i\downarrow 0$ yield $c_1^\prime(X;\eta)=0$ and $c_2^\prime(X;\eta)_j\le 0$, respectively.
    
Conversely, given a direction $\eta$ with $c_1^\prime(X;\eta)=0$, which implies that $c_1(X+t\eta)-c_1(X)=o(t)$ and thus $c_1(X+t\eta)=o(t)$. For $j\in I_0(X)$, if $c_2^{\prime}(X;\eta)_j<0$, it holds that $c_2(X+t\eta)_j<0$ for all sufficiently small $t>0$. If $c_2^{\prime}(X;\eta)_j=0$, we have $c_2(X+t\eta)_j=c_2(X+t\eta)_j-c_2(X)_j=o(t)$. Consequently, we have $\dist(X+t\eta,\manifold)=o(t)$ since $\dist(X+t\eta,\manifold)\le\rho\|(c_1(X+t\eta),\max\{c_2(X+t\eta),0\})\|$. Therefore, it is concluded that $\eta\in\tangent_\manifold(X)$ by definition of the Bouligand tangent cone.

(ii) If $\zeta\in\tangenttwo_\manifold(X;\eta)$, there exist $t_i\downarrow 0$ and $\zeta_i\to\zeta$ such that $X+t_i\eta+\frac{1}{2}t_i^2\zeta_i\in\manifold$, which reveals that $0=c_1(X+t_i\eta+\frac{1}{2}t_i^2\zeta_i)-c_1(X)=\frac{1}{2}t_i^2c_1^{\prime\prime}(X;\eta,\zeta)+o(t_i^2)$, and $0\ge c_2(X+t_i\eta+\frac{1}{2}t_i^2\zeta_i)_j-c_2(X)_j=\frac{1}{2}t_i^2c_2^{\prime\prime}(X;\eta,\zeta)_j+o(t_i^2)$ for $j\in I_1(X;\eta)$, by the local Lipschitzness of $c_1,c_2$. Hence, we divide the (in)equalities by $t_i^2$ and let $t_i\downarrow 0$ to obtain $c_1^{\prime\prime}(X;\eta,\zeta)=0$ and $c_2^{\prime\prime}(X;\eta,\zeta)_j\le 0$.

Conversely, if $c_1^{\prime\prime}(X;\eta,\zeta)=0$, it holds that $c_1(X+t\eta+\frac{1}{2}t^2\zeta)-c_1(X)=o(t^2)$. Additionally, for $j\in I_1(X;\eta)$, if $c_2^{\prime\prime}(X;\eta,\zeta)_j<0$, it holds that $c_2(X+t\eta+\frac{1}{2}t^2\zeta)_j<0$ for all sufficiently small $t>0$; if $c_2^{\prime\prime}(X;\eta,\zeta)_j=0$, it holds that $c_2(X+t\eta+\frac{1}{2}t^2\zeta)_j=o(t^2)$. Therefore, we derive that $\dist(X+t\eta+\frac{1}{2}t^2\zeta,\manifold)=o(t^2)$ since $\dist(X+t\eta+\frac{1}{2}t^2\zeta,\manifold)\le\rho\|(c_1(X+t\eta+\frac{1}{2}t^2\zeta),\max\{o(t^2),0\})\|$. By definition of the second-order tangent set, we have $\zeta\in\tangenttwo_\manifold(X;\eta)$.
\end{proof}

\begin{remark}
    In fact, the inequality $\dist(\tilde{X},\manifold)\le\rho\|(c_1(\tilde{X}),\,[c_2(\tilde{X})]_+)\|_2$ in Assumption~\ref{assu:errorbound} is the so-called \emph{error bound property} \cite{luo1993errorbound,luo1994errorboundapplication}, which has garnered wide interest in optimization and variational analysis; see \cite{luo1993errorbound,luo1994errorboundapplication,aze2003errorboundsurvey} and references therein for more details.
Moreover, the inclusions ``$\subseteq$'' in Theorem~\ref{the:cal_tangentsets}(i) and (ii) always hold, i.e., the zero sets of the directional derivatives are natural candidates for the tangent sets, and the error bound is only used for the reverse inclusions.
\end{remark}

\begin{remark}
    Under the error bound property, first- and second-order tangent sets to sets $\{X\in\mbR^q\mid h(X)\in\canifold\}$ with a smooth mapping $h$ and a closed set $\canifold$ have been characterized via derivatives in \cite{henrion2005calmness} and \cite{gfrerer2022secondorder}, respectively; Theorem~\ref{the:cal_tangentsets} extends this principle to Lipschitz and directionally differentiable mappings such as $\sigma_{r+1}$.
\end{remark}

\subsection{Directional derivatives of singular values}\label{sec:directionalderisigma}
We aim at applying Theorem~\ref{the:cal_tangentsets} to the determinantal variety $\boundedrank=\{X\in\mbRmn\mid\sigma_{r+1}(X)=0\}$. By \cite[Theorem~3.1]{zhang2013secondordersingular}, $\sigma_{r+1}$ is directionally differentiable and admits parabolic second-order directional derivatives; together with its Lipschitz continuity, $\sigma_{r+1}$ satisfies the requirements in Theorem~\ref{the:cal_tangentsets}. To this end, we review the first- and second-order directional derivatives of singular-value mappings, as developed in \cite{lewis2005nonsmoothPartI,zhang2013secondordersingular,ding2014introductiontomatrixCP}. These will be instrumental in identifying the zeros of $\sigma_{r+1}^\prime$ and $\sigma_{r+1}^{\prime\prime}$, which in turn characterize the tangent sets to $\boundedrank$.

We follow the notation introduced in \cite{zhang2013secondordersingular}, which, although somewhat involved, is essential for the analysis. Without loss of generality, assume that $m\le n$. Given $X\in\mbRmn$ with the full SVD,
\begin{equation}\label{eq:fullSVD}
    X=\bar{U}[\bar{\varSigma}\ 0]\bar{V}^\top,
\end{equation}
where $\ddiag(\bar{\varSigma})=(\sigma_1(X),\sigma_2(X),\ldots,\sigma_m(X))$. The set of  $(\bU,\bV)$ is defined by
\begin{equation*}
    \orth^{m,n}(X):=\{(\bU^\prime,\bV^\prime)\in\orth(m)\times \orth(n)\mid X=\bU^\prime[\bar{\varSigma}\ 0]\bV^{\prime\top}\}.
\end{equation*}
Similarly, if $X \in \sym(m)$, we define a set according to the spectral decomposition,
\begin{equation*}
    \orth^{n}(X):=\{\bU^\prime\in \orth(n)\mid X=\bU^\prime \bar{\varLambda} \bU^{\prime\top}\},
\end{equation*}
where $\bar{\varLambda}=\Diag(\lambda_1(X),\lambda_2(X),\ldots,\lambda_m(X))$ collects the eigenvalues of $X$. Let $I \subseteq \{1,2,\ldots,m\}$ and $J \subseteq \{1,2,\ldots,n\}$ be index sets. For a matrix $Z$, we denote by $Z_{IJ}$ the submatrix obtained by selecting
the rows indexed by $I$ and the columns indexed by $J$.
Likewise, $Z_{I}$ denotes the submatrix of $Z$ formed by the columns indexed by $I$.

The index set $\{1,2,\ldots,n\}$ is divided into three subsets: $\alpha =\{i \mid \sigma_i(X)>0,1 \leq i \leq m\}$, $\beta =\{i \mid \sigma_i(X)=0,1 \leq i \leq m\}$, and $\beta_0 =\{m+1, \ldots, n\}$. Suppose that $X$ admits $t+1$ distinct singular values, $\mu_1>\mu_2>\ldots>\mu_t>\mu_{t+1}=0$, based on which we categorize the index set $\alpha$ into $\alpha_k=\left\{i\mid \sigma_i(X)=\mu_k, 1 \leq i \leq m\right\}$ for $k=1, \ldots, t$. Denote $\alpha_{t+1}=\beta$ and $\widehat{\beta}=\beta\cup\beta_0$. We then partition the matrices as $\bU=[\bU_{\alpha_1}\ \bU_{\alpha_2}\ \cdots\ \bU_{\alpha_{t+1}}]$ and $\bV=[\bV_{\alpha_1}\ \bV_{\alpha_2}\ \cdots\ \bV_{\alpha_t}\ \bV_{\widehat{\beta}}]$, where $\bU_{\alpha_k}\in\mbR^{m\times|\alpha_k|}$ for $k=1,2,\ldots,t+1$, $\bV_{\alpha_k}\in\mbR^{n\times|\alpha_k|}$ for $k=1,2,\ldots,t$, and $\bV_{\widehat{\beta}}\in\mbR^{n\times|\widehat{\beta}|}$.

Given a direction $\eta\in\mbRmn$, we denote $\bareta=\bU^\top\eta\bV$. For $k=1,2,\ldots,t$, let the eigenvalues of $\frac{1}{2}(\bareta_{\alpha_k \alpha_k}+\bareta_{\alpha_k \alpha_k}^\top)$ be $\lambda_i^k$ ($i=1,2,\ldots,|\alpha_k|$) and the distinct eigenvalues be $\theta_1^k>\theta_2^k>\ldots>\theta_{N_k}^k$, which induce the partitions $\beta_j^k:=\{i\mid \lambda_i^k=\theta_j^k, i=1, \ldots,|\alpha_k|\}$ for $j=1,2,\ldots,N_k$. In parallel, letting the distinct singular values of $\bareta_{\beta \widehat{\beta}}$ be $\theta_1^{t+1}>\theta_2^{t+1}>\ldots>\theta_{N_{t+1}+1}^{t+1}=0$, we define $\beta_j^{t+1}=\{i\mid\sigma_i(\bareta_{\beta \widehat{\beta}})=\theta_j^{t+1}, i=1, \ldots,|\beta|\}$ for $j=1,2,\ldots,N_{t+1}+1$. \myfig\ref{fig:illustration_index} illustrates the partitions of the index sets. Moreover, we need the following mappings between index sets,
\begin{equation}\label{eq:indexmappings}
\begin{aligned}
q_a &:\{1, \ldots, m\} \rightarrow\{1, \ldots, t+1\},\ q_a(i)=k, \text { if } i \in \alpha_k, \\
l &:\{1, \ldots, m\} \rightarrow \mathbb{N},\ l(i)=i-\kappa_{q_a(i)-1}, \\
q_b &:\{1, \ldots, m\} \rightarrow \mathbb{N},\ q_b(i)=e,\ \text{if}\ l(i) \in \beta_e^{q_a(i)},\\
l^{\prime} &:\{1, \ldots, m\} \rightarrow \mathbb{N},\ l^{\prime}(i)=l(i)-\kappa_{q_b(i)-1}^{(q_a(i))},
\end{aligned}  
\end{equation}
where $\kappa_i:=\sum_{j=1}^i|\alpha_j|$ and $\kappa_i^{(k)}:=\sum_{j=1}^i|\beta_j^k|$.

\begin{figure}[htbp]
\begin{minipage}{\textwidth}
\centering
\begin{tikzpicture}[thick, >=Stealth, font=\small]

    \def\layerDist{1.7}      
    \def\sPos{5}             
    \def\mPos{7.5}           
    \def\nPos{10}            
    \def\gap{0.08}           
    
    \def\akStart{2.8}
    \def\akEnd{3.6}
    \def\akCenter{3.2}

    
    \draw[colNormal] (0, 0) -- (\akStart, 0);
    
    \draw[colFocusAlpha, line width=1.2pt] (\akStart, 0) -- (\akEnd, 0);
    
    \draw[colNormal] (\akEnd, 0) -- (\sPos, 0);
    
    \draw[colFocusBeta, line width=1.2pt] (\sPos, 0) -- (\mPos, 0);
    
    \draw[colNormal] (\mPos, 0) -- (\nPos, 0);

    \foreach \x in {0, \nPos} {
        \draw (\x, 0) -- (\x, 0.15); 
    }
    
    \node[below] at (0, -0.1) {1};
    \node[below] at (\sPos, -0.1) {$s$};
    \node[below] at (\mPos, -0.1) {$m$};
    \node[below] at (\nPos, -0.1) {$n$};

    \node at (0.5, 0.4) {$\alpha_1$};
    \node at (1.65, 0.4) {$\cdots$};
    \node[colFocusAlpha] at (\akCenter, 0.4) {$\alpha_k$}; 
    \node at (4.0, 0.4) {$\cdots$};
    \node at (4.6, 0.4) {$\alpha_t$};
    
    \draw[colFocusAlpha] (\akStart, 0) -- (\akStart, 0.15);
    \draw[colFocusAlpha] (\akEnd, 0) -- (\akEnd, 0.15);

    \draw[colFocusBeta] (\sPos, 0) -- (\sPos, 0.15);
    \draw[colFocusBeta] (\mPos, 0) -- (\mPos, 0.15);

    \def\braceRaise{18pt} 
    
    \draw [decorate, decoration={brace, amplitude=6pt, raise=\braceRaise}, colNormal]
        (0, 0) -- (\sPos-\gap, 0) node [midway, above=26pt] {$\alpha$};
    
    \draw [decorate, decoration={brace, amplitude=6pt, raise=\braceRaise}, colFocusBeta, thick]
        (\sPos+\gap, 0) -- (\mPos-\gap, 0) node [midway, above=26pt] {$\beta=\alpha_{t+1}$};
        
    \draw [decorate, decoration={brace, amplitude=6pt, raise=\braceRaise}, colNormal]
        (\mPos+\gap, 0) -- (\nPos, 0) node [midway, above=26pt] {$\beta_0$};

    \coordinate (ak_bottom) at (\akCenter, 0);
    \coordinate (beta_bottom) at (6.25, 0); 

    \begin{scope}[shift={(0, -\layerDist)}]
        \def\w{3.5}
        \draw[colFocusAlpha] (0, 0) -- (\w, 0);
        \draw[colFocusAlpha] (0, 0) -- (0, 0.15);
        \draw[colFocusAlpha] (\w, 0) -- (\w, 0.15);
        
        \node[colFocusAlpha] at (0.5, 0.4) {$\beta_1^k$};
        \draw[colFocusAlpha] (1.0, 0) -- (1.0, 0.15);
        
        \node[colFocusAlpha] at (1.75, 0.3) {$\cdots$};
        
        \draw[colFocusAlpha] (2.5, 0) -- (2.5, 0.15);
        \node[colFocusAlpha] at (3.0, 0.4) {$\beta_{N_k}^k$};
        
        \node[colFocusAlpha] at (\w/2, -0.3) {$\alpha_k$};
        
        \coordinate (sub_ak_top) at (\w/2, 0.6);
    \end{scope}

    \begin{scope}[shift={(6.5, -\layerDist)}]
        \def\w{4.5}
        \draw[colFocusBeta] (0, 0) -- (\w, 0);
        \draw[colFocusBeta] (0, 0) -- (0, 0.15);
        \draw[colFocusBeta] (\w, 0) -- (\w, 0.15);
        
        \node[colFocusBeta] at (0.6, 0.4) {$\beta_1^{t+1}$};
        \draw[colFocusBeta] (1.2, 0) -- (1.2, 0.15);
        
        \node[colFocusBeta] at (2.25, 0.3) {$\cdots$};
        
        \draw[colFocusBeta] (3.3, 0) -- (3.3, 0.15);
        \node[colFocusBeta] at (3.9, 0.4) {$\beta_{N_{t+1}+1}^{t+1}$};
        
        \node[colFocusBeta] at (\w/2, -0.3) {$\beta$};
        
        \coordinate (sub_beta_top) at (\w/2, 0.6);
    \end{scope}
    
    
    \draw[->, colFocusAlpha!60, thick] (ak_bottom) to[out=-90, in=90, looseness=0.8] ($(sub_ak_top)+(0pt,0pt)$);;
    \draw[->, colFocusBeta!60, thick] (beta_bottom) to[out=-90, in=90, looseness=0.8] ($(sub_beta_top)+(0pt,0pt)$);

\end{tikzpicture}
\end{minipage}
\caption{Illustration of partitions for the index sets, where $s$ denotes the rank of $X$, $\alpha$ corresponds to the indices of nonzero singular values, $\beta$ corresponds to the zero singular values, and $\beta_0$ represents the indices of the remaining dimensions $m+1, \dots, n$.}
\label{fig:illustration_index}
\end{figure}
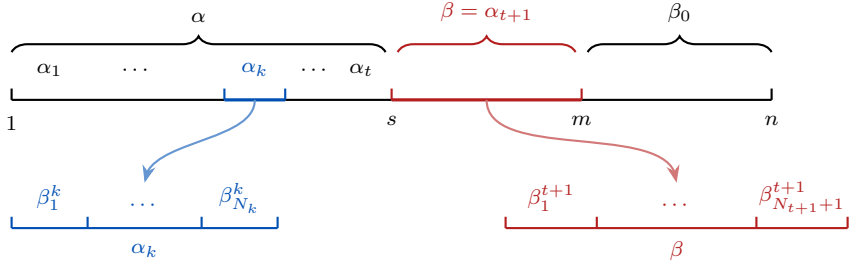

Given an index $i \in \{1, 2, \ldots, m\}$, the rules for computing the first- and second-order directional derivatives of $\sigma_i$ vary depending on the location of $i$, determined by the values $q_a(i)$, $l(i)$, $q_b(i)$, and $l'(i)$. We organized the results of \cite[Theorem 3.1]{zhang2013secondordersingular} as the flowchart in \myfig\ref{fig:flowchart_2ndderivative}, where the expression of $\widehat{V}_k$ for $k\in\{1,2,\ldots,t+1\}$ is detailed as follows,
\begin{equation*}\label{eq:widehatV}
    \widehat{V}_k(\eta,\zeta)\!=\!\begin{cases}
	\frac{\barzeta_{\alpha_k\alpha_k}+\barzeta^\top_{\alpha_k\alpha_k}}{2}+\frac{1}{\mu_k}(\frac{\bareta_{\alpha_k\alpha_k}-\bareta^\top_{\alpha_k\alpha_k}}{2})^\top(\frac{\bareta_{\alpha_k\alpha_k}-\bareta^\top_{\alpha_k\alpha_k}}{2})
    \\
    +\frac{1}{\mu_k}(\bareta^\top_{\beta\alpha_k}\bareta_{\beta\alpha_k}+\bareta_{\alpha_k\beta}\bareta^\top_{\alpha_k\beta}+\bareta_{\alpha_k\beta_0}\bareta^\top_{\alpha_k\beta_0})
    \\
    +\!\!\sum\limits_{\substack{j \ne k \\ j \le t}}\frac{\mu_j\bareta_{\alpha_k\alpha_j}\bareta_{\alpha_j\alpha_k}+\mu_k\bareta^\top_{\alpha_j\alpha_k}\bareta_{\alpha_j\alpha_k}+\mu_k\bareta_{\alpha_k\alpha_j}\bareta^\top_{\alpha_k\alpha_j}+\mu_j\bareta^\top_{\alpha_j\alpha_k}\bareta^\top_{\alpha_k\alpha_j}}{2(\mu_k^2-\mu_j^2)}, \hspace{1mm} \text{if}\ k\le t,
    \\[1mm]
    \barzeta_{\beta\widehat{\beta}}-2\bareta_{\beta\alpha}\bar{\varSigma}^{-1}_{\alpha\alpha}\bareta_{\alpha\widehat{\beta}},\hspace{57mm} \text{if}\ k=t+1,
    \end{cases}
\end{equation*}
and the matrices $Q^k$ and $(Q_{\beta\beta},\widehat{Q}_{\widehat{\beta}\widehat{\beta}})$ are arbitrarily chosen such that
\begin{equation}\label{eq:choseQ}
    Q^k\in\orth^{|\alpha_k|}\kh{\frac{\bareta_{\alpha_k\alpha_k}+\bareta_{\alpha_k\alpha_k}^\top}{2}}\ \ \text{and}\ \ (Q_{\beta\beta},\widehat{Q}_{\widehat{\beta}\widehat{\beta}})\in\orth^{|\beta|,|\widehat{\beta}|}(\bareta_{\beta\widehat{\beta}}).
\end{equation}

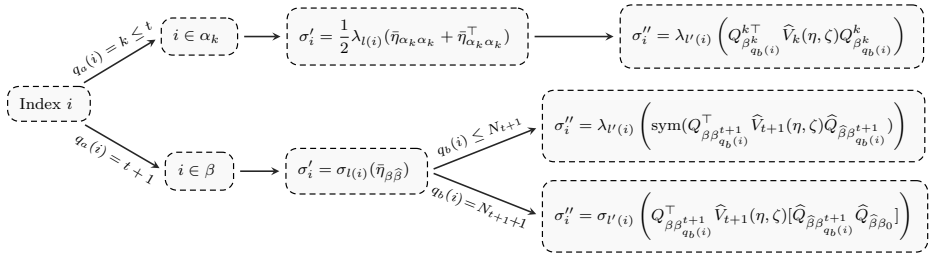
\begin{figure}[htbp]
\centering
\resizebox{\textwidth}{!}{%
\begin{tikzpicture}[
    node distance=2.0cm and 1.5cm, 
    dashedbox/.style={
        rectangle,
        draw=black,
        fill=gray!5,
        line width=0.6pt,
        dash pattern=on 3pt off 2pt,
        rounded corners=2mm,
        inner sep=6pt, 
        align=center,
        font=\footnotesize,
        minimum width=1.2cm, 
        execute at begin node={\everymath{\displaystyle}} 
    },
    arrowline/.style={
        ->,
        thick,
        >=stealth,
        color=black!85,
        shorten >=3pt,
        shorten <=3pt
    },
    arrowlabel/.style={
        midway,
        sloped,
        font=\scriptsize,
        fill=white,
        inner sep=1.5pt,
        text=black
    }
]


    \node[dashedbox] (start1) {
        $i\in\alpha_k$
    };

    \node[dashedbox, right=0.85cm of start1] (mid1) {
        $\sigma_i^{\prime}=\frac{1}{2} \lambda_{l(i)}(\bareta_{\alpha_k \alpha_k}+\bareta_{\alpha_k \alpha_k}^\top)$
    };

    \node[dashedbox, right=of mid1, xshift=0.1cm] (end1) {
        $\sigma_i^{\prime \prime} = \lambda_{l^{\prime}(i)}\kh{Q_{\beta_{q_b(i)}^k}^{k\top} \widehat{V}_k(\eta, \zeta) Q_{\beta_{q_b(i)}^k}^k}$
    };

    \draw[arrowline] (start1) -- (mid1);
    \draw[arrowline] (mid1) -- (end1);


    \node[dashedbox, below=1.7cm of start1] (start2) {
        $i\in\beta$
    };

    \node[dashedbox, right=0.9cm of start2] (mid2) {
        $\sigma_i^{\prime}=\sigma_{l(i)}(\bareta_{\beta \widehat{\beta}})$
    };

    
    \node[dashedbox, above right=0.8cm and 2.0cm of mid2.east, anchor=west] (end2_upper) {
        $\sigma_i^{\prime \prime} = \lambda_{l^{\prime}(i)}\kh{\mathrm{sym}(Q_{\beta \beta_{q_b(i)}^{t+1}}^\top \widehat{V}_{t+1}(\eta,\zeta) \widehat{Q}_{\widehat{\beta} \beta_{q_b(i)}^{t+1}})}$
    };

    \node[dashedbox, below right=0.8cm and 2.0cm of mid2.east, anchor=west] (end2_lower) {
        $\sigma_i^{\prime \prime} = \sigma_{l^{\prime}(i)}\kh{Q_{\beta \beta_{q_b(i)}^{t+1}}^\top \widehat{V}_{t+1}(\eta,\zeta)[\widehat{Q}_{\widehat{\beta} \beta_{q_b(i)}^{t+1}} \widehat{Q}_{\widehat{\beta}\beta_0}]}$\!
    };

    \draw[arrowline] (start2) -- (mid2);

    \draw[arrowline] (mid2.east) -- (end2_upper.west) 
        node[arrowlabel, above=2pt] {$q_b(i)\le N_{t+1}$};
        
    \draw[arrowline] (mid2.east) -- (end2_lower.west) 
        node[arrowlabel, below=2pt] {$q_b(i)\!=\!N_{t+1}\!\!+\!\!1$};


    \node[dashedbox] (root) at ($(start1.west)!0.5!(start2.west) - (2, 0)$) {
        Index $i$
    };

    \draw[arrowline] (root) -- (start1.west) 
        node[arrowlabel, above=2pt] {$q_a(i)=k \le t$};

    \draw[arrowline] (root) -- (start2.west) 
        node[arrowlabel, below=2pt] {$q_a(i)=t+1$};

\end{tikzpicture}
}
\caption{Flowchart of computing directional derivatives of singular values in directions $\eta,\zeta\in\mbRmn$, where we abbreviate $\sigma_i^\prime(X;\eta)$ and $\sigma_i^{\prime\prime}(X;\eta,\zeta)$ as $\sigma_i^\prime$ and $\sigma_i^{\prime\prime}$, respectively.}
\label{fig:flowchart_2ndderivative}
\end{figure}

Generally, to compute the directional derivatives, it first determines whether the singular value associated with $i$ is nonzero: if $1\le i\le\rank(X)$, i.e., $i\in\alpha_k$ for some $k\le t$, the first row in the flowchart applies. Otherwise, we have $i\in\beta$, and a second stage further determines how the second-order derivative is evaluated, depending on whether the index $q_b(i)$ corresponds to a zero singular value of $\bareta_{\beta\widehat{\beta}}$.

\subsection{Tangent sets to $\boundedrank$}\label{sec:tangenttoMr}
Given $X\in\boundedrank$, we begin by identifying the mapping $\sigmar^{\prime}(X;\cdot)$, and then find its zeros to explicitly give the tangent cone to the determinantal variety, as an application of Theorem~\ref{the:cal_tangentsets}. Section~\ref{sec:directionalderisigma} reveals that computing directional derivatives of $\sigma_{r+1}$ is generally nontrivial. Nevertheless, we then show that the condition $\sigmar(X)=0$ and the geometry of $\boundedrank$ can facilitate the derivation.

Assume that $\rank(X)=s$ with the full SVD $X=\bU[\bar{\varSigma}\ 0]\bV^\top$ as in~\eqref{eq:fullSVD}. According to \myfig\ref{fig:flowchart_2ndderivative} and letting $i=r+1$, we first notice that $q_a(r+1)=t+1$ since $\sigmar(X)=0$, which implies that the index $r+1\in\beta$. Therefore, the computation follows the second row of the flowchart. Moreover, by definition of the mapping $l$ in~\eqref{eq:indexmappings}, we have $l(r+1)=r+1-\kappa_{q_a(r+1)-1}=r-s+1$, since $\kappa_{q_a(r+1)-1}$ counts the number of nonzero singular values of $X$ in this case. Consequently, it holds that $\sigmar^\prime(X;\eta)=\sigma_{l(r+1)}(\bareta_{\beta\widehat{\beta}})=\sigma_{r-s+1}(\bareta_{\beta\widehat{\beta}})$ for any direction $\eta\in\mbRmn$. Hence, $\sigmar^\prime(X;\eta)=0$ is equivalent to the condition $\rank(\bareta_{\beta\widehat{\beta}})\le r-s$.

Denote $U=[\bU_{\alpha_1}\ \bU_{\alpha_2}\ \cdots\ \bU_{\alpha_t}]\in\stiefel(m,s)$, $U_\bot=\bU_{\beta}\in\stiefel(m,m-s)$, $V=[\bV_{\alpha_1}\ \bV_{\alpha_2}\ \cdots\ \bV_{\alpha_t}]\in\stiefel(n,s)$, and $V_\bot=\bV_{\widehat{\beta}}\in\stiefel(n,n-s)$.  Applying Theorem~\ref{the:cal_tangentsets} and substituting $\bareta_{\beta\widehat{\beta}}=U_\bot^\top\eta V_\bot$, we obtain
\begin{align}
    \Btangent_{\boundedrank}(X)&=\{\eta\in\mbRmn\mid\sigmar^\prime(X;\eta)=0\}  \label{eq:Btangent_cone_boundedrank_sigma_descrip} 
    \\
    &=\{\eta\in\mbRmn\mid\rank(U_\bot^\top\eta V_\bot)\le r-s\}.   \nonumber
\end{align}
This recovers the existing result~\eqref{eq:Btangent_cone_boundedrank}, which follows by parameterizing $\eta$ by $\eta=UW_1V^\top+UW_2V_\bot^\top+U_\bot W_3V^\top+R$ with $W_i$ ($i=1,2,3$) being matrices of appropriate dimensions and $R\in\normal_{\lowrank}(X)$. Additionally, a recent work \cite{olikier2025fourtangentproof} focuses on various descriptions of the tangent cone to $\boundedrank$; we remark that the characterization~\eqref{eq:Btangent_cone_boundedrank_sigma_descrip} serves as an alternative perspective.

More importantly, Theorem~\ref{the:cal_tangentsets} enlightens the following proposition, which provides a closed-form expression for the second-order tangent set to $\boundedrank$---a new development on the geometry of the determinantal variety.

\begin{proposition}\label{pro:tangenttwo_mr}
    Given $X\in\boundedrank$ and $\eta\in\tangent_{\boundedrank}(X)$, where $\rank{(X)}=s$ and $\rank(\projection_{\normal_{\lowrank}(X)}(\eta))=\ell-s$ for some $s\le \ell\le r$. Let the compact SVDs be $X=U\varSigma V^\top$ and $\projection_{\normal_{\lowrank}(X)}(\eta)=U_{\eta} \varSigma_\eta V_{\eta}^\top$, respectively. Take ${U}_{\eta\bot}$ and ${V}_{\eta\bot}$ such that $[U\ U_\eta\ U_{\eta\bot}]\in\orth(m)$ and $[V\ V_\eta\ V_{\eta\bot}]\in\orth(n)$, and denote $U^+=[U\ U_{\eta}]$, $V^+=[V\ V_{\eta}]$. It holds that
    \begin{equation}\label{eq:tangenttwo_mr}
        \tangenttwo_{\boundedrank}(X;\eta)=\hkh{2\eta X^\dagger\eta + [U^+\ U_{\eta\bot}]\left[ \begin{matrix}
        	W_1&		W_2\\
        	W_3&		J\\
        \end{matrix} \right][V^+\ V_{\eta\bot}]^\top \left|\,\begin{array}{l}
        W_1\in\mbR^{\ell\times\ell},
        \\
        W_2\in\mbR^{\ell\times (n-\ell)},
        \\
        W_3\in\mbR^{(m-\ell)\times \ell},
        \\
        J\in\mbR^{(m-\ell)\times (n-\ell)},
        \\
        \rank(J)\le r-\ell
        \end{array}\right.
        }.
    \end{equation}
\end{proposition}
\begin{proof}
    According to Theorem~\ref{the:cal_tangentsets} and expression~\eqref{eq:Btangent_cone_boundedrank_sigma_descrip}, given the $\eta\in\tangent_{\boundedrank}(X)$, the second-order tangent set to $\boundedrank$ at $X$ in direction $\eta$ coincides with the zeros of the mapping $\sigmar^{\prime\prime}(X;\eta,\cdot)$.

    Without loss of generality, we can assume that the full SVD of $X$ is $X=\bU[\bar{\varSigma}\ 0]\bV^\top$, and it holds that $U=[\bU_{\alpha_1}\ \bU_{\alpha_2}\ \cdots\ \bU_{\alpha_t}]\in\stiefel(m,s)$, $[U_\eta\ U_{\eta\bot}]=\bU_{\beta}\in\stiefel(m,m-s)$, $V=[\bV_{\alpha_1}\ \bV_{\alpha_2}\ \cdots\ \bV_{\alpha_t}]\in\stiefel(n,s)$, and $[V_\eta\ V_{\eta\bot}]=\bV_{\widehat{\beta}}\in\stiefel(n,n-s)$. In this case, 
    \begin{equation}\label{eq:decompose_bareta}
        \bareta_{\beta\widehat{\beta}}=\bU_\beta^\top\eta \bV_{\widehat{\beta}}=[U_\eta\ U_{\eta\bot}]^\top\eta[V_\eta\ V_{\eta\bot}]=\left[ \begin{matrix}
        	\varSigma _{\eta}&		0_{(\ell-s)\times (n-\ell)}\\
        	0_{(m-\ell)\times(\ell-s)}&		0_{(m-\ell)\times(n-\ell)}\\
        \end{matrix} \right].
    \end{equation}
    Following the derivation in the first-order case~\eqref{eq:Btangent_cone_boundedrank_sigma_descrip}, we have $q_a(r+1)=t+1$ and $l(r+1)=r-s+1$. Through the second row of the flowchart in \myfig\ref{fig:flowchart_2ndderivative}, the next step is to determine the indices $q_b(r+1)$ and $l^\prime(r+1)$. 

    We notice that $\rank(\bareta_{\beta\widehat{\beta}})=\rank(\varSigma_\eta)=\ell-s\le r-s$, and thus $l(r+1)=r-s+1$ indexes a zero singular value of $\bareta_{\beta\widehat{\beta}}$. Hence, the definition of the mapping $q_b$ in~\eqref{eq:indexmappings} reveals that $q_b(r+1)=N_{t+1}+1$. Furthermore, substituting the values of $q_a(r+1)$, $l(r+1)$, and $q_b(r+1)$, we obtain
    \begin{equation*}
        l^\prime(r+1)=l(r+1)-\kappa^{(q_a(r+1))}_{q_b(r+1)-1}=r-s+1 - \kappa^{(t+1)}_{N_{t+1}}=r-\ell+1,
    \end{equation*}
    where the last equality holds since $\kappa^{(t+1)}_{N_{t+1}}=\rank(\bareta_{\beta\widehat{\beta}})=\ell-s$ counts the number of nonzero singular values of $\bareta_{\beta\widehat{\beta}}$.

    Consequently, the indices $r+1\in\beta$ and $q_b(r+1)=N_{t+1}+1$ point the computation of $\sigma_{r+1}^{\prime\prime}$ to the branch at the bottom of the flowchart (see \myfig\ref{fig:flowchart_2ndderivative}):
    \begin{align}        
        \sigma_{r+1}^{\prime \prime}(X;\eta,\zeta)=&\ \sigma_{l^{\prime}(r+1)}(Q_{\beta \beta_{q_b(r+1)}^{t+1}}^\top \widehat{V}_{t+1}(\eta,\zeta)[\widehat{Q}_{\widehat{\beta} \beta_{q_b(r+1)}^{t+1}} \widehat{Q}_{\widehat{\beta}\beta_0}])    \nonumber
        \\
        =&\ \sigma_{r-\ell+1}(Q_{\beta \beta_{N_{t+1}+1}^{t+1}}^\top \widehat{V}_{t+1}(\eta,\zeta)[\widehat{Q}_{\widehat{\beta} \beta_{N_{t+1}+1}^{t+1}} \widehat{Q}_{\widehat{\beta}\beta_0}]).   \label{eq:sigmarpp}
    \end{align}
    Then we calculate $\widehat{V}_{t+1}(\eta,\zeta)$ according to the expressions provided in section~\ref{sec:directionalderisigma}:
    \begin{align}
        \widehat{V}_{t+1}(\eta,\zeta)&=[U_\eta\ U_{\eta\bot}]^\top(\zeta-2\eta V\varSigma^{-1}U^\top\eta)[V_\eta\ V_{\eta\bot}]     \nonumber
        \\
        &=[U_\eta\ U_{\eta\bot}]^\top(\zeta-2\eta X^\dagger\eta)[V_\eta\ V_{\eta\bot}].     \label{eq:wideV_t+1}
    \end{align}
    Moreover, noticing from~\eqref{eq:choseQ} and~\eqref{eq:decompose_bareta} that $Q_{\beta\beta}$ and $\widehat{Q}_{\widehat{\beta}\widehat{\beta}}$ can be chosen as $Q_{\beta\beta}=I_{m-s}$ and $Q_{\widehat{\beta}\widehat{\beta}}=I_{n-s}$, which reveals that 
    $$Q_{\beta \beta_{N_{t+1}+1}^{t+1}}=[0_{(m-\ell)\times(\ell-s)}\ I_{m-\ell}]^\top\ \text{and}\ \ [\widehat{Q}_{\widehat{\beta} \beta_{N_{t+1}+1}^{t+1}} \widehat{Q}_{\widehat{\beta}\beta_0}]=[0_{(n-\ell)\times(\ell-s)}\ I_{n-\ell}]^\top.$$
    Substituting the above equalities and the expression~\eqref{eq:wideV_t+1} into~\eqref{eq:sigmarpp} yields
    \begin{equation*}
        \sigma_{r+1}^{\prime \prime}(X;\eta,\zeta) = \sigma_{r-\ell+1}(U_{\eta\bot}^\top(\zeta-2\eta X^\dagger\eta)V_{\eta\bot}). 
    \end{equation*}
    Therefore, $\sigmar^{\prime\prime}(X;\eta,\zeta)=0$ is equivalent to $\rank(U_{\eta\bot}^\top(\zeta-2\eta X^\dagger\eta)V_{\eta\bot})\le r-\ell$. Recall that $U^+=[U\ U_{\eta}]$, $V^+=[V\ V_{\eta}]$, and consider the decomposition $\zeta-2\eta X^\dagger\eta=U^+W_1V^{+\top}+U^+W_2V_{\eta\bot}^\top+U_{\eta\bot}W_3V^{+\top}+U_{\eta\bot}JV_{\eta\bot}^\top$ with $W_i$ ($i=1,2,3$) and $J$ being matrices of appropriate dimensions. In this view, $\sigmar^{\prime\prime}(X;\eta,\zeta)=0$ is equivalent to $\rank(J)\le r-\ell$. Applying Theorem~\ref{the:cal_tangentsets} concludes the expression~\eqref{eq:tangenttwo_mr}.
\end{proof}

Note that $\fixedrank$ is relatively open in $\boundedrank$, i.e., for any $X\in\fixedrank$, there exists a neighborhood $\neighbor$ around $X$ such that $\boundedrank\cap\neighbor\subseteq\fixedrank$ \cite{olikier2022continuity}. Therefore, the second-order tangent set to $\fixedrank$ at $X$ coincides with that to $\boundedrank$. Specifically, we inherit the notation from Proposition~\ref{pro:tangenttwo_mr} and find that $\ell=\rank(X)=r$, implying that $U_\eta$ and $V_\eta$ vanish, and thus $U_{\eta\bot}$ and $V_{\eta\bot}$ become $U_\bot$ and $V_\bot$, respectively. Consequently, the expression~\eqref{eq:tangenttwo_mr} reduces to the tangent set to the fixed-rank manifold, i.e., $\tangenttwo_{\fixedrank}(X;\eta)=\tangenttwo_{\boundedrank}(X;\eta)$, and more concisely,
\begin{equation}\label{eq:tangenttwo_fixedrank}
    \tangenttwo_{\fixedrank}(X;\eta)=\hkh{2\eta X^\dagger\eta + [U\ U_{\bot}]\left[ \begin{matrix}
        W_1&		W_2\\
        W_3&		0\\
    \end{matrix} \right][V\ V_{\bot}]^\top \left|\,\begin{array}{l}
    W_1\in\mbR^{r\times r},
    \\
    W_2\in\mbR^{r\times (n-r)},
    \\
    W_3\in\mbR^{(m-r)\times r}
    \end{array}\right.
    },
\end{equation}
for any $\eta\in\tangent_{\fixedrank}(X)$. This explicit formula for $\tangenttwo_{\fixedrank}(X;\eta)$ is also a new result to our knowledge, shedding light on the (second-order) geometry of $\fixedrank$.

\section{Extensions to structured low-rank sets}\label{sec:extent_structured}
In this section, we show that the developed analysis can be extended to more scenarios. Specifically, Theorem~\ref{the:cal_tangentsets} is applied to low-rank tensor varieties~\eqref{eq:httensor}, symmetric matrices~\eqref{eq:characterization_Sr}, and positive semidefinite matrices~\eqref{eq:boundedranksdp}. Moreover, following the spirit of Theorem~\ref{the:cal_tangentsets}, we establish in Theorem~\ref{the:expressions_McapK} the intersection rules for the intersection of two general sets, which are then applied to obtain tangent sets to an array of structured low-rank sets; all the results are summarized in Table~\ref{tab:tangentsets} and appendices.

\subsection{Tangent sets to tensor varieties}
Low-rank tensor decompositions compactly represent multi-dimensional data, capturing essential structure with far less storage. The Tucker format \cite{tucker1964extension}, the tensor train (TT) format \cite{oseledets2011TTSVD}, and more generally, the hierarchical Tucker (HT) format \cite{grasedyck2010hierarchicalSVD} are among the most typical formats. Each format induces a kind of tensor variety when the low-rank structure is imposed, and in particular, the Tucker and TT varieties are special cases of the HT variety (see Appendix~\ref{app:reduction}). Therefore, in this section, we calculate the tangent sets to the HT variety, and then obtain the results for the Tucker and TT varieties as immediate reductions.

We call $\tensX$ a tensor if it is a $d$-dimensional array in the space $\tensorspace$, and introduce in Appendix~\ref{app:HTvariety} the preliminaries for HT varieties, including the \emph{dimension tree} $T$ with nodes denoted by $t\in T$ and the associated dimensions $\{n_t,n_{t_-}\}_{t\in T}$, \emph{mode-$t$ matricization} $X_{(t)}^\mht$ and \emph{tensorization} $\tensorizeht_{(t)}(\cdot)$, and the \emph{HT rank} $\rankht(\tensX)$. Given an array $\vecr=(r_t)_{t\in T}$ of positive integers indexed by nodes of $T$, we define $\boundedht$ as the set of tensors with an HT rank at most $\vecr$:
\begin{equation}\label{eq:httensor}
    \boundedht = \{\tensX\in\tensorspace \mid \rankht(\tensX)\le\vecr\},
\end{equation}
where the ``$\le$" is understood component-wise. Through the lens of matricization, $\boundedht$ coincides with the intersection of tensorized matrix varieties along different modes, i.e.,
\begin{equation}\label{eq:defineht_intersecion}
    \manifold_{\le\vecr}^\mht =  \bigcap_{t\in T}{\tensorize_{(t)}^\mht\kh{\mbR^{n_t\times n_{t_-}}_{\le r_t}}},
\end{equation}
where we adopt $\mbR^{n_1\times n_2}_{\le r}:=\{X\in\mbR^{n_1\times n_2}\mid \rank(X)\le r\}$ to explicitly reveal the shape of the matrices. Therefore, $\boundedht$ is also a real algebraic variety; we refer to it as the \emph{HT variety}.

Motivated by \eqref{eq:defineht_intersecion}, we extend the perspective~\eqref{eq:sigmarboundedrank} to the tensor scenario:
\begin{equation}\label{eq:tensor_defined_sigma}
    \manifold_{\le\vecr}^\mht = \hkh{\tensX\in\tensorspace \mid \sigma_{r_t+1}(X^\mht_{(t)})=0\ \text{for all}\ t\in T},
\end{equation}
which enlightens the application of Theorem~\ref{the:cal_tangentsets}. To this end, it suffices to verify that the $\boundedtensor$ given through~\eqref{eq:tensor_defined_sigma} satisfies Assumption~\ref{assu:errorbound}. 

Firstly, note that the mapping $\tensX\mapsto\sigma_{r_t+1}(X^\mht_{(t)})$ is Lipschitz continuous for all $t\in T$. Then, for any $\tensY\in\tensorspace$, we can find a $\tensY_p\in\boundedtensor$ such that $\|\tensY-\tensY_p\|_{\frob}$ can be bounded by the values of $\{\sigma_{r_t+1}(Y^\mht_{(t)})\}_{t\in T}$. In fact, we resort to the hierarchical SVD \cite{grasedyck2010hierarchicalSVD} to produce a low-rank truncation of $\tensY$ as the candidate for $\tensY_p$; see Appendix~\ref{app:tensorvariety} for more details. 

Consequently, we apply Theorem~\ref{the:cal_tangentsets} to the variety $\boundedtensor$, unveiling the intersection rules for the associated tangent sets.

\begin{proposition}\label{pro:tangenttovarieties}
    The tangent sets to the tensor varieties $\boundedtensor$ equal the intersection of tensorized tangent sets to unfolding matrices along different modes, i.e.,
    \begin{align}
        \Btangent_{\boundedtensor}(\tensX) &= \bigcap_{t\in T}{\tensorize_{(t)}^\mht\kh{\Btangent_{\ranifold_t}(X^\mht_{(t)})}},    \label{eq:expressions_tensor_1}
        \\
        \tangenttwo_{\boundedtensor}(\tensX;\tenseta) &= \bigcap_{t\in T}{\tensorize_{(t)}^\mht\kh{\tangenttwo_{\ranifold_t}(X^\mht_{(t)};\eta^\mht_{(t)})}} \ \ \text{for any}\ \tenseta\in \Btangent_{\boundedtensor}(\tensX).    \label{eq:expressions_tensor_2}
    \end{align}
    where we denote $\ranifold_t:=\mbR_{\le r_t}^{n_t\times n_{t_-}}$.
\end{proposition}
\begin{proof}
    See Appendix~\ref{app:prooftangentvariety}.
\end{proof}

To our knowledge, it is the first time that the intersection rules~\eqref{eq:expressions_tensor_1} and~\eqref{eq:expressions_tensor_2} are identified for low-rank HT varieties. In addition, with appropriate dimension trees, the results realize the tangent sets to TT varieties and to Tucker varieties as immediate reductions (see Appendix~\ref{app:reduction}). Specifically, the first-order rule~\eqref{eq:expressions_tensor_1} for TT and Tucker varieties have been reported in \cite[Corollary 2.9]{kutschan2018tangentTT} and \cite[Corollary 1]{gao2025lowranktucker}, respectively; while the second-order counterparts are new results.

\subsection{Tangent sets to intersection of sets}\label{sec:McapH}
A natural question is whether the preceding analysis can be extended to the intersection of the low-rank set, e.g. $\boundedrank$, with another set $\kanifold$ defined as a level set of a mapping $h$---a setting attracting growing interest in recent years \cite{cason2013iterative,li2020jotaspectral,li2023normalboundedaffine,yang2025spacedecouple}. To this end, following the spirit of section~\ref{sec:generaltengentset}, we analyze the tangent sets to a general intersection $\manifold\cap\kanifold$, and then apply the results to several specific instances in sections~\ref{sec:lowrank_rectangular}-\ref{sec:lowrank_posi}; results are summarized in Table~\ref{tab:tangentsets}.

Given an Euclidean space $\mbR^q$, consider two sets as follows: 
\begin{equation}\label{eq:def_MK}
    \manifold=\{\tilde{X}\in\mbR^q \mid c_1(\tilde{X})=0\}\ \ \text{and}\ \ \kanifold=\{\tilde{X}\in\mbR^q\mid h(\tilde{X})=0\},
\end{equation}
where $c_1:\mbR^q\to\mbR^{n_1}$ is a possibly nonsmooth mapping, and $h:\mbR^q\to\mbR^{n_2}$ is a smooth mapping. Studying the geometry of $\manifold\cap\kanifold$ is obstructed in two respects: 1) the set $\manifold$ can be nonsmooth, and thus the intersection rule~\eqref{eq:cone_tranverse}, which relies on transversality, becomes invalid; 2) application of the developed Theorem~\ref{the:cal_tangentsets} is not straightforward, as it is uncertain whether the intersection $\manifold\cap\kanifold$ satisfies Assumption~\ref{assu:errorbound}(ii)---partly because the metric projection onto the coupled set remains unclear in general.

To circumvent the nonsmooth geometry, we employ the idea of \emph{smooth parameterization}, which introduces a smooth manifold $\bmanifold\subseteq\mathbb{R}^{\bar q}$ and a smooth mapping $\phi:\mathbb{R}^{\bar q}\to\mathbb{R}^{q}$ with $\phi(\bmanifold)=\manifold$, originally proposed as a remedy for nonsmooth optimization problems~\cite{levin2023remedy,levin2025effect}. Subsequently, we pull back $\kanifold$ through $\phi$ to obtain
\[\bkanifold:=\{\tilde{x}\in\mbR^{\bq}\mid h(\phi(\tilde{x}))=0\}=\phi^{-1}(\kanifold).\]
Then $\mbR^{\bq}$ is viewed as an auxiliary space, and it is hoped that the benign properties of $\bmanifold\cap\bkanifold$ in the lift space $\mbR^{\bq}$ can shed light on the analysis for $\manifold\cap\kanifold$ in the original space $\mbR^q$. Consequently, we identify some  regularity conditions as follows, which are illustrated in \myfig\ref{fig:diagramMK}.

\begin{assumption}\label{assu:MintersecK}
    Given $\manifold$ and $\kanifold$ as in~\eqref{eq:def_MK}. At $X\in\manifold\cap\kanifold$, $\manifold$ and $\kanifold$ satisfy Assumption~\ref{assu:errorbound}, respectively. Moreover, it admits a smooth parameterization $(\bmanifold\subseteq\mbR^{\bq},\phi)$ of $\manifold$ and a point $x\in\bmanifold\cap\bkanifold$ with $\phi(x)=X$ satisfying the following conditions:
    \begin{itemize}
        \item [\emph{(i)}] the differential $\diff (h\circ\phi)$ has constant rank in a neighborhood of $x$, which implies that $\bkanifold$ is a smooth manifold around $x$;
        \item [\emph{(ii)}] the manifolds $\bmanifold$ and $\bkanifold$ intersect transversally at $x$;
        \item [\emph{(iii)}]  the restriction $\phi|_{\bmanifold}:\bmanifold\to\manifold$ is open\footnote{The mapping $\phi_{\bmanifold}:\bmanifold\to\manifold$ is said to be open at $x\in\bmanifold$ if it maps neighborhoods of $x$---sets in $\bmanifold$ containing $x$ in their interior---to neighborhoods of $\phi(x)\in\manifold$ endowed with the subspace topology inherited from the ambient space.} at $x$.
    \end{itemize}
\end{assumption}

\begin{figure}[htbp]
    \vspace{0mm}
    \newcommand{\ratio}{0.5}
    \newcommand{\customrectangle}[7]{
            \draw[densely dashed, draw=#6, fill=#7] 
                (#1-#3/2, #2-#4/2) rectangle 
                (#1+#3/2, #2+#4/2); 
            \node at (#1, #2) {#5}; 
        }
    \begin{minipage}{1\textwidth}
    \begin{center}
    \begin{tikzpicture}
        \newcommand{\nodegap}{1.3}
        \tikzset{
        node distance=\nodegap cm,
        post/.style={->,shorten >=2pt,shorten <=2pt,>={Stealth[round]},thick},
        space/.style={
          draw=none, 
          fill=none, 
          inner sep=0pt,
          minimum size=6mm
        },
        hookarrow/.style={{Hooks[left]}->}
        }
        
        \node[space] (Y) {\normalsize $\bmanifold\cap\bkanifold\subseteq\mbR^{\bar{q}}$};
        \node[space] (X) [below=1.2cm of Y] {\normalsize $\manifold\cap\kanifold\subseteq \mbR^q$};
        \node[space] (R) [right=2.5cm of X] {\normalsize $\mbR^{n_2}$};
        
        \draw[post] (Y) to node[midway, left, xshift=-2pt] (phi) {\normalsize $\phi$} (X);
        \draw[post] (Y) to node[midway, above, xshift=42pt] {\normalsize $h\circ \phi$\ \,(rank-constant)} (R);
        \draw[->,shorten <=4pt,>={Stealth[round]},thick] (X) to node[midway, below, yshift=-1pt] {\normalsize $h$} (R);

        \node[space] (tran) [left=0.2cm of Y] {\normalsize (transversal)};
        \node[space] (open) [left=0.0cm of phi] {\normalsize (open)};

    \end{tikzpicture}
    \end{center}
    \end{minipage}
    \caption{Illustration of Assumption~\ref{assu:MintersecK} with the three regularity conditions.}
    \label{fig:diagramMK}
\end{figure}
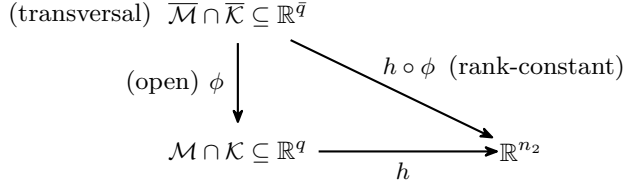

\begin{remark}
Assumption~\ref{assu:MintersecK}(ii) ensures the regularity of the intersection $\bmanifold \cap \bkanifold$, while Assumption~\ref{assu:MintersecK}(iii) helps transfer local structure from the auxiliary space to the original space. In fact, these two conditions can be readily verified in certain scenarios. For example, when considering the LR parameterization for $\boundedrank$, i.e., 
\begin{equation}\label{eq:MH_LR}
    \bmanifold = \mbR^{m\times r} \times \mbR^{n\times r}\ \ \ \text{and}\ \ \ \phi: (L,R)\mapsto LR^\top,
\end{equation}
the manifold $\bmanifold$ coincides with the ambient Euclidean space, and thus the transversal property naturally holds for $\bmanifold\cap\bkanifold$ provided $\bkanifold$ is a manifold. In addition, by \cite[Theorem~2.3 and Proposition~2.8]{levin2025effect}, $\phi$ is open at $(L,R)$ if and only if $\rank(L)=\rank(R)=\rank(LR^\top)$; such a preimage exists for every $X=U\varSigma V^\top$, e.g., $(L,R)=([U\varSigma^{1/2}\ 0],[V\varSigma^{1/2}\ 0])$, which aligns with Assumption~\ref{assu:MintersecK}(iii).
\end{remark}

We prove in the following theorem that if $\manifold$\ and $\kanifold$ satisfy Assumption~\ref{assu:MintersecK}, the intersection $\manifold\cap\kanifold$ satisfies Assumption~\ref{assu:errorbound}, thereby validating the application of the developed Theorem~\ref{the:cal_tangentsets} to $\manifold\cap\kanifold$.

\begin{theorem}[Intersection rule]\label{the:expressions_McapK}
    Suppose that $\manifold$ and $\kanifold$ satisfy Assumption~\ref{assu:MintersecK} at $X\in\manifold\cap\kanifold$. We have the following intersection rules for the tangent sets to $\manifold\cap\kanifold$,
    \begin{enumerate}
    \item[\emph{(i)}] \emph{(First-order)} If $c_1$ is directionally differentiable at $X$, then
    \begin{equation}\label{eq:decoupletangentMK}
        \Btangent_{\manifold\cap\kanifold}(X) = \Btangent_{\manifold}(X)\cap\Btangent_{\kanifold}(X);
    \end{equation}
    \item[\emph{(ii)}] \emph{(Second-order)} If, in addition, $c_1$ admits parabolic second-order directional derivatives at $X$ for every direction pair $(\eta,\zeta)$, then for any $\eta\in\Btangent_{\manifold\cap\kanifold}(X)$,
    \begin{equation}\label{eq:decoupletangenttwoMK}
        \tangenttwo_{\manifold\cap\kanifold}(X;\eta) = \tangenttwo_{\manifold}(X;\eta)\cap\tangenttwo_{\kanifold}(X;\eta).
    \end{equation}
    \end{enumerate}
\end{theorem}
\begin{proof}
    See Appendix~\ref{app:McapK}.
\end{proof}

\begin{remark}
    The proof of Theorem~\ref{the:expressions_McapK} only uses Assumption~\ref{assu:MintersecK} to show that $\manifold$ and $\kanifold$ are subtransversal at $X$~\cite{kruger2018setregularities}, i.e., $\dist(Y,\manifold\cap\kanifold)\le\kappa(\dist(Y,\manifold)+\dist(Y,\kanifold))$ for all $Y$ near $X$; thus the conclusions of Theorem~\ref{the:expressions_McapK} hold given Assumption~\ref{assu:errorbound} and the subtransversality, and Assumption~\ref{assu:MintersecK} serves as a more verifiable sufficient condition.
\end{remark}

\begin{remark}
    If Assumption~\ref{assu:MintersecK} is relaxed to require Assumption~\ref{assu:errorbound} only for $\manifold$ rather than for both $\manifold$ and $\kanifold$, the explicit calculation $\Btangent_{\manifold\cap\kanifold}(X) = \{\eta\in\mbR^q\mid c_1^\prime(X;\eta)=0,\,h^\prime(X;\eta)=0\}$ still holds. The role of Assumption~\ref{assu:errorbound} for $\kanifold$ is to guarantee the decoupling principle $\{\eta\in\mbR^q\mid c_1^\prime(X;\eta)=0,\,h^\prime(X;\eta)=0\}=\Btangent_{\manifold}(X)\cap\Btangent_{\kanifold}(X)$; an analogous rationale applies to the second-order counterpart.
\end{remark}

Theorem~\ref{the:expressions_McapK} serves as an extension of Theorem~\ref{the:cal_tangentsets}, incorporating an additional set $\kanifold$ and decoupling the computation of the tangent sets to $\manifold\cap\kanifold$ into the intersection of those to each component. In essence, Theorem~\ref{the:expressions_McapK} provides a unified perspective to unveil the first-order rule~\eqref{eq:decoupletangentMK}, which were previously verified on a case-by-case basis for certain choices of $(\manifold,\kanifold)$~\cite{cason2013iterative,li2023normalboundedaffine,yang2025spacedecouple,peng2025normalizedTT}. More importantly, the second-order intersection rule~\eqref{eq:decoupletangenttwoMK} yields new theoretical insights into the geometry of the intersection $\manifold\cap\kanifold$.

Next, we demonstrate the broad applicability of the developed Theorems~\ref{the:cal_tangentsets} and \ref{the:expressions_McapK} through several instances where low-rank sets intersect with additional structured sets. For clarity, we divide the discussion into three groups: first, low-rank rectangular matrices in $\mathbb{R}^{m\times n}$; second, low-rank symmetric matrices; and third, low-rank positive semidefinite matrices; see Table~\ref{tab:tangentsets} for a summary.

\subsubsection{Low-rank rectangular matrices}\label{sec:lowrank_rectangular}
We begin with the intersection of the determinantal variety $\boundedrank$ and another structured set $\hanifold\subseteq\mbRmn$. Four typical choices of $\hanifold$ considered in existing literature are the affine manifold~\cite{li2023normalboundedaffine}, the Frobenius sphere~\cite{cason2013iterative,yang2025spacedecouple,peng2025normalizedTT}, the oblique manifold~\cite{yang2025spacedecouple}, and the hyperbolic manifold~\cite{jawanpuria2019lowrankhyperbolic}:
\begin{equation}\label{eq:exampleHi}
    \begin{aligned}
        \hanifold_1 &=\affine(m,n):= \{X\in\mbRmn\,|\ \aanifold(X)-b={0}\},
        \\
        \hanifold_2 &=\Fsphere(m,n):= \{X\in\mbRmn\,|\ \norm{X}^2_{\frob}-1=0\},
        \\
\hanifold_3 &=\oblique(m,n):= \{X\in\mbRmn\,|\ \ddiag(XX^\top)-{\bf 1}={0}\},
        \\
        \hanifold_4 &= \matrixhyperboloid:=\{X\in\mbRmn\mid X_i\in\hyperboloid\ \text{for}\ i=1,2,\ldots,n\},
    \end{aligned}
\end{equation}
where ${\bf 1}\in\mathbb{R}^{m}$ denotes an all-ones vector, $\hyperboloid:=\{x\in\mathbb{R}^m\mid-x_1^2 + \sum_{i=2}^m x_i^2=-1,\ x_1>0\}$, and $X_{i}$ extracts the $i$-th column of $X$. Checking that all the $\boundedrank\cap\hanifold_j$ ($j=1,2,3,4$) satisfy Assumption~\ref{assu:MintersecK} by choosing $(\bmanifold,\phi)$ as the LR parameterization~\eqref{eq:MH_LR}, we then apply Theorem~\ref{the:expressions_McapK} to obtain the following intersection rules,
\begin{equation*}
\begin{aligned}
    \Btangent_{\boundedrank\cap\hanifold_j}(X) &= \Btangent_{\boundedrank}(X)\cap\Btangent_{\hanifold_j}(X),
    \\
    \tangenttwo_{\boundedrank\cap\hanifold_j}(X;\eta) &= \tangenttwo_{\boundedrank}(X;\eta)\cap\tangenttwo_{\hanifold_j}(X;\eta)\ \ \text{for any}\ \eta\in\Btangent_{\boundedrank\cap\hanifold_j}(X).
\end{aligned}
\end{equation*}
The above results enable us to derive the closed-form expressions for the tangent sets to $\boundedrank\cap\hanifold_j$; more details are given in Appendix~\ref{app:McapH_application}.

\subsubsection{Low-rank symmetric matrices}\label{sec:lowrank_sym}
Subsequently, we turn to the symmetric scenario when $m=n$ and $\hanifold=\sdp(n)\cap\uanifold$ (resp. $\sdp^+(n)\cap\uanifold$) for some $\uanifold\subseteq\sym(n)$---this is a topic of independent interest~\cite{pataki1998lowranksolution,li2020jotaspectral}, and thus we rewrite the intersection $\boundedrank\cap\hanifold$ in the more specific form $\sym_{\le r}(n)\cap\uanifold$ (resp. $\boundedranksdp(n)\cap\uanifold$), where it is recalled that
\begin{equation}\label{eq:define_Srn}
    \boundedranks = \hkh{X\in\sanifold(n)\mid \rank(X)\le r}.
\end{equation}

\begin{remark}
    One may also regard $\boundedranks$ as the intersection $\boundedrank\cap\sym(n)$ and take $(\manifold,\kanifold)=(\boundedrank,\sym(n))$ in Theorem~\ref{the:expressions_McapK}. However, Assumption~\ref{assu:MintersecK}(i) fails for this choice under the LR parameterization, since the rank of $\diff(h\circ\phi)$ with $h(X)=X-X^\top$ is not constant in general. Therefore, we derive Proposition~\ref{pro:tangenttoS} via the decomposition~\eqref{eq:characterization_Sr} instead.
\end{remark}

Inspired by the perspective~\eqref{eq:sigmarboundedrank}, we consider the eigenvalue mappings arranged in a non-increasing order, $\lambda_1(X)\ge\lambda_2(X)\ge\cdots\ge\lambda_n(X)$, and then draw on the following characterization of $\boundedranks$, 
\begin{equation}\label{eq:characterization_Sr}
    \sym_{\le r}(n)=\bigcup_{j=1}^{r+1} \sanifold_j,\ \text{with}\ \sanifold_j:=\{X\in\sym(n)\mid\lambda_{j}(X)=0,\,\lambda_{j+n-r-1}(X)=0\}.
\end{equation}
We briefly explain the above decomposition. In fact, any $X\in\sym(n)$ with $\rank(X)\le r$ must have at least $n-r$ consecutive eigenvalues $\lambda_i$ equal zero, and by the non-increasing ordering of the $\lambda_i$, we conclude that $X\in\sanifold_j$ for some $1\le j \le r+1$, i.e.,
$\lambda_1(X) \ge  \cdots \ge \lambda_{j-1}(X) \ge 0 = \lambda_j(X) = \cdots =\lambda_{j+n-r-1}(X) = 0 \ge \lambda_{j+n-r}(X)\ge\!\cdots \ge \lambda_n(X)$.

Combining the rules~\eqref{eq:tangentcupXi} and~\eqref{eq:characterization_Sr}, we can identify the tangent sets to $\boundedranks$ by taking the union of those to each $\sanifold_j$, which shifts our focus to $\sanifold_j$. In fact, it can be verified that for $1\le j\le r+1$, the set $\sanifold_j$ satisfies Assumption~\ref{assu:errorbound}. Subsequently, we apply Theorem~\ref{the:cal_tangentsets} to $\sanifold_j$, obtaining the associated tangent sets as the zeros of directional derivatives of eigenvalue mappings, which is achievable since explicit expressions for directional derivatives of $\lambda_i$ are given in \cite{torki2001secondtoeigen,zhang2013secondordersingular}; the tangent sets to $\sanifold_j$ are derived in Proposition~\ref{pro:tangenttoSj} of Appendix~\ref{app:tangent_Sj}. Collecting the results produces tangent sets to $\sym_{\le r}(n)$, as presented in the following Proposition.

\begin{proposition}\label{pro:tangenttoS}
    Given $X\in\boundedranks$ with $\rank{(X)}=s$ and the spectral decomposition $X=U\varLambda U^\top$ with $U\in\stiefel(n,s)$. The tangent cone to $\boundedranks$ can be characterized by
    \begin{equation}\label{eq:tangent_Sr}
        \Btangent_{\boundedranks}(X)=\hkh{[U\ U_{\bot}]\left[ \begin{matrix}
        	W_1&		W_2\\
        	W_2^\top&		J\\
        \end{matrix} \right][U\ U_{\bot}]^\top \left|\,\begin{array}{l}
        W_1\in\sym(s),
        \\
        W_2\in\mbR^{s\times (n-s)},
        \\
        J\in\sym_{\le r-s}(n-s)
        \end{array}\right.
        }.
    \end{equation}
    Additionally, given a direction $\eta\in\Btangent_{\boundedranks}(X)$ parameterized in the above manner with $\rank(J)=\ell-s$ for some $s\le\ell\le r$. Let the spectral decomposition of $U_\bot J U^\top_\bot$ be $U_\bot J U^\top_\bot=U_{\eta} \varSigma_\eta U_{\eta}^\top$ with $U_{\eta}\in\stiefel(n,\ell-s)$. Take ${U}_{\eta\bot}$ such that $[U\ U_\eta\ U_{\eta\bot}]\in\orth(n)$, and denote $U^+=[U\ U_{\eta}]$. It holds that
    \begin{equation}\label{eq:tangenttwo_Sr}
        \tangenttwo_{\boundedranks}(X;\eta)=\hkh{2\eta X^\dagger\eta + [U^+\ U_{\eta\bot}]\left[ \begin{matrix}
        	W_1&		W_2\\
        	W_2^\top&		L\\
        \end{matrix} \right][U^+\ U_{\eta\bot}]^\top \left|\,\begin{array}{l}
        W_1\in\sym(\ell),
        \\
        W_2\in\mbR^{\ell\times (n-\ell)},
        \\
        L\in\sym_{\le r-\ell}(n-\ell)
        \end{array}\right.
        }.
    \end{equation}
\end{proposition}
\begin{proof}
    See Appendix~\ref{app:SU}.
\end{proof}

\smallskip

Furthermore, imposing an additional constraint $\uanifold$ on $\boundedranks$ has recently received increasing attention. A representative example is $\uanifold=\{X\in\sym(n)\mid \|X\|_{\frob}^2-1=0\}$~\cite{cason2013iterative,li2020jotaspectral}. Extending the spirit of the rule~\eqref{eq:tangentcupXi} and the decomposition~\eqref{eq:characterization_Sr}, we have $\sym_{\le r}(n)\cap\uanifold=\bigcup_{j=1}^{r+1} (\sanifold_j\cap\uanifold)$, and thus it suffices to compute the tangent sets to each $\sanifold_j\cap\uanifold$, followed by assembling them to obtain those to the union $\sym_{\le r}(n)\cap\uanifold$; the results are provided in Appendix~\ref{app:SU}.

\subsubsection{Low-rank positive semidefinite matrices}\label{sec:lowrank_posi}
As shown by Pataki~\cite{pataki1998lowranksolution}, semidefinite programs (SDPs) often admit low-rank solutions, underlining the importance of the geometry of low-rank positive semidefinite matrices:
\begin{equation}\label{eq:boundedranksdp}
    \boundedranksdp(n) = \hkh{X \in \sym(n)\ |\  X\succeq 0,\,\rank(X) \leq r}.
\end{equation}
It is noteworthy to observe that $\boundedranksdp(n)$ coincides with $\sanifold_{r+1}$ defined in~\eqref{eq:characterization_Sr}, i.e.,
\begin{equation}\label{eq:characterization_S+r}
    \boundedranksdp(n)=\sanifold_{r+1}=\{X\in\sym(n)\mid\lambda_{r+1}(X)=0,\lambda_n(X)=0\}.
\end{equation}
The equality holds from the non-increasing ordering of eigenvalues, $\lambda_1(X)\ge\cdots\ge\lambda_r(X)\ge 0 = \lambda_{r+1}(X)=\cdots=\lambda_n(X)$. Since the tangent sets to the $\sanifold_j$ in~\eqref{eq:characterization_Sr} have been derived in Proposition~\ref{pro:tangenttoSj}, we specify the computation in the following proposition by taking $j=r+1$.

\begin{proposition}\label{pro:tangenttoS+r}
    Given $X\in\boundedranksdp(n)$ with $\rank{(X)}=s$ and the spectral decomposition $X=U\varLambda U^\top$, where $U\in\stiefel(n,s)$. The tangent cone to $\boundedranksdp(n)$ can be characterized by
    \begin{equation*}
        \Btangent_{\boundedranksdp(n)}(X)=\hkh{[U\ U_{\bot}]\left[ \begin{matrix}
        	W_1&		W_2\\
        	W_2^\top&		J\\
        \end{matrix} \right][U\ U_{\bot}]^\top \left|\,\begin{array}{l}
        W_1\in\sym(s),
        \\
        W_2\in\mbR^{s\times (n-s)},
        \\
        J\in\sym^+_{\le r-s}(n-s)
        \end{array}\right.
        }.
    \end{equation*}
    Additionally, given a direction $\eta\in\Btangent_{\boundedranksdp(n)}(X)$ parameterized in the above manner with $\rank(J)=\ell-s$ for some $s\le\ell\le r$. Let the spectral decomposition of $U_\bot J U^\top_\bot$ be $U_\bot J U^\top_\bot=U_{\eta} \varSigma_\eta U_{\eta}^\top$ with $U_{\eta}\in\stiefel(n,\ell-s)$. Take ${U}_{\eta\bot}$ such that $[U\ U_\eta\ U_{\eta\bot}]\in\orth(n)$, and denote $U^+=[U\ U_{\eta}]$. It holds that
    \begin{equation*}
        \tangenttwo_{\boundedranksdp(n)}(X;\eta)=\hkh{2\eta X^\dagger\eta + [U^+\ U_{\eta\bot}]\left[ \begin{matrix}
        	W_1&		W_2\\
        	W_2^\top&		L\\
        \end{matrix} \right][U^+\ U_{\eta\bot}]^\top \left|\,\begin{array}{l}
        W_1\in\sym(\ell),
        \\
        W_2\in\mbR^{\ell\times (n-\ell)},
        \\
        L\in\sym^+_{\le r-\ell}(n-\ell)
        \end{array}\right.
        }.
    \end{equation*}
\end{proposition}
\begin{proof}
Setting $j=r+1$ in Proposition~\ref{pro:tangenttoSj}, the associated quantities $s_+$ and $\ell_+$ reduce to $s_+=s$ and $\ell_+=\ell-s$, respectively, and thus the tangent sets of $\sanifold_j$ in~\eqref{eq:tangenttoSj} and~\eqref{eq:tangenttwotoSj} reduce directly to those of $\sanifold_{r+1}=\boundedranksdp(n)$.
\end{proof}

Subsequently, taking into account an additional structured set $\uanifold\subseteq\sym(n)$, the geometry of the coupled set $\boundedranksdp(n)\cap\uanifold$ becomes more complicated. Typically, low-rank SDPs with linear equality constraints have found a broad range of applications, which motivates the study of $\boundedranksdp(n)\cap\uanifold$ with $\uanifold=\{X\in\sym(n)\mid\aanifold(X)-b=0\}$~\cite{boumal2020deterministic,levin2025effect}. In this case, the tangent cone to $\boundedranksdp(n)\cap\uanifold$ was first derived in \cite[Corollary 4.12]{levin2025effect}, and we note that applying the established Theorem~\ref{the:expressions_McapK} not only recovers the first-order result directly, but also identifies the second-order tangent set to $\boundedranksdp(n)\cap\uanifold$. Specifically, we employ the Burer--Monteiro parameterization~\cite{burer2003BM}

\begin{equation}\label{eq:BMpara}
    (\manifoldBM,\phiBM)=(\mbR^{n\times r},\, R\mapsto RR^\top).
\end{equation}
Since $\manifoldBM$ is the whole space $\mbR^{n\times r}$, Assumption~\ref{assu:MintersecK}(ii) holds whenever $\buanifold:=\phiBM^{-1}(\uanifold)$ is a manifold around $R$, and $\phiBM$ is open at every $R\in\mbR^{n\times r}$ by~\cite[Theorem~2.3 and Proposition~2.7]{levin2025effect}. Hence Theorem~\ref{the:expressions_McapK} yields the following corollary.
\begin{corollary}\label{cor:expressions_S+capU}
    Suppose that $h:\sym(n)\to\mbR^q$ is smooth and $\uanifold=\{\tilde{X}\in\sym(n)\mid h(\tilde{X})=0\}$ satisfies Assumption~\ref{assu:errorbound} at the point $X\in\boundedranksdp(n)\cap\uanifold$. Additionally, the differential of $h\circ\phiBM$ has constant rank in a neighborhood of some $R\in\phiBM^{-1}(X)$. Then, the following intersection rules hold,
    \begin{equation}\label{eq:expressions_S+capU}
    \begin{aligned}
        \Btangent_{\boundedranksdp(n)\cap\uanifold}(X) &= \Btangent_{\boundedranksdp(n)}(X)\cap\Btangent_{\uanifold}(X),
        \\
        \tangenttwo_{\boundedranksdp(n)\cap\uanifold}(X;\eta) &= \tangenttwo_{\boundedranksdp(n)}(X;\eta)\cap\tangenttwo_{\uanifold}(X;\eta)\ \ \text{for any}\ \eta\in\Btangent_{\boundedranksdp(n)\cap\uanifold}(X).
    \end{aligned}
    \end{equation}
\end{corollary}
For an affine set $\uanifold$, the assumptions of Corollary~\ref{cor:expressions_S+capU} are verified in Appendix~\ref{app:S+U}.

\section{Tangent sets bridge optimization landscapes}\label{sec:tangent_bridge_land}
Consider a general constrained optimization problem as follows,
\begin{equation*}\label{eq:general_xanifold}
    \min_{X\in\manifold}\ f(X),
\end{equation*}
where the feasible region $\manifold$ is possibly nonsmooth and nonconvex, with a typical example being $\manifold=\boundedrank$. To circumvent the irregularity of $\manifold$, a common approach is to employ the technique of {smooth parameterization}~\cite{levin2023remedy,rebjock2024boundedrank,gao2024desingularizationtensor,levin2025effect,yang2025spacedecouple}, which introduces a smooth manifold $\bmanifold$ to (over)parameterize $\manifold$, thereby inducing a Riemannian optimization problem; see \myfig\ref{fig:diagram} for an illustration and see~\cite{absil2008optimization,boumal2023introduction} for more background of Riemannian optimization.

Then, the central question is: what is the relationship between the landscapes---or more precisely, the stationary points---of the reformulated problem~\eqref{eq:PM} and the original problem~\eqref{eq:P}? To answer this, we unveil that the first- and second-order tangent sets to $\manifold$ and $\bmanifold$ play a crucial role, indeed bridging the landscapes of the two optimization problems. Specifically, section~\ref{sec:smoothpara_opt} reviews existing results on smooth parameterization. More essentially, section~\ref{sec:equi_2ndpoints} presents the key finding: we precisely identify a sufficient and necessary condition under which the second-order stationary points of~\eqref{eq:PM} map to those of~\eqref{eq:P}.

\subsection{Smooth parameterization for optimization problems}\label{sec:smoothpara_opt}
We begin by revisiting the first- and second-order stationarity for a constrained optimization problem; see \cite{ruszczynski2006nonlinearopt}. Specifically, consider the problem $\min_{X\in\manifold} f(X)$. We say a point $X^*\in\manifold$ is \emph{first-order stationary} if $\projection_{\Btangent_{\manifold}(X^*)}(-\nabla f(X^*))=0$, or equivalently, $-\nabla f(X^*)\in\Fnormal_\manifold(X^*)$; and $X^*$ is \emph{second-order stationary} if, in addition, it satisfies
\begin{equation}\label{eq:def2optcon}
    \innerp{\nabla f(X^*),\zeta}+\innerp{\eta,\nabla^2 f(X^*)[\eta]} \ge 0,\quad\text{for all}\ \zeta\in\tangenttwo_\manifold(X^*;\eta),
\end{equation}
for every $\eta\in\Btangent_\manifold(X^*)$ such that $\innerp{\nabla f(X^*),\eta}=0$. The definitions are general, and apply analogously when the region~$\manifold$ is replaced by other sets, such as~$\bmanifold$.

Implementing algorithms directly on~$\manifold$ may suffer from the nonsmoothness. To address this, smooth parameterization is introduced~\cite{levin2023remedy,rebjock2024boundedrank,levin2025effect,yang2025spacedecouple}. Specifically, as demonstrated in \myfig\ref{fig:diagram}, let $\manifold$ and $\bmanifold$ denote a (possibly nonsmooth) set in $\eanifold$ and a smooth manifold embedded in $\wanifold$, respectively, and let $\phi:\wanifold\to\eanifold$ be a smooth mapping between the two Euclidean spaces such that $\phi(\bmanifold)=\manifold$. Through the parameterization~$(\bmanifold,\phi)$, the original nonsmooth problem~\eqref{eq:P} can thus be cast as a smooth optimization problem on the manifold~\eqref{eq:PM}. 

However, nonlinear parameterizations may distort the optimization landscape, underscoring the importance of studying the relationship between the stationary points of the two problems. More formally, we say that the parameterization $(\bmanifold,\phi)$ satisfies ``${p\Rightarrow q}$" ($p,q=1$ or $2$) at $Y$, if for any objective function $f$, $Y$ being a $p$-th-order stationary point for problem \eqref{eq:PM} implies that $\phi(Y)$ is a $q$-th-order stationary point for problem~\eqref{eq:P}.

Levin et al.~\cite{levin2025effect} established a comprehensive analysis to answer when the implications ``$1\!\Rightarrow\!1$" or ``$2\!\Rightarrow\!1$" hold. Specifically, let $X=\phi(Y)$, and it is proved in~\cite[Theorem 2.4]{levin2025effect} that ``$1\!\Rightarrow\!1$" holds at $Y$ if and only if the image $\ima(\diff \phi_Y(\tangent_{\bmanifold}(Y)))=\Btangent_\manifold(X)$---that is, the differential $\diff \phi$ fully preserves the information contained in the Bouligand tangent cone. Moreover, \cite[Theorem 3.23]{levin2025effect} provides several conditions to guarantee the property ``$2\!\Rightarrow\!1$". These results collectively reveal that the parameterization approach is effective in finding first-order stationary points on the nonsmooth $\manifold$.

However, exactly identifying the conditions under which \twototwo\ holds remains an open problem, since second-order stationarity on the (possibly nonsmooth) set $\manifold$ is more involved than its first-order counterpart, as remarked in~\cite[\S 6]{levin2025effect}. 

\subsection{Equivalence between second-order stationary points}\label{sec:equi_2ndpoints}
We provide in this section a sufficient and necessary condition to characterize when second-order stationary points of~\eqref{eq:PM} map to those of~\eqref{eq:P}, i.e., \twototwo\ holds. We define the following mappings, which borrow the idea from \cite{levin2025effect},
\begin{align*}
    &\tanL_Y:\tangent_{\bmanifold}(Y)\to\eanifold:\ v\mapsto \diff\phi_Y[v],\ \ \text{for}\ Y\in\bmanifold,
    \\
    &\tanQ_{Y,v}:\tangenttwo_{\bmanifold}(Y;v)\to\eanifold:\ u_v\mapsto \diff\phi_Y[u_v]+\diff^2\phi_Y [v,v],\ \ \text{for}\ Y\in\bmanifold\ \text{and}\ v\in\tangent_{\bmanifold}(Y),
\end{align*}
where $\diff\phi_Y$ and $\diff^2\phi_Y$ denote the first and second derivatives of $\phi$ at $Y$ on the ambient space $\wanifold$. Let $X=\phi(Y)$. In fact, the mappings $\tanL$ and $\tanQ$ convey the geometric information encoded in the tangent sets of $\bmanifold$ to those of $\manifold$. To see this, we note that for the manifold $\bmanifold$, any $v\in\tangent_{\bmanifold}(Y)$ admits a curve $\gamma(t)$ such that $\gamma(0)=Y,\gamma^\prime(0)=v$, and thus $\tanL_Y(v)=(\phi\circ\gamma)^{\prime}(0)$, which implies $\ima(\tanL_Y)\subseteq\tangent_{\manifold}(X)$. Moreover, it is revealed from \cite[Proposition 13.13]{rockafellar2009variationalanalysis} that given any $u_v\in\tangenttwo_{\bmanifold}(Y;v)$ associated with $v\in\Btangent_{\bmanifold}(Y)$, there exists a curve $\beta(t)$ satisfying $\beta(0)=Y$, $\beta^\prime(0)=v$, and $\beta^{\prime\prime}(0)=u_v$. Hence, we have $\tanQ_{Y,v}(u_v)=(\phi\circ \beta)^{\prime\prime}(0)$, indicating that $\ima(\tanQ_{Y,v})\subseteq \tangenttwo_{\manifold}(X;\tanL_Y(v))$. The construction of the curve $\beta(t)$ indeed introduces the concept of \emph{second fundamental form}; see \cite[\S 6.2]{do1992riemannian} and~\cite[\S 5.11]{boumal2023introduction} for more details. Specifically, we can define a quadratic form\footnote{Different choices of $\beta$ yield equal $\beta^{\prime\prime}(0)$, modulo $\tangent_{\bmanifold}(Y)$, ensuring no ambiguity in the definition of $\secfunda_Y(v,v)$; see, e.g.,~\cite[formula~(3.7)]{levin2025effect}.} $\secfunda_Y(v, v) = \projection_{\normal_{\bmanifold}(Y)}(\beta''(0))$, for $v\in\tangent_{\bmanifold}(Y)$, which induces the second fundamental form $\secfunda_Y(v,v^\prime)$:
\begin{equation*}
    \secfunda_Y(v, v^\prime) := \frac{1}{4}\left[\secfunda_Y(v+v^\prime, v+v^\prime) - \secfunda_Y(v-v^\prime, v-v^\prime)\right],
\end{equation*}
for $(v,v^\prime)\in\tangent_{\bmanifold}(Y)\times\tangent_{\bmanifold}(Y)$. In this view, we have another characterization of the second-order tangent set to the manifold, 
\[\tangenttwo_{\bmanifold}(Y; v) = \secfunda_Y(v, v) + \tangent_{\bmanifold}(Y);\]
see Lemma~\ref{lem:funda_tangent2}. Therefore, the image of $\tanQ_{Y,v}$ is indeed a translate of $\ima(\tanL_Y)$, i.e.,
\begin{equation}\label{eq:imaQ_secfunda}
    \ima(\tanQ_{Y,v}) = \diff\phi_Y[\secfunda_Y(v, v)] + \diff^2\phi_Y[v, v] + \ima(\tanL_Y).
\end{equation}

It is concluded in \cite[Theorem 2.4]{levin2025effect} that ``$1\!\Rightarrow\!1$" fails at $Y$ only when $\tanL_Y$ loses information, i.e., $\ima(\tanL_Y)\subsetneq\tangent_{\manifold}(X)$, which suggests that a ``comparison'' between the second-order tangent sets to $\bmanifold$ and to $\manifold$ through $\phi$ would facilitate characterizing the \twototwo\ property. This, however, appears more intricate than the first-order counterpart, suffering from two pains: 1) a direction $\eta=\tanL_Y(v)$ may admit multiple preimages $v^\prime$ under $\tanL_Y$; 2) the asymptotic behavior of $v_i$ approaching such $v^\prime$ (in the sense of $\tanL_Y(v_i)\to\tanL_Y(v^\prime)=\eta$) also plays a role, which necessitates aggregating the images of $\tanQ_{Y,v_i}$ to capture $\tangenttwo_\manifold(X;\eta)$. Motivated by the two considerations, we formalize the idea in Theorem~\ref{the:2to2}.

Before delving into the analysis, we present some basic computations. The derivatives of $\bar{f}=f\circ\phi$ can be computed as follows,
\begin{equation}\label{eq:compute_barf}
    \nabla \bar{f}(Y)=\diff\phi^*_Y(\nabla f(X))\ \ \text{and}\ \ \nabla^2 \bar{f}(Y) = \diff\phi^*_Y\circ \nabla^2 f(X)\circ\diff\phi_Y + \nabla^2\phi_{\nabla f(X)}(Y),
\end{equation}
where $\diff\phi^*_Y:\eanifold\to\wanifold$ is the adjoint of $\diff\phi_Y$, and the mapping $\phi_\eta:\wanifold\to\mbR$ is given by $\phi_{\eta}(W)=\innerp{\eta,\phi(W)}$ for any $(\eta,W)\in(\eanifold,\wanifold)$. Then, given $v\in\tangent_{\bmanifold}(Y)$ and $u_v\in\tangenttwo_{\bmanifold}(Y;v)$, we can specify the following computation,
\begin{align}
        &\ \innerp{\nabla \bar{f}(Y),u_v}+\innerp{v,\nabla^2 
        \bar{f}(Y)[v]}  \nonumber
        \\
        =&\ \innerp{\diff\phi^*_Y(\nabla f(X)),u_v} + \left\langle v, \kh{\diff\phi^*_Y\circ \nabla^2 f(X)\circ\diff\phi_Y + \nabla^2\phi_{\nabla f(X)}(Y)}[v]\right \rangle    \nonumber
        \\
        =&\ \innerp{\nabla f(X),\diff\phi_Y(u_v)} + \left\langle\tanL_Y(v),\nabla^2 f(X)[\tanL_Y(v)]\right \rangle + \innerp{\nabla f(X),\diff^2\phi_Y[v,v]}    \nonumber
        \\
        =&\ \innerp{\nabla f(X),\tanQ_{Y,v}(u_v)}  + \left\langle\tanL_Y(v),\nabla^2 f(X)[\tanL_Y(v)]\right \rangle, \label{eq:2order_deri}
\end{align}
where the first equality comes by substituting the derivatives of $\bar{f}$~\eqref{eq:compute_barf}, and the second equality holds by identifying $\diff\phi_Y=\tanL_Y$ on $\tangent_{\bmanifold}(Y)$. Moreover, we say that a sequence $\{\zeta_i+Z\}_{i\in\mathbb{N}}$ of translates of a subspace $Z\subseteq\eanifold$ converges (necessarily to a translate of $Z$) if there exists a sequence $\{z_i\}_{i\in\mathbb{N}}\subseteq Z$ such that $\{\zeta_i+z_i\}_{i\in\mathbb{N}}$ converges. 

In Theorem~\ref{the:2to2}, We resort to the concept of the following sets:
\begin{align}
    &\qanifold_Y(v):=\bigcup_{\{v_i\}_{i\in\mbN}:\,\tanL_Y(v_i)\to \tanL_Y(v)}\lim_{i\to\infty}\kh{\tanQ_{Y,v_i}(u_{v_i}) + \ima(\tanL_{Y})},\ \ \text{for}\ v\in\tangent_{\bmanifold}(Y),  \nonumber
    \\
    &\varGamma_Y:=\hkh{\projection_{\Fnormal_{\manifold}(\phi(Y))}\kh{\diff\phi_Y[\secfunda(v_0,v_1)]+\diff^2\phi_Y[v_0,v_1]}\mid v_0\in\kernel(\tanL_Y),\,v_1\in\tangent_{\bmanifold}(Y)}. \label{eq:defGammaY}
\end{align}
We use $\closeconv(\cdot)$ to denote the closed convex hull of a set, and propose a sufficient and necessary condition for the property \twototwo. More specifically, when the condition is violated, the failure of \twototwo\ can be witnessed by an explicitly constructed objective. The analysis invokes several lemmas capturing the properties of the constructed sets $\qanifold_Y(v)$ and $\varGamma_Y$, which are presented in Appendix~\ref{app:lemmas2to2}.

\begin{theorem}\label{the:2to2}
    Let $X=\phi(Y)$ and suppose that $\tangenttwo_\manifold(X;\eta)\neq\emptyset$ for all $\eta\in\Btangent_\manifold(X)$. The parameterization $(\bmanifold,\phi)$ satisfies ``\,$2\!\Rightarrow\!2$'' at $Y\in\bmanifold$ if and only if $\ima(\tanL_Y)=\Btangent_{\manifold}(X)$, and for all $v\in\tangent_{\bmanifold}(Y)$,
    \begin{equation}\label{eq:nece_suff_2to2}
        \!\!\tangenttwo_\manifold(X;\tanL_Y(v)) \subseteq \closeconv \kh{\qanifold_Y(v)+\qanifold_Y(0)+\varGamma_Y}.
    \end{equation}
    If ``\,$2\!\Rightarrow\!2$'' does not hold, there always exists a smooth function $f$ such that $Y$ is second-order stationary for~\eqref{eq:PM} while $\phi(Y)$ is not second-order stationary for~\eqref{eq:P}.
\end{theorem} \begin{proof}
    See Appendix~\ref{app:lemmas2to2}.
\end{proof}

To broaden the applicability, we extend the result to composition of parameterizations, which is inspired by~\cite[\S 3.3]{levin2025effect}

\begin{proposition}\label{pro:composition}
    Let $(\bmanifold,\phi)$ be a smooth parameterization of $\manifold$. Given another smooth manifold $\zanifold$, let $\varphi:\zanifold\to\bmanifold$ be a smooth mapping such that $\psi:=\phi\circ\varphi$ is surjective. Then $(\zanifold,\psi)$ is a smooth parameterization of $\manifold$. Moreover, for $Z\in\zanifold$ and $Y:=\varphi(Z)\in\bmanifold$, the following properties hold.
    \begin{itemize}
        \item [\emph{(i)}] If $(\zanifold,\psi)$ satisfies ``\,$2\!\Rightarrow\!2$'' at $Z$, then $(\bmanifold,\phi)$ satisfies ``\,$2\!\Rightarrow\!2$'' at $Y$.
        \item [\emph{(ii)}] If $\varphi$ is a submersion at $Z$ and $(\bmanifold,\phi)$ satisfies ``\,$2\!\Rightarrow\!2$'' at $Y$, then $(\zanifold,\psi)$ satisfies ``\,$2\!\Rightarrow\!2$'' at $Z$.
    \end{itemize}
\end{proposition}
\begin{proof}
    See Appendix~\ref{app:lemmas2to2}.
\end{proof}

In summary, Theorem~\ref{the:2to2} establishes that the parameterization $(\bmanifold,\phi)$ avoids introducing spurious second-order stationary points only when the mappings $\tanL$ and $\tanQ$ preserve the tangent information, that is, $\ima(\tanL_Y)=\Btangent_{\manifold}(X)$ and condition~\eqref{eq:nece_suff_2to2} holds. Moreover, Proposition~\ref{pro:composition} offers an additional perspective for verifying the \twototwo\ property through the lens of compositions. These results will later be applied to two well-known parameterizations of $\boundedrank$, the LR parameterization and the desingularization~\cite{khrulkov2018desingularization,rebjock2024boundedrank,levin2025effect}, enabling us to exactly determine the points on $\boundedrank$ where the \twototwo\ property holds, and in section~\ref{sec:PSD_BM} to the Burer--Monteiro parameterization.

\section{Second-order optimality on bounded-rank matrices }\label{sec:SecondorderstationaryMr}

This section applies the framework developed in sections~\ref{sec:2tangentset}-\ref{sec:tangent_bridge_land} to low-rank optimization problems. We begin with problem~\eqref{eq:boundedrankopt}, where the feasible region is the matrix variety $\boundedrank$ coupled with an additional structured constraint $\hanifold$. By substituting the characterizations of the first- and second-order tangent sets, we derive the first- and second-order optimality conditions in section~\ref{sec:SOCMH}. Then, the focus is shifted to the scenario when $\hanifold=\mbRmn$, leading to the following formulation,
\begin{equation}\label{eq:boundedrank_noH}
\begin{aligned}
    \min_{X\in\mbR^{m\times n}}\ \ & f(X)\\[-1.5mm]
    \mathrm{s.\,t.}\ \ \ \ & X\in\boundedrank,
    \end{aligned}
\end{equation}
which is of independent interest~\cite{schneider2015Lojaconvergence,levin2023remedy}. It is shown in section~\ref{sec:NP_versec} that checking second-order optimality over $\boundedrank$ is NP-hard in general. Nevertheless, in section~\ref{sec:parameter_matvariety}, we clarify that second-order stationarity remains attainable in certain cases, through the lens of smooth parameterizations.

\subsection{Second-order optimality conditions}\label{sec:SOCMH}
Regarding the feasible region $\boundedrank\cap\hanifold$ of problem~\eqref{eq:boundedrankopt}, we recall the intersection rules developed in section~\ref{sec:lowrank_rectangular},
\begin{equation}\label{eq:calculationrule_H}
    \begin{aligned}
        \tangent_{\boundedrank\cap\hanifold}(X)&=\tangent_{\boundedrank}(X)\cap\tangent_{\hanifold}(X),
        \\
        \tangenttwo_{\boundedrank\cap\hanifold}(X;\eta)&=\tangenttwo_{\boundedrank}(X;\eta)\cap\tangenttwo_{\hanifold}(X;\eta),\ \ \text{for all}\ \eta\in\tangent_{\boundedrank\cap\hanifold}(X),
    \end{aligned}
\end{equation}
where $\hanifold$ accommodates the ambient space $\mbRmn$, the affine manifold~\cite{li2023normalboundedaffine}, the Frobenius sphere~\cite{cason2013iterative,yang2025spacedecouple}, the oblique manifold~\cite{yang2025spacedecouple}, and the hyperbolic manifold~\cite{jawanpuria2019lowrankhyperbolic}, as outlined in Table~\ref{tab:tangentsets} and~\eqref{eq:exampleHi}. Then the first- and second-order optimality conditions are derived in the following proposition for $\hanifold$ in~\eqref{eq:exampleHi}.

\begin{proposition}\label{pro:2order_opt_H}
     Given $X^*\in\boundedrank\cap\hanifold$, it is a first-order stationary point for problem~\eqref{eq:boundedrankopt} if for all $\eta\in\tangent_{\boundedrank}(X^*)\cap\tangent_{\hanifold}(X^*)$, it holds that $\innerp{\nabla f(X^*),\eta}=0$; and $X^*$ is second-order stationary if it additionally satisfies that for all $\eta\in\tangent_{\boundedrank}(X^*)\cap\tangent_{\hanifold}(X^*)$,
     \begin{equation*}
        \innerp{\nabla f(X^*),\zeta}+\innerp{\eta,\nabla^2 f(X^*)[\eta]} \ge 0,\quad\text{for all}\ \zeta\in\tangenttwo_{\boundedrank}(X^*;\eta)\cap\tangenttwo_{\hanifold}(X^*;\eta).
    \end{equation*}
\end{proposition}
\begin{proof}
    Notice that each $\hanifold$ in~\eqref{eq:exampleHi} is a manifold, which implies that $\tangent_\hanifold(X^*)$ is always a linear space. Combining this observation with the expression~\eqref{eq:Btangent_cone_boundedrank} and the rule~\eqref{eq:calculationrule_H} reveals that a direction $\eta\in\tangent_{\boundedrank\cap\hanifold}(X^*)$ if and only if all $\rho\in\mbR$ yield $\rho\eta\in\tangent_{\boundedrank\cap\hanifold}(X^*)$. Hence we have $-\nabla f(X^*)\in\Fnormal_{\boundedrank\cap\hanifold}(X^*)$ if and only if $\innerp{\nabla f(X^*),\eta}=0$ for all $\eta\in\tangent_{\boundedrank\cap\hanifold}(X^*)$, which, together with~\eqref{eq:calculationrule_H}, can be substituted into~\eqref{eq:def2optcon} to produce the desired conclusion.
\end{proof}

Taking $\hanifold=\mbRmn$ in~\eqref{eq:boundedrankopt}, the problem reduces to \eqref{eq:boundedrank_noH}---in this case, specifying the computation in Proposition~\ref{pro:2order_opt_H} is of independent interest.
\begin{proposition}\label{pro:sec_opt_Mr}
     Given $X^*\in\boundedrank$ with $\rank(X^*)=s$, it is a second-order stationary point for~\eqref{eq:boundedrank_noH} if it satisfies
     \begin{equation}\label{eq:Mr_2order_opt}
         \begin{cases}
            \nabla_{\fixedrank} f(X^*)=0\ \ \text{and}\ \ \nabla^2_{\fixedrank} f(X^*)\succeq 0,\ \ \hspace{3.4cm}\text{if}\ s=r,
    	\\
            \nabla f(X^*)=0\ \ \text{and}\ \ \innerp{\eta,\nabla^2 f(X^*)[\eta]}\ge 0\ \text{for all}\ \eta\in\Btangent_{\boundedrank}(X^*),\ \ \text{if}\ s<r,
            \\
        \end{cases}
     \end{equation}
     where $\nabla_{\fixedrank}$ and $\nabla^2_{\fixedrank}$ denote the Riemannian gradient and Riemannian Hessian on $\fixedrank$, respectively.
\end{proposition}
\begin{proof}
    The computation of Riemannian derivatives on $\fixedrank$ can be found in~\cite[\S7.5]{boumal2023introduction}. Let the SVD of $X^*$ be $X^*=U\varSigma V^\top$. When $s = r$, the condition $-\nabla f(X^*) \in \normal_{\boundedrank}(X^*)$ is equivalent, by~\eqref{eq:Tcone_lowrank}, to $\nabla f(X^*)-P_{U_\bot}\nabla f(X^*)P_{V_\bot}=0$, which, in turn, is equivalent to the Riemannian first-order optimality condition $\nabla_{\fixedrank} f(X^*)=\projection_{\tangent_{\fixedrank}X^*}(\nabla f(X^*))=0$. For the second-order condition, substituting the explicit form of $\tangenttwo_{\boundedrank}(X^*;\eta)$ in~\eqref{eq:tangenttwo_fixedrank} into~\eqref{eq:def2optcon} yields
    \begin{equation*}
        \begin{aligned}
           0\le &\ \innerp{\nabla f(X^*),\zeta}+\innerp{\eta,\nabla^2 f(X^*)[\eta]} 
            \\
            =&\ \innerp{\nabla f(X^*),2\eta X^{*\dagger}\eta + [U\ U_{\bot}]\left[ \begin{matrix}
                W_1&		W_2\\
                W_3&		0\\
            \end{matrix} \right][V\ V_{\bot}]^\top} + \innerp{\eta,\nabla^2 f(X^*)[\eta]}
            \\
            =&\ \innerp{\nabla f(X^*),2\eta X^{*\dagger}\eta} + \innerp{\eta,\nabla^2 f(X^*)[\eta]},
        \end{aligned}
    \end{equation*}
    where the last equality holds from $\nabla f(X^*)=P_{U_\bot}\nabla f(X^*)P_{V_\bot}$. The condition $\innerp{\nabla f(X^*),2\eta X^{*\dagger}\eta} + \innerp{\eta,\nabla^2 f(X^*)[\eta]}\ge 0$ for all $\eta\in\tangent_{\fixedrank}(X^*)$ is equivalent to $\nabla^2_{\fixedrank} f(X^*)\succeq 0$, as stated in \cite[\S 7.5]{boumal2023introduction}.

    When $s<r$, considering the expression~\eqref{eq:Fnormal_cone_boundedrank}, the first-order optimality requires that $-\nabla f(X^*)\in\Fnormal_{\boundedrank}(X^*)=\{0\}$. Taking $\nabla f(X^*)=0$ in~\eqref{eq:def2optcon} reveals that $\innerp{\eta,\nabla^2 f(X^*)[\eta]}\ge 0\ \text{for all}\ \eta\in\Btangent_{\boundedrank}(X^*)$.
\end{proof}

In fact, for the optimization problem over the determinantal variety~\eqref{eq:boundedrank_noH}, Proposition~\ref{pro:sec_opt_Mr} shows that the second-order optimality condition coincides with the Riemannian one at points of rank~$r$, whereas, at rank-deficient points, i.e., $\rank(X^*)<r$, the second-order condition requires that the Euclidean Hessian admit no negative curvature directions along the tangent cone.

\subsection{{NP-hardness} of verifying second-order optimality}\label{sec:NP_versec}
Building on the optimality condition characterized in \eqref{eq:Mr_2order_opt}, one might expect that identifying a second-order stationary point would be a tractable task. However, we find that this problem is NP-hard in general. Worse still, unless P=NP, there does not exist a fully polynomial-time algorithm for deciding whether a given point is second-order stationary in an approximate sense. This section aims to establish the hardness of verifying second-order optimality for \eqref{eq:boundedrank_noH} through a reduction from a combinatorial problem that is known to be NP-complete.

Given a point $X\in\boundedrank$ with $\rank(X)=s<r$ and the SVD $X=U\varSigma V^\top$, checking the first-order optimality can be accomplished in polynomial time~\cite{cason2013iterative,schneider2015Lojaconvergence}. Hence, according to~\eqref{eq:Mr_2order_opt}, the main bottleneck lies in deciding the existence of a negative curvature direction $\eta\in\Btangent_{\boundedrank}(X)$, which can be formulated as the following problem
\begin{equation}\label{eq:lreig}
\begin{aligned}
    \min_{\eta\in\mbR^{m\times n}}\ \ & \innerp{\eta,\mathscr{G}(\eta)}
    \\[-1.5mm]
    \mathrm{s.\,t.}\ \ \ \ & \|\eta\|_{\frob} = 1,
    \\
    &\ \eta\in\Btangent_{\boundedrank}(X).
    \end{aligned}
\end{equation}
Note that we abstract the Hessian as a symmetric operator $\mathscr{G}:\mbRmn\to\mbRmn$. Let $\lambda^\star$ denote the optimal value, and then the task of \underline{ver}ifying \underline{s}econd-order \underline{o}ptimality \underline{c}onditions,  abbreviated as \verifysec, is equivalent to checking whether $\lambda^\star<0$.
\begin{tcolorbox}[
colframe=black,
colback=white,
boxrule=0.5pt,    
arc=2mm,          
left=2mm, right=2mm, top=2mm, bottom=2mm 
]
\textbf{Problem:} \verifysec
\\
\textbf{Input:} Parameters $m,n,r$; point $X\in\boundedrank$; symmetric operator $\mathscr{G}$.
\\
\textbf{Question:} Does the optimal value of \eqref{eq:lreig} $\lambda^\star<0$?
\end{tcolorbox}

Next, we introduce a combinatorial problem. Consider an undirected graph $G = (\vanifold, E)$ where $\vanifold=\{1,2,\ldots,n\}$ is the vertex set and $E$ is the edge set with $(i, j) \in E$ indicating the existence of an edge between vertices $i$ and~$j$; by convention, we add both $(i,j)$ and $(j,i)$ to $E$ for each such edge. Additionally, a subset $S\subseteq \vanifold$ is called a \emph{clique} if every pair of vertices in $S$ is connected by an edge in $E$, and accordingly, the \emph{clique number} of the graph is defined by $\omega(G):=\max\{|S|\mid S\subseteq \vanifold\ \text{is a clique}\}$. The \emph{clique decision problem}, abbreviated as \clique, asks whether there exists a clique of size $K$, or equivalently, whether $\omega(G)\ge K$.
\begin{tcolorbox}[
colframe=black,
colback=white,
boxrule=0.5pt,    
arc=2mm,          
left=2mm, right=2mm, top=2mm, bottom=2mm 
]
\textbf{Problem:} \clique
\\
\textbf{Input:} Undirected graph $G = (\vanifold, E)$; clique size $K$.
\\
\textbf{Question:} Does there exist a clique of size $K$ in $G$?
\end{tcolorbox}

In fact, \clique\ is one of the original $21$ NP-complete problems~\cite{karp1972reducibility}. More importantly, Motzkin and Straus \cite{motzkin1965maxima} bridged the quantity $\omega(G)$ with an optimization problem via the following formula,
\begin{equation}\label{eg:MS_omegaG}
    1-\frac{1}{\omega(G)}=\max_{x\in\mathrm{\Delta}_n}\sum_{(i,j)\in E}x_ix_j,
\end{equation}
where $\mathrm{\Delta}_n:=\{x\in\mbR^n:0\le x_i\le 1, i=1,2,\ldots,n,\,x_1+x_2+\cdots+x_n=1\}$ is the unit simplex. Let $e_1,e_2,\ldots,e_n$ be the standard basis in $\mbR^n$.

Now, we are ready to show the reduction from \clique\ to \verifysec, with the main idea outlined as follows---given an input $(G,K)$ for \clique, we can construct an input $(m,n,r,X,\mathscr{G})$ for \verifysec\ 
with $m=n$, $r=1$, $X=0$ and the symmetric operator $\mathscr{G}$ satisfying
\[
\mathscr{G}:\mbRmn\to\mbRmn:\,\eta\mapsto (1-\frac{1}{K-1})\eta-\frac{1}{4}\sum_{(i,j)\in E} (e_ie_j^\top+e_je_i^\top)\eta(e_ie_j^\top+e_je_i^\top).
\]
In this manner, $\Btangent_{\boundedrank}(X)$ coincides with $\{\eta\in\mbRmn\mid\rank(\eta)\le 1\}$, and thus~\eqref{eq:lreig} can be specified as follows,
\begin{equation}\label{eq:lreig_rank1}
\begin{aligned}
    \min_{u,v\in\mbR^n}\ \ & 1-\frac{1}{K-1} - \sum_{(i,j)\in E} \innerp{uv^\top,A_{ij}uv^\top A_{ij}}
    \\[-1.5mm]
    \mathrm{s.\,t.}\ \ \ \ & \|u\|_{\frob} = \|v\|_{\frob} = 1.
    \end{aligned}
\end{equation}
where we denote $A_{ij}=\frac{1}{2}(e_ie_j^\top+e_je_i^\top)$ for $(i,j)\in E$. Then it can be shown that $G$ admits a $K$-clique if and only if the optimal value of \eqref{eq:lreig_rank1} $\lambda^\star<0$.

\begin{theorem}[NP-hardness]\label{the:np}
The problem \clique\ is polynomially reducible to \verifysec, and thus verifying second-order optimality is NP-hard.
\end{theorem}
\begin{proof}
    Notice that the matrices $A_{ij}$ are symmetric matrices. Therefore, we  derive
    \begin{align}
            \lambda^\star=&\min_{\|u\|_\frob = \|v\|_\frob = 1}\ 1-\frac{1}{K-1} - \sum_{(i,j)\in E} \innerp{uv^\top,A_{ij}uv^\top A_{ij}}  \nonumber
            \\
            =&\min_{\|u\|_\frob = 1}\ 1-\frac{1}{K-1} - \sum_{(i,j)\in E} (u^\top A_{ij}u)^2    \label{eq:uvtou}
            \\
            =&\min_{\|u\|_\frob = 1}\ 1-\frac{1}{K-1} - \sum_{(i,j)\in E} u_i^2u_j^2    \nonumber
            \\
            =&\ \frac{1}{\omega(G)} - \frac{1}{K-1},      \label{eq:omega-K}
    \end{align}
    where~\eqref{eq:uvtou} is obtained from the Cauchy--Schwarz inequality, with equality attained at $u=v$, and~\eqref{eq:omega-K} comes by considering $\norm{u}_\frob=1$ and parameterizing $x_i=u_i^2$ in~\eqref{eg:MS_omegaG}.
    
If the graph $G$ admits a clique $S$ of size $K$, i.e., $\omega(G)>K-1$. Then, the derivation~\eqref{eq:omega-K} indicates that $\lambda^\star=\frac{1}{\omega(G)}-\frac{1}{K-1}<0$. Conversely, if there does not exist a clique of size $K$ in $G$, i.e., $\omega(G)\le K-1$, we have $\lambda^\star\ge 0$. Then the result for \verifysec\ answers \clique, which implies that \verifysec\ is NP-hard.
\end{proof}

The above theorem concludes the NP-hardness of verifying second-order optimality for the problem~\eqref{eq:boundedrankopt}. Furthermore, as we shall show, even finding an approximate solution for \eqref{eq:lreig}---which one might hope could alleviate the hardness---still does not admit a polynomial-time strategy, unless P=NP. 

Specifically, we call an algorithm a fully polynomial-time approximation scheme (or FPTAS), if given any $\varepsilon\in(0,1)$, it can return an $\varepsilon$-approximate solution $\tilde{\lambda}$ for \eqref{eq:lreig} in the sense that $\tilde{\lambda}-\lambda^\star\le\varepsilon$, and the complexity is upper bounded by a polynomial function of the problem size and $\varepsilon^{-1}$.
\begin{theorem}[No FPTAS]\label{the:np_approx}
Unless P=NP, there is no FPTAS for verifying whether a point is second-order stationary for \eqref{eq:boundedrankopt}.
\end{theorem}
\begin{proof}
    Suppose that there exists a polynomial-time scheme, and we take $\varepsilon=\frac{1}{2K(K-1)}$. Therefore, if the graph admits a $K$-size clique, the FPTAS will yield $\tilde{\lambda}\le\lambda^\star+\varepsilon=\frac{1}{\omega(G)}-\frac{1}{2}(\frac{1}{K-1}+\frac{1}{K})<0$. Conversely, if there is no $K$-size clique, we have $\tilde{\lambda}\ge\lambda^\star\ge 0$. In conclusion, we can answer \clique\ if there were an FPTAS for \verifysec, which is possible only when P=NP.
\end{proof}

\subsection{Two parameterizations for {matrix varieties}}\label{sec:parameter_matvariety}
Although section~\ref{sec:NP_versec} establishes the NP-hardness of identifying second-order stationary points in general, it remains natural to ask whether finding such points is still possible in special cases.
To this end, we resort to the technique of {smooth parameterization}~\cite{levin2025effect}, since the conditions to guarantee \twototwo\ have been developed in section~\ref{sec:equi_2ndpoints}. 

Specifically, two parameterizations for $\boundedrank$ are considered: the LR parameterization~\cite{mishra2014fixedLR},
\begin{equation}\label{eq:LR}\tag{LR}
\manifold_{\mathrm{LR}} = \mbR^{m\times r}\times \mbR^{n\times r},\ \ \ \ \ \phiLR(L,R)=LR^\top,
\end{equation}
and the desingularization~\cite{khrulkov2018desingularization,rebjock2024boundedrank},
\begin{equation}\label{eq:desing}\tag{Desing}
\!\!\manifolddesing=\{(X, G) \in \mathbb{R}^{m \times n} \times \grassmann(n, n-r)\mid X G=0\},\ \ \phidesing(X,G)=X,
\end{equation}
where the \emph{Grassmann manifold}~\cite{bendokat2024grassmann} is viewed as an embedded submanifold in $\sym(n)$, i.e., $\grassmann(n,s):=\{G\in\sym(n)\mid G^2=G,\,\rank(G)=s\}$.
More background on the two parameterizations can be found in~\cite{khrulkov2018desingularization,rebjock2024boundedrank,levin2025effect,yang2025spacedecouple}.

We now apply Theorem~\ref{the:2to2} to the two parameterizations in turn. The main principle is to compute explicitly the images of $\tanL$ and $\tanQ$ defined in section~\ref{sec:equi_2ndpoints}, and substitute them into the conditions identified in Theorem~\ref{the:2to2}, thereby determining when the parameterization produces second-order stationary points on $\boundedrank$.

\begin{proposition}\label{pro:LR_2to2}
    The LR parameterization of $\boundedrank$ given by~\eqref{eq:LR} satisfies ``\,$2\!\Rightarrow\!2$'' at $(L,R)$ if and only if $\rank(LR^\top)=r$.
\end{proposition}
\begin{proof}
If $\rank(LR^\top)<r$, the expression~\eqref{eq:Btangent_cone_boundedrank} reveals that $\Btangent_{\boundedrank}(LR^\top)$ is not a linear space, and thus $\ima(\tanL_{(L,R)})\neq \Btangent_{\boundedrank}(LR^\top)$. Applying Theorem~\ref{the:2to2} shows the necessity of $\rank(L)=\rank(R)=r$.

Then we turn to the ``if\,'' part. In preparation, we present the computations of $\tanL$ and $\tanQ$ in this context:
\begin{equation}\label{eq:LR_tanLtanQ}
    \!\!\!\tanL_{(L,R)}:(\dot{L},\dot{R})\mapsto\dot{L}R^\top+L\dot{R}^\top,\ \tanQ_{(L,R),(\dot{L},\dot{R})}: (L^\prime,R^\prime)\mapsto 2\dot{L}\dot{R}^\top  + L{R}^{\prime\top} + L^{\prime}R^{\top}.
\end{equation}
Let $X=LR^\top$ and the SVD of $X$ be $X=U\varSigma V^\top$. Then $\rank(X)=r$ indicates that $(L,R)$ can be expressed as $(L,R)=(UB,VC)$ for some invertible $B,C\in\mbR^{r\times r}$ satisfying $BC^\top=\varSigma$. Given any $\eta\in\Btangent_{\boundedrank}(X)$ written as $\eta=U\dot{A}V^\top+U\dot{B}V_\bot^\top+U_\bot\dot{C}V^\top$, we can construct $\dot{L}=U\dot{A}C^{-\top}+U_\bot\dot{C}C^{-\top}$ and $\dot{R}=V_\bot\dot{B}^\top B^{-\top}$ to obtain $\tanL_{(L,R)}(\dot{L},\dot{R})=\eta$. Hence, the arbitrariness of $\eta$ concludes that $\ima(\tanL_{(L,R)})=\Btangent_{\boundedrank}(X)$. Subsequently, we aim to show that 
\begin{equation}\label{eq:second_inclusion_LR}
    \tangenttwo_{\boundedrank}(X;\eta)=\tanQ_{(L,R),(\dot{L},\dot{R})}(L^\prime,R^\prime)+\ima(\tanL_{(L,R)})
\end{equation}
for some $(L^\prime,R^\prime)$, thereby implying the inclusion~\eqref{eq:nece_suff_2to2}. On the one hand, according to the closed-form expression~\eqref{eq:tangenttwo_fixedrank}, we have
\begin{equation}\label{eq:etaXeta_t2}
    \tangenttwo_{\boundedrank}(X;\eta)=2\eta X^\dagger\eta+\tangent_{\fixedrank}(X)= 2U_\bot \dot{C}\varSigma^{-1} \dot{B} V_\bot^\top + \tangent_{\fixedrank}(X),
\end{equation}
where the second equality follows by substituting the expressions of $\eta$ and $X^\dagger=V\varSigma^{-1}U^\top$. On the other hand, taking the expressions of $(\dot{L},\dot{R})$ into~\eqref{eq:LR_tanLtanQ} leads to
\begin{equation*}
    \tanQ_{(L,R),(\dot{L},\dot{R})}(L^\prime,R^\prime)
    \equiv 2U_\bot \dot{C} C^{-\top} B^{-1} \dot{B} V_\bot^\top 
    \ (\bmod\ \tangent_{\fixedrank}(X)),
\end{equation*}
where ``$\equiv$'' indicates equality modulo the tangent space. Substituting $BC^\top=\varSigma$ and combining the result with~\eqref{eq:etaXeta_t2} yields~\eqref{eq:second_inclusion_LR}. Therefore, applying Theorem~\ref{the:2to2} verifies the \twototwo\ property at $(L,R)$ whenever $\rank(LR^\top)=r$.
\end{proof}

\begin{proposition}\label{pro:desing_2to2}
    The desingularization of $\boundedrank$ given by~\eqref{eq:desing} satisfies ``\,$2\!\Rightarrow\!2$'' at $(X,G)$ if and only if $\rank(X) = r$.
\end{proposition}
\begin{proof}
    It has been proved in \cite[Proposition 2.9]{levin2025effect} that $\ima(\tanL_{(X,G)})=\Btangent_{\boundedrank}(X)$ if and only if $\rank(X)=r$, which validates the ``only if\,'' part of our proposition.
    
    Then we focus on the ``if\,'' part by assuming $\rank(X)=r$. Following the proof of~\cite[Theorem~3]{yang2025spacedecouple} We resort to the manifold $\mbR^{m\times r} \times \stiefel(n, r)$, together with the smooth mapping $\varphi:\mbR^{m\times r}\times \stiefel(n, r) \rightarrow \manifolddesing:(\tilde{L},\tilde{R})\mapsto (\tilde{L}\tilde{R}^\top, I-\tilde{R}\tilde{R}^\top)$, which is a submersion onto $\manifolddesing$. We then introduce the composition ${\psi}:=\phi\circ\varphi$, and aim to show the \twototwo\ property of $(\mbR^{m\times r} \times \stiefel(n, r),\psi)$, which implies the \twototwo\ property of $(\manifolddesing,\phi)$, as supported by Proposition~\ref{pro:composition}.

    Let $\psi(L,R)=X$, i.e., $X=LR^\top$, and the SVD of $X$ be $X=U\varSigma V^\top$ (which indicates that $(L,R)=(UB,VC)$ with $B=\varSigma V^\top R$ and $C=\varSigma U^\top L(L^\top L)^{-1}$). The computations of $\tanL$ and $\tanQ$ directly follows~\eqref{eq:LR_tanLtanQ}. Given any $\eta\in\Btangent_{\boundedrank}(X)$ written as $\eta=U\dot{A}V^\top+U\dot{B}V_\bot^\top+U_\bot\dot{C}V^\top$, we can construct $\dot{L}=U\dot{A}C^{-\top}+U_\bot\dot{C}C^{-\top}$ and $\dot{R}=V_\bot\dot{B}^\top B^{-\top}\in\tangent_{\stiefel(n,r)}(R)$ to obtain $\tanL_{(L,R)}(\dot{L},\dot{R})=\eta$. Hence, the arbitrariness of $\eta$ concludes that $\ima(\tanL_{(L,R)})=\Btangent_{\boundedrank}(X)$.     
    
    Then, we move on to prove \eqref{eq:second_inclusion_LR}, thereby verifying condition \eqref{eq:nece_suff_2to2}. The analysis parallels that in Proposition \ref{pro:LR_2to2}. Therefore, applying Theorem~\ref{the:2to2} confirms the \twototwo\ property of $(\mbR^{m\times r}\times\stiefel(n,r),\psi)$ at $(L,R)$, while Proposition~\ref{pro:composition} further implies that $(\manifolddesing,\phidesing)$ also satisfies the \twototwo\ property at $(X,G)$.
\end{proof}

Riemannian trust-region algorithms are guaranteed to accumulate at second-order stationary points on smooth manifolds~\cite{absil2007trustRTR,boumal2019global}, and thus Propositions~\ref{pro:LR_2to2} and~\ref{pro:desing_2to2} reveal that smooth parameterizations may find second-order stationary points on $\boundedrank$, provided that the returned point happens to have rank $r$. From this perspective, the NP-hardness characterized in section~\ref{sec:NP_versec} can be essentially attributed to the singularities of $\boundedrank$, i.e., points with rank strictly lower than $r$.

\section{\texorpdfstring{Low-rank semidefinite programs and the Burer--Monteiro parameterization}{Low-rank semidefinite programs and the Burer--Monteiro parameterization}}\label{sec:PSD_BM}

This section studies low-rank semidefinite programs with affine constraints,
\begin{equation}\label{eq:Fproblem}\tag{L-SDP}
    \min_{X\in\sym(n)}\ f(X)\quad \st\ \ X\in\boundedranksdp(n)\cap\uanifold,
\end{equation}
where $f$ is twice continuously differentiable, $\uanifold=\{X\in\sym(n)\mid\aanifold(X)=b\}$, and $\aanifold(X)_i=\innerp{A_i,X}$ with $A_i\in\sym(n)$ for $i=1,\ldots,q$. The Burer--Monteiro parameterization~\eqref{eq:BMpara}, restricted to $\buanifold=\phiBM^{-1}(\uanifold)$, yields the parameterized problem
\begin{equation}\label{eq:BM}\tag{BM}
    \min_{R\in\buanifold}\ \bar f(R):=f(RR^\top).
\end{equation}
We derive stationarity conditions for~\eqref{eq:Fproblem}, compare them with those for~\eqref{eq:BM}, and unveil the relationship between the stationary points of the two problems, borrowing some ideas from~\cite{habib2026sphere}.

Throughout, $X\in\boundedranksdp(n)\cap\uanifold$ has rank $s\le r$ and the spectral decomposition $X=U\varLambda U^\top$ as in Proposition~\ref{pro:tangenttoS+r}, and $J(\eta):=U_\bot^\top\eta U_\bot$ for $\eta\in\sym(n)$. We say that the \emph{linear independence constraint qualification} (LICQ) holds at $X$ if $A_1U,\ldots,A_qU$ are linearly independent. The proofs of the results in this section are given in Appendix~\ref{app:PSD_BM}.

\subsection{\texorpdfstring{Stationarity conditions}{Stationarity conditions}}\label{sec:PSD_SOC}

Under LICQ, Corollary~\ref{cor:expressions_S+capU} with $\tangent_\uanifold(X)=\tangenttwo_\uanifold(X;\eta)=\kernel(\aanifold)$ gives
\begin{equation}\label{eq:FtangentF}
    \begin{aligned}
        \tangent_{\boundedranksdp(n)\cap\uanifold}(X)&=\tangent_{\boundedranksdp(n)}(X)\cap\kernel(\aanifold),
        \\
        \tangenttwo_{\boundedranksdp(n)\cap\uanifold}(X;\eta)&=\tangenttwo_{\boundedranksdp(n)}(X;\eta)\cap\kernel(\aanifold)\ \ \text{for}\ \eta\in\tangent_{\boundedranksdp(n)\cap\uanifold}(X),
    \end{aligned}
\end{equation}
since LICQ makes the differential of $\bar h(R'):=\aanifold(R'R'^\top)-b$ surjective near $\phiBM^{-1}(X)$; see Appendix~\ref{app:F1}. The following proposition characterizes first-order stationarity for~\eqref{eq:Fproblem} through the lens of Lagrange multipliers.

\begin{proposition}\label{pro:Ffirst}
    Let $X\in\boundedranksdp(n)\cap\uanifold$ with $\rank(X)=s\le r$ satisfy LICQ, and let $S:=\nabla f(X)-\aanifold^*(\mu)$ with $\aanifold^*(\mu)=\sum_{i=1}^q\mu_iA_i$. If $s=r$, then $X$ is first-order stationary for~\eqref{eq:Fproblem} if and only if $SX=0$ for some $\mu\in\mbR^q$; if $s<r$, then $X$ is first-order stationary if and only if $SX=0$ and $S\succeq0$ for some $\mu\in\mbR^q$. In both cases, such $\mu$ is unique.
\end{proposition}
\begin{proof}
    See Appendix~\ref{app:F1}.
\end{proof}

We then turn to second-order stationarity. Let $X$ satisfy LICQ and be first-order stationary, with $S$ as in Proposition~\ref{pro:Ffirst}. Considering condition~\eqref{eq:def2optcon}, the next proposition shows that, for $\eta\in\tangent_{\boundedranksdp(n)\cap\uanifold}(X)\cap\vecspan\{\nabla f(X)\}^\bot$, the minimum of the left-hand side of~\eqref{eq:def2optcon} over $\zeta\in\tangenttwo_{\boundedranksdp(n)\cap\uanifold}(X;\eta)$ is the quadratic form
\begin{equation}\label{eq:FQ}
    Q_S(\eta):=\innerp{\eta,\nabla^2f(X)[\eta]}+2\innerp{S,\eta X^\dagger\eta},
\end{equation}
and consequently, $Q_S\ge0$ characterizes second-order stationarity.

\begin{proposition}\label{pro:Fopt}
    Let $X\in\boundedranksdp(n)\cap\uanifold$ with $\rank(X)=s\le r$ satisfy LICQ.
    \begin{itemize}
        \item[\emph{(i)}] $X$ is second-order stationary for~\eqref{eq:Fproblem} if and only if $X$ is first-order stationary and $Q_S(\eta)\ge0$ for all $\eta\in\tangent_{\boundedranksdp(n)\cap\uanifold}(X)\cap\vecspan\{\nabla f(X)\}^\bot$, with $S$ as in Proposition~\ref{pro:Ffirst}.
        \item[\emph{(ii)}] If $X$ is first-order stationary and $Q_S(\eta)>0$ for all nonzero $\eta\in\tangent_{\boundedranksdp(n)\cap\uanifold}(X)\cap\vecspan\{\nabla f(X)\}^\bot$, then there exist $\gamma>0$ and $\delta>0$ such that $f(Y)\ge f(X)+\gamma\|Y-X\|^2$ for all $Y\in\boundedranksdp(n)\cap\uanifold$ with $\|Y-X\|\le\delta$; thus $X$ is a strict local minimizer of~\eqref{eq:Fproblem}.
    \end{itemize}
\end{proposition}
\begin{proof}
    See Appendix~\ref{app:F1}.
\end{proof}

Table~\ref{tab:Fliterature} summarizes Propositions~\ref{pro:Ffirst} and~\ref{pro:Fopt} together with existing results.

\begin{table}[htbp]
    \centering
    \caption{Stationarity conditions for~\eqref{eq:Fproblem} at $X$ with $\rank(X)=s$ under LICQ. SDP and low-rank SDP refer to~\eqref{eq:Fproblem} with $r=n$ and $r<n$, respectively. For $s=r$, $\boundedranksdp(n)\cap\uanifold$ is a smooth manifold around $X$; for $q=0$, the linear constraints vanish.}
    \label{tab:Fliterature}
    \renewcommand{\arraystretch}{1.3}
    \setlength{\tabcolsep}{8pt}
    \begin{tabular}{lllcc}
        \toprule
        \multicolumn{3}{c}{} & First-order & Second-order \\
        \midrule
        \multicolumn{3}{l}{SDP} & \cite{bonnans2000perturbationanalysis,shapiro1997first} & \cite{bonnans1999parabolicset,bonnans2000perturbationanalysis,shapiro1997first,sun2006strong} \\
        \midrule
        \multirow{3}{*}{Low-rank SDP} & $s=r$ & $q\ge0$ & \cite{absil2008optimization,boumal2023introduction} & \cite{absil2008optimization,boumal2023introduction} \\
        \cmidrule{2-5}
        & \multirow{2}{*}{$s<r$} & $q=0$ & \cite{tam2017sparsesdp} & --- \\
        & & $q\ge1$ & --- & --- \\
        \midrule
        \multicolumn{3}{l}{General} & Proposition~\ref{pro:Ffirst} & Proposition~\ref{pro:Fopt} \\
        \bottomrule
    \end{tabular}
\end{table}

For~\eqref{eq:BM}, LICQ at $X$ makes $\buanifold$ an embedded submanifold of $\mbR^{n\times r}$ around every $R\in\phiBM^{-1}(X)$, and the mapping $\tanL_R$ of section~\ref{sec:equi_2ndpoints} reads $\tanL_R(\dot R)=\dot RR^\top+R\dot R^\top$ for $\dot R\in\tangent_{\buanifold}(R)$, with $\ima(\tanL_R)=\{\eta\in\kernel(\aanifold)\mid J(\eta)=0\}$ by~\eqref{eq:FimL}. The following proposition characterizes first- and second-order stationarity for~\eqref{eq:BM} through a Lagrange multiplier and the quadratic form $Q_S$ in~\eqref{eq:FQ}.

\begin{proposition}\label{pro:FBM}
    Let $X\in\boundedranksdp(n)\cap\uanifold$ with $\rank(X)=s\le r$ satisfy LICQ, and let $R\in\phiBM^{-1}(X)$.
    \begin{itemize}
        \item[\emph{(i)}] $R$ is first-order stationary for~\eqref{eq:BM} if and only if $SR=0$, where $S:=\nabla f(X)-\aanifold^*(\mu)$, for some $\mu\in\mbR^q$; such $\mu$ is unique.
        \item[\emph{(ii)}] Let $R$ be first-order stationary. If $s=r$, then $R$ is second-order stationary for~\eqref{eq:BM} if and only if $Q_S(\eta)\ge0$ for all $\eta\in\ima(\tanL_R)$; if $s<r$, then $R$ is second-order stationary if and only if $S\succeq0$ and $Q_S(\eta)\ge0$ for all $\eta\in\ima(\tanL_R)$.
    \end{itemize}
\end{proposition}
\begin{proof}
    See Appendix~\ref{app:F1}.
\end{proof}

\begin{remark}\label{rem:FBMX}
    Since $\range(R)=\range(X)$, combining Propositions~\ref{pro:Ffirst} and~\ref{pro:FBM}(ii) shows that $R$ is second-order stationary for~\eqref{eq:BM} if and only if $X$ is first-order stationary for~\eqref{eq:Fproblem} and $Q_S(\eta)\ge0$ for all $\eta\in\ima(\tanL_R)$.
\end{remark}

\subsection{\texorpdfstring{Equivalence between stationary points}{Equivalence between stationary points}}\label{sec:PSD_4to2}

The following proposition applies Theorem~\ref{the:2to2} to $(\buanifold,\phiBM)$.

\begin{proposition}\label{pro:BM_2to2}
    Let $X\in\boundedranksdp(n)\cap\uanifold$ satisfy LICQ, and let $R\in\phiBM^{-1}(X)$. Then $(\buanifold,\phiBM)$ satisfies ``\,$2\!\Rightarrow\!2$'' at $R$ if and only if $\rank(R)=r$.
\end{proposition}
\begin{proof}
    See Appendix~\ref{app:F4}.
\end{proof}

For $\rank(R)=r$, ``\,$2\!\Rightarrow\!2$'' also follows from~\cite[Propositions~2.7 and~2.12]{levin2025effect}, since $\boundedranksdp(n)\cap\uanifold$ is a smooth manifold around $X$ when $\rank(X)=r$. For $\rank(X)<r$, the second-order stationarity conditions for~\eqref{eq:Fproblem} and~\eqref{eq:BM} test $Q_S$ on $\tangent_{\boundedranksdp(n)\cap\uanifold}(X)\cap\vecspan\{\nabla f(X)\}^\bot$ and on $\ima(\tanL_R)$, respectively, according to Propositions~\ref{pro:Fopt} and~\ref{pro:FBM}. The following proposition identifies, for a given gradient $\nabla f(X)$, when the two sets coincide and when second-order stationarity of $R$ yields second-order stationarity of $X$.

\begin{proposition}\label{pro:Fcompare}
    Let $X\in\boundedranksdp(n)\cap\uanifold$ with $\rank(X)=s<r$ satisfy LICQ and be first-order stationary for~\eqref{eq:Fproblem}, with $S$ as in Proposition~\ref{pro:Ffirst}, and let $R\in\phiBM^{-1}(X)$.
    \begin{itemize}
        \item[\emph{(i)}] $\ima(\tanL_R)\subseteq\tangent_{\boundedranksdp(n)\cap\uanifold}(X)\cap\vecspan\{\nabla f(X)\}^\bot$, with equality if and only if $\rank(X)+\rank(S)=n$.
        \item[\emph{(ii)}] If $X$ is second-order stationary for~\eqref{eq:Fproblem}, then $R$ is second-order stationary for~\eqref{eq:BM}. Conversely, if $R$ is second-order stationary for~\eqref{eq:BM} and $\rank(X)+\rank(S)=n$, then $X$ is second-order stationary for~\eqref{eq:Fproblem}.
        \item[\emph{(iii)}] If $\rank(X)+\rank(S)<n$, there exists a quadratic function $\tilde f:\sym(n)\to\mbR$ with $\nabla\tilde f(X)=\nabla f(X)$ such that $R$ is second-order stationary for $\tilde f\circ\phiBM$ on $\buanifold$, while $X$ is not second-order stationary for $\tilde f$ on $\boundedranksdp(n)\cap\uanifold$.
    \end{itemize}
\end{proposition}
\begin{proof}
    See Appendix~\ref{app:F4}.
\end{proof}

By Proposition~\ref{pro:Fcompare}, for a given gradient $\nabla f(X)$, second-order stationarity of $R$ implies that of $X$ if and only if $\rank(X)+\rank(S)=n$, i.e., under strict complementarity.

We now turn to higher-order stationarity, which is measured along curves on manifolds, as in~\cite{cartis2018second,levin2025effect,habib2026sphere}.
\begin{definition}\label{def:Fcrit}
    Let $k\ge1$, and let $(\bmanifold,\phi)$ be a smooth parameterization of $\manifold$. A point $Y\in\bmanifold$ is \emph{$k$-th-order stationary} for $\bar f:=f\circ\phi$ on $\bmanifold$ if
    \begin{equation*}
        \liminf_{\tau\downarrow0}\tau^{-k}\big(\bar f(c(\tau))-\bar f(Y)\big)\ge0
    \end{equation*}
    for every $C^\infty$ curve $c:(-\varepsilon,\varepsilon)\to\bmanifold$ with $c(0)=Y$. The parameterization $(\bmanifold,\phi)$ satisfies ``\,$k\!\Rightarrow\!2$'' (resp.\ ``\,$k\!\Rightarrow\!1$'') at $Y$ if, for every twice continuously differentiable $f$, $Y$ being $k$-th-order stationary for $f\circ\phi$ on $\bmanifold$ implies that $\phi(Y)$ is second-order (resp.\ first-order) stationary for~\eqref{eq:P}.
\end{definition}
Note that Definition~\ref{def:Fcrit} with $k=1,2$ agrees with first- and second-order stationarity in section~\ref{sec:smoothpara_opt}. The following proposition shows that, if $\rank(R)<r$, even third-order stationarity of $R$ for~\eqref{eq:BM} does not imply second-order stationarity of $X$ for~\eqref{eq:Fproblem}.

\begin{proposition}\label{pro:F2to2}
    Let $X\in\boundedranksdp(n)\cap\uanifold$ satisfy LICQ, and let $R\in\phiBM^{-1}(X)$ with $\rank(R)<r$. Then, for every unit vector $v\in\mbR^n$ with $Xv=0$ and $f_v(Y):=-(v^\top Yv)^2$, $R$ is third-order stationary for $f_v\circ\phiBM$ on $\buanifold$, while $X$ is not second-order stationary for $f_v$ on $\boundedranksdp(n)\cap\uanifold$. Hence ``\,$3\!\Rightarrow\!2$'' does not hold at $R$.
\end{proposition}
\begin{proof}
    See Appendix~\ref{app:F4}.
\end{proof}

To deduce second-order stationarity of $X$ from fourth-order stationarity of $R$, Lemma~\ref{lem:Fjet} in Appendix~\ref{app:F4} constructs, for every first- and second-order tangent direction $\eta$ and $\zeta$ to $\boundedranksdp(n)\cap\uanifold$ at $X$, a curve $c$ in $\buanifold$ through $R$ such that $\phiBM(c(\tau))=X+\tau^2\eta+\frac12\tau^4\zeta+O(\tau^6)$. If $\innerp{\nabla f(X),\eta}=0$, the second-order Taylor expansion of $f$ at $X$ then gives
\begin{equation*}
    \bar f(c(\tau))=\bar f(R)+\frac{\tau^4}{2}\big(\innerp{\nabla f(X),\zeta}+\innerp{\eta,\nabla^2f(X)[\eta]}\big)+o(\tau^4),
\end{equation*}
and hence fourth-order stationarity of $R$ yields second-order stationarity of $X$, which is inspired by~\cite{habib2026sphere}.

\begin{theorem}\label{the:F4to2}
    Let $X\in\boundedranksdp(n)\cap\uanifold$ satisfy LICQ, and let $R\in\phiBM^{-1}(X)$. If $R$ is fourth-order stationary for~\eqref{eq:BM}, then $X$ is second-order stationary for~\eqref{eq:Fproblem}.
\end{theorem}
\begin{proof}
    See Appendix~\ref{app:F4}.
\end{proof}

\begin{remark}\label{rem:Fanswer}
    Habib, Musat and Boumal~\cite{habib2026sphere} established ``\,$4\!\Rightarrow\!2$'' and the failure of ``\,$3\!\Rightarrow\!2$'' for the lift $y\mapsto y\odot y$ from the unit sphere onto the simplex, and asked whether it holds for the Burer--Monteiro parameterization of $\boundedranksdp(n)$. Theorem~\ref{the:F4to2} answers this question affirmatively.
\end{remark}

\section{Geometry of the graph of the normal cone mapping}\label{sec:Geometryofgraph}
Viewing the Mordukhovich normal cone induced by the determinantal variety as a set-valued mapping, that is,
$$\Lnormal_{\boundedrank}: \mbRmn\rightrightarrows\mbRmn:\,X\mapsto \Lnormal_{\boundedrank}(X),$$
we aim to give an explicit formula for the Mordukhovich normal cone to $\graph\Lnormal_{\boundedrank}$. Specifically, according to~\eqref{eq:Lnormal_cone_boundedrank}, the graph of $\normal_{\boundedrank}$ can be characterized by
\begin{equation*}
    \graph\Lnormal_{\boundedrank} = \hkh{(X,Y)\in\boundedrank\times \mbRmn\mid Y\in\normal_{\lowrank}(X),\ \rank(Y)\le \min\{m,n\}-r},
\end{equation*}
where we denote $s=\rank(X)$.

In this section, we investigate the variational geometry of $\graph\Lnormal_{\boundedrank}$, with the derivation illustrated below.
\begin{figure}[H]
	\centering
	\vspace{-4mm}
	\begin{tikzpicture}
		\node (M) at (0,0) {\normalsize ${\Btangent_{\graph\Lnormal_{\boundedrank}}}$};
		\node (Tc) at (4,0) {\normalsize ${\Fnormal_{\graph\Lnormal_{\boundedrank}}}$};
		\node (R) at (8,0) {\normalsize $\Lnormal_{\graph\Lnormal_{\boundedrank}}$};
		
		\draw[->,>={Stealth[round]}] (Tc) -- (R);
		\draw[->,>={Stealth[round]}] (M) -- (Tc);
		
		\node[above] at ($0.5*(M)+0.5*(Tc)$) {\normalsize $\mathrm{polar}$};
		\node[above] at ($0.5*(R)+0.5*(Tc)+(0.,0)$) {\normalsize $\lim$};
	\end{tikzpicture}
\end{figure}

\vspace{-4mm}
\noindent In detail, we first characterize the Bouligand tangent cone to $\graph\Lnormal_{\boundedrank}$ in Theorem~\ref{the:Btangent_graph}, and then take the polar operation to obtain the Fr\'echet normal cone in Corollaries~\ref{cor:Fnormal_graph_1}-\ref{cor:Fnormal_graph_2}. Consequently, in Theorem~\ref{the:Lnormal_graph}, the Mordukhovich normal cone to $\graph\Lnormal_{\boundedrank}$ is identified as the outer limit of the developed Fr\'echet normal cone. 

In preparation, we introduce some notation used throughout this section. We denote $k:=\min\{m,n\}$. Given $(X,Y)$, the ranks of $X$ and $Y$ are represented by $s$ and $k-\ell$, respectively, and specifically, the condition $(X,Y)\in\graph\Lnormal_{\boundedrank}$ implies that $0\le s\le r \le \ell \le k$.

\subsection{Bouligand tangent cone to the graph}
As a preview, we note that the derived tangent cone~\eqref{eq:Btangent_graph} is characterized via a parameterization built upon the SVD of the reference point $(X,Y)\in\graph\Lnormal_{\boundedrank}$. It is admitted that the coupling relationships among parameters are slightly involved, and thus we extract part of them in the following lemma, which appears technical but forms the basis of Theorem~\ref{the:Btangent_graph}.

\begin{lemma}\label{lem:constrained_boundedrank_curve}
    Given $0\le s\le r\le k=\min\{m,n\}$ and $R\in\mbR^{(m-s)\times (n-s)}_{\le k-r}$ with the compact SVD, $R=U_R\varSigma_RV_R^\top$. Suppose that $K\in\mbR^{(m-s)\times (n-s)}_{\le r-s}$ satisfies $K^\top R=0$ and $RK^\top=0$.  Then, for any 
    \begin{equation}\label{eq:nanifold}
    \!\!\tilde{D}\in\{D\in\mbR^{(m-s)\times (n-s)}\mid K^\top D V_{R\bot}=0,\,U_{R\bot}^\top D K^\top=0\}\cap\Btangent_{\mbR^{(m-s)\times(n-s)}_{\le k-r}}(R),
    \end{equation}
    there exist smooth curves $(R(t),K(t))\in\mbR^{(m-s)\times(n-s)}_{\le k-r}\times \mbR^{(m-s)\times(n-s)}_{\le r-s}$ such that $R^\prime(0)=\tilde{D}$, $(R(0),K(0))=(R,K)$, $K(t)^\top R(t)=0$ and $R(t)K(t)^\top=0$.
\end{lemma}
\begin{proof}
    Let $\rank(R)=k-\ell$ with $r\le \ell\le k$. Denote the set on the right side of~\eqref{eq:nanifold} by $\nanifold(R,K)$, and we can give an explicit characterization for it, that is,
    \begin{equation}\label{eq:zeta1_zeta2}
        \nanifold(R,K) = \hkh{\zeta_1+\zeta_2\ \left|\,\begin{array}{l}\zeta_1\in\Btangent_{\mbR^{(m-s)\times(n-s)}_{k-\ell}}(R),\ \zeta_2\in\normal_{\mbR^{(m-s)\times(n-s)}_{k-\ell}}(R),
        \\[2mm]
         \rank(\zeta_2)\le \ell-r,\ K^\top \zeta_2=0,\ \zeta_2K^\top=0\end{array}\right.},
    \end{equation}
    which can be verified by considering the conditions $K^\top R=0$, $RK^\top=0$, and the expressions from~\eqref{eq:Tcone_lowrank} to~\eqref{eq:Btangent_cone_boundedrank}.

    Then, for $\tilde{D}=\zeta_1+\zeta_2\in\nanifold(R,K)$, we prove the lemma by construction. In view of~\eqref{eq:zeta1_zeta2}, given any tangent vector $\zeta_1$ to the analytic manifold $\mbR^{(m-s)\times(n-s)}_{k-\ell}$ at $R$, there exists an analytic curve $R_1(t)$ on the manifold with $R_1(0)=R$ and $R_1^\prime(0)=\zeta_1$. Subsequently, \cite[Theorem 1]{bunse1991analyticSVD} reveals that $R_1(t)$ admits an analytic singular value decomposition, i.e.,
    \begin{equation}\label{eq:curve_R1}
                R_1(t)=\zkh{U_R(t)\ U_{R\bot}(t)}\left[ \begin{matrix}
    \varSigma _{R}\left( t \right)&		0\\
    0&		\varSigma _{R\bot}\left( t \right)\\
        \end{matrix} \right] \zkh{V_R(t)\ V_{R \bot}(t)}^\top. 
    \end{equation}   
    Without loss of generality, suppose $[U_R(0)\ U_{R\bot}(0)] = [U_R\ U_{R\bot}]$, $\varSigma_R(0)=\varSigma_R$, and $[V_R(0)\ V_{R\bot}(0)]=[V_R\ V_{R\bot}]$. Since $\rank(R_1(t))\equiv k-\ell$ and $\rank(\varSigma_R(0))=k-\ell$, we can find an $\varepsilon>0$ and the interval $\kh{-\varepsilon,+\varepsilon}$ such that $\varSigma_{R\bot}(t)\equiv 0$ for all $t\in(-\varepsilon,+\varepsilon)$,  which means $R_1(t)=U_R(t)\varSigma_R(t)V^\top_R(t)$. 
    
    The next step is to additionally introduce the direction~$\zeta_2$ in~\eqref{eq:zeta1_zeta2}. Let $\rank(K)=r-c\le r-s$. The conditions $K^\top R=0$ and $RK^\top=0$ indicate that $K=U_{R\bot} U_{R\bot}^\top K V_{R\bot} V_{R\bot}^\top$, and thus it admits the decomposition $K=U_{R\bot}U_K\varSigma_{K}V_K^\top V_{R\bot}^\top$ for some $U_K\in\stiefel(m-k+\ell-s,r-c)$ and $V_K\in\stiefel(n-k+\ell-s,r-c)$, the complements of which are denoted by $U_{K\bot}$ and $V_{K\bot}$, respectively. For $\zeta_2\in\normal_{\mbR^{(m-s)\times(n-s)}_{k-\ell}}(R)$ of $\rank(\zeta_2)\le\ell-r$, the conditions $K^\top \zeta_2=0$ and $\zeta_2K^\top=0$ imply that $\zeta_2$ can be parameterized by $\zeta_2 = U_{R\bot}U_{K\bot}{Z}V_{K\bot}^\top V_{R\bot}^\top$ for some $Z$ with $\rank(Z)=\rank(\zeta_2)$.

    Collecting $R_1(t)$ given in~\eqref{eq:curve_R1} and $R_2(t):=tU_{R\bot}(t)U_{K\bot} ZV_{K\bot}^\top V_{R\bot}(t)^\top$, we obtain $R(t):=R_1(t)+R_2(t)\in\mbR^{(m-s)\times(n-s)}_{\le k-r}$ and $K(t):=U_{R\bot}(t)U_K\varSigma_K V_K^\top V_{R\bot}(t)^\top\in\mbR^{(m-s)\times(n-s)}_{\le r-s}$ satisfying $R^\prime(0)=\tilde{D}=\zeta_1+\zeta_2$ and $(R(0),K(0))=(R,K)$. Moreover, $K(t)^\top R(t)=0$ and $R(t)K(t)^\top=0$ hold in the interval $(-\varepsilon,+\varepsilon)$.
\end{proof}

We now proceed to derive the tangent cone to the graph.

\begin{theorem}[Bouligand tangent cone]\label{the:Btangent_graph}
    Given $(X,Y)\in\graph\Lnormal_{\boundedrank}$ with $\rank(X)=s$, the SVD $X=U\varSigma V^\top$, and $Y=U^{}_{\bot} R_{}^{}V_{\bot}^\top$ for some $R\in\mbR^{(m-s)\times (n-s)}$. Suppose that $R$ admits the compact SVD $R=U_R\varSigma_RV_R^\top$. Then, the Bouligand tangent cone at $(X,Y)$ can be expressed as follows,
    \begin{equation}\label{eq:Btangent_graph}
        \Btangent_{\graph\Lnormal_{\boundedrank}}(X,Y)=\hkh{ (\eta,\xi)\ \left|\,\begin{array}{l}
            \eta = UAV^\top + U_\bot B\varSigma V^\top + U\varSigma C^\top V_\bot^\top + U_\bot K V^\top_\bot
           \\
           \xi = U_\bot DV_\bot ^\top - U B^\top RV_\bot^\top - U_\bot R CV^\top, 
            \\
            A\in\mbR^{s\times s},\,B\in\mbR^{(m-s)\times s},\,C\in\mbR^{(n-s)\times s},
            \\
            K\in\mbR^{(m-s)\times(n-s)}_{\le r-s},\, K^\top R=0,\,R K^\top=0,
            \\
            D\in\tangent_{\mbR_{\le k-r}^{(m-s)\times (n-s)}}(R),
            \\
            K_{}^\top D_{}^{}V^{}_{R\bot}=0,\ U_{R\bot}^\top D_{}^{}K_{}^\top =0
        \end{array} \right.
        }\!.\!\!\!\!\!
    \end{equation}
\end{theorem}
\begin{proof}
We begin by proving the ``$\supseteq$" part of~\eqref{eq:Btangent_graph}. Given $K\in\mbR^{(m-s)\times(n-s)}_{\le r-s}$ satisfying $K^\top R=0$ and $RK^\top=0$, we can construct curves $R(t)\in\mbR^{(m-s)\times(n-s)}_{\le k-r}$ and $K(t)\in\mbR^{(m-s)\times(n-s)}_{\le r-s}$ as stated in Lemma~\ref{lem:constrained_boundedrank_curve}. Consider, in addition, curves $[U(t)\ U_\bot(t)]\in\orth(m)$ and $[V(t)\ V_\bot(t)]\in\orth(n)$ passing through $[U\ U_\bot]$ and $[V\ V_\bot]$ at $t=0$, respectively. Then, we assemble the curves in the following manner:
$$\gamma(t):=U(t)\varSigma(t)V(t)^\top+U_\bot(t)\cdot tK(t)\cdot V_\bot(t)^\top\ \ \text{and}\ \  \beta(t):=U_\bot(t)R(t)V_\bot(t)^\top,$$
where $\varSigma(t)\in\mbR^{s\times s}$ with $\varSigma(0)=\varSigma$. 

Note that $\beta(t)^\top \gamma(t)=0$, $\gamma(t)\beta(t)^\top=0$, $\rank(\gamma(t))\le r$, and $\rank(\beta(t))\le k-r$. Hence, we obtain a smooth curve $\alpha(t):=(\gamma(t),\beta(t))$ in $\graph\Lnormal_{\boundedrank}$. Differentiating $\alpha(t)$ at $t=0$ yields $\alpha^\prime(0)=(\gamma^\prime(0),\beta^\prime(0))\in\Btangent_{\graph\Lnormal_{\boundedrank}}(X,Y)$ with
\begin{equation}\label{eq:gammap_betap_eta}
    \begin{aligned}
    \gamma^\prime(0) &= U^\prime(0)\varSigma V^\top + U\varSigma^\prime(0)V^\top + U\varSigma V^\prime(0)^\top+U_{\bot}KV_{\bot}^\top,
    \\
    \beta^\prime(0) &= U_{\bot}^\prime(0)RV_{\bot}^\top + U_{\bot} R^\prime(0)V_{\bot}^\top + U_{\bot} RV_{\bot}^\prime(0)^\top.
    \end{aligned}
\end{equation}
Taking into account the tangent space $\tangent_{\orth(m)}([U\ U_\bot])$ (see~\cite[\S3.5]{absil2008optimization}), for any $\Omega_u\in\sksym(s)$, $\Omega_{u\bot}\in\sksym(m-s)$, and $B\in\mbR^{(m-s)\times s}$, it is reasonable to construct the curve such that $U^\prime(0)=U\Omega_u+U_\bot B$ and $U_\bot^\prime(0)=U_\bot \Omega_{u\bot}-UB^\top$. Similarly, we can arrange to have $V^\prime(0)=V\Omega_v+V_\bot C$ and $V_\bot^\prime(0)=V_\bot\Omega_{v\bot}-VC^\top$ for any $\Omega_v\in\sksym(s)$, $\Omega_{v\bot}\in\sksym(n-s)$, and $C\in\mbR^{(n-s)\times s}$. Moreover, notice that $\varSigma(t)\in\mbR^{s\times s}$ is unconstrained, and thus $\varSigma^\prime(0)$ is allowed to be any $L\in\mbR^{s\times s}$. Substituting the discussed quantities into~\eqref{eq:gammap_betap_eta} leads to
\begin{equation}\label{eq:gammap_betap_eta2}
    \begin{aligned}
    \gamma^\prime(0) &= U(\Omega_u\varSigma+L-\varSigma\Omega_v)V^\top + U_\bot B\varSigma V^\top + U\varSigma C^\top V_\bot^\top + U_{\bot}KV_{\bot}^\top,
    \\
    \beta^\prime(0) &= U_\bot(\Omega_{u\bot}R+R^\prime(0)-R\Omega_{v\bot})V_\bot ^\top - UB^\top R V_\bot^\top - U_\bot R CV^\top.
    \end{aligned}
\end{equation}
We denote $A=\Omega_u\varSigma+L-\varSigma\Omega_v$ and $D=\Omega_{u\bot}R+R^\prime(0)-R\Omega_{v\bot}$. The arbitrariness of $L\in\mbR^{s\times s}$ reveals that of $A\in\mbR^{s\times s}$. Additionally, the conditions $K^\top R=0$ and $RK^\top=0$, together with the freedom in choosing $R^\prime(0)\in\nanifold(R,K)$ (cf.~Lemma~\ref{lem:constrained_boundedrank_curve}), reveal the arbitrariness of $D\in\nanifold(R,K)$, which concludes the ``$\supseteq$" part.

Next, we turn to show the ``$\subseteq$" in~\eqref{eq:Btangent_graph}. Given any $(\eta,\xi)\in\Btangent_{\graph\Lnormal_{\boundedrank}}(X,Y)$, by definition of the Bouligand tangent cone, it admits sequences $\{t_i\}\subseteq \mbR_+$ and $\{(\eta_i,\xi_i)\}\subseteq \mbRmn\times\mbRmn$ such that $t_i\to 0$, $(\eta_i,\xi_i)\to(\eta,\xi)$, and $(X,Y)+t_i(\eta_i,\xi_i)\in\graph\Lnormal_{\boundedrank}$, or equivalently,
\begin{equation}\label{eq:relation_sequence}
    X+t_i\eta_i\in\boundedrank\ \ \text{and}\ \ Y+t_i\xi_i\in\Lnormal_{\boundedrank}(X+t_i\eta_i).
\end{equation}
We note that $\eta\in\tangent_{\boundedrank}(X)$, and according to the expression~\eqref{eq:Btangent_cone_boundedrank}, $(\eta,\xi)$ has the following form,
\begin{equation}\label{eq:eta_four_term}
\begin{aligned}
\eta=&\ UA_\eta V^\top+U_\bot B_\eta V^\top+UC_\eta V^\top_\bot+U_{\bot}K V^\top_{\bot},      
\\
\xi=&\ UA_\xi V^\top+U_\bot B_\xi V^\top+UC_\xi V^\top_\bot+U_{\bot}D V^\top_{\bot},  
\end{aligned}
\end{equation}
where $\rank(K)\le r-s$. Then we will determine the relationships between the involved parameters (e.g., $A_\eta,A_\xi$).

Firstly, the requirement~\eqref{eq:relation_sequence} reveals that
\begin{equation}\label{eq:XYtk0}
(X+t_i\eta_i)^\top(Y+t_i\xi_i)=t_i\eta_i^\top Y+t_iX^\top \xi_i+t_i^2\eta_i^\top\xi_i=0.  
\end{equation}
Dividing the equation by $t_i$ and letting $i\to+\infty$ yield $\eta^\top Y+X^\top\xi=0$, i.e.,
$VB_\eta^\top RV_\bot^\top + V_\bot K^\top RV_\bot^\top + V\varSigma A_\xi V^\top + V\varSigma C_\xi V_\bot^\top = 0$,
which implies $C_\xi =-\varSigma^{-1}B_\eta^\top R$, $A_\xi=0$, and $K^\top R=0$. In a similar way, we can derive $B_\xi=-RC^\top_\eta\varSigma^{-1}$ and $RK^\top=0$. Substituting these equalities into~\eqref{eq:eta_four_term} and comparing the expression with~\eqref{eq:Btangent_graph} reduce the task to verifying that $D\in\nanifold(R,K)$ as defined in Lemma~\ref{lem:constrained_boundedrank_curve}. To see this, noticing that $U_\bot^\top\xi_iV_\bot\to U_\bot^\top\xi V_\bot = D$, and
\begin{equation*}
    U_\bot^\top(Y+t_i\xi_i)V_\bot=R+t_iU_\bot^\top\xi_iV_\bot\in \mbR^{(m-s)\times (n-s)}_{\le k-r},
\end{equation*}
we have $D\in\tangent_{\mbR^{(m-s)\times (n-s)}_{\le k-r}}(R)$. Moreover, pre- and postmultiplying~\eqref{eq:XYtk0} by $V_\bot V_\bot^\top$ and $ V_\bot V_{R\bot}$, respectively, we obtain $V_\bot V_\bot^\top \eta_i^\top\xi_iV_\bot V_{R\bot}=0$. Letting $i\to+\infty$ and incorporating the expression of $(\eta,\xi)$~\eqref{eq:eta_four_term} reveal that $V_\bot(K^\top U_\bot^\top+C^\top_\eta U^\top)(UC_\xi+U_\bot D)V_{R\bot}=0$, which is simplified to $V_\bot(K^\top D+C_\eta^\top C_\xi)V_{R\bot}=0$. Consequently, we use $C_\xi V_{R\bot}=-\varSigma^{-1}B_\eta^\top RV_{R\bot}=0$ to derive $K^\top DV_{R\bot}=0$, and similarly, we can also find that $U_{R\bot}^\top DK^\top=0$, indicating that $D\in\nanifold(R,K)$. Therefore, we have identified all the relationships to conclude that $(\eta,\xi)$ belongs to the set on the right side of~\eqref{eq:Btangent_graph}.
\end{proof}

\subsection{Fr\'echet normal cone to the graph}
As shown in~\eqref{eq:Btangent_graph}, the characterization of the tangent cone at $(X,Y)$ is relevant to the rank of $X$. This observation leads to different treatments of $\Fnormal_{\graph\Lnormal_{\boundedrank}}(X,Y)$ in two cases, depending on whether $X$ attains rank $r$, which mirrors the first-order result in~\eqref{eq:Fnormal_cone_boundedrank}. Accordingly, Corollaries~\ref{cor:Fnormal_graph_1} and~\ref{cor:Fnormal_graph_2} address the cases $\rank(X)<r$ and $\rank(X)=r$, respectively.

\begin{corollary}\label{cor:Fnormal_graph_1}
    Given $(X,Y)\in\graph\Lnormal_{\boundedrank}$ with $\rank(X)=s<r$, the SVD $X=U\varSigma V^\top$, and $Y=U^{}_{\bot} R_{}^{}V_{\bot}^\top$ for some $R\in\mbR^{(m-s)\times (n-s)}$. Suppose that $\rank(R)=k-\ell$ and it admits the compact SVD $R=U_R\varSigma_RV_R^\top$. Then, the Fr\'echet normal cone at $(X,Y)$ can be expressed as follows,
    \begin{equation}\label{eq:Fnormal_graph_deficient}
        \Fnormal_{\graph\Lnormal_{\boundedrank}}(X,Y)=\hkh{ (\hat{\upsilon},\hat{\omega})\ \left|\,\begin{array}{l}
            \hat{\upsilon} =  U_\bot U_R\varSigma_R\hat{C}_1^\top V^\top + U\hat{B}_1^\top \varSigma_RV_R^\top V_\bot^\top
            \\
            \hspace{7mm}+ U_\bot Z V^\top_\bot,
           \\
           \hat{\omega} = U\hat{A}V^\top + U_\bot U_R\hat{B}_1\varSigma V^\top + U_\bot U_{R\bot}\hat{B}_2 V^\top 
           \\
           \hspace{7mm}+ U \varSigma\hat{C}_1V_R^\top V_\bot^\top+ U \hat{C}_2V_{R\bot}^\top V_\bot^\top + U_\bot\hat{Z} V_\bot^\top
            \\
            \hat{A}\in\mbR^{s\times s},\,\hat{B}_1\in\mbR^{(k-\ell)\times s},\,\hat{B}_2\in\mbR^{(m-k+\ell-s)\times s}
            \\
            \hat{C}_1\in\mbR^{s\times (k-\ell)},\,\hat{C}_2\in\mbR^{s\times(n-k+\ell-s)},
            \\
            {Z}\in \Btangent_{\mbR^{(m-s)\times(n-s)}_{k-\ell}}(R),\,\hat{Z}\in\Fnormal_{\mbR^{(m-s)\times(n-s)}_{\le k-r}}(R)
        \end{array} \right.
        }.
    \end{equation}
\end{corollary}
\begin{proof}
    Taking the polar operation on both sides of~\eqref{eq:Btangent_graph} will yield the desired Fr\'echet normal cone. Specifically, notice that $(\eta,\xi)$ belongs to $\Btangent_{\graph\Lnormal_{\boundedrank}}(X,Y)$ if and only if $(-\eta,-\xi)$ is in it, and thus $(\hat{\upsilon},\hat{\omega})\in\Fnormal_{\graph\Lnormal_{\boundedrank}}(X,Y)$ is equivalent to $\innerp{(\hat{\upsilon},\hat{\omega}),(\eta,\xi)}=0$ for all $(\eta,\xi)\in\Btangent_{\graph\Lnormal_{\boundedrank}}(X,Y)$. 

    Given $(\hat{\upsilon},\hat{\omega})\in\mbRmn\times\mbRmn$, we can parameterize them as follows,
    \begin{equation}\label{eq:hatupsilon_hatomega}
        \begin{aligned}
            \hat\upsilon &=  U\hat{A}_{\upsilon}V^\top + U_\bot \hat{B}_\upsilon V^\top + U\hat{C}_\upsilon V_\bot^\top + U_\bot Z V_\bot^\top,
           \\
           \hat\omega &= U\hat{A}_{\omega}V^\top + U_\bot \hat{B}_\omega V^\top + U\hat{C}_\omega V_\bot^\top + U_\bot \hat{Z} V_\bot^\top.
        \end{aligned}
    \end{equation}
    Considering $(\eta,\xi)$ as expressed in~\eqref{eq:Btangent_graph}, we have
    \begin{equation}\label{eq:innerto0}
        \begin{aligned}
        \innerp{(\hat{\upsilon},\hat{\omega}),(\eta,\xi)} = &\ \innerp{\hat{A}_\upsilon,A} + \innerp{\hat{A}_\omega,0}+\innerp{\hat{B}_\upsilon\varSigma-R\hat{C}_\omega^\top,B} 
        \\
        &+ \innerp{\hat{C}_\upsilon^\top\varSigma-R^\top\hat{B}_\omega,C} + \innerp{Z,K} + \innerp{\hat{Z},D}            
        \end{aligned}
    \end{equation}
    Letting $A,B,C,K$ be zero matrices of corresponding shapes, the orthogonal requirement $\langle(\hat{\upsilon},\hat{\omega}),(\eta,\xi)\rangle=0$ implies that $\innerp{\hat{Z},D}=0$ for all $D\in\tangent_{\mbR^{(m-s)\times(n-s)}_{\le k-r}}(R)$; similar processes conclude that $\hat{A}_\upsilon=0$, $\hat{A}_\omega$ is freely chosen from $\mbR^{s\times s}$, and $\innerp{Z,K}=0$ for all $K$ described in~\eqref{eq:Btangent_graph}. Moreover, we let $A,K,D$ be zero matrices to find that
    \begin{equation*}
        \innerp{\hat{B}_\upsilon \varSigma-R\hat{C}^\top_\omega,B}+\innerp{\hat{C}_\upsilon^\top\varSigma-R^\top \hat{B}_\omega,C}=0,\ \ \text{for all}\ (B,C)\in\mbR^{(m-s)\times s}\times \mbR^{(n-s)\times s},
    \end{equation*}
     which, therefore, enforces that $\hat{B}_\upsilon=R\hat{C}^\top_\omega\varSigma^{-1}$ and $\hat{C}_\upsilon=\varSigma^{-1}\hat{B}_\omega^\top R$.  Taking  the relationships, and substituting $\hat{C}_\omega=\varSigma\hat{C}_1V^\top_R+\hat{C}_2V_{R\bot}^\top$ and $\hat{B}_\omega=U_R\hat{B}_1\varSigma+U_{R\bot}\hat{B}_2$ into~\eqref{eq:hatupsilon_hatomega} lead to the formulation~\eqref{eq:Fnormal_graph_deficient}.
\end{proof} 

We then turn to the case $\rank(X)=r$. The main difference is that when $X$ attains rank $r$, the parameter $K$ in~\eqref{eq:Btangent_graph} is forced to vanish, implying that the component of $\eta$ spanned by $U_\bot$ and $V_\bot$ disappears. 

\begin{corollary}\label{cor:Fnormal_graph_2}
    Given $(X,Y)\in\graph\Lnormal_{\boundedrank}$ with $\rank(X)=r$, the SVD $X=U\varSigma V^\top$, and $Y=U^{}_{\bot} R_{}^{}V_{\bot}^\top$ for some $R\in\mbR^{(m-s)\times (n-s)}$. Suppose that $\rank(R)=k-\ell$ and it admits the compact SVD $R=U_R\varSigma_RV_R^\top$. Then, the Fr\'echet normal cone at $(X,Y)$ can be expressed as follows,
    \begin{equation}\label{eq:Fnormal_graph_full}
        \Fnormal_{\graph\Lnormal_{\boundedrank}}(X,Y)=\hkh{ (\hat{\upsilon},\hat{\omega})\ \left|\,\begin{array}{l}
            \hat{\upsilon} =  U_\bot U_R\varSigma_R\hat{C}_1^\top V^\top + U\hat{B}_1^\top \varSigma_RV_R^\top V_\bot^\top
            \\
            \hspace{7mm}+ U_\bot Z V^\top_\bot,
           \\
           \hat{\omega} = U\hat{A}V^\top + U_\bot U_R\hat{B}_1\varSigma V^\top + U_\bot U_{R\bot}\hat{B}_2 V^\top
           \\
           \hspace{7mm}+ U \varSigma\hat{C}_1V_R^\top V_\bot^\top+ U \hat{C}_2V_{R\bot}^\top V_\bot^\top +U_\bot\hat{Z} V_\bot^\top 
            \\
            \hat{A}\in\mbR^{s\times s},\,\hat{B}_1\in\mbR^{(k-\ell)\times s},\,\hat{B}_2\in\mbR^{(m-k+\ell-s)\times s}
            \\
            \hat{C}_1\in\mbR^{s\times (k-\ell)},\,\hat{C}_2\in\mbR^{s\times(n-k+\ell-s)},
            \\
            Z\in\mbR^{(m-r)\times(n-r)},\,\hat{Z}\in\Fnormal_{\mbR^{(m-s)\times(n-s)}_{\le k-r}}(R)
        \end{array} \right.
        }.
    \end{equation}
\end{corollary}
\begin{proof}
The derivation follows the same reasoning as in the proof of Corollary~\ref{cor:Fnormal_graph_1}, except that $\rank(X)=r$ enforces $K=0$ in~\eqref{eq:innerto0}. Consequently, $Z$ becomes unconstrained in $\mbR^{(m-r)\times(n-r)}$.
\end{proof}

It is worth noting that $\fixedrank$ is relatively open in $\boundedrank$~\cite{olikier2022continuity}. That is, in a neighborhood of a point $X$ with $\rank(X)=r$, the determinantal variety $\boundedrank$ coincides with the smooth manifold $\fixedrank$. Since $\fixedrank$ is an embedded submanifold of $\mathbb{R}^{m\times n}$, Theorem~6.23 in~\cite{lee2012manifolds} implies that the normal bundle $\normal\fixedrank$ is an embedded submanifold of $\mathbb{R}^{m\times n}\times\mathbb{R}^{m\times n}$ of dimension $mn$. Consequently, when $\rank(X)=r$, we have $\tangent_{\normal \fixedrank}(X,Y) = \Btangent_{\graph\Lnormal_{\boundedrank}}(X,Y)$, and $\normal_{\normal \fixedrank}(X,Y) = \Fnormal_{\graph\Lnormal_{\boundedrank}}(X,Y) = \Lnormal_{\graph\Lnormal_{\boundedrank}}(X,Y)$, as characterized by Corollary~\ref{cor:Fnormal_graph_2}.

\subsection{Mordukhovich normal cone to the graph}\label{sec:Mordukhovichnormalcone}
We are now in a position to develop the Mordukhovich normal cone, for which two auxiliary lemmas are required.

Let $\mathbb{R} _{\ge}^{d}$ be the set of vectors in $\mathbb{R}^{d}$ with elements being positive and arranged in a non-increasing order, i.e.,
\begin{equation*}
    \mathbb{R} _{\ge}^{d}:=\{ x\in \mathbb{R} ^d\mid x_1\ge x_2\ge \cdots \ge x_d>0 \}.
\end{equation*}
We define the mapping $\mathfrak{D}:\mathbb{R} _{\ge}^{d_1}\times \mathbb{R} _{\ge}^{d_2}\rightarrow \mathbb{R} ^{d_1\times d_2}$ as follows,
\begin{equation*}
    \mathfrak{D}( x,y) _{jt}:=\frac{x_j}{x_j+y_t},\ \ \text{for}\ \ j=1,2,\dots,d_1\ \text{and}\ t=1,2,\ldots,d_2.
\end{equation*}
Then, the mapping $\mathfrak{D}$ introduces the following set,
\begin{equation}\label{eq:varTheta}
    \!\!\mathrm{\Theta}( d_1,d_2 ) :=\hkh{ \lim _{i\rightarrow \infty}\mathfrak{D}( z_{1}^{i},z_{2}^{i} )\in \mathbb{R} ^{d_1\times d_2}\mid z_{1}^{i}\rightarrow 0,z_{2}^{i}\rightarrow 0,z_{1}^{i}\in \mathbb{R} _{\ge}^{d_1},z_{2}^{i}\in \mathbb{R} _{\ge}^{d_2} }.
\end{equation}
In fact, the set $\mathrm{\Theta}$ helps characterize the asymptotic behavior of two families of positive singular values (see the proof of Theorem~\ref{the:Lnormal_graph}). Additionally, we remark that the definitions of $\mathbb{R} _{\ge}^{d}$, $\mathfrak{D}$, and $\mathrm{\Theta}$ are inspired by~\cite[\S3.2]{ding2014SDCMPCC} and \cite[\S3.2]{wu2014SDCMPCC}, where related concepts were introduced for studying normal cones to $\graph\Lnormal_{\sym^+(n)}$.

The subsequent lemma identifies a basic relation in matrix computation, which was proved in~\cite[Lemma~2.2]{wu2014SDCMPCC}. 

\begin{lemma}\label{lem:D_equivalence}
    Given vectors $b\in\mbR^{d_1}_{\ge}$, $q\in\mbR^{d_2}_\ge$ and matrices $B,Q\in\mbR^{d_1\times d_2}$, it holds that
    \begin{equation*}
        \Diag\left( b \right) B=Q\Diag\left( q \right) \Longleftrightarrow \mathfrak{D}(b,q) \odot B+\left(\mathfrak{D}(b,q) -{\bm 1}_{d_1\times d_2} \right) \odot Q=0,
    \end{equation*}
    where $\odot$ denotes the Hadamard product and ${\bm 1}_{d_1\times d_2}$ is the all-ones matrix of size $d_1\times d_2$.
\end{lemma}

Note that the Mordukhovich normal cone is defined through the outer limit \eqref{eq:liminting_normalcone}, and thus we introduce the next lemma to capture the behavior of sequences convergent to $(X,Y)\in\graph\Lnormal_{\boundedrank}$. Specifically, it is shown that the convergence of ${(X_i,Y_i)}$ to $(X,Y)$ induces the convergence of the subspaces spanned by $X_i$ and $Y_i$ as well as of their orthogonal complements (possibly after taking a subsequence), which indeed extends the idea of~\cite[Lemma 4.2]{olikier2022continuity}.

\begin{lemma}\label{lem:fivepart_convergence}
 Given $(X,Y)\in\graph\Lnormal_{\boundedrank}$ with $\rank(X)=s$, $\rank(Y)=k-\ell$, and the SVDs $X=U\varSigma V^\top$ and $Y=U_Y\varSigma_YV_Y^\top$. Let $\{(X_i,Y_i)\}_{i\in\mbN}\subseteq\manifold_{\downr}\times\manifold_{k-\upr} $ be a sequence in $\graph\Lnormal_{\boundedrank}$ converging to $(X,Y)$, where $s\le \downr\le r\le \upr\le\ell\le k$. Then, there exist $\tilde{U}_\bot\in\stiefel(m,\downr-s)$, $\breve{U}_\bot\in\stiefel(m,m-k+\upr-\downr)$, $\tilde{U}_{Y\bot}\in\stiefel(m,\ell-\upr)$, $\tilde{V}_\bot\in\stiefel(n,\downr-s)$, $\breve{V}_\bot\in\stiefel(n,n-k+\upr-\downr)$, and $\tilde{V}_{Y\bot}\in\stiefel(n,\ell-\upr)$ such that $[U\ \tilde{U}_\bot\ \breve{U}_\bot\ \tilde{U}_{Y\bot}\  U_Y]\in\orth(m)$ and $[V\ \tilde{V}_\bot\ \breve{V}_\bot\ \tilde{V}_{Y\bot}\ V_Y]\in\orth(n)$, together with a subsequence $\{(X_{i_j},Y_{i_j})\}_{j\in\mbN}$ and two associated sequences
    \begin{equation*}\label{eq:U_V_spacesequence}
        \begin{aligned}
            \{(U_{i_j},\tilde{U}_{i_j\bot},\breve{U}_{i_j\bot},\tilde{U}_{i_jY\bot},U_{i_j Y})\}_{j\in\mbN} \ &\subseteq\ \orth(m),
            \\
            \{(V_{i_j},\tilde{V}_{i_j\bot},\breve{V}_{i_j\bot},\tilde{V}_{i_jY\bot},V_{i_j Y})\}_{j\in\mbN} \ &\subseteq\ \orth(n),
        \end{aligned}
    \end{equation*}
    satisfying the following properties.
    \begin{itemize}[left=0cm]
    \item For all $i_j$, $\ima[U_{i_j}\ \tilde{U}_{i_j\bot}]=\ima X_{i_j}$, $\ima[V_{i_j}\ \tilde{V}_{i_j\bot}]=\ima X^\top_{i_j}$, $\ima[\tilde{U}_{i_j Y\bot}\ U_{i_j Y}]=\ima Y_{i_j}$, and $\ima[\tilde{V}_{i_j Y\bot}\ {V}_{i_j Y}]=\ima Y^\top_{i_j}$.
    \item It holds that 
    \begin{equation}\label{eq:subspace_seq}
        \begin{aligned}
            \lim_{j\to\infty}[U_{i_j}\ \tilde{U}_{i_j\bot}\ \breve{U}_{i_j\bot}\ \tilde{U}_{i_jY\bot}\ U_{i_j Y}] &= [U\ \tilde{U}_\bot\ \breve{U}_\bot\ \tilde{U}_{Y\bot}\ U_Y],
            \\
            \lim_{j\to\infty}[V_{i_j}\ \tilde{V}_{i_j\bot}\ \breve{V}_{i_j\bot}\ \tilde{V}_{i_jY\bot}\ V_{i_j Y}] &= [V\ \tilde{V}_\bot\ \breve{V}_\bot\ \tilde{V}_{Y\bot}\ V_Y].
        \end{aligned}
    \end{equation}
    \end{itemize}
\end{lemma}
\begin{proof}
    See Appendix~\ref{sec:proof_Mor}.
\end{proof}

Recalling from the definition of the Mordukhovich normal cone~\eqref{eq:liminting_normalcone}, the direction $(\upsilon,\omega)\in\Lnormal_{\graph\Lnormal_{\boundedrank}}(X,Y)$ if and only if there exist a sequence $\{(X_i,Y_i)\}_{i\in\mbN}$ and the associated $(\hat{\upsilon}_i,\hat{\omega}_i)\in\Fnormal_{\graph\Lnormal_{\boundedrank}}(X_i,Y_i)$ such that 
\begin{equation}\label{eq:limit_viwi}
    \lim_{i\to\infty} (X_i,Y_i)=(X,Y)\ \ \text{and}\ \ \lim_{i\to\infty} (\hat{\upsilon}_i,\hat{\omega}_i)=(\upsilon,\omega).
\end{equation}
Then, we arrive at the characterization of the Mordukhovich normal cone.

\begin{theorem}[Mordukhovich normal cone]\label{the:Lnormal_graph}
    Given $(X,Y)\in\graph\Lnormal_{\boundedrank}$, where $\rank(X)=s$ and $\rank(Y)=k-\ell$. Let the SVDs be $X=U\varSigma V^\top$ and $Y=U_Y\varSigma_YV_Y^\top$. Then the element $(\upsilon,\omega)$ belongs to $\Lnormal_{\graph\Lnormal_{\boundedrank}}(X,Y)$ if and only if there exist $\downr,\upr$ with $s\le\downr\le r\le\upr\le\ell$, $\tilde{U}_\bot\in\stiefel(m,\downr-s)$, $\tilde{U}_{Y\bot}\in\stiefel(m,\ell-\upr)$, $\breve{U}_\bot\in\stiefel(m,m-k+\upr-\downr)$, $\tilde{V}_\bot\in\stiefel(n,\downr-s)$, $\tilde{V}_{Y\bot}\in\stiefel(n,\ell-\upr)$, $\breve{V}_\bot\in\stiefel(n,n-k+\upr-\downr)$ with $[U\ \tilde{U}_\bot\ \breve{U}_\bot\  U_Y\ \tilde{U}_{Y\bot}]\in\orth(m)$ and $[V\ \tilde{V}_\bot\ \breve{V}_\bot\ V_Y\ \tilde{V}_{Y\bot}]\in\orth(n)$ such that $(\upsilon,\omega)$ can be expressed by
    \begin{equation}\label{eq:vw_expression}
        \begin{aligned}
        &\upsilon =\left[ U^+\,\,\breve{U}_\bot\,\,U_{Y}^{+} \right] \left[ 
        \begin{array}{ccc}
            0 & 0 & B^{\upsilon} \\
            0 &  Z_1 & Z_2\\
            C^\upsilon & Z_3 & Z_4
        \end{array} 
        \right] \left[ V^+\,\,\breve{V}_\bot\,\,V_{Y}^{+} \right] ^{\top},
        \\
        &\omega = \left[ U^+\,\,\breve{U}_\bot\,\,U_{Y}^{+} \right] \left[ \begin{matrix}
        	A&		C&		C^{\omega}\\
        	B&		0&		0\\
        	B^{\omega}&		0&		\hat{Z}\\
        \end{matrix} \right] \left[ V^+\,\,\breve{V}_\bot\,\,V_{Y}^{+} \right] ^{\top},
        \end{aligned}
    \end{equation}
    where $U^+:=[U\ \tilde{U}_\bot]$, $U^+_{Y}:=[U_{Y}\ \tilde{U}_{Y\bot}]$, $V^+:=[V\ \tilde{V}_{\bot}]$, $V^+_{Y}:=[V_{Y}\ \tilde{V}_{Y\bot}]$; $Z_1=0$ if $\downr<r$ otherwise $Z_1$ is not restricted; $\hat{Z}=0$ if $\upr>r$ otherwise $\hat{Z}$ is not restricted; and $B^\upsilon$, $B^\omega$, $C^\upsilon$, $C^\omega$ have the following forms,
    \begin{equation}\label{eq:BBCC}
        B^{\upsilon}=\left[ \begin{matrix}
                \varSigma ^{-1}G_{1}^{\top}\varSigma _Y&		0\\
                E_{1}^{\upsilon}&		F_{1}^{\upsilon}\\
            \end{matrix} \right],\,
        B^{\omega}=\left[ \begin{matrix}
            	G_1&		0\\
            	E_{1}^{\omega}&		F_{1}^{\omega}\\
                \end{matrix} \right],\,
        C^{\upsilon}=\left[ \begin{matrix}
                	\varSigma _YG_{2}^{\top}\varSigma ^{-1}&		E_{2}^{\upsilon}\\
                	0&		F_{2}^{\upsilon}\\
                \end{matrix} \right],\,
        C^{\omega}=\left[ \begin{matrix}
        	G_2&		E_{2}^{\omega}\\
        	0&		F_{2}^{\omega}\\
        \end{matrix} \right],
    \end{equation}
    satisfying 
    \begin{equation}\label{eq:A}
        \begin{aligned}
            &D \odot F_1^\omega+(D-{\bm 1}_{(\ell-\upr)\times(\downr-s)})\odot(F^\upsilon_1)^\top=0,
            \\
            &D^\top \odot F_2^\omega+(D^\top-{\bm 1}_{(\downr-s)\times (\ell-\upr)})\odot(F^\upsilon_2)^\top=0,
        \end{aligned}
    \end{equation}
    {for some} $D\in\mathrm{\Theta}(\ell-\upr, \downr-s)$. The dimensions of the matrix parameters are summarized in Table~\ref{tab:mat_dims}.
\end{theorem}
\begin{table}[h]
\caption{Dimensions of matrix parameters in Theorem~\ref{the:Lnormal_graph}, where $\upr-\downr$ is abbreviated as $\Delta r$.}
\label{tab:mat_dims}
\centering
\renewcommand{\arraystretch}{1.}  
\setlength{\tabcolsep}{4.5pt}        
\begin{tabular}{lccc}
\toprule
Param. & $A$ & $B$ & $C$  \\
\midrule
Dim. & ${\downr\times \downr}$ & ${(m-k+\Delta r)\times \downr}$ &  ${\downr\times(n-k+\Delta r)}$ \\
\midrule
Param. & $D$ & $G_{1}$, $G_2^\top$ & $E_{1}^\omega$, $(E_{2}^\omega)^\top$
\\
\midrule
Dim. &  ${(\ell-\upr)\times(\downr-s)}$ &  ${(k-\ell)\times s}$ &
${(\ell-\upr)\times s}$
\\
\midrule
Param. &  $E_{1}^\upsilon$, $(E_{2}^\upsilon)^\top$ & $F^\upsilon_1$, $(F_1^\omega)^\top$, $(F^\upsilon_2)^\top$, $F_2^\omega$  & $Z_1$\\[2pt]
\midrule
Dim.  & ${(\downr-s)\times(k-\ell)}$ & 
${(\downr-s)\times (\ell-\upr)}$ & $(m-k+\Delta r)\times (n-k+\Delta r)$ \\[2pt]
\midrule
Param. & $Z_2$ & $Z_3$ & $Z_4$,\,$\hat{Z}$ \\
\midrule
Dim. & $(m-k+\Delta r)\times (k-\upr)$ &  $(k-\upr)\times (n-k+\Delta r)$ & {$(k-\upr)\times (k-\upr)$} \\
\bottomrule
\end{tabular}
\end{table}
\begin{proof}
    We remark that Corollaries~\ref{cor:Fnormal_graph_1}–\ref{cor:Fnormal_graph_2} provide the parameterization in terms of $Y = U_\bot R V_\bot^\top$ and $R = U_R \varSigma_R V_R^\top$. To align with the notation used in this theorem, we can therefore identify $(U_Y, \varSigma_Y, V_Y) = (U_\bot U_R, \varSigma_R, V_\bot V_R)$; a similar argument applies when considering normal cones at $(X_i,Y_i)$.

    To prove the ``if" part of the theorem, we will construct a sequence $\{(X_i,Y_i)\}_{i\in\mbN}$ with $\rank(X_i)=\downr$, $\rank(Y_i)=k-\upr$ and the associated $(\hat{\upsilon}_i,\hat{\omega}_i)$, which produces $(\upsilon,\omega)$ expressed by~\eqref{eq:vw_expression} in the manner of~\eqref{eq:limit_viwi}. To this end, taking into account the matrix $D\in\mathrm{\Theta}(\ell-\upr, \downr-s)$ in~\eqref{eq:A}, we assume that $D$ is generated by the sequences $\{z^i_1\}\subseteq\mbR_{\ge}^{\ell-\upr}$ and $\{z^i_2\}\subseteq\mbR_{\ge}^{\downr-s}$ as in~\eqref{eq:varTheta}. Subsequently, the $(X_i,Y_i)$ can be given by
    \begin{equation*}
        X_i=\left[ U\,\,\tilde{U}_{\bot} \right] \left[ \begin{matrix}
	\varSigma&		0\\
	0&		\Diag(z^i_2)\\
        \end{matrix} \right] \left[ V\,\,\tilde{V}_{\bot} \right] ^{\top}\!\!,\ 
        Y_i=\left[ U_Y\,\,\tilde{U}_{Y\bot} \right] \left[ \begin{matrix}
        	\varSigma _Y&		0\\
        	0&		\Diag (z^i_1)\\
        \end{matrix} \right] \left[ V_Y\,\,\tilde{V}_{Y\bot} \right] ^{\top}.
    \end{equation*}
    According to~\eqref{eq:Fnormal_graph_deficient} and~\eqref{eq:Fnormal_graph_full}, we can design $(\hat{\upsilon}_i,\hat{\omega}_i)\in\Fnormal_{\graph\Lnormal_{\boundedrank}}(X_i,Y_i)$ in the following form,
    \begin{equation}\label{eq:viwi_expression}
        \begin{aligned}
        &\hat{\upsilon}_i :=\left[ U^+\,\,\breve{U}_\bot\,\,U_{Y}^{+} \right] \left[ 
        \begin{array}{ccc}
            0 & 0 & B_i^{\upsilon} \\
            0 & Z_1 & Z_2\\
            C^\upsilon_i & Z_3 & Z_4
        \end{array} 
        \right] \left[ V^+\,\,\breve{V}_\bot\,\,V_{Y}^{+} \right] ^{\top},
        \\
        &\hat{\omega}_i := \left[ U^+\,\,\breve{U}_\bot\,\,U_{Y}^{+} \right] \left[ \begin{matrix}
        	A&		C&		C_i^{\omega}\\
        	B&		0&		0\\
        	B_i^{\omega}&		0&  \hat{Z}\\
        \end{matrix} \right] \left[ V^+\,\,\breve{V}_\bot\,\,V_{Y}^{+} \right] ^{\top},
        \end{aligned}
    \end{equation}
    where $Z_1=0$ if $\downr<r$ otherwise $Z_1$ is not restricted, and $\hat{Z}=0$ if $\upr>r$ otherwise $\hat{Z}$ is not restricted. Next, we detail the construction of $(B^\omega_i,B^\upsilon_i)$, which is required to satisfy the following coupling relation, as revealed by the expression~\eqref{eq:Fnormal_graph_full},
    \begin{equation}\label{eq:BwBv_constraint}
        \left[ \begin{matrix}
	\varSigma _Y&		0\\
	0&		\Diag( z^i_1)\\
        \end{matrix} \right] B_{i}^{\omega}=\left( B_{i}^{\upsilon} \right) ^{\top}\left[ \begin{matrix}
        	\varSigma&		0\\
        	0&		\Diag (z^i_2)\\
        \end{matrix} \right].
    \end{equation}
   We denote $D^i=\mathfrak{D}(z^i_1,z^i_2)$, $D^i_{jt}=D^i(j,t)$, and $D_{jt}=\lim_{i\to\infty}D^i_{jt}$, where ``$(j,t)$" refers to the entry in the $j$-th row and $t$-th column of a matrix. Then, two sequences of matrices, $\{F^\omega_{1i}\}_{i\in\mbN}$ and $\{F^\upsilon_{1i}\}_{i\in\mbN}$ are defined as follows, 
    \begin{equation}\label{eq:construct_F}
        \left( F_{1i}^{\omega}\left( j,t \right) ,F_{1i}^{\upsilon}\left( t,j \right) \right) :=\begin{cases}
        	\left( F_{1}^{\omega}\left( j,t \right) ,\frac{D_{jt}^{i}}{1-D_{jt}^{i}}F_{1}^{\omega}\left( j,t \right) \right) ,\ \    D_{jt}\ne 1,\\
        	\left( \frac{1-D_{jt}^{i}}{D_{jt}^{i}}F_{1}^{\upsilon}\left( t,j \right) ,F_{1}^{\upsilon}\left( t,j \right) \right) ,\ \ \,    D_{jt}=1.\\
        \end{cases}
    \end{equation}
    for $i\in\mbN$, $j=1,2,\ldots,\ell-\upr$, and $t=1,2,\ldots,\downr-s$. Finally, the construction of $(B^\omega_i,B^\upsilon_i)$ is divided into four blocks, which is aligned with~\eqref{eq:BBCC},
    \begin{equation*}
        B_i^{\omega}:=\left[ \begin{matrix}
	G_1&		\ \varSigma^{-1}_Y(E^\upsilon_1)^\top\Diag(z_2^i)\\
	E_{1}^{\omega}&		F_{1i}^{\omega}\\
        \end{matrix} \right],\ \ \  B_i^{\upsilon}:=\,\,\left[ \begin{matrix}
        	\varSigma ^{-1}G_{1}^{\top}\varSigma _Y&		\ \varSigma^{-1}(E^\omega_1)^\top\Diag(z^i_1)\\
        	E_{1}^{\upsilon}&		F_{1i}^{\upsilon}\\
        \end{matrix} \right].
    \end{equation*}
    It can be verified that $(B^\omega_i,B^\upsilon_i)$ satisfies~\eqref{eq:BwBv_constraint} according to~\eqref{eq:construct_F}. Moreover, by the relationship \eqref{eq:A}, taking $i\to\infty$ in~\eqref{eq:construct_F} shows that $\left( F_{1i}^{\omega} ,F_{1i}^{\upsilon} \right) \to (F_{1}^{\omega} ,F_{1}^{\upsilon})$. Additionally, the matrices $(C_i^\omega,C^\upsilon_i)$ can be constructed in a similar manner. Consequently, taking $i\to\infty$, the limit of $(\hat{\upsilon}_i,\hat{\omega}_i)$ in \eqref{eq:viwi_expression} yields the desirable $(\upsilon,\omega)$ in~\eqref{eq:vw_expression}.

    We then turn to the ``only if" part of the theorem. Suppose that $(\upsilon,\omega)\in\Lnormal_{\graph\Lnormal_{\boundedrank}}(X,Y)$ is generated by the sequence $\{(X_i,Y_i)\}_{i\in\mbN}$ in the sense of~~\eqref{eq:limit_viwi}, and then we will identify that the structure of $(\upsilon,\omega)$ coincides with~\eqref{eq:vw_expression}. To this end, note that there exists $(\downr,\upr)$ such that $s\le \downr\le r\le \upr\le\ell\le k$ and a subsequence $\{(X_{i_j},Y_{i_j})\}_{j\in\mbN}$ satisfying $\rank(X_{i_j})=\downr$ and $\rank(Y_{i_j})=k-\upr$ for every $j\in\mbN$, and the associated subspace sequences~\eqref{eq:subspace_seq} as constructed in Lemma~\ref{lem:fivepart_convergence}. 

    Subsequently, we concentrate on the sequence $\{(X_{i_j},Y_{i_j})\}_{j\in\mbN}$ and re-assign the index as $i\in\mbN$, e.g., $\{(X_i,Y_i)\}_{i\in\mbN}$ for simplicity. Inheriting the notation in the statement of Lemma~\ref{lem:fivepart_convergence} and following the expression~\eqref{eq:Fnormal_graph_deficient}, we have
    \begin{equation*}
        \begin{aligned}
            \hat{\upsilon}_i &=  \left[ U_{i}^{+}\,\,\breve{U}_i\,\,U_{iY}^{+} \right] \left[ \begin{matrix}
                0 & 0 & B_i^{\upsilon} \\
                0 & Z_{1i} & Z_{2i}\\
                C^\upsilon_i & Z_{3i} & Z_{4i}
            \end{matrix} \right] \left[ V_{i}^{+}\,\,\breve{V}_i\,\,V_{iY}^{+} \right] ^{\top},
           \\
           \hat{\omega}_i &= \left[ U_{i}^{+}\,\,\breve{U}_i\,\,U_{iY}^{+} \right] \left[ \begin{matrix}
            	A_{i}^{\omega}&		C_{i2}^{\omega}&		C_{i1}^{\omega}\\
            	B_{i2}^{\omega}&		0&		0\\
            	B_{i1}^{\omega}&		0&		\hat{Z}_i\\
            \end{matrix} \right] \left[ V_{i}^{+}\,\,\breve{V}_i\,\,V_{iY}^{+} \right] ^{\top},
        \end{aligned}
    \end{equation*}
    where $U^+_i:=[U_i\ \tilde{U}_{i\bot}]$, $U^+_{iY}:=[\tilde{U}_{iY\bot}\ U_{iY}]$, $V^+_i:=[V_i\ \tilde{V}_{i\bot}]$, $V^+_{iY}:=[\tilde{V}_{iY\bot}\ V_{iY}]$; $Z_{1i}=0$ if $\downr<r$ otherwise $Z_{1i}$ is not restricted; and $\hat{Z}_i=0$ if $\upr>r$ otherwise $\hat{Z}_i$ is not restricted. Moreover, we note that
    \begin{equation*}
        \begin{aligned}
            &B^\upsilon_i=(U^+_i)^\top \hat{\upsilon}_iV^+_{iY},\  C^\upsilon_i=(U^+_{iY})^\top\hat{\upsilon}_iV_i^+,\ Z_{1i}=\breve{U}_i ^{\top}\hat{\upsilon}_i\breve{V}_i,
            \\
            & Z_{2i}=\breve{U}_i ^{\top}\hat{\upsilon}_i{V}_{iY}^{+},\ Z_{3i}=({U}_{iY}^{+})^{\top}\hat{\upsilon}_i\breve{V}_i,\ Z_{4i}=({U}_{iY}^{+})^{\top}\hat{\upsilon}_i{V}_{iY}^{+},
            \\
            & A_{i}^{\omega}=( U_{i}^{+} ) ^{\top}\hat\omega _iV_{i}^{+} ,\ B_{i1}^{\omega}=( U_{iY}^{+} ) ^{\top}\hat\omega _i V_{i}^{+} ,\ B_{i2}^{\omega}=( \breve{U}_i ) ^{\top}\hat\omega _i V_{i}^{+} , 
            \\
            &C_{i1}^{\omega}=( U_{i}^{+} ) ^{\top}\hat\omega _i V_{iY}^{+},\ C_{i2}^{\omega}=( U_{i}^{+} ) ^{\top}\hat\omega _i\breve{V}_i,\ \hat{Z}_i=({U}_{iY}^{+})^{\top}\hat{\omega}_i{V}_{iY}^{+}.
        \end{aligned}
    \end{equation*}
    All the quantities listed above converge (taking a subsequence if necessary), and thus, letting  $i\to+\infty$ leads to the following equalities,
    \begin{equation}\label{eq:vw_decomposition}
        \begin{aligned}
        &\upsilon =\left[ U^+\,\,\breve{U}_\bot\,\,U_{Y}^{+} \right] \left[ \begin{matrix}
            0 & 0 & B^{\upsilon} \\
            0 & Z_1 & Z_2\\
            C^\upsilon & Z_3 & Z_4
        \end{matrix} \right] \left[ V^+\,\,\breve{V}_\bot\,\,V_{Y}^{+} \right] ^{\top},
        \\
        &\omega = \left[ U^+\,\,\breve{U}_\bot\,\,U_{Y}^{+} \right] \left[ \begin{matrix}
        	A^{\omega}&		C_{2}^{\omega}&		C_{1}^{\omega}\\
        	B_{2}^{\omega}&		0&		0\\
        	B_{1}^{\omega}&		0&		\hat{Z}\\
        \end{matrix} \right] \left[ V^+\,\,\breve{V}_\bot\,\,V_{Y}^{+} \right] ^{\top},
        \end{aligned}
    \end{equation}
    Next, we examine the relation between the blocks $B_1^\omega$ and $B^\upsilon$. Through~\eqref{eq:Fnormal_graph_full}, we have $\varSigma_{R_i}B^\omega_{i1}=(B_i^\upsilon)^\top\varSigma_i$, which, by Lemma~\ref{lem:D_equivalence}, is equivalent to 
    \begin{equation}\label{eq:SigmaBC}
        {\varLambda}_i \odot B_{i1}^{\omega}+\left( {\varLambda}_i -{\bm 1}_{\left( k-\ell \right) \times \downr} \right) \odot \left( B_{i}^{\upsilon} \right) ^{\top}=0,
    \end{equation}
    where $\varLambda_i=\mathfrak{D}(\ddiag(\varSigma_{R_i}),\ddiag(\varSigma_i))$. Notice that the last $(\ell - \upr)$ diagonal entries of $\varSigma_{R_i}$ vanish in the limit, and similarly, the last $(\downr - s)$ diagonal entries of $\varSigma_i$ vanish. Accordingly, we partition the matrices into four blocks $B_{1}^{\omega}=\left[ \begin{matrix}
	G^{\omega}&		J^{\omega}\\
	E^{\omega}&		F^{\omega}\\
    \end{matrix} \right] \ \text{and}\ B^{\upsilon}=\left[ \begin{matrix}
    	G^{\upsilon}&		J^{\upsilon}\\
    	E^{\upsilon}&		F^{\upsilon}\\
    \end{matrix} \right] $, and then take $i\to+\infty$ in~\eqref{eq:SigmaBC} to obtain
    \begin{equation*}
        \left[ \begin{matrix}
        	\varXi&		{\bm 1}_{( k-\ell ) \times ( \downr-s )}\\
        	{\bm 0}_{( \ell-\upr ) \times s}&		D\\
        \end{matrix} \right] \odot \left[ \begin{matrix}
        	G^{\omega}&		J^{\omega}\\
        	E^{\omega}&		F^{\omega}\\
        \end{matrix} \right]   +\left[ \begin{matrix}
        	\varXi -{\bm 1}&		{\bm0}_{( k-\ell ) \times ( \downr-s )}\\
        	-{\bm 1}_{( \ell-\upr ) \times s}&		D -{\bm1}\\
        \end{matrix} \right] \odot\left[ \begin{matrix}
        	G^{\upsilon}&		J^{\upsilon}\\
        	E^{\upsilon}&		F^{\upsilon}\\
        \end{matrix} \right]   ^{\top}\!\!\!=0,
    \end{equation*}
    where $\varXi=\mathfrak{D}(\ddiag(\varSigma_{Y}),\ddiag(\varSigma))$, and $D$ is an element in $\mathrm{\Theta}(\ell-\upr,\downr-s)$. Therefore, it is concluded that $J^\omega=0$, $J^{\upsilon}=0$, $\varSigma_Y G^\omega=(G^\upsilon)^\top\varSigma$, and $(F^\omega,F^\upsilon)$ satisfies $D\odot F^\omega+(D-{\bm 1}_{(\ell-\upr)\times(\downr-s)})\odot(F^\upsilon)^\top=0$. A parallel analysis for the pair $(C_1^\omega,C^\upsilon)$ in~\eqref{eq:vw_decomposition} can be conducted, thereby completing the proof.
\end{proof}

\smallskip

Recalling the definition from \cite{mordukhovich2006variationalI}, the (Mordukhovich) coderivative to $\Lnormal_{\boundedrank}(\cdot)$ at $(X,Y)\in\graph\Lnormal_{\boundedrank}$ is a set-valued mapping $\diff^* \Lnormal_{\boundedrank}(X,Y): \mbRmn\rightrightarrows\mbRmn$ given as follows, 
\begin{equation*}
    \diff^* \Lnormal_{\boundedrank}(X,Y)[\omega^*] = \hkh{\upsilon^*\in \mbRmn\mid (\upsilon^*,-\omega^*)\in\Lnormal_{\graph\Lnormal_{\boundedrank}}(X,Y)},
\end{equation*}
for all $\omega^*\in\mbRmn$. Consequently, the explicit formula of $\Lnormal_{\graph\Lnormal_{\boundedrank}}(X,Y)$ identified in Theorem~\ref{the:Lnormal_graph} allows for the direct computation of the coderivative. 

\section{Bilevel programming problems with low-rank structure}\label{sec:LRBP}
Bilevel optimization, in which upper- and lower-level problems are nested with each other, has witnessed various applications~\cite{yang2025sobiRL,yang2025lancbio} and theoretical developments~\cite{lin2014solving}. When the lower-level problem possesses specific structures, e.g., the semidefinite constraints, a more tailored treatment is required~\cite{ding2014SDCMPCC,wu2014SDCMPCC,dempe2018optimality}.

In this section, we consider the bilevel programming problem~\eqref{eq:lowrank_BiO}, where the lower level seeks a solution constrained to the set of bounded-rank matrices. Our goal is to derive an optimality condition, as a direct application of the results developed in Section~\ref{sec:Geometryofgraph}.

\subsection{Motivating applications}\label{sec:bilevel_app}
We now present two representative applications falling into the scope of the formulation~\eqref{eq:lowrank_BiO}.

\paragraph{\emph{\textbf{Bilevel optimization with low-rank adaptation.}}} Natural language processing has increasingly adopted bilevel optimization to address various tasks \cite{grangier2023LLMshift,shen2025seal,zangrando2025debora}. Specifically, the upper level introduces a task-oriented variable $x\in\mbR^q$ while the lower level trains a large language model (LLM), which resorts to the popular parameter-efficient fine-tuning approach \cite{hu2022lora}, i.e., freezing the pretrained model weight $\bar{X}\in\mbRmn$ and optimizing an additive low-rank trainable matrix $X\in\boundedrank$. Therefore, the following bilevel formulation summarizes the discussed applications,
\begin{equation*}
\begin{array}{cl}
         \min\limits_{x\in\mbR^q, X^*\in\mbRmn}\ & \ \frac{1}{\left|\mathcal{D}_1\right|} \sum_{y_i \in \mathcal{D}_{1}} \lanifold(x,\bar{X}+X^*;y_i) 
         \\[3mm]
         \mathrm{s.\,t.}\!\!\!\!&\ G(x)\le 0,
         \\
         &\ 
        \begin{aligned}[t]
            X^* \in \argmin_{X\in\mbRmn}&\ \ \frac{1}{\left|\mathcal{D}_{2}\right|} \sum_{y_i \in \mathcal{D}_{2}} \lanifold(x, \bar{X}+X; y_i),
            \\
            \st&\ \ X\in\boundedrank,
        \end{aligned}
\end{array}
\end{equation*}
where $\lanifold$ denotes the loss function, $\danifold_j$ ($j=1,2$) are different datasets, and $\{y_i\}$ label the associated data points. 

\paragraph{\emph{\textbf{Data hyper-cleaning with low-rank model.}}} A line of applications in machine learning community only receives corrupted or noisy data while is required to train a reliable model. To this end, the approach, data hyper-cleaning~\cite{shaban2019truncated} formulates the task as a bilevel problem, 
\begin{equation*}
\begin{array}{cl}
         \min\limits_{w\in\mbR^q, X^*\in\mbRmn}\ \ & \ \frac{1}{\left|\mathcal{D}_1\right|} \sum_{y_i \in \mathcal{D}_{1}} \lanifold(w,X^*;y_i) 
         \\[2mm]
         \mathrm{s.\,t.}&\ 
        \begin{aligned}
            X^* &\in \argmin_{X\in\mbRmn}\ \frac{1}{\left|\mathcal{D}_{2}\right|} \sum_{y_i \in \mathcal{D}_{2}} c(w_i) \lanifold(w, X; y_i),
        \end{aligned}
\end{array}
\end{equation*}
where the upper level searches for a weight $w$ deciding the confidence of each data through a mapping $c:\mbR^q\to\mbR_+$, and the lower level trains a model according to the weighted data. When the lower-level model $X\in\mbRmn$ possesses a low-rank structure---typical examples including image recovery \cite{zhang2013hyperspectral,wang2017reweighted} and network training \cite{idelbayev2020lowrankcompress,yaras2024compressible}---it is advantageous to impose the constraint $\rank(X)\le r$, which will effectively circumvent parameter redundancy while preserving a decent performance.

\subsection{Optimality conditions via a relaxation}
Note that finding a global minimizer of a function subject to the bounded-rank constraint is NP-hard in general~\cite{gillis2011NPlowrank}. Nevertheless, existing literature~\cite{schneider2015Lojaconvergence,levin2023remedy,olikier2023RFDR} is able to find a first-order point in the sense that the antigradient belongs to the Mordukhovich normal cone of the determinantal variety. Therefore, we turn to the formulation~\eqref{eq:lowrank_BiO_M1}, which serves as a relaxation for \eqref{eq:lowrank_BiO} by replacing the lower-level global optimality with the first-order stationarity. Subsequently, introducing a slack variable $Y\in\mbR^{m\times n}$, we obtain the following formulation equivalent to~\eqref{eq:lowrank_BiO_M1},
\begin{equation}\label{eq:lowrank_BiO_M}\tag{M-LRBP}
\begin{aligned}
    \min_{x\in\mbR^q,X\in\mbR^{m\times n}}&\ \lanifold(x,X)\\[-1.5mm]
    \mathrm{s.\,t.}\ \ \ \ \ \,&\ G(x)\le 0,
        \\
        &\ \nabla_X F(x,X)+Y=0,
        \\
        &\ (X,Y)\in\graph \Lnormal_{\boundedrank}.
    \end{aligned}
\end{equation}
Since the relaxation is based on the Mordukhovich normal cone, we prefix the name of~\eqref{eq:lowrank_BiO} with an additional ``M-''. Moreover, we assume that the mappings $\lanifold:\mbR^q\times\mbRmn\to\mbR$ and $G:\mbR^q\to\mbR^p$ are continuously differentiable, while $F:\mbR^q\times\mbRmn\to\mbR$ is twice continuously differentiable.

We then investigate the relationship between~\eqref{eq:lowrank_BiO} and~\eqref{eq:lowrank_BiO_M} in terms of \emph{local optimal solutions}, namely, points that minimize the objective $\lanifold$ over a neighborhood in the feasible region.
\begin{proposition}
    If $(\tilde{x},\tilde{X},\tilde{Y})$ is a local optimal solution of~\eqref{eq:lowrank_BiO_M}, and additionally, $\tilde{X}\in\argmin_{X\in\boundedrank} F(\tilde{x},X)$, then $(\tilde{x},\tilde{X})$ is a local optimal solution of~\eqref{eq:lowrank_BiO}.
\end{proposition}
\begin{proof}
    Suppose, toward a contradiction, that $(\tilde{x},\tilde{X})$ is not locally optimal for problem~\eqref{eq:lowrank_BiO}. i.e., there exist $(\tilde{x}_i,\tilde{X}_i)\to(\tilde{x},\tilde{X})$ feasible for~\eqref{eq:lowrank_BiO} such that $\lanifold(\tilde{x}_i,\tilde{X}_i)<\lanifold(\tilde{x},\tilde{X})$. Note that the optimality $\tilde{X}_i\in\argmin_{X\in\boundedrank} F(\tilde{x}_i,X)$ always implies the first-order condition $-\nabla_X F(\tilde{x}_i,\tilde{X}_i)\in\Lnormal_{\boundedrank}(\tilde{X}_i)$; see~\cite{schneider2015Lojaconvergence}. Therefore, any feasible points $(\tilde{x}_i,\tilde{X}_i)$ of~\eqref{eq:lowrank_BiO} induce an $(\tilde{x}_i,\tilde{X}_i,-\nabla_XF(\tilde{x}_i,\tilde{X}_i))$ feasible for~\eqref{eq:lowrank_BiO_M}. The condition $\lanifold(\tilde{x}_i,\tilde{X}_i)<\lanifold(\tilde{x},\tilde{X})$ contradicts the local optimality of $(\tilde{x},\tilde{X},\tilde{Y})$.
\end{proof}

Finally, taking into account the coderivative of the Mordukhovich normal cone mapping developed in section~\ref{sec:Geometryofgraph}, we can give a Fritz John type M-stationary condition for~\eqref{eq:lowrank_BiO_M}, where the notation follows from Theorem~\ref{the:Lnormal_graph}, e.g., $k:=\min\{m,n\}$ and $0\le s\le r \le \ell \le k$, and we denote by $\janifold_X$ the partial Jacobian of a mapping with respect to $X$.

\begin{proposition}\label{pro:M-stationary}
    Let $(\tilde{x},\tilde{X},\tilde{Y})$ be a local optimal solution to problem~\eqref{eq:lowrank_BiO_M}. Suppose that $\rank(\tilde{X})=s$ and $\rank(\tilde{Y})=k-\ell$, and let the SVDs be $\tilde{X}=U\varSigma V^\top$ and $\tilde{Y}=U_Y\varSigma_YV_Y^\top$. Then there exist a multiplier $(\mu,\lambda,\delta)\in\mbR\times\mbR^{p}\times\mbR^{m\times n}$ and matrices $(\upsilon,\omega)\in\mbRmn\times\mbRmn$ such that
    \begin{align}
        &\mu\nabla_x\lanifold(\tilde{x},\tilde{X})+\nabla G(\tilde{x})\lambda+\janifold_X(\nabla_xF)(\tilde{x},\tilde{X})[\delta] = 0,  \label{eq:lowrank_BiO_optimilaty_1}
        \\
        &\mu\nabla_X\lanifold(\tilde{x},\tilde{X})+\janifold_X (\nabla_X F)(\tilde{x},\tilde{X})[\delta] + \upsilon = 0,     \label{eq:lowrank_BiO_optimilaty_2}
        \\
        &\delta+\omega = 0,   \label{eq:lowrank_BiO_optimilaty_3}
        \\
        & \innerp{G(\tilde{x}),\lambda}_{\mbR^p}= 0,\ \ \lambda\ge 0,   \label{eq:lowrank_BiO_optimilaty_4}
        \\
        &(\upsilon,\omega)\ \text{are expressed as in~\eqref{eq:vw_expression}}.    \nonumber
    \end{align}
\end{proposition}
\begin{proof}
    The formulation~\eqref{eq:lowrank_BiO_M} is a program with inequality constraints, equality constraints, and a geometric constraint with $(\tilde{x},\tilde{X},\tilde{Y})$ as a local solution. Applying \cite[Theorem 5.21]{mordukhovich2006variationalII} implies the existence of a multiplier $(\mu,\lambda,\delta)\in\mbR\times\mbR^{p}\times\mbR^{m\times n}$ with $(\mu,\lambda,\delta)\neq 0$ and matrices $(\upsilon,\omega)\in\Lnormal_{\graph\Lnormal_{\boundedrank}}(\tilde{X},\tilde{Y})$ satisfying conditions~\eqref{eq:lowrank_BiO_optimilaty_1}-\eqref{eq:lowrank_BiO_optimilaty_4}. In addition, we note that the feasibility of $(\tilde{x},\tilde{X},\tilde{Y})$ reveals that $\tilde{Y}=-\nabla_X F(\tilde{x},\tilde{X})$ and $(\tilde{X},\tilde{Y})\in\graph \Lnormal_{\boundedrank}$, which validates the application of Theorem~\ref{the:Lnormal_graph} to express $(\upsilon,\omega)\in\Lnormal_{\graph\Lnormal_{\boundedrank}}(\tilde{X},\tilde{Y})$ via~\eqref{eq:vw_expression}.
\end{proof}

\section{Conclusions and perspectives}\label{sec:conclusion}
In this paper, we conduct the variational analysis of determinantal varieties. Specifically, we provide a unified framework for analyzing first- and second-order tangent sets to various low-rank set, recovering existing results, and revealing a range of new ones. Drawing on the tangent sets, we establish a sufficient and necessary condition to characterize the second-order equivalence between a general nonsmooth problem and its smooth parameterization. The developed framework is applied to low-rank optimization and to low-rank semidefinite programs with the Burer--Monteiro parameterization. In another thread of analysis, we investigate the geometry of the graph of the Mordukhovich normal cone to the matrix variety, which plays a role in low-rank bilevel programs. We conclude with several remarks and outline potential directions for future research inspired by this work.

\emph{Extension to sparsity constraints.} 
In fact, Theorem~\ref{the:cal_tangentsets} also finds potential applications in sparse scenarios.  Specifically, consider the sparse set $\canifold_s := \{x\in\mbR^q \mid \|x\|_0 \le s\}$, where $\|\cdot\|_0$ counts the number of nonzero entries of a vector. 
Let $|x|^\downarrow$ denote the vector obtained by sorting $(|x_1|,\dots,|x_q|)$ in a non-increasing order. Then, analogous to~\eqref{eq:sigmarboundedrank}, the sparse set admits the characterization $\canifold_s=\{x\in\mbR^q \mid (|x|^\downarrow)_{s+1}=0\}$. It can be verified that $\canifold_s$ satisfies the error bound condition: $\dist(y,\canifold_s)\le \sqrt{q-s}\,(|y|^\downarrow)_{s+1}$ for any $y\in\mbR^q$. Additionally, we note that $(|x|^\downarrow)_{s+1}=\lambda_{s+1}(\Diag(|x|))$, where both the mappings $x\mapsto \Diag(|x|)$ and $\lambda_{s+1}(\cdot)$ are locally Lipschitz and admit first- and second-order directional derivatives; hence, the composite mapping $x\mapsto (|x|^\downarrow)_{s+1}$ inherits the same properties. Therefore, Theorem~\ref{the:cal_tangentsets} can be invoked to derive the first- and second-order tangent sets to $\canifold_s$. Extending the spirit, Theorem~\ref{the:expressions_McapK} may further be employed to develop the intersection rules, when an additional constraint is imposed to the sparse set.

{\emph{Low-rank sets intersecting with inequality constraints.}} An immediate extension of Theorem~\ref{the:expressions_McapK} involves cases where $\kanifold$ is defined by a system including inequalities, i.e., $\kanifold=\{X\in\mathbb{R}^q \mid h(X)=0,\; g(X)\le 0\}$, mirroring the setup in Theorem~\ref{the:cal_tangentsets}. In such a scenario, the parameterization $(\bmanifold,\phi)$ may pull $\kanifold$ back to a \emph{manifold with boundary}~\cite{lee2012manifolds}. An appropriate generalization of Theorem~\ref{the:expressions_McapK} holds promise for applications to the intersection of $\boundedrank$ with constraints such as the closed unit Frobenius ball, the symmetric box, or the spectrahedron, as discussed in~\cite{li2020jotaspectral}.

\emph{Algorithms for low-rank bilevel programming problems.} While Proposition~\ref{pro:M-stationary} gives an optimality condition for problem~\eqref{eq:lowrank_BiO_M}, designing an algorithm provable to accumulate at such stationary points remains a challenge. We envision that progress can be made by initially restricting attention to scenarios where the lower-level low-rank problem minimizes a strongly convex objective~\cite{park2018findingefficiently}, a direction inspired by the avenue developed in the existing bilevel optimization literature~\cite{ghadimi2018approximation}.

\emph{Higher-order stationarity of parameterizations.} Theorem~\ref{the:F4to2} shows that the Burer--Monteiro parameterization satisfies ``\,$4\!\Rightarrow\!2$'' for~\eqref{eq:Fproblem}. Recalling that ``\,$2\!\Rightarrow\!1$'' also holds~\cite[Proposition~2.7]{levin2025effect}, it is natural to ask whether ``\,$2k\!\Rightarrow\!k$'' holds for every $k\ge1$, with $k$-th-order stationarity for~\eqref{eq:Fproblem} suitably defined for $k\ge3$. It also remains open whether ``\,$4\!\Rightarrow\!2$'' holds for the LR parameterization~\eqref{eq:LR} at $(L,R)$ with $\rank(L)=\rank(R)=\rank(LR^\top)<r$, and whether similar results extend to parameterizations of tensor varieties.

\appendix

\section{Tangent sets to tensor varieties}\label{app:tensorvariety}
In this section, we introduce the hierarchical Tucker (HT) variety \cite{hackbusch2009newHT,grasedyck2010hierarchicalSVD}, and then show that the Tucker and tensor train (TT) varieties \cite{tucker1964extension,oseledets2011TTSVD} arise as specific cases of the HT variety. Finally, a proof for Proposition~\ref{pro:tangenttovarieties} is provided.

\subsection{Hierarchical Tucker varieties}\label{app:HTvariety}
Given the order $d$, a \emph{dimension tree} $T$ on $\{1,2,\ldots,d\}$ is a binary tree whose nodes are nonempty subsets $t\subseteq\{1,\ldots,d\}$ such that: 1)~the root is $t_r=\{1,\ldots,d\}$; 2)~the leaves are the singletons $\{k\}$, $k=1,2,\ldots,d$; 3)~if $t$ is an internal node with children $t_1$ and $t_2$, then $t=t_1\cup t_2$ and $k_1<k_2$ for all $k_1\in t_1$ and $k_2\in t_2$. For each subset $t\subseteq\{1,2,\ldots,d\}$, we denote the associated dimension by $n_t:=\prod_{k\in t}n_k$, and the set complementary to $t$ by $t_{-}=\{1,2,\ldots,d\} \backslash t$.

Based on a fixed dimension tree $T$, the \emph{HT mode-$t$ unfolding} (or matricization) of $\tensX\in\tensorspace$ is formed by arranging the modes in $t=\{\mu_1,\ldots,\mu_p\}$ along the row dimension and those in $t_-=\{v_1,\ldots,v_{d-p}\}$ along the column dimension:
$$
X^\mht_{(t)} \in \mbR^{n_t\times n_{t_-}}\ \ \text {with}\ \ X^\mht_{(t)}(i_{\mu_1}, i_{\mu_2}, \ldots, i_{\mu_p};i_{\nu_1}, i_{\nu_2}, \ldots, i_{\nu_{d-p}})=\mathbf{X}_{i_1, \ldots, i_d},
$$
Each matricization operator is invertible, as it is a one-to-one rearrangement of the tensor entries. Hence the mapping $\tensX \mapsto X^{\mht}_{(t)}$ admits a unique inverse $\tensorizeht_{(t)}:\mathbb{R}^{n_t\times n_{t_-}}\to\tensorspace$, called the \emph{mode-$t$ tensorization}.

The \emph{HT rank} of a tensor $\tensX$ is defined as the tuple
$$
\rankht(\tensX)=\left(r_t\right)_{t\in T}\  \text{ with }\  r_t=\rank(X^\mht_{(t)})\ \text{for}\ t\in T.
$$
Consequently, we can introduce the set of hierarchical tensors with an HT rank at most $\vecr\in\mbN^{|T|}$:
\begin{equation*}
    \boundedht = \{\tensX\in\tensorspace \mid \rankht(\tensX)\le\vecr\}.
\end{equation*}
Note that $\boundedht$ is the common zero set of all $(r_t+1)$-minors of the mode-$t$ unfolding matrices $\{X^{\mht}_{(t)}\}_{t\in T}$, i.e., $\rank(X^\mht_{(t)})\le r_t$ for all $t\in T$. Hence $\boundedht$ is a real algebraic variety; we refer to it as the \emph{HT variety}.

\subsection{\texorpdfstring{Reduction to Tucker varieties and tensor train varieties}{Reduction to Tucker varieties and tensor train varieties}}\label{app:reduction}
Different choices of the dimension tree $T$ with the associated rank parameter $\vecr=(r_t)_{t\in T}$ will yield varieties with different structures. We then specify the constructions of $T$ and $(r_t)_{t\in T}$ to reduce the HT variety to two typical instances---the \emph{Tucker variety} \cite{tucker1964extension,gao2025lowranktucker} and the \emph{tensor train variety} \cite{oseledets2011TTSVD,kutschan2018tangentTT}, respectively.

To derive the concept of the Tucker variety, let $T$ be the dimension tree on $\{1,2,\ldots,d\}$ that, at each level, divides the first mode index to the left child and assigns the rest to the right child. Formally, set $t_{r}=\{1,2,\ldots,d\}$, and recursively, for the internal node $t=\{k,\ldots,d\}$ with $1\le k\le d-1$, define its two children by $t_1=\{k\}$ and $t_2=\{k+1,\ldots,d\}$. Moreover, letting $r_{\{k\}}$ denote the rank of the mode-$\{k\}$ unfolding of a given $\tensX$ for $k=1,2,\ldots,d$, we can define the mapping $\rank_\mtc:\tensorspace\to \mathbb{N}^d:\,\tensX\mapsto(r_{\{1\}},r_{\{2\}},\ldots,r_{\{d\}})$. Consequently, the Tucker variety can be defined as
\begin{equation}\label{eq:tctensor}
    \boundedtucker = \{\tensX\in\tensorspace \mid \ranktc(\tensX)\le\vecr^\mtc\},
\end{equation}
where $\vecr^\mtc=(r^\mtc_1,r^\mtc_2,\ldots, r^\mtc_d)\in\mathbb{N}^d$ is a given Tucker rank parameter such that $r^\mtc_k\le\min\{n_{\{k\}},n_{\{k\}_-}\}$ for $k\in\{1,2,\ldots,d\}$.

In parallel with the derivation of~\eqref{eq:tctensor}, we can also view the TT variety as a special case of the HT variety, which has been discussed in~\cite[\S5]{uschmajew2013geometryofHT}; for completeness, we briefly recall the construction. We adopt the same partition of the tree $T$ as in the Tucker case, namely, setting the root $t_{r}=\{1,2,\ldots,d\}$, and recursively, for the internal node $t=\{k,\ldots,d\}$ with $1\le k\le d-1$, defining the children by $t_1=\{k\}$ and $t_2=\{k+1,\ldots,d\}$. Then, letting $r_{\{k,\ldots,d\}}$ denote the rank of the mode-$\{k,\ldots,d\}$ unfolding of a given $\tensX$ for $k=2,\ldots,d$, we can define the mapping $\rank_\mtt:\tensorspace\to \mathbb{N}^{d-1}:\,\tensX\mapsto(r_{\{2,\ldots,d\}},r_{\{3,\ldots,d\}},\ldots,r_{\{d\}})$, which introduces the definition of the TT variety as follows,
\begin{equation}\label{eq:tttensor}
    \boundedtt = \{\tensX\in\tensorspace \mid \ranktt(\tensX)\le\vecr^\mtt\},
\end{equation}
where $\vecr^\mtt=(r^\mtt_1,r^\mtt_2,\ldots, r^\mtt_{d-1})\in\mathbb{N}^{d-1}$ is a given TT rank parameter such that $r^\mtt_k\le\min\{n_{\{1,\ldots,k\}},n_{\{1,\ldots,k\}_-}\}$ for $k\in\{1,2,\ldots,d-1\}$.
In summary, $\boundedtucker$ and $\boundedtt$ are HT varieties on the same dimension tree, where only the ranks at the leaves $\{k\}$ and at the nodes $\{k+1,\ldots,d\}$ are constrained, respectively.
\subsection{Proof of Proposition~\ref{pro:tangenttovarieties}}\label{app:prooftangentvariety}
\begin{proof}
We begin by formally verifying that the $\boundedtensor$ given through~\eqref{eq:tensor_defined_sigma} satisfies Assumption~\ref{assu:errorbound}. Firstly, note that the mapping $\tensX\mapsto\sigma_{r_t+1}(X^\mht_{(t)})$ is Lipschitz continuous for all $t\in T$. We then turn to Assumption~\ref{assu:errorbound}(ii). For an arbitrary $\tensY\in\tensorspace$, we consider the hierarchical truncation constructed by \cite[Lemma~3.15]{grasedyck2010hierarchicalSVD} and denote it as $\projection_{\le \vecr}(\tensY)$.  It is revealed from \cite[Theorem~3.11]{grasedyck2010hierarchicalSVD} that 
\begin{equation*}
    \dist(\tensY,\boundedht)\le \|\tensY-\projection_{\le \vecr}(\tensY)\|_{\frob} \le \sqrt{\sum_{t\in T}\sum_{i>r_t}\sigma^2_{i}(Y^{\mht}_{(t)})} \le  \rho\sqrt{\sum_{t\in T} \sigma^2_{r_t+1}(Y^{\mht}_{(t)})}\,,
\end{equation*}
where $\rho=\sqrt{\max_{t\in T}\{\min\{n_t,n_{t_-}\}-r_t\}}$. Therefore, we can apply Theorem~\ref{the:cal_tangentsets} to $\boundedtensor$ given through~\eqref{eq:tensor_defined_sigma}.

Regarding the tangent cone, we have
\begin{align*}
    \Btangent_{\boundedht}(\tensX) &= \hkh{\tenseta\in\tensorspace \mid \sigma^\prime_{r_t+1}(X^\mht_{(t)};\eta^\mht_{(t)})=0\ \text{for}\ t\in T}
    \\
    &=\bigcap_{t\in T} \hkh{\tenseta\in\tensorspace \mid \sigma^\prime_{r_t+1}(X^\mht_{(t)};\eta^\mht_{(t)})=0}
    \\
    &= \bigcap_{t\in T}{\tensorize_{(t)}^\mht\kh{\Btangent_{\ranifold_t}(X^\mht_{(t)})}}
\end{align*}
where $\ranifold_t:=\mbR_{\le r_t}^{n_t\times n_{t_-}}$ and the last equality comes from~\eqref{eq:Btangent_cone_boundedrank_sigma_descrip}. Similarly, let $\tenseta\in \Btangent_{\boundedtensor}(\tensX)$, and it holds that
\begin{align*}
    \tangenttwo_{\boundedht}(\tensX;\tenseta) &= \hkh{\tenszeta\in\tensorspace \mid \sigma^{\prime\prime}_{r_t+1}(X^\mht_{(t)};\eta^\mht_{(t)},\zeta^\mht_{(t)})=0\ \text{for}\ t\in T}
    \\
    &=\bigcap_{t\in T} \hkh{\tenszeta\in\tensorspace \mid \sigma^{\prime\prime}_{r_t+1}(X^\mht_{(t)};\eta^\mht_{(t)},\zeta^\mht_{(t)})=0}
    \\
    &= \bigcap_{t\in T}{\tensorize_{(t)}^\mht\kh{\tangenttwo_{\ranifold_t}(X^\mht_{(t)};\eta^\mht_{(t)})}}.
\end{align*}
\end{proof}

\section{Proof of Theorem~\ref{the:expressions_McapK}}\label{app:McapK}
In this section, we provide the proof for Theorem~\ref{the:expressions_McapK}. The main idea is to show that if $(\manifold,\kanifold)$ satisfies Assumption~\ref{assu:MintersecK}, their intersection $\manifold\cap\kanifold$ satisfies Assumption~\ref{assu:errorbound}, thereby applying the developed Theorem~\ref{the:cal_tangentsets} gives the conclusion.

As a preliminary, we show that any set $\kanifold$ realized as a level set of a smooth mapping with constant rank satisfies the error bound property, which is closely related to Assumption~\ref{assu:MintersecK}(i).

\begin{lemma}\label{lem:local_error_bound_constant_rank}
Let $h:\mathbb{R}^q\to\mathbb{R}^k$ be smooth and let $\kanifold=\{\tilde X\in\mathbb{R}^q \mid h(\tilde X)=0\}$. Given $X\in\kanifold$, suppose that there exists a neighborhood $\neighbor_1\subseteq\mbR^q$ around $X$ where the Jacobian $\diff h$ has constant rank $p$. Then it admits a neighborhood $\neighbor_2\subseteq\mbR^q$ around $X$ and a constant $\rho>0$ such that $\dist(Y,\kanifold)\le \rho\|h(Y)\|_2$ for all $Y\in\neighbor_2$.
\end{lemma}
\begin{proof}
By the constant rank theorem \cite[Theorem~4.12]{lee2012manifolds}, there exist open neighborhoods $\neighbor\subseteq\mathbb{R}^q$ of $X$ and $\wanifold\subseteq\mathbb{R}^k$ of $h(X)=0\in\mbR^k$, a diffeomorphism $\theta:\neighbor\to \neighbor'\subseteq\mbR^q$ with $\theta(X)=0$, and a diffeomorphism $\psi:\wanifold\to \wanifold'\subseteq\mbR^k$ with $\psi(0)=0$, such that in the coordinates $(u,v)\in\mathbb{R}^p\times\mathbb{R}^{q-p}$ one has
\[
(\psi\circ h\circ \theta^{-1})(u,v)=(u,0)\in\mathbb{R}^p\times\mathbb{R}^{\,k-p}.
\]
In these coordinates, we have $\kanifold\cap \neighbor=\{\tilde{X}\in \neighbor\mid \theta(\tilde{X})=(0,v)\ \text{for some}\ v\in \mbR^{q-p}\}$.

Shrink $\neighbor$ if necessary such that $\neighbor'=\theta(\neighbor)$ is a ball centered at $0$, and $\theta$ and $\theta^{-1}$ are Lipschitz on $\neighbor$ and $\neighbor^\prime$ with constants $L_{\theta}$ and $L_{\theta^{-1}}$, respectively; and similarly, shrink $\wanifold$ such that $\psi$ is Lipschitz on $\wanifold$ with a constant $L_{\psi}$. Moreover, take the neighborhood $\neighbor_2\subseteq \neighbor$ such that $h(\neighbor_2)\subseteq \wanifold$. Given any $Y\in \neighbor_2$, let $\theta(Y)=(u,v)$ and $Y_p\in\kanifold$ be the point with the coordinate $\theta(Y_p)=(0,v)$, which leads to
\begin{equation}\label{eq:dist_YK}
\dist(Y,\kanifold)\le\|Y-Y_p\|_2\le L_{\theta^{-1}}\| (u,v)-(0,v)\|_2=L_{\theta^{-1}}\|u\|_2.
\end{equation}
On the other hand, it holds that $\|u\|_2=\|(u,0)\|_2=\|\psi\circ h\circ \theta^{-1}(u,v)\|_2=\|\psi(h(Y))-\psi(0)\|_2\le L_{\psi}\|h(Y)\|_2$. Combining it with~\eqref{eq:dist_YK} gives $\dist(Y,\kanifold)\le L_{\theta^{-1}}L_{\psi}\|h(Y)\|_2$. Setting $\rho:=L_{\theta^{-1}}L_{\psi}$ completes the proof.
\end{proof}

Subsequently, we show that when two manifolds intersect transversally, the distance to their intersection can be bounded by the distances to each manifold. This property resonates with Assumption~\ref{assu:MintersecK}(ii).
\begin{lemma}\label{lem:transverse_local_error_bound}
Let $\kanifold_1,\kanifold_2\subseteq\mathbb{R}^q$ be smooth embedded manifolds, and let $X\in \kanifold_1\cap\kanifold_2$. Suppose that $\kanifold_1$ and $\kanifold_2$ intersect transversally at $X$, i.e., $\tangent_{\kanifold_1}(X)+\tangent_{\kanifold_2}(X)=\mathbb{R}^q$. Then there exist a neighborhood $\neighbor\subseteq\mathbb{R}^q$ of $X$ and a constant $C>0$ such that
\[
\dist(Y,\kanifold_1\cap\kanifold_2)\ \le\ C\Bigl(\dist(Y,\kanifold_1)+\dist(Y,\kanifold_2)\Bigr)\ \ \text{for all}\ Y\in \neighbor.
\]
\end{lemma}

\begin{proof}
Since $\kanifold_i$ are embedded, according to \cite[Theorem 8.75]{boumal2023introduction}, there exist neighborhoods $\neighbor_i$ of $X$ and smooth $h_i:\neighbor_i\to\mathbb{R}^{k_i}$ with full-rank Jacobians at $X$, such that $\kanifold_i\cap \neighbor_i=\{\,\tilde X\in \neighbor_i\mid\ h_i(\tilde X)=0\}$ for $i=1,2$.

Let $\neighbor:=\neighbor_1\cap \neighbor_2$ and $h:\neighbor\to \mbR^{k_1}\times \mbR^{k_2}:\, \tilde{X}\mapsto(h_1(\tilde{X}),h_2(\tilde{X}))$. Transversality at $X$ is equivalent to $\rank(\diff h_X)=k_1+k_2$. By continuity of $\diff h$, shrinking $\neighbor$ if necessary, we may assume $\rank(\diff h_{\tilde{X}})\equiv k_1+k_2$ on $\neighbor$. Apply Lemma~\ref{lem:local_error_bound_constant_rank} on the manifold $\{\tilde{X}\in \neighbor\mid h(\tilde{X})=0\}$
to give a $\rho_h>0$ such that
\begin{equation}\label{eq:EB-int}
\dist(Y,\kanifold_1\cap\kanifold_2)\le\rho_h\|h(Y)\|_2
\ \ \text{for all}\ Y\in\neighbor.
\end{equation}
Shrinking the neighborhood $\neighbor$, there exist $L_i>0$ such that $h_i$ are $L_i$-Lipschitz in $\neighbor$ ($i=1,2$).

Consequently, given any $Y\in\neighbor$, we can pick $\projection_i(Y)\in\kanifold_i$ with $\|Y-\projection_i(Y)\|_2=\dist(Y,\kanifold_i)$, and thus we have $\|h_i(Y)\|_2=\|h_i(Y)-h_i(\projection_i(Y))\|_2\le L_i\dist(Y,\kanifold_i)$ for $i=1,2$. Additionally, incorporating $\|h(Y)\|_2\le \|h_1(Y)\|_2+\|h_2(Y)\|_2$ into \eqref{eq:EB-int} yields $\dist(Y,\kanifold_1\cap\kanifold_2)\le\rho_h\!\left(L_1\,\dist(Y,\kanifold_1)+L_2\,\dist(Y,\kanifold_2)\right)$, and taking $C=\rho_h\max\{L_1,L_2\}$ completes the proof.
\end{proof}

\bigskip

We are now in a position to prove Theorem~\ref{the:expressions_McapK}. The proof proceeds by combining Lemma~\ref{lem:local_error_bound_constant_rank} and Lemma~\ref{lem:transverse_local_error_bound} to establish the error bound property of $\bmanifold \cap \bkanifold$ in the auxiliary space $\mbR^{\bq}$, and then transferring this property to $\manifold \cap \kanifold$ via the smooth mapping $\phi:\mbR^{\bq}\to\mbR^q$. As a result, $\manifold \cap \kanifold$ satisfies Assumption~\ref{assu:errorbound}, and thus falls within the scope of Theorem~\ref{the:cal_tangentsets}.

\medskip

\noindent\textit{Proof of Theorem~\ref{the:expressions_McapK}}\ \ \ 
Since $\manifold$ satisfies Assumption~\ref{assu:errorbound} and $h$ is smooth, we can take a neighborhood $\neighbor_p\subseteq\mbR^q$ of $X$ such that $c_1$ and $h$ are $L_c$- and $L_h$-Lipschitz on $\neighbor_p$, respectively, and $\dist({Y},\manifold)\le \rho_\manifold\|c_1({Y})\|_2$ holds for a constant $\rho_\manifold>0$ and any ${Y}\in \neighbor_p$. Therefore, we can find a neighborhood $\neighbor\subseteq \neighbor_p$ such that given any $Y\in \neighbor$, there exists $Y_p\in\manifold\cap \neighbor_p$ such that $\|Y-Y_p\|_2=\dist(Y,\manifold)$. 

Since $\phi_{\bmanifold}:\bmanifold\to\manifold$ is smooth and open at $x\in\phi^{-1}(X)$, we can shrink $(\neighbor_p,\neighbor)$ until it admits a neighborhood $\lineB_p\subseteq\mbR^{\bq}$ of $x$ such that $\manifold\cap\neighbor_p\subseteq \phi(\bmanifold\cap\lineB_p)$ and $\phi$ is $L_\phi$-Lipschitz on $\lineB_p$. Therefore, we can find a preimage of $Y_p$, denoted by $y_p\in\bmanifold\cap\lineB_p$. By Lemma~\ref{lem:transverse_local_error_bound} and shrinking $(\neighbor_p,\neighbor,\lineB_p)$ if necessary, we can find a $z_p\in(\bmanifold\cap\bkanifold)\cap\lineB_p$ such that $\|y_p-z_p\|_2=\dist(y_p,\bmanifold\cap\bkanifold)$, and the transversality of $\bmanifold\cap\bkanifold$ reveals that for a constant $C>0$, we have
\begin{equation}\label{eq:ypzp}
\|y_p-z_p\|_2\le C\kh{\dist(y_p,\bmanifold)+\dist(y_p,\bkanifold)}.    
\end{equation}
Note that $\dist(y_p,\bmanifold)=0$ due to $y_p\in\bmanifold$. Moreover, since $\bkanifold=\{\tilde{x}\in\mbR^{\bq}\mid h\circ\phi (\tilde{x})=0\}$ satisfies Assumption~\ref{assu:MintersecK}(i), we apply Lemma~\ref{lem:local_error_bound_constant_rank} on $\bkanifold$ to obtain $\dist(y_p,\bkanifold)\le \rho_{\bkanifold}\|h(\phi(y_p))\|_2$ for a constant $\rho_{\bkanifold}>0$. Therefore, \eqref{eq:ypzp} indicates that
\begin{equation}
    \|y_p-z_p\|_2\le C\dist(y_p,\bkanifold)\le C\rho_{\bkanifold}\|h(\phi(y_p))\|_2 = C\rho_{\bkanifold}\|h(Y_p)\|_2.
\end{equation}
Consequently, letting $Z_p=\phi(z_p)\in\manifold\cap\kanifold$, it holds that for all $Y\in\neighbor$,
\begin{equation*}
\begin{aligned}
    \|Y-Z_p\|_2 &\le \|Y-Y_p\|_2 + \|Y_p-Z_p\|_2 
    \\
    &\le \rho_\manifold\|c_1(Y)\|_2 + L_{\phi}\|y_p-z_p\|_2   
    \\
    &\le \rho_\manifold\|c_1(Y)\|_2 + L_{\phi}C\rho_{\bkanifold}\|h(Y_p)\|_2
    \\
    &\le \rho_\manifold\|c_1(Y)\|_2 + L_{\phi}C\rho_{\bkanifold} \kh{L_h\|Y-Y_p\|_2 + \|h(Y)\|_2}
    \\
    &\le \rho_\manifold(1+L_{\phi}L_hC\rho_{\bkanifold} )\|c_1(Y)\|_2 + L_{\phi}C\rho_{\bkanifold} \|h(Y)\|_2,
\end{aligned}
\end{equation*}
where we employ  the triangle inequality $\|h(Y_p)\|_2 - \|h(Y)\|_2\le \|h(Y_p)-h(Y)\|_2$, and the Lipschitz continuity of $\phi$ and $h$. Therefore, setting $\rho=\rho_\manifold(1+L_{\phi}L_hC\rho_{\bkanifold} )+L_{\phi}C\rho_{\bkanifold}$ concludes that $\dist(Y,\manifold\cap\kanifold)\le \rho\|(c_1(Y),h(Y))\|_2$ for all $Y\in\neighbor$, which together with the locally Lipschitz property of $(c_1,h)$, verifies that $\manifold\cap\kanifold$ satisfies Assumption~\ref{assu:errorbound}. Applying Theorem~\ref{the:cal_tangentsets} to $\manifold\cap\kanifold=\{\tilde{X}\in\mbR^q\mid c_1(\tilde{X})=0,\,h(\tilde{X})=0\}$ gives
\begin{align}
    \Btangent_{\manifold\cap\kanifold}(X) &= \{\eta\in\mbR^q\mid c_1^\prime(X;\eta)=0,\,h^\prime(X;\eta)=0\}    \nonumber
    \\
    &= \{\eta\in\mbR^q\mid c_1^\prime(X;\eta)=0\}\cap\{\eta\in\mbR^q\mid h^\prime(X;\eta)=0\}   \nonumber
    \\
    &= \Btangent_{\manifold}(X)\cap\Btangent_{\kanifold}(X),  \nonumber
\end{align}
where the last equality holds since both $\manifold$ and $\kanifold$ satisfy Assumption~\ref{assu:errorbound}. Similarly, letting $\eta\in\Btangent_{\manifold\cap\kanifold}(X)$, we have
\begin{align}
    \tangenttwo_{\manifold\cap\kanifold}(X;\eta) &= \{\zeta\in\mbR^q\mid c_1^{\prime\prime}(X;\eta,\zeta)=0,\,h^{\prime\prime}(X;\eta,\zeta)=0\}    \nonumber
    \\
    &= \{\zeta\in\mbR^q\mid c_1^{\prime\prime}(X;\eta,\zeta)=0\}\cap\{\zeta\in\mbR^q\mid h^{\prime\prime}(X;\eta,\zeta)=0\}   \nonumber
    \\
    &= \tangenttwo_{\manifold}(X;\eta)\cap\tangenttwo_{\kanifold}(X;\eta). \nonumber
\end{align}
\qed

\section{Tangent sets to $\boundedrank\cap\hanifold$}\label{app:McapH_application}
In this section, which supplements section~\ref{sec:lowrank_rectangular}, we apply Theorem~\ref{the:expressions_McapK} to sets of low-rank rectangular matrices in the form of $\boundedrank\cap\hanifold$, recovering the first-order results in \cite{li2023normalboundedaffine} and \cite{yang2025spacedecouple}, respectively, and further characterizing the second-order tangent sets. Finally, we consider the case $\hanifold = \matrixhyperboloid$ and derive the first- and second-order tangent sets of $\boundedrank \cap \matrixhyperboloid$, which represent novel contributions.

The derivation proceeds by checking that $\boundedrank\cap\hanifold$ satisfies Assumption~\ref{assu:MintersecK}, which mainly resorts to the LR parameterization for $\boundedrank$,
\begin{equation}\label{eq:LR_para}
(\manifold_{\mathrm{LR}},\phi_{\mathrm{LR}})=(\mbR^{m\times r}\times \mbR^{n\times r},\, (L,R)\mapsto LR^\top).
\end{equation}
Then, a straightforward application of Theorem~\ref{the:expressions_McapK} leads to the intersection rules
\begin{equation}\label{eq:expressions_McapH}
\begin{aligned}
    \Btangent_{\boundedrank\cap\hanifold}(X) &= \Btangent_{\boundedrank}(X)\cap\Btangent_{\hanifold}(X),
    \\
    \tangenttwo_{\boundedrank\cap\hanifold}(X;\eta) &= \tangenttwo_{\boundedrank}(X;\eta)\cap\tangenttwo_{\hanifold}(X;\eta)\ \ \text{for any}\ \eta\in\Btangent_{\boundedrank\cap\hanifold}(X).
\end{aligned}
\end{equation}

To facilitate the discussion, we note that $\manifoldLR$ coincides with the whole ambient Euclidean space $\mbR^{m\times r}\times \mbR^{n\times r}$, and thus the transversality property in Assumption~\ref{assu:MintersecK}(ii) naturally holds. Moreover, $\phiLR|_{\manifoldLR}$ is open at $(L, R)\in\manifoldLR$ if and only if $\rank(L) = \rank(R) = \rank(LR^\top)$, according to~\cite[Theorem~2.3, Proposition~2.8]{levin2025effect}. Therefore, when $(\manifoldLR, \phiLR)$ is chosen as the smooth parameterization of $\boundedrank$, conditions (ii) and (iii) of Assumption~\ref{assu:MintersecK} are automatically satisfied. It then remains to verify condition (i) for applying Theorem~\ref{the:expressions_McapK} to $\boundedrank \cap \hanifold$. This observation motivates the following corollary, which can be viewed as an instance of Theorem~\ref{the:expressions_McapK} specified for $\boundedrank\cap\hanifold$.

\begin{corollary}\label{cor:expressions_MrcapH}
    Suppose that $h:\mbRmn\to\mbR^q$ is smooth and $\hanifold=\{\tilde{X}\in\mbRmn\mid h(\tilde{X})=0\}$ satisfies Assumption~\ref{assu:errorbound} at point $X\in\boundedrank\cap\hanifold$. Additionally, the differential of the $h\circ\phiLR$ has constant rank in a neighborhood of $$\bhanifold:=\{(L,R)\in\mbR^{m\times r}\times \mbR^{n\times r}\mid h(\phiLR(L,R))=0\}=\phiLR^{-1}(\hanifold).$$ Then, the intersection rules~\eqref{eq:expressions_McapH} for the associated tangent sets hold.
\end{corollary}

\subsection{$\hanifold$ as an affine manifold}\label{app:Haffine}
Consider the case $\hanifold=\affine(m,n)= \{X\in\mbRmn\,|\ \aanifold(X)-b={0}\}$. Let $A_1,A_2,\ldots,A_q\in\mbRmn$ be the matrices constituting the mapping $\aanifold:\mbRmn\to\mbR^q$, i.e., $\aanifold(X)_i=\innerp{A_i,X}$ for $i=1,2,\ldots,q$. Given $X\in\boundedrank$ with $\rank(X)=s$ and the SVD $X=U\varSigma V^\top$, denoting
\begin{equation*}
    T^i_X = \left[ \begin{matrix}
	U^{\top}A_iV&		U^{\top}A_iV_{\bot}\\
	U_{\bot}^{\top}A_iV&		0\\
\end{matrix} \right]\ \ \text{and}\ \ E^i_X=U^\top A_i V\ \ \text{for}\ i=1,2,\ldots,q,
\end{equation*}
Li and Luo~\cite{li2023normalboundedaffine} proposed the following constraint qualification.
\begin{assumption}{\cite[Assumptions~3.3 and 3.4]{li2023normalboundedaffine}}\label{assu:affinemanifold}
    When $s=r$, the matrices $T^i_X$, $i=1,2,...,q$, are linearly independent; when $s<r$, the matrices $E^i_X$, $i=1,2,...,q$, are linearly independent.
\end{assumption}

Taking into account the LR parameterization $(\manifoldLR,\phiLR)$ in~\eqref{eq:LR_para}, we present the preimage of $\hanifold=\affine(m,n)$ under the mapping $\phiLR$ as follows,
\begin{equation}\label{eq:barH_affine}
    \bhanifold = \phiLR^{-1}(\hanifold)=\{(L,R)\in\mbR^{m\times r}\times \mbR^{n\times r}\mid \innerp{A_i,LR^\top}-b_i=0,\,i=1,2,\ldots,q\}.
\end{equation}
In addition, the set $\bhanifold$ can be characterized as the level set of $\bar{h}:\mbR^{m\times r}\times \mbR^{n\times r}\to\mbR^{q}:\,\bar{h}(L,R)_i:=\innerp{A_i,LR^\top}-b_i,\ \text{for}\ i=1,2,\ldots,q$. In this view, we then show that $\bhanifold$ is indeed an embedded submanifold in $\mathbb{R}^{m\times r}\times\mathbb{R}^{n\times r}$ under Assumption~\ref{assu:affinemanifold}.

\begin{proposition}\label{pro:manifold_affine}
    Suppose that Assumption~\ref{assu:affinemanifold} holds at all $X\in\boundedrank\cap\affine(m,n)$. The set $\bhanifold = \phiLR^{-1}(\affine(m,n))$ is a smooth submanifold embedded in $\mathbb{R}^{m\times r}\times\mathbb{R}^{n\times r}$ of dimension $(mr+nr-q)$.
\end{proposition}
\begin{proof}
We compute the differential of $\bar{h}$ below,
\begin{equation}\label{eq:diffbarh_affine}
    \kh{\diff \bar{h}_{(L,R)}[\dot{L},\dot{R}]}_i =\innerp{A_i,L\dot{R}^\top} + \innerp{A_i,\dot{L}R^\top},\ \text{for}\ i=1,2,\ldots,q,
\end{equation}    
and it suffices to prove that $\rank(\diff\bar{h}_{(L,R)})=q$ for any $(L,R)\in\bhanifold$, according to~\cite[Corollary 5.14]{lee2012manifolds}. To this end, let $X=LR^\top$ with the SVD $X=U\varSigma V^\top$.

We first consider the case $\rank(X)=r$, which indicates that $(L,R)$ can be written as $(L,R)=(UB,VC)$ for some invertible $B,C\in\mbR^{r\times r}$. Taking $\dot{L}=(U\dot{B}_1+U_\bot \dot{B}_2)C^{-\top}$ and $\dot{R}=(V\dot{C}_1+V_\bot \dot{C}_2)B^{-\top}$ in~\eqref{eq:diffbarh_affine} yields
\begin{equation*}
\diff \bar{h}_{(L,R)}[\dot{L},\dot{R}]_i = \Big\langle T^i_X, \left[ \begin{matrix}
\dot{C}_{1}^{\top}+\dot{B}_1&		\dot{C}_{2}^{\top}\\
\dot{B}_2&		0\\
\end{matrix} \right] \Big\rangle\ \ \text{for}\ i=1,2,\ldots,q.
\end{equation*}
The linear independence of $T^i_X$ and the arbitrariness of $\dot{B}_1,\dot{C}_1\in\mbR^{r\times r}$, $\dot{B}_2\in\mbR^{(m-r)\times r}$, $\dot{C}_2\in\mbR^{(n-r)\times r}$ reveal that $\rank \diff\bar{h}_{(L,R)} = q$.

The second case is $\rank(X)=s<r$. We parameterize $(L,R)$ by $(L,R)=(UB_1+U_\bot B_2, VC_1+V_\bot C_2)$, and the SVD of $X$ implies that
$U\varSigma V^\top = LR^\top = UB_1C^\top_1 V^\top + UB_1C_2^\top V_\bot^\top + U_\bot B_2C_1^\top V^\top + U_\bot B_2 C_2^\top V_\bot^\top$.
Therefore, we have $B_1^{}C_1^\top=\varSigma$, $B_1^{}C^\top_2=0$, $B_2^{}C_1^\top=0$, and $B_2^{}C^\top_2=0$. Taking $(\dot{L},\dot{R})=(0,V\dot{C}\varSigma^{-1}C_1)$ with $\dot{C}\in\mbR^{s\times s}$ in~\eqref{eq:diffbarh_affine} yields
\begin{equation*}
    \diff \bar{h}_{(L,R)}[\dot{L},\dot{R}]_i =\innerp{A_i,L\dot{R}^\top}=\innerp{ A_i,U\dot{C}^\top V^\top}=\innerp{E^i_X,\dot{C}^\top},\ \text{for}\ i=1,2,\ldots,q.   
\end{equation*}
The linear independence of $E^i_X$ indicates that $\rank \diff\bar{h}_{(L,R)} = q$.
\end{proof}

The above proof concludes that the differential of $\bar{h}$ has full rank $q$ in the level set $\bhanifold$, and thus applying Corollary~\ref{cor:expressions_MrcapH} directly gives~\eqref{eq:expressions_McapH} with $\hanifold=\affine(m,n)$.

We then delve into the closed-form formula of the tangent cone at a point $X\in\boundedrank\cap\affine(m,n)$ when $\rank(X)=s$. Recalling~\eqref{eq:Btangent_cone_boundedrank}, any $\eta\in\Btangent_{\boundedrank}(X)$ can be parameterized as $\eta=UW_1V^\top + UW_2V^\top_\bot + U_\bot W_3V^\top + U_\bot JV^\top_\bot$ with $\rank(J)\le r-s$. If, additionally, $\eta$ belongs to $\Btangent_{\affine(m,n)}(X)$, i.e., $\innerp{A_i,\eta}=0$ for $i=1,2,\ldots,q$, we have $\innerp{U^\top A_i V, W_1} + \innerp{U^\top A_iV_\bot,W_2} + \innerp{U_\bot^\top A_i V,W_3} + \innerp{U_\bot^\top A_iV_\bot,J}=0$, and then substitute the expression of $T^i_X$ to obtain
\begin{equation*}
\Big\langle T^i_X, \left[ \begin{matrix}
W_1 &		W_2\\
W_3 &		0\\
\end{matrix} \right] \Big\rangle +\innerp{U_\bot^\top A_iV_\bot,J} =0\ \ \text{for}\ i=1,2,\ldots,q.
\end{equation*}
Therefore, the tangent cone to $\boundedrank\cap\affine(m,n)$ admits the following characterizations,
\begin{align}
&\Btangent_{\boundedrank\cap\affine(m,n)}(X) =\ \Btangent_{\boundedrank}(X)\cap\Btangent_{\affine(m,n)}(X)   \nonumber
\\[1mm]
&=\hkh{\left[ U\,\,U_{\bot} \right] \left[ \begin{matrix}
W_1 &	W_2\\
W_3 &   J\\
\end{matrix} \right] \left[ V\,\,V_{\bot} \right] ^{\top}\left|\,\begin{array}{l}
W_1\in\mbR^{s\times s},\, W_2\in\mathbb{R} ^{s\times \left( n-s \right)},
\\
W_3\in\mathbb{R} ^{\left( m-s \right) \times s},\, J\in \mbR^{\left( m-s \right) \times \left( n-s \right)},
\\
\rank(J)\le r-s,
\\[1mm]
\Big\langle T^i_X, \left[ \begin{matrix}
    W_1 &		W_2\\
    W_3 &		0\\
\end{matrix} \right] \Big\rangle +\innerp{U_\bot^\top A_iV_\bot,J} =0, i\in [q] \label{eq:tangent_MHaffine}
\end{array}
\right.},    
\end{align}
where we denote $[q]=\{1,2,\ldots,q\}$. Taking the polar operation on the above equality yields the Fr\'echet normal cone as a byproduct,
\begin{equation}\label{eq:normal_MHaffine}
    \Fnormal_{\boundedrank\cap\affine(m,n)}(X) = \Fnormal_{\boundedrank}(X) + \Fnormal_{\affine(m,n)}(X),
\end{equation}
where $\Fnormal_{\boundedrank}(X)$ is presented in~\eqref{eq:Fnormal_cone_boundedrank} and $\Fnormal_{\affine(m,n)}(X)=\{\sum_{i=1}^q c_iA_i\mid c_i\in\mbR,\,i\in[q]\}$. We remark that normal cone~\eqref{eq:normal_MHaffine} recovers \cite[Theorem 3.7]{li2023normalboundedaffine}, and the developed tangent cone~\eqref{eq:tangent_MHaffine} serves as a new result.

Finally, given any $\eta\in\Btangent_{\boundedrank\cap\affine(m,n)}(X)$, we note that $\tangenttwo_{\affine(m,n)}(X;\eta)=\Btangent_{\affine(m,n)}(X) $ by definition~\eqref{eq:second_tangentcone} and the affine structure of $\affine(m,n)$. Hence, an explicit characterization of $\tangenttwo_{\boundedrank\cap\affine(m,n)}(X;\eta)$ is attainable by intersecting the result in~\eqref{eq:tangenttwo_mr} with the fixed subspace $\Btangent_{\affine(m,n)}(X)$.

\subsection{$\hanifold$ as an orthogonally invariant manifold}\label{app:Horthogonalinvariant}

Yang et al.~\cite{yang2025spacedecouple} considered $\boundedrank\cap\hanifold$ with the $\hanifold$ encompassing $\Fsphere(m,n)$ and $\oblique(m,n)$ as specific instances. They resorted to the concept of ``orthogonal invariance'' as follows.

\begin{assumption}{\cite[Assumption~1]{yang2025spacedecouple}}\label{assu:h_orth_invariant}
    The set $\hanifold$ is the level set of a smooth and orthogonally invariant mapping $h:\,\mathbb{R}^{m\times n}\rightarrow \mathbb{R}^{q}$ in the sense that 
    \begin{equation}\label{eq:orth_invariant_h}
       \hanifold=\{X\in\mbRmn\mid h(X)=0\},\ \ \text{and}\ \  h(XQ) = h(X),\ \text{for all}\ Q\in\orth(n).
    \end{equation}
    Moreover, $h$ has full rank $q$ in $\hanifold$, i.e., $\rank(\diff h_X)=q$ for {all} $X\in\hanifold$.
\end{assumption}
We utilize the parameterization $(\manifoldLR,\phiLR)$ to lift $\hanifold$ through the mapping $\phiLR$:
\begin{equation}\label{eq:barH_orth}
    \bhanifold = \phiLR^{-1}(\hanifold)=\{(L,R)\in\mbR^{m\times r}\times \mbR^{n\times r}\mid \bar{h}(L,R):=h(LR^\top)=0\}.
\end{equation}
Then, the aim is to show that $\diff \bar{h}$ has full rank in the level set $\bhanifold$, and thus $\bhanifold$ is an embedded submanifold; see the following proposition.
\begin{proposition}\label{pro:manifold_orth_inv}
    Suppose that $\hanifold$ satisfies Assumption~\ref{assu:h_orth_invariant}. The set $\bhanifold = \phiLR^{-1}(\hanifold)$ is a smooth submanifold embedded in $\mathbb{R}^{m\times r}\times\mathbb{R}^{n\times r}$ of dimension $(mr+nr-q)$.
\end{proposition}
\begin{proof}
The differential of $\bar{h}$ can be computed as
\begin{equation}\label{eq:diffbarh_orth}
    \diff \bar{h}_{(L,R)}[\dot{L},\dot{R}]=\diff h_X[L\dot{R}^\top] + \diff h_X[\dot{L}R^\top], 
\end{equation}
where we denote $X=LR^\top$. Suppose that $\rank(X)=s\le r$ and let the SVD of $X$ be $X=U\varSigma V^\top$. We parameterize $(L,R)$ by $(L,R)=(UB_1+U_\bot B_2, VC_1+V_\bot C_2)$, and the SVD of $X$ implies that
$U\varSigma V^\top = LR^\top = UB_1C^\top_1 V^\top + UB_1C_2^\top V_\bot^\top + U_\bot B_2C_1^\top V^\top + U_\bot B_2 C_2^\top V_\bot^\top$. Therefore, it holds that $B_1^{}C_1^\top=\varSigma$ and $B_1^{}C_2^\top=0$. Moreover, using \cite[Proposition~1]{yang2025spacedecouple} shows that $\{BV^\top_\bot\mid B\in\mbR^{m\times(n-s)}\} \subseteq\kernel(\diff {h}_X)$. This observation, together with the full-rankness of $\diff h_X$, reveals that any $b\in\mbR^q$ admits a preimage $\eta$ in the form of $\eta = \tilde{B}V^\top \in\mbR^{m\times n}$ for some $\tilde{B}\in\mbR^{m\times s}$, i.e., $\diff h_{X}[\tilde{B}V^\top] =b$. Taking $(\dot{L},\dot{R})=(\tilde{B}\varSigma^{-1}B_1,0)$ in~\eqref{eq:diffbarh_orth} shows that $\diff\bar{h}_{(L,R)}[\dot{L},\dot{R}]=b$. Consequently, the arbitrariness of $b\in\mbR^q$ implies that $\diff \bar{h}$ has the full rank $q$ in $\bhanifold$, which completes the proof by invoking \cite[Corollary 5.14]{lee2012manifolds}.
\end{proof}

Consequently, applying Corollary~\ref{cor:expressions_MrcapH} to $\boundedrank \cap \hanifold$ yields the intersection rules for the tangent sets in~\eqref{eq:expressions_McapH}. Motivated by this result, we now derive the closed-form expressions for the cases $\hanifold = \Fsphere(m,n)$ and $\hanifold = \oblique(m,n)$. Specifically, we note that $\tangent_{\Fsphere(m,n)}(X)=\{\eta\in\mbRmn\mid\trace(X^\top \eta)=0\}$, and thus
\begin{equation*}
\begin{aligned}
\Btangent_{\boundedrank\cap\Fsphere(m,n)}(X)&= \tangent_{\boundedrank}(X)\cap\tangent_{\Fsphere(m,n)}(X)
\\
&=\hkh{\left[ U\,\,U_{\bot} \right] \left[ \begin{matrix}
W_1 &		W_2\\
W_3&		J\\
\end{matrix} \right] \left[ V\,\,V_{\bot} \right] ^{\top}\left|\,\begin{array}{l}
W_1\in\mbR^{s\times s},\, W_2\in\mathbb{R} ^{s\times \left( n-s \right)},
\\
W_3\in\mathbb{R} ^{\left( m-s \right) \times s},
\\
J\in \mbR^{\left( m-s \right) \times \left( n-s \right)},
\\
\rank(J)\le r-s,\, \trace(\varSigma W_1) = 0
\end{array}
\right.}. 
\end{aligned}
\end{equation*}
Similarly, from $\tangent_{\oblique(m,n)}(X)=\{\eta\in\mbRmn\mid\ddiag(X\eta^\top)=0\}$, we obtain
\begin{equation*}
\begin{aligned}
\Btangent_{\boundedrank\cap\oblique(m,n)}(X)&= \tangent_{\boundedrank}(X)\cap\tangent_{\oblique(m,n)}(X)
\\
&=\hkh{\left[ U\,\,U_{\bot} \right] \left[ \begin{matrix}
W_1 &	W_2\\
W_3&		J\\
\end{matrix} \right] \left[ V\,\,V_{\bot} \right] ^{\top}\left|\,\begin{array}{l}
W_1\in\mbR^{s\times s},\, W_2\in\mathbb{R} ^{s\times \left( n-s \right)},
\\
W_3\in\mathbb{R} ^{\left( m-s \right) \times s},
\\
J\in \mbR^{\left( m-s \right) \times \left( n-s \right)},
\\
\rank(J)\le r-s,
\\
UW_1+U_\bot W_3\in\tangent_{\oblique(m,s)}(U\varSigma)
\end{array}
\right.}.
\end{aligned}
\end{equation*}
The explicit formulas of the tangent cones recover the results in \cite[Theorem~6.1]{cason2013iterative} and \cite[Theorem~1]{yang2025spacedecouple}.

\subsection{$\hanifold$ as the product of hyperbolic manifolds}\label{app:hyperbolic}
Let $\lanifold = \mathrm{Diag}(-1,1,\ldots,1) \in \mathbb{R}^{m\times m}$. For $x, y \in \mathbb{R}^m$, we define the Lorentzian inner product as $\langle x,y\rangle_\lanifold:=x^\top \lanifold y = -x_1y_1 + \sum_{i=2}^m x_iy_i$. Then, we consider the upper sheet of an $(m-1)$-dimensional hyperboloid to define \emph{hyperbolic manifold}:
$\mathbb{H}_{m-1}=\{x\in\mathbb{R}^m\mid\langle x,x\rangle_\lanifold=-1,\ x_1>0\}$.
Stacking $n$ vectors in $\mathbb{H}_{m-1}$ gives rise to the product manifold,
\begin{equation}\label{eq:hyperbolicmatrixset}
    \mathbb{H}^n_{m-1}=\{X\in\mbRmn\mid X_i\in\hyperboloid\ \text{for}\ i=1,2,\ldots,n\},
\end{equation}
where $X_{i}$ extracts the $i$-th column of $X$.

The hyperbolic manifold is a smooth manifold with negative constant curvature, and it has attracted recent interest in the machine learning community for learning hyperbolic embeddings of entities~\cite{nickel2018learninghierarchieslorentz}.
For computational efficiency, Jawanpuria et al.~\cite{jawanpuria2019lowrankhyperbolic} proposed learning hyperbolic embeddings within a latent low-dimensional subspace. Specifically, they searched for a low-rank matrix in $\mathbb{R}^{m\times n}$ with columns encoding $(m-1)$-dimensional hyperbolic embeddings corresponding to $n$ data points, which, in turn, motivates our study on the geometry of the feasible region $\boundedrank \cap \matrixhyperboloid$.

To align with the spirit of Assumption~\ref{assu:MintersecK}, we then treat $\matrixhyperboloid$ as the zero set of a sequence of functions $h_i:\mbRmn\to\mbR:\,X\mapsto\innerp{X_i,X_i}_{\lanifold}+1$ ($i=1,2,\ldots,n$), that is, $\matrixhyperboloid=\{X\in\mbRmn\mid h_i(X)=0,\,X_{i,1}>0\ \text{for}\ i=1,2,\ldots,n\}$, where $X_{i,1}$ denotes the first element of $X_i$. Let $h:=(h_1,h_2,\ldots,h_n)$, and a direct computation tells that $\diff h$ has full rank on $\matrixhyperboloid$. 

Furthermore, define the smooth functions $\bar{h}_i:\mathbb{R}^{m\times r}\times\mathbb{R}^{n\times r}\to\mathbb{R}$ by 
$$
\bar{h}_i(L,R):=\innerp{(LR^\top)_i, (LR^\top)_i}_{\lanifold} + 1,\ \text{for}\ i=1,2,\dots,n,
$$
and the mapping $\bar{h}:=(\bar{h}_1,\bar{h}_2\dots,\bar{h}_n)$. We note that $\bar{h}=h\circ\phiLR$, and are interested in the preimage of $\hanifold=\matrixhyperboloid$ through the smooth mapping $\phiLR$:
\begin{equation}\label{eq:barH_hyperbolic}
    \bhanifold = \phiLR^{-1}(\matrixhyperboloid)=\{(L,R)\in\mbR^{m\times r}\times \mbR^{n\times r}\mid \bar{h}(L,R)=0\}\cap\zanifold,
\end{equation}
where $\zanifold:=\{(L,R)\in\mbR^{m\times r}\times \mbR^{n\times r}\mid (LR^\top)_{i,1}>0\ \text{for}\ i=1,2,\ldots,n\}$. Notice that $\zanifold$ is an open subset, and the following proposition reveals that $\bhanifold$ is a smooth manifold.

\begin{proposition}\label{pro:manifold_hyperbolic}
     The set $\bhanifold = \phiLR^{-1}(\matrixhyperboloid)$ is a smooth submanifold embedded in $\mathbb{R}^{m\times r}\times\mathbb{R}^{n\times r}$ of dimension $(mr+nr-n)$.
\end{proposition}
\begin{proof}
    According to~\cite[Corollary 5.14]{lee2012manifolds} and the characterization~\eqref{eq:barH_hyperbolic}, it suffices to prove that $\rank(\diff\bar{h}_{(L,R)})=n$ for every $(L,R)\in\bhanifold$. In preparation, write $R^i\in\mathbb{R}^{1\times r}$ for the $i$-th row of $R$ and note that $x_i=LR^{i\top}\in\mathbb{R}^m$. We then compute the gradients of $\bar{h}_i$ with respect to $R$ as follows:
    \begin{equation}\label{eq:gradsbarhL}
    \begin{aligned}
        (\nabla_R \bar{h}_i)^{i}&= 2x_i^\top \lanifold L \in\mathbb{R}^{1\times r}, 
        \\
        (\nabla_R \bar{h}_i)^{j} &= 0\ \ \text{for}\ j=1,\ldots,i-1,i+1,\ldots,n,
    \end{aligned}
    \end{equation}
    where we use the superscripts to extract the corresponding rows from a matrix. Moreover, $\nabla_L \bar{h}_i = 2\lanifold x_i R^i \in \mathbb{R}^{m\times r}$. Hence, given  $(L,R)\in\bhanifold$, we have $x_i^\top \lanifold x_i
    = (LR^{i\top})^\top \lanifold (LR^{i\top})
    =R^i(L^\top \lanifold L) R^{i\top}
    = -1$, and thus it holds that 
    \begin{equation}\label{eq:neqRLx}
    (\nabla_R \bar{h}_i)^{i\top} = 2L^\top \lanifold x_i= 2(L^\top \lanifold L)R^{i\top}\neq 0,\ \ \text{for}\ i=1,2,\ldots,n.
    \end{equation}

    We now show that the differentials of $\bar{h}_1,\bar{h}_2\dots,\bar{h}_n$ are linearly independent, for which it suffices to consider only the $\nabla_R \bar{h}_i$. In detail, the computation~\eqref{eq:gradsbarhL} reveals that each $\nabla_R \bar{h}_i$ is supported exclusively on the $i$-th row. Therefore, if there exists $\alpha\in\mbR^n$ such that $\sum_{i=1}^n\alpha_i\,\nabla_R \bar{h}_i=0$, then the $i$-th row of the sum equals $2\alpha_i x_i^\top \lanifold L=0$, which forces $\alpha_i=0$ for all $i$ by~\eqref{eq:neqRLx}. This concludes that $\{\nabla_R \bar{h}_i\}_{i=1}^n$ are linearly independent, and thus the full gradients $\{(\nabla_L \bar{h}_i,\nabla_R \bar{h}_i)\}_{i=1}^n$ are also linearly independent. In other words, we have $\rank(\diff\bar{h}_{(L,R)})=n$ on $\bhanifold$.
\end{proof}

The above proof shows that $\boundedrank \cap \matrixhyperboloid$ falls within the scope of Theorem~\ref{the:expressions_McapK}---or more specifically, Corollary~\ref{cor:expressions_MrcapH}---and therefore we obtain
\begin{equation*}
\begin{aligned}
\Btangent_{\boundedrank\cap\matrixhyperboloid}(X)&= \tangent_{\boundedrank}(X)\cap\tangent_{\matrixhyperboloid}(X)
\\
&=\hkh{\eta = \left[ U\,\,U_{\bot} \right] \left[ \begin{matrix}
W_1 &		W_2\\
W_3&		J\\
\end{matrix} \right] \left[ V\,\,V_{\bot} \right] ^{\top}\left|\,\begin{array}{l}
W_1\in\mbR^{s\times s},\, W_2\in\mathbb{R} ^{s\times \left( n-s \right)},
\\
W_3\in\mathbb{R} ^{\left( m-s \right) \times s},
\\
J\in \mbR^{\left( m-s \right) \times \left( n-s \right)},
\\
\rank(J)\le r-s,
\\
\ddiag(X^\top\lanifold \eta) = 0
\end{array}
\right.},
\end{aligned}
\end{equation*}
where the second equality holds by substituting the formula $\tangent_{\matrixhyperboloid}(X)=\{\eta\in\mbRmn\mid \ddiag(X^\top\lanifold \eta)= {0}\}$ (see~\cite{nickel2018learninghierarchieslorentz}). Similarly, the second-order intersection rule  also holds: 
\begin{equation*}
\tangenttwo_{\boundedrank\cap\matrixhyperboloid}(X;\eta) = \tangenttwo_{\boundedrank}(X;\eta)\cap\tangenttwo_{\matrixhyperboloid}(X;\eta)\ \ \text{for any}\ \eta\in\Btangent_{\boundedrank\cap\matrixhyperboloid}(X).
\end{equation*}

\section{Tangent sets to low-rank symmetric sets}\label{app:lowranksym}
This section serves as a supplement to sections~\ref{sec:lowrank_sym}-\ref{sec:lowrank_posi}, and the organization is outlined as follows. Appendix~\ref{app:tangent_Sj} applies Theorem~\ref{the:cal_tangentsets} to $\sanifold_j$ for $j \in \{1,2,\ldots,r+1\}$, characterizing the corresponding tangent sets. Appendix~\ref{app:SU} then aggregates these results to derive the tangent sets to $\sanifold(n)$. Finally, in Appendix~\ref{app:S+U}, we identify $\boundedranksdp(n)$ with $\sanifold_{r+1}$, and thus translate the obtained results on $\sanifold_{r+1}$ to yield the tangent sets to $\boundedranksdp(n)$; moreover, applying Theorem~\ref{the:expressions_McapK} derives the tangent sets to $\boundedranksdp \cap \uanifold$ when $\uanifold$ is an affine set, a case that has garnered increasing interest in recent work \cite{boumal2020deterministic,levin2025effect}.

\smallskip

We first verify that each $\sanifold_j$ satisfies Assumption~\ref{assu:errorbound}, which is defined as follows,
\begin{equation}\label{eq:characterize_Sj}
 \sanifold_j=\{X\in\sym(n)\mid\lambda_{j}(X)=0,\,\lambda_{j+n-r-1}(X)=0\},\ \ \text{for}\ j=1,2,\ldots,r+1.
\end{equation}
Give an index $j\in\{1,2,\ldots, r+1\}$ and any $\tilde{X} \in \sym(n)$ that admits a spectral decomposition $\tilde{X}=\bU \bar{\varSigma}\bU^\top$ where $\bU\in\orth(n)$ and $\bar{\varSigma}(i,i)=\lambda_i(\tilde{X})$ for $i=1,2,\ldots,n$. We can construct a point $\tilde{X}_P=\bU\bar{\varSigma}_P\bU^\top\in\sanifold_j$, 
where $\bar{\varSigma}_P$ is a diagonal matrix with the entries defined as follows,
\begin{equation}\label{eq:proj_Sj}
    \bar{\varSigma}_P(i,i) = \begin{cases}
    \max\{0,\lambda_i(\tilde{X})\},\  \ \ \ \ \text{if}\ 1\le i \le j-1,
	\\
    0, \hspace{24.5mm}\ \text{if}\ j\le i \le j+n-r-1,
    \\
    \min\{0,\lambda_i(\tilde{X})\},\  \ \ \ \ \ \text{if}\ j+n-r\le i \le n.
\end{cases}
\end{equation}
Therefore, we have
\begin{equation}\label{eq:disttosr}
    \begin{aligned}
    \!\!\!\dist(\tilde{X},\sanifold_j)^2 \le \|\tilde{X}-\tilde{X}_P\|_\frob^2 =&\!\ \sum_{i=1}^{j-1} |\lambda_i(\tilde{X})-\max\{0,\lambda_i(\tilde{X})\}|^2 + \!\sum_{i=j}^{j+n-r-1}\lambda_i(\tilde{X})^2\!\!\!
    \\
    &+\sum_{i=j+n-r}^n |\lambda_i(\tilde{X})-\min\{0,\lambda_i(\tilde{X})\}|^2.
    \end{aligned}
\end{equation}
If $\lambda_{j}(\tilde{X})\ge 0$, it holds that $\sum_{i=1}^{j-1} |\lambda_i(\tilde{X})-\max\{0,\lambda_i(\tilde{X})\}|^2 = 0$. Otherwise, if $\lambda_{j}(\tilde{X}) < 0$, there exists an index $e$ ($1 \le e \le j$), which is the smallest index with $\lambda_e(\tilde{X}) < 0$, implying that
\begin{equation}\label{eq:case2}
   \sum_{i=1}^{j-1} |\lambda_i(\tilde{X})-\max\{0,\lambda_i(\tilde{X})\}|^2= \sum_{i=e}^{j-1}\lambda_i(\tilde{X})^2\le n|\lambda_{j}(\tilde{X})|^2.
\end{equation}
A parallel discussion on the sign of $\lambda_{j+n-r-1}(\tilde{X})$ leads to 
\begin{equation}\label{eq:case3}
    \sum_{i=j+n-r}^n |\lambda_i(\tilde{X})-\min\{0,\lambda_i(\tilde{X})\}|^2 \le n |\lambda_{j+n-r-1}(\tilde{X})|^2.
\end{equation}
Consequently, we collect the estimates~\eqref{eq:disttosr}, \eqref{eq:case2}, and \eqref{eq:case3} to obtain
\begin{align}
\|\tilde{X}-\tilde{X}_P\|_\frob^2 &\le n|\lambda_{j}(\tilde{X})|^2 + \sum_{i=j}^{j+n-r-1}\lambda_i(\tilde{X})^2 + n |\lambda_{j+n-r-1}(\tilde{X})|^2    \nonumber
\\
&\le 3n \max\{|\lambda_{j}(\tilde{X})|^2,|\lambda_{j+n-r-1}(\tilde{X})|^2\},    \label{eq:estimate_X-Xp}
\end{align}
which reveals that $\sanifold_j$ given by~\eqref{eq:characterize_Sj} satisfies the error bound condition. Additionally, Weyl's inequality \cite{weyl1912inequality} indicates the Lipschitz continuity of the mapping $\tilde{X}\mapsto(\lambda_{j}(\tilde{X}),\lambda_{j+n-r-1}(\tilde{X}))$, which justifies the application of Theorem~\ref{the:cal_tangentsets} to $\sanifold_j$.

\subsection{Tangent sets to $\sanifold_j$}\label{app:tangent_Sj}
Applying Theorem~\ref{the:cal_tangentsets} to $\sanifold_j$ reduces to finding the zeros of the directional derivatives of eigenvalue mappings. To this end, we note that the explicit expressions for the directional derivatives of $\lambda_i$ ($i=1,2,\ldots,n$) are provided in \cite{torki2001secondtoeigen,zhang2013secondordersingular}. In preparation, we recall some necessary notation from \cite{zhang2013secondordersingular} and section~\ref{sec:directionalderisigma}.

For the spectral decomposition of a symmetric matrix $X\in\sym(n)$, we adopt notation consistent with the full SVD in~\eqref{eq:fullSVD}, 
serving as its symmetric counterpart. Specifically, we replace $[\bar{\varSigma}\ 0]$ with $\bar{\Lambda}$ and $\bar{V}$ with $\bar{U}$, yielding
$X=\bar{U}\bar{\Lambda}\bar{U}^\top$, where $\bar{U}\in\orth(n)$ and $\bar{\Lambda}=\Diag(\lambda_1(X),\ldots,\lambda_n(X))$. We slightly abuse the notation of ${\mu}$ and $\alpha_k$ by adapting them to this symmetric scenario. In detail, suppose that $X$ admits $t$ distinct eigenvalues, $\mu_1>\mu_2>\ldots>\mu_t$, based on which we categorize the index set $\{1,\ldots,n\}$ into $\alpha_k=\left\{i\mid \lambda_i(X)=\mu_k, 1 \leq i \leq n\right\}$ for $k=1, \ldots, t$. Moreover, the definitions of $N_k$ and $\beta_j^k$ for $j=1,2,\ldots,N_k$ follow those in section~\ref{sec:directionalderisigma} when $m=n$. The index mappings also require a minor adaptation as follows:
\begin{equation}\label{eq:adapt_indexmapping}
\begin{aligned}
q_a &:\{1, \ldots, n\} \rightarrow\{1, \ldots, t\},\ q_a(i)=k, \text { if } i \in \alpha_k, \\
l &:\{1, \ldots, n\} \rightarrow \mathbb{N},\ l(i)=i-\kappa_{q_a(i)-1}, \\
q_b &:\{1, \ldots, n\} \rightarrow \mathbb{N},\ q_b(i)=e,\ \text{if}\ l(i) \in \beta_e^{q_a(i)},\\
l^{\prime} &:\{1, \ldots, n\} \rightarrow \mathbb{N},\ l^{\prime}(i)=l(i)-\kappa_{q_b(i)-1}^{(q_a(i))},
\end{aligned}  
\end{equation}
where $\kappa_i:=\sum_{j=1}^i|\alpha_j|$ and $\kappa_i^{(k)}:=\sum_{j=1}^i|\beta_j^k|$. We then outline the results of \cite[Theorem 2.1]{zhang2013secondordersingular}. Given $i\in\alpha_k$ and $\eta\in\sym(n)$, the first-order directional derivative of $\lambda_i$ is computed by
\begin{equation}\label{eq:lambdaprime}
    \lambda^{\prime}_i(X;\eta) = \lambda_{l(i)} (\bareta_{\alpha_k\alpha_k}),
\end{equation}
where $\bareta_{\alpha_k\alpha_k}=\bU_{\alpha_k}^\top\eta \bU_{\alpha_k}$. Given, in addition, a direction $\zeta\in\sym(n)$, the second-order directional derivative is expressed by
\begin{equation}\label{eq:lambdaprimeprime}
    \lambda^{\prime\prime}_i(X;\eta,\zeta) = \lambda_{l^\prime(i)}\kh{(Q^k_{\beta^k_{q_b(i)}})^\top \bU_{\alpha_k}^\top[\zeta-2\eta(X-\lambda_i(X)I)^\dagger\eta]\bU_{\alpha_k} Q^k_{\beta^k_{q_b(i)}}},
\end{equation}
where $Q^k\in\orth^{|\alpha_k|}(\bareta_{\alpha_k\alpha_k})$, and the subscript $\beta^k_{q_b(i)}$ extracts the columns of $Q^k$ indexed by the set $\beta^k_{q_b(i)}$.

We now proceed to derive the tangent sets to $\sanifold_j$ by determining the zeros of $\lambda^\prime$ and $\lambda^{\prime\prime}$ given in~\eqref{eq:lambdaprime} and~\eqref{eq:lambdaprimeprime}, respectively; see the following proposition.

\begin{proposition}\label{pro:tangenttoSj}
    Given $j\in\{1,2,\ldots,r+1\}$, $X\in\sanifold_j$ with $\rank{(X)}=s$, and the spectral decomposition $X=U\varLambda U^\top$ with $U\in\stiefel(n,s)$. Let $s_+\le s$ count the positive eigenvalue of $X$. The tangent cone to $\sanifold_j$ can be characterized by
    \begin{equation}\label{eq:tangenttoSj}
        \Btangent_{\sanifold_j}(X)=\hkh{[U\ U_{\bot}]\left[ \begin{matrix}
        	W_1&		W_2\\
        	W_2^\top&		J\\
        \end{matrix} \right][U\ U_{\bot}]^\top \left|\,\begin{array}{l}
        W_1\in\sym(s),
        \\
        W_2\in\mbR^{s\times (n-s)},
        \\
        J \in \sym(n-s),
        \\
        \lambda_{j-s_+}(J)=0,
        \\
        \lambda_{j+n-r-1-s_+}(J)=0
        \end{array}\right.
        }.
    \end{equation}
    Additionally, given a direction $\eta\in\Btangent_{\sanifold_j}(X)$ parameterized in the above manner with $\rank(J)=\ell-s$ for some $s\le\ell\le r$. Let $\ell_+\le \ell -s$ count the positive eigenvalue of $J$. Let the spectral decomposition of $U_\bot J U^\top_\bot$ be $U_\bot J U^\top_\bot=U_{\eta} \varSigma_\eta U_{\eta}^\top$ with $U_{\eta}\in\stiefel(n,\ell-s)$. Take ${U}_{\eta\bot}$ such that $[U\ U_\eta\ U_{\eta\bot}]\in\orth(n)$, and denote $U^+=[U\ U_{\eta}]$. It holds that
    \begin{equation}\label{eq:tangenttwotoSj}
        \tangenttwo_{\sanifold_j}(X;\eta)=\hkh{2\eta X^\dagger\eta + [U^+\ U_{\eta\bot}]\left[ \begin{matrix}
        	W_1&		W_2\\
        	W_2^\top&		L\\
        \end{matrix} \right][U^+\ U_{\eta\bot}]^\top \left|\,\begin{array}{l}
        W_1\in\sym(\ell),
        \\
        W_2\in\mbR^{\ell\times (n-\ell)},
        \\
        L\in\sym(n-\ell),
        \\
        \lambda_{j-s_+-\ell_+}(L) = 0,
        \\
        \lambda_{j+n-r-1-s_+-\ell_+}(L) = 0
        \end{array}\right.
        }.
    \end{equation}
\end{proposition}
\begin{proof}
It suffices to apply Theorem~\ref{the:cal_tangentsets} to $\sanifold_j(n)$, and then to obtain the  tangent sets as the zeros of the directional derivatives of the eigenvalue mappings $\tilde{X}\mapsto(\lambda_{j}(\tilde{X}),\lambda_{j+n-r-1}(\tilde{X}))$. In fact, computing the zeros of~\eqref{eq:lambdaprime} and~\eqref{eq:lambdaprimeprime} proceeds in parallel with the analysis in sections~\ref{sec:directionalderisigma}-\ref{sec:tangenttoMr}.

The condition $X\in\sanifold_j$ implies that $\lambda_{j}({X})=\lambda_{j+n-r-1}({X})=0$, and thus the indices $q_a(j)=q_a(j+n-r-1)$ point to the zero eigenvalue of $X$. Therefore, taking $i=j$ in~\eqref{eq:adapt_indexmapping} yields the index set $\alpha_k$ with $k:=q_a(j)$, which corresponds to the set of all zero eigenvalues of $X$. Hence we have $\bareta_{\alpha_k\alpha_k}=U_\bot^\top \eta U_\bot$ in this scenario. By definition of the index mapping~\eqref{eq:adapt_indexmapping}, we have $l(j)=j-s_+$ and $l(j+n-r-1)=j+n-r-1-s_+$, with $s_+\le s$ counting the positive eigenvalues of $X$. By~\eqref{eq:lambdaprime}, the condition $\lambda^\prime_{j}(X;\eta)=\lambda_{j+n-r-1}^\prime(X;\eta)=0$ is equivalent to
\begin{equation}\label{eq:eigen_J_lambda}
    \lambda_{j-s_+}(U_\bot^\top\eta U_\bot) = \lambda_{j+n-r-1-s_+}(U_\bot^\top\eta U_\bot) = 0,
\end{equation}
which yields the expression~\eqref{eq:tangenttoSj} by denoting $J=U_\bot^\top\eta U_\bot$.

The derivation of the second-order tangent set~\eqref{eq:tangenttwotoSj} parallels the proof of Proposition~\ref{pro:tangenttwo_mr}, which starts by identifying the values of the associated index mappings. Given the direction $\eta\in\Btangent_{\sanifold_j}(X)$ with the associated $\rank(J)=\ell-s$. We notice from~\eqref{eq:eigen_J_lambda} that the singular value of $J=\bareta_{\alpha_k\alpha_k}$ indexed by $l(j)=j-s_+$ is zero, and thus $q_b(j)$ points to the zero eigenvalue of $J$. Therefore, taking $i=j$ in~\eqref{eq:lambdaprimeprime} reveals that the index set $\beta_{q_b(j)}^k$ corresponds to the set of all zero eigenvalues of $J$. Hence we can identify the quantity in~\eqref{eq:lambdaprimeprime} as $\bU_{\alpha_k} Q^k_{\beta^k_{q_b(j)}}=U_{\eta\bot}$ without loss of generality. Moreover, the index $l^\prime(j)$ counts the position of $\lambda_{l(j)}(J)$ in the zero eigenvalues of $J$, indicating that $l^\prime(j)=l(j)-\ell_+=j-s_+-\ell_+$. A similar analysis can be implemented by letting $i=j+n-r-1$ in~\eqref{eq:adapt_indexmapping} to obtain $l^\prime(j+n-r-1)=j+n-r-1-s_+-\ell_+$. Consequently, by~\eqref{eq:lambdaprimeprime}, the condition $\lambda^{\prime\prime}_{j}(X;\eta,\zeta)=\lambda^{\prime\prime}_{j+n-r-1}(X;\eta,\zeta)=0$ is equivalent to
\begin{equation*}
    \lambda_{j-s_+-\ell_+}(U_{\eta\bot}^\top[\zeta-2\eta X^\dagger\eta]U_{\eta\bot}) = \lambda_{j+n-r-1-s_+-\ell_+}(U_{\eta\bot}^\top[\zeta-2\eta X^\dagger\eta]U_{\eta\bot}) = 0,
\end{equation*}
which yields the expression~\eqref{eq:tangenttwotoSj} by denoting $L=U_{\eta\bot}^\top[\zeta-2\eta X^\dagger\eta]U_{\eta\bot}$.

\end{proof}

\subsection{Tangent sets to $\boundedranks\cap\uanifold$}\label{app:SU}
We first investigate the geometry of $\boundedranks\cap\uanifold$ when $\uanifold$ is the trivial ambient space, i.e., $\uanifold=\sym(n)$. To this end, we collect the derived tangent sets to $\sanifold_j$ for $j\in\{1,2,\ldots,r+1\}$, thereby obtaining those to $\boundedranks=\bigcup_{j=1}^{r+1}\sanifold_j$.

\vspace{1mm}

\noindent\textit{Proof of Proposition~\ref{pro:tangenttoS}}\ \ \ The condition $\rank(X)=s\le r$ implies that $X\in \sanifold_j$ if and only if $j\in \janifold(X):=\{s_++1, \ldots, r+1 -s_-\}$, where $s_+$ and $s_-$ count the positive and negative eigenvalues of $X$, respectively. Therefore, we have
$\tangent_{\boundedranks}(X) = \bigcup_{j\in \janifold(X)}\tangent_{\sanifold_j}(X)$. Enlightened by~\eqref{eq:tangenttoSj}, we observe that
$\sym_{\le r-s}(n-s) = \bigcup_{j\in\janifold(X)}\{J\in\sym(n-s)\mid \lambda_{j-s_+}(J)=\lambda_{j+n-r-1-s_+}(J)=0\}$,
which, together with~\eqref{eq:tangenttoSj}, produces the expression~\eqref{eq:tangent_Sr}.

Furthermore, given the direction $\eta\in\Btangent_{\boundedranks}(X)$ parameterized in the form of~\eqref{eq:tangent_Sr} with $\rank(J)=\rank(U^\top_\bot \eta U_\bot)=\ell-s$. According to~\eqref{eq:tangenttoSj}, it holds that $X\in\sanifold_j$ and $\eta\in\Btangent_{\sanifold_j}(X)$ if and only if $j\in\janifold^\prime(X;\eta):=\{s_++\ell_++1,\ldots,r+1-s_--\ell_-\}$, where $\ell_+$ and $\ell_-$ count the positive and negative eigenvalues of $J=U^\top_\bot \eta U_\bot\in\sym(n-s)$, respectively. Therefore, we have
\begin{equation}\label{eq:tangenttwo_unionSj}
    \tangenttwo_{\boundedranks}(X;\eta) = \bigcup_{j\in \janifold^\prime(X;\eta)}\tangenttwo_{\sanifold_j}(X;\eta),\ \ \text{for any}\ \eta\in\tangent_{\boundedranks}(X).
\end{equation}
Then, the relations $s_++s_-=s$, $\ell_+ + \ell_- =\ell-s $, and the observation that
\begin{equation*}
    \bigcup_{j\in\janifold^\prime(X;\eta)}\{L\in\sym(n-\ell)\mid \lambda_{j-s_+-\ell_+}(L)=\lambda_{j+n-r-1-s_+-\ell_+}(L)=0\} = \sym_{\le r-\ell}(n-\ell)
\end{equation*}
conclude~\eqref{eq:tangenttwo_Sr} by substituting~\eqref{eq:tangenttwotoSj} into~\eqref{eq:tangenttwo_unionSj}.\qed

\subsubsection{$\uanifold$ as a Frobenius sphere}
We then impose an additional constraint $\uanifold$ on $\boundedranks$ by considering the representative example $\uanifold=\{X\in\sym(n)\mid \|X\|_{\frob}^2-1=0\}$~\cite{cason2013iterative,li2020jotaspectral}. The main principle is to apply Theorem~\ref{the:cal_tangentsets} to $\sanifold_j\cap\uanifold$, and then assemble the results to obtain the tangent sets of the union $\sym_{\le r}(n)\cap\uanifold=\bigcup_{j=1}^{r+1}(\sanifold_j\cap\uanifold)$ according to the rule~\eqref{eq:tangentcupXi}.

We first verify that each 
$$\sanifold_j\cap\uanifold=\{X\in\sym(n)\mid\lambda_{j}(X)=0,\,\lambda_{j+n-r-1}(X)=0,\,h(X)=0\}$$
satisfies Assumption~\ref{assu:errorbound}, where $h(X):=\|X\|_{\frob}^2-1$. To see this, given $\tilde{X} \in \sym(n)$ that admits an spectral decomposition $\tilde{X}=\bU \bar{\varSigma}\bU^\top$ where $\bU\in\orth(n)$ and $\bar{\varSigma}(i,i)=\lambda_i(\tilde{X})$ for $i=1,2,\ldots,n$. We can construct a point $\tilde{X}^{\diamond}_P=\bU\bar{\varSigma}^{\diamond}_P\bU^\top\in\sanifold_j\cap\uanifold$, 
where $\bar{\varSigma}^{\diamond}_P:=\bar{\varSigma}_P/\|\bar{\varSigma}_P\|_\frob$ with $\bar{\varSigma}_P$ given in~\eqref{eq:proj_Sj}. Therefore, letting $\tilde{X}_P=\bU\bar{\varSigma}_P\bU^\top$ yields
\begin{align}
\dist(\tilde{X},\sanifold_j\cap\uanifold) \le  \|\tilde{X}-\tilde{X}_P\|_\frob + |\|\tilde{X}_P\|_\frob-1| \le 2\|\tilde{X}-\tilde{X}_P\|_\frob + |\|\tilde{X}\|_\frob-1|, \label{eq:pro_SjcapU}
\end{align}
where we use the triangle inequality $|\|\tilde{X}_P\|_\frob-1|\le |\|\tilde{X}\|_\frob-1| + |\|\tilde{X}\|_\frob-\|\tilde{X}_P\|_\frob|\le |\|\tilde{X}\|_\frob-1| + \|\tilde{X}-\tilde{X}_P\|_\frob$. Consequently, we can confirm the error bound property of $\sanifold_j\cap\uanifold$ by incorporating~\eqref{eq:estimate_X-Xp} and $|\|\tilde{X}\|_\frob-1| \le |h(\tilde{X})|$ into~\eqref{eq:pro_SjcapU}. Applying Theorem~\ref{the:cal_tangentsets} to $\sanifold_j\cap\uanifold$ and taking into account the rule~\eqref{eq:tangentcupXi} show that
\begin{align*}
    \tangent_{\boundedranks\cap\uanifold}(X) = \bigcup_{j=1}^{r+1} \tangent_{\sanifold_j\cap\uanifold}(X) =\bigcup_{j=1}^{r+1}(\tangent_{\sanifold_j}(X) \cap \tangent_{\uanifold}(X)) = \tangent_{\boundedranks}(X)\cap\tangent_\uanifold(X).
\end{align*}
Similarly, given any $\eta\in\tangent_{\boundedranks\cap\uanifold}(X)$, we have
$\tangenttwo_{\boundedranks\cap\uanifold}(X;\eta) = \tangenttwo_{\boundedranks}(X;\eta)\cap\tangenttwo_\uanifold(X;\eta)$.

\subsection{Tangent sets to~$\boundedranksdp(n)\cap\uanifold$}\label{app:S+U}
We have clarified in section~\ref{sec:lowrank_posi} that the set $\boundedranksdp(n)$ coincides with $\sanifold_{r+1}$ defined in~\eqref{eq:characterize_Sj}, and thus Proposition~\ref{pro:tangenttoS+r} directly inherits the results derived in Proposition~\ref{pro:tangenttoSj}.

The intersection $\boundedranksdp(n)\cap\uanifold$ with a nontrivial $\uanifold$ is handled by Corollary~\ref{cor:expressions_S+capU} in section~\ref{sec:lowrank_posi}.

\subsubsection{$\uanifold$ as an affine set}
We apply Corollary~\ref{cor:expressions_S+capU} to $\boundedranksdp(n)\cap\uanifold$ when $\uanifold$ is an affine set. In fact, the study of $\boundedranksdp(n)\cap\uanifold$ with $\uanifold=\
\{X\in\sym(n)\mid\aanifold(X)-b=0\}$~\cite{boumal2020deterministic,levin2025effect} has been motivated by low-rank SDPs with linear equality constraints, which have witnessed a wide range of applications. Specifically, let $A_1,A_2,\ldots,A_q\in\sym(n)$ be the matrices constituting the linear mapping $\aanifold:\sym(n)\to\mbR^q$, i.e., $\aanifold(X)_i=\innerp{A_i,X}$ for $i=1,2,\ldots,q$; and the tangent cone to $\boundedranksdp(n)\cap\uanifold$ was first derived in \cite{levin2025effect} based on the standard assumption below.

\begin{assumption}{\cite[Assumption~1.1]{boumal2020deterministic}}\label{assu:h_affinesdp} 
    The set $\uanifold=\{X\in\sym(n)\mid\aanifold(X)-b=0\}$ defined by $A_1,A_2,\ldots,A_q\in\sym(n)$ satisfies at least one of the following conditions.
    \begin{itemize}
        \item [\emph{(i)}] $\{A_1R,A_2R,\ldots,A_qR\}$ are linearly independent in $\mbR^{n\times r}$;
        \item [\emph{(ii)}] $\{A_1R,A_2R,\ldots,A_qR\}$ span a subspace of constant dimension in $\mbR^{n\times r}$ for all $R$ in an open neighborhood of\ $\buanifold=\{R\in\mbR^{n\times r}\mid\aanifold(RR^\top)-b=0\}$.
    \end{itemize}
\end{assumption}

Denote $h:\sym(n)\to\mbR^q: X\mapsto\aanifold(X)-b$. Under Assumption~\ref{assu:h_affinesdp}, the analysis in \cite[Appendix A]{boumal2020deterministic} shows that the differential of $h\circ\phiBM$ at $R$ has rank equal to the dimension of the space spanned by $\{A_1R,A_2R,\ldots,A_qR\}$, which validates the application of Corollary~\ref{cor:expressions_S+capU}. Therefore, the intersection rules~\eqref{eq:expressions_S+capU} hold with $\uanifold$ as an affine set satisfying Assumption~\ref{assu:h_affinesdp}.

Finally, given $X\in\boundedranksdp(n)\cap\uanifold$ with $\rank{(X)}=s$ and the spectral decomposition $X=U\varLambda U^\top$, we present the closed-form formula of the tangent cone:
\begin{equation*}
\begin{aligned}
\Btangent_{\boundedranksdp(n)\cap\uanifold}(X)&= \tangent_{\boundedranksdp(n)}(X)\cap\tangent_{\uanifold}(X)
\\
&=\hkh{\eta = [U\ U_{\bot}]\left[ \begin{matrix}
        	W_1&		W_2\\
        	W_2^\top&		J\\
        \end{matrix} \right][U\ U_{\bot}]^\top \left|\,\begin{array}{l}
        W_1\in\sym(s),
        \\
        W_2\in\mbR^{s\times (n-s)},
        \\
        J\in\sym^+_{\le r-s}(n-s),
        \\
        \innerp{A_i,\eta}=0\ \text{for}\ i\in[q]
        \end{array}\right.
        },
\end{aligned}
\end{equation*}
which recovers the result in~\cite[Corollary 4.12]{levin2025effect}. In addition, given any $\eta\in\Btangent_{\boundedranksdp(n)\cap\uanifold}(X)$, we have $\tangenttwo_{\uanifold}(X;\eta)=\Btangent_{\uanifold}(X) $ according to definition~\eqref{eq:second_tangentcone} and the affine structure of $\uanifold$. Therefore, the second-order tangent set $\tangenttwo_{\boundedranksdp(n)\cap\uanifold}(X;\eta)=\tangenttwo_{\boundedranksdp(n)}(X;\eta)\cap\tangenttwo_{\uanifold}(X;\eta)$ can be characterized by substituting the formula of $\tangenttwo_{\boundedranksdp(n)}(X;\eta)$ developed in~Proposition~\ref{pro:tangenttoS+r}. Section~\ref{sec:PSD_BM} uses these tangent sets to derive stationarity conditions for~\eqref{eq:Fproblem} and to relate them to stationarity conditions for the parameterized problem~\eqref{eq:BM}.

\section{\texorpdfstring{Proofs for section~\ref{sec:equi_2ndpoints}}{Proofs for section \ref{sec:equi_2ndpoints}}}\label{app:lemmas2to2}

This section is devoted to proofs for Theorem~\ref{the:2to2} and Proposition~\ref{pro:composition} developed in section~\ref{sec:equi_2ndpoints}. We start with several technical lemmas essential for the analysis. The first lemma relates the second-order tangent set to the second fundamental form on a smooth manifold.

\begin{lemma}\label{lem:funda_tangent2}
Given a smooth manifold $\bmanifold$ embedded in an Euclidean space $\wanifold$. Let $Y \in \bmanifold$ and $v \in \tangent_{\bmanifold}(Y)$. The associated second-order tangent set can be identified as an affine subspace, $ \tangenttwo_{\bmanifold}(Y; v) = \secfunda_Y(v, v) + \tangent_{\bmanifold}(Y)$.
\end{lemma}

\begin{proof}
    As shown in \cite[Proposition~5.16]{lee2012manifolds}, the submanifold $\bmanifold$ can be locally realized as the level set of a smooth submersion $h: \wanifold \to \mathbb{R}^k$ (where $k$ is the codimension of $\bmanifold$). That is, there exists a neighborhood $\neighbor$ of $Y$ in $\wanifold$ such that $\bmanifold\cap\neighbor=\{\tilde{Y}\in\neighbor\mid h(\tilde{Y})=0\}$. Additionally, it holds that $\tangent_{\bmanifold}(Y)=\kernel(\diff h_Y)$.
    
    It is revealed in~\cite[formula~(3.7)]{levin2025effect} that the second-order tangent set to the manifold can be characterized by $\tangenttwo_{\bmanifold}(Y;v)=\{u_v\in\wanifold\mid\diff h_Y[u_v]+\diff^2h_Y[v,v]=0\}$. That is, $\tangenttwo_{\bmanifold}(Y;v)$ consists of all vectors $u_v \in \wanifold$ satisfying the following linear equation:
    \begin{equation}\label{eq:linearequation}
        \diff h_Y[u_v]=-\diff^2h_Y[v,v],
    \end{equation}
    the solution set of which is an affine subspace given by the sum of a particular solution and the kernel of the linear operator $\diff h_Y$. Note that $\kernel(\diff h_Y)=\tangent_{\bmanifold}(Y)$, and thus it suffices to show that $\secfunda_Y(v,v)$ is a particular solution of~\eqref{eq:linearequation}.
    
Let $\beta(t)$ be a smooth curve on $\bmanifold$ such that $\beta(0) = Y$, $\beta^\prime(0) = v$, and we denote $w = \beta^{\prime\prime}(0)$. By definition of the second fundamental form, we have $\secfunda(v,v)=\projection_{\normal_{\bmanifold}(Y)}(w)$. Moreover, since the curve lies on the manifold, we have the identity $h(\beta(t)) \equiv 0$. Twice differentiating it with respect to $t$ yields
    $\diff h_{\beta(t)}[\beta^{\prime\prime}(t)] + \diff^2h_{\beta(t)}[\beta^\prime(t), \beta^\prime(t)] = 0$.
    Evaluating at $t=0$ and substituting $(\beta(0),\beta^\prime(0),\beta^{\prime\prime}(0))=(Y,v,w)$, we obtain $\diff h_{Y}[w]=-\diff^2h_{Y}[v,v]$. Substituting $w=\secfunda_Y(v,v)+\projection_{\tangent_{\bmanifold}(Y)}(w)$ and recalling $\kernel(\diff h_Y)=\tangent_{\bmanifold}(Y)$ lead to $\diff h_{Y}[\secfunda_Y(v,v)]=-\diff^2h_{Y}[v,v]$, and thus $\secfunda_Y(v,v)$ is a particular solution of~\eqref{eq:linearequation}.
\end{proof}

\bigskip

The subsequent two lemmas examine the properties of the set $\qanifold_Y(v)$ defined in~\eqref{eq:defGammaY}. Specifically, Lemma~\ref{lem:cone_QY0} characterizes the structure of $\qanifold_Y(0)$ for the case $v=0$, while Lemma~\ref{lem:second_Q} demonstrates the role of $\qanifold_Y(v)$ in optimization landscapes for general tangent vectors $v$.
In fact, $\qanifold_Y(0)$ coincides with the set $B_Y$ in \cite[Definition~3.20]{levin2025effect}, and Lemma~\ref{lem:cone_QY0} is part of \cite[Proposition~3.21]{levin2025effect}; we include a proof for completeness.
\begin{lemma}\label{lem:cone_QY0}
    The set $\qanifold_Y(0)$ is a cone, i.e., if $\zeta \in \qanifold_Y(0)$, then $c \zeta \in \qanifold_Y(0)$ for all $c\ge 0$. Moreover, it holds that $\ima(\tanL_Y)\subseteq\qanifold_Y(0)$.
\end{lemma}

\begin{proof}
    Let $\zeta \in \qanifold_Y(0)$ and $c \ge 0$. By definition of $\qanifold_Y(0)$ in~\eqref{eq:defGammaY}, there exist $\{v_i\}_{i\in\mbN} \subset \tangent_{\bmanifold}(Y)$ with $\tanL_Y(v_i) \to 0$, $u_i \in \tangenttwo_{\bmanifold}(Y; v_i)$, and $\{z_i\}_{i\in\mbN} \subseteq \ima(\tanL_Y)$ such that $\zeta = \lim_{i\to\infty} (\tanQ_{Y,v_i}(u_i) + z_i)$.
    We consider the characterization of the second-order tangent set: $\tangenttwo_{\bmanifold}(Y; v) = \secfunda_Y(v, v) + \tangent_{\bmanifold}(Y)$. Thus, for each $i$, we can decompose $u_i$ as $u_i = \secfunda_Y(v_i, v_i) + w_i$ for some $w_i \in \tangent_{\bmanifold}(Y)$. Then, we can introduce the scaled sequences $\tilde{v}_i := \sqrt{c} v_i$ and $\tilde{u}_i := c u_i$, which directly yields $\tanL_Y(\tilde{v}_i) = \sqrt{c} \tanL_Y(v_i) \to 0$. In addition, we have
    $\tilde{u}_i = c \secfunda_Y(v_i, v_i) + c w_i = \secfunda_Y(\sqrt{c} v_i, \sqrt{c} v_i) + c w_i = \secfunda_Y(\tilde{v}_i, \tilde{v}_i) + c w_i$,
    revealing that $\tilde{u}_i \in \tangenttwo_{\bmanifold}(Y; \tilde{v}_i)$. Finally, setting $\tilde{z}_i = c z_i \in \ima(\tanL_Y)$, the limit becomes $\lim_{i\to\infty} \kh{\tanQ_{Y,\tilde{v}_i}(\tilde{u}_i) + \tilde{z}_i} 
    = c \lim_{i\to\infty} \kh{\tanQ_{Y,v_i}(u_i) + z_i} 
    = c \zeta\in \qanifold_Y(0)$.

    Moreover, we can take $v_i=0$, which implies that $\tangenttwo_{\bmanifold}(Y;v_i)=\tangent_{\bmanifold}(Y)$~\cite[Proposition 9]{giorgi2010overviewsectengent}. Therefore, for any $u\in\ima(\tanL_Y)$, set $u_i=0$, and $z_i=u$ to obtain $\tanQ_{Y,v_i}(u_i)+z_i=u$ for all $i\in\mbN$, indicating that $u = \lim_{i\to\infty}\tanQ_{Y,v_i}(u_i)+z_i\in\qanifold_Y(0)$. The arbitrariness of $u\in\ima(\tanL_Y)$ concludes that $\ima(\tanL_Y)\subseteq\qanifold_Y(0)$. 
\end{proof}

\begin{lemma}\label{lem:second_Q}
    Given a smooth function $f$ defined on $\eanifold$ and a parameterization $(\bmanifold,\phi)$ for $\manifold$, suppose that $Y\in\bmanifold$ is a second-order stationary point for problem~\eqref{eq:PM} and denote $X=\phi(Y)$. Then for any $v\in\tangent_{\bmanifold}(Y)$, it holds that 
    \begin{equation*}
        \innerp{\nabla f(X),\zeta}  + \left\langle\tanL_Y(v),\nabla^2 f(X)[\tanL_Y(v)]\right \rangle \ge 0,\ \ \text{for all}\ \zeta\in\qanifold_Y(v).
    \end{equation*}
\end{lemma}
\begin{proof}
    The first-order stationarity of $Y$ reveals that $\innerp{\nabla \bar{f}(Y),v^\prime}=0$ for all $v^\prime \in\tangent_{\bmanifold}(Y)$. Substituting $\nabla\bar{f}(Y)=\diff\phi_Y^*(\nabla f(X))$ shows that $\nabla f(X)\perp \ima(\tanL_Y)$. According to the definition of $\qanifold_Y(v)$, any element $\zeta\in\qanifold_Y(v)$ can be realized by $\zeta = \lim_{i\to\infty} \tanQ_{Y,v_i}(u_{i})+z_i$ with $\{v_i\}_{i\in\mbN}\subseteq\tangent_{\bmanifold}(Y)$, $\{z_i\}_{i\in\mbN}\subseteq\ima(\tanL_Y)$, and the associated $u_{i}\in\tangenttwo_{\bmanifold}(Y;v_i)$ satisfying $\tanL_Y(v_i) \to \tanL_Y(v)$. Since $Y$ is second-order stationary for~\eqref{eq:PM}, the condition~\eqref{eq:def2optcon} indicates that for any $(v_i,u_i)$, it holds that $\innerp{\nabla \bar{f}(Y),u_{i}}+\innerp{v_{i},\nabla^2 \bar{f}(Y)[v_{i}]} \ge 0$. Applying the equality~\eqref{eq:2order_deri}, and substituting $z_i\in\ima(\tanL_Y)$ together with the first-order condition $\nabla f(X)\bot\ima(\tanL_Y)$ yield that $\innerp{\nabla f(X),\tanQ_{Y,v_{i}}(u_{i})+z_i}  + \left\langle\tanL_Y(v_{i}),\nabla^2 f(X)[\tanL_Y(v_{i})]\right \rangle \ge 0$. Letting $i\to\infty$ in the inequality, we obtain $\innerp{\nabla f(X),\zeta}  + \left\langle\tanL_Y(v),\nabla^2 f(X)[\tanL_Y(v)]\right \rangle \ge 0$. 
\end{proof}

\bigskip

The next lemma investigates the property of $\varGamma_Y$ constructed in~\eqref{eq:defGammaY}.
\begin{lemma}\label{lem:Gamma_Y}
    Given a smooth function $f$ defined on $\eanifold$ and a parameterization $(\bmanifold,\phi)$ for $\manifold$, suppose that $Y\in\bmanifold$ is a second-order stationary point for problem~\eqref{eq:PM}, and $\ima(\tanL_Y)=\tangent_{\manifold}(X)$ with $X\!=\phi(Y)$. Then for any $\tau\in\varGamma_Y$, it holds that $\innerp{\nabla f(X),\tau}=0$.
\end{lemma}
\begin{proof}
Recalling the definition of $\varGamma_Y$~\eqref{eq:defGammaY}, we denote $\tau \in \varGamma_Y$ by 
\[
\tau=\projection_{\Fnormal_{\manifold}(X)}\kh{\diff\phi_Y(\secfunda(v_0,v_1))+\diff^2\phi_Y[v_0,v_1]}\ \text{for some}\ v_0\in\ker(\tanL_Y),\, v_1\in\tangent_{\bmanifold}(Y).
\]
Define a symmetric bilinear function on $\Btangent_{\bmanifold}(Y)\times\Btangent_{\bmanifold}(Y)$ as follows,
\begin{equation*}
q(\tilde{v}_0,\tilde{v}_1):=\langle\nabla f(X),\diff\phi_Y(\secfunda(\tilde{v}_0,\tilde{v}_1)) +\diff^2\phi_Y[\tilde{v}_0,\tilde{v}_1]\rangle + \innerp{\tanL_Y(\tilde{v}_0),\nabla^2 f(X)[\tanL_Y(\tilde{v}_1)]}.
\end{equation*}
Since $v_0\in\ker(\tanL_Y)$, we have $\tanL_Y(v_0)=0$, and thus
\begin{align*}
    \innerp{\nabla f(X),\tau} = \innerp{\nabla f(X),\projection_{\Fnormal_{\manifold}(X)}\kh{\diff\phi_Y(\secfunda(v_0,v_1))+\diff^2\phi_Y[v_0,v_1]}} = q(v_0,v_1).
\end{align*}
The last equality comes from the first-order optimality $\nabla \bar{f}(Y)=\diff\phi_Y^*(\nabla f(X))\perp \tangent_{\bmanifold}(Y)$, which, together with $\ima(\tanL_Y)=\tangent_{\manifold}(X)$, reveals that $\nabla f(X)\in\Fnormal_{\manifold}(X)$. Then, to prove $\innerp{\nabla f(X),\tau}=0$, it suffices to prove $q(v_0,v_1)=0$ for all $v_0\in\ker(\tanL_Y),\, v_1\in\tangent_{\bmanifold}(Y)$. To see this, we note from~\eqref{eq:imaQ_secfunda} and~\eqref{eq:2order_deri} that $q(\tilde{v},\tilde{v})=\innerp{\nabla\bar{f}(Y),u_{\tilde{v}}}+\innerp{\tilde{v},\nabla^2\bar{f}(Y)[\tilde{v}]}$, and thus the second-order stationarity of $Y$ reveals that $q(\tilde{v},\tilde{v})\ge 0$ for all $\tilde{v}\in\Btangent_{\bmanifold}(Y)$. Considering the direction $v_0 + t v_1$, we have $q(v_0+tv_1,v_0+tv_1) \ge 0$ for all $t$, which can be expanded as follows,
\begin{equation}\label{eq:expand_q}
    q(v_0+tv_1,v_0+tv_1) = q(v_0,v_0) + 2t q(v_0,v_1) + t^2 q(v_1,v_1)\ge 0.
\end{equation}
Based on $\tanL_Y(v_0)=0$, it holds that $\diff\phi_Y(\secfunda(v_0,v_0))+\diff^2\phi_Y[v_0,v_0]\in \ima(\tanQ_{Y,v_0})\subseteq\tangenttwo_{\manifold}(X;0)=\tangent_{\manifold}(X)$, where the equality holds from~\cite[Proposition 9]{giorgi2010overviewsectengent}. Therefore, we have 
\begin{equation*}
q(v_0,v_0)\!=\!\underbrace{\langle\nabla f(X),\diff\phi_Y(\secfunda(v_0,v_0) +\diff^2\phi_Y[v_0,v_0]\rangle}_{=0,\ \text{since}\ \nabla f(X)\perp \tangent_{\manifold}(X)} + \innerp{\underbrace{\tanL_Y(v_0)}_{=0^{}_{}},\nabla^2 f(X)[\tanL_Y(v_0)]} \!= 0,
\end{equation*}
which enforces $q(v_0,v_1)\ge 0$ to guarantee~\eqref{eq:expand_q}. Replacing $v_1$ with $-v_1$ and repeating the analysis show that $q(v_0,v_1) = 0$.
\end{proof}

\smallskip

We now arrive at the proof of the main theorem.

\smallskip

\noindent\textit{Proof of Theorem~\ref{the:2to2}}\ \ 
    \textbf{(Sufficiency)} We first assume that $\ima(\tanL_Y)=\Btangent_{\manifold}(X)$ and the inclusion~\eqref{eq:nece_suff_2to2} holds for all $v\in\Btangent_{\bmanifold}(Y)$. Suppose that $Y$ is second-order stationary for problem~\eqref{eq:PM}. Then, the first-order condition $\nabla \bar{f}(Y)\in\normal_{\bmanifold}(Y)$ implies that $0=\innerp{v,\nabla\bar{f}(Y)}=\innerp{v,\diff\phi_Y^*(\nabla f(X))}=\innerp{\tanL_Y(v),\nabla f(X)}$, {for all} $v\in\Btangent_{\bmanifold}(Y)$,
    which, together with $\ima(\tanL_Y)=\Btangent_{\manifold}(X)$ confirms the first-order stationarity of $X$.

    Subsequently, we turn to the second-order condition. Given an arbitrary $\eta\in \Btangent_{\manifold}(X)$ and an associated $\zeta\in \tangenttwo_{\manifold}(X;\eta)$. By $\ima(\tanL_Y)=\Btangent_{\manifold}(X)$, pick $v\in \tangent_{\bmanifold}(Y)$ with $\eta=\tanL_Y(v)$. According to the assumed inclusion~\eqref{eq:nece_suff_2to2}, for any $\varepsilon>0$, there exist $\{(s^{(j)},b^{(j)},\tau^{(j)})\}_{j=1}^N\subseteq\qanifold_Y(v)\times \qanifold_Y(0)\times \varGamma_Y$, together with the coefficients $\{\mu_j\}_{j=1}^N$ such that     
    \begin{equation}\label{eq:varepsilon_sj}
        \|\,\zeta - \sum_{j=1}^N \mu_j (s^{(j)}+b^{(j)}+\tau^{(j)})\,\|<\varepsilon,
    \end{equation}
    where $\mu_j\ge0$ and $\sum_{j=1}^N\mu_j=1$.
    For all $j=1,2,\ldots,N$, applying Lemmas~\ref{lem:second_Q} and~\ref{lem:Gamma_Y}, we have
    \[
    \innerp{\nabla f(X),s^{(j)}}  + \left\langle\eta,\nabla^2 f(X)[\eta]\right \rangle \ge 0,\ \ \langle\nabla f(X),b^{(j)}\rangle{\ge 0},\ \ \text{and}\ \innerp{\nabla f(X),\tau^{(j)}}=0.
    \]
    Taking the convex combination with the coefficients $\{\mu_j\}_{j=1}^N$ gives $\Big\langle \nabla f(X),\,\sum_{j=1}^N \mu_j (s^{(j)}+b^{(j)}+\tau^{(j)})\Big\rangle + \langle \eta,\nabla^2 f(X)[\eta]\rangle\ge 0$.
    Finally, letting the parameter $\varepsilon$ in~\eqref{eq:varepsilon_sj} tend to $0$, and using the continuity of $s\mapsto \langle \nabla f(X),s\rangle$, we have $\langle \nabla f(X),\zeta\rangle\ +\ \langle \eta,\nabla^2 f(X)[\eta]\rangle\ \ge\ 0$. By the arbitrariness of $\eta\in\Btangent_\manifold(X)$ and $\zeta\in\tangenttwo_{\manifold}(X;\eta)$, we conclude the second-order stationarity of $X=\phi(Y)\in\manifold$ for problem~\eqref{eq:P}.

    \textbf{(Necessity)} We then turn to the ``only if\," part. Firstly, suppose that $\ima(\tanL_Y)\subsetneq \Btangent_{\manifold}(X)$. Taking polars reverses the inclusion $(\Btangent_\manifold(X))^{\circ}\subsetneq(\ima(\tanL_Y))^{\circ}$, where we note that the strict inclusion still holds since $\ima(\tanL_Y)$ is a linear space. Pick $w\in \Btangent_{\manifold}(X)\setminus \ima(\tanL_Y)$. Let $w_\bot:=w-\projection_{\ima(\tanL_Y)}(w)$, and the operator $H:z\mapsto-\innerp{w_\bot,z}w_\bot$. Define $f(\tilde{X})=\tfrac12\innerp{\tilde{X}-X, H(\tilde{X}-X)}$. Then $\nabla f(X)=0$ and $\nabla^2 f(X)=H$. By the computation \eqref{eq:compute_barf}, $\nabla\bar f(Y)=0$ and $\innerp{v,\nabla^2 \bar{f}(Y)[v]}=\langle\tanL_Y(v), H\circ\tanL_Y(v)\rangle= 0$ for all $v\in\tangent_{\bmanifold}(Y)$. Hence $Y$ is second-order stationary for~\eqref{eq:PM}. However, at $X$, $\nabla f(X)=0$ and $\innerp{w,\nabla^2 f(X)[w]}=-\|w_\bot\|^4<0$, violating the second-order condition for~\eqref{eq:P}. Therefore, to guarantee the property \twototwo, $\ima(\tanL_Y)=\Btangent_{\manifold}(X)$ must hold.

Then, we assume that condition~\eqref{eq:nece_suff_2to2} fails for some $v$. Let $\eta = \tanL_Y(v)$. If $\eta=0$, we have 
\begin{equation*}
\tangenttwo_{\manifold}(X;\eta)=\tangenttwo_{\manifold}(X;0)=\tangent_{\manifold}(X)=\ima(\tanL_Y)\subseteq \qanifold_Y(0),
\end{equation*}
where the second equality holds from~\cite[Proposition 9]{giorgi2010overviewsectengent}, and the last inclusion comes from Lemma~\ref{lem:cone_QY0}. Therefore, the inclusion~\eqref{eq:nece_suff_2to2} always holds for $v\in\kernel(\tanL_Y)$, and thus we assume that~\eqref{eq:nece_suff_2to2} fails for some $v$ with $\|\eta\|>0$. Then there exists $\zeta \in \tangenttwo_{\manifold}(X;\eta)$ outside the closed convex set $\mathcal{C}_v := \closeconv(\mathcal{Q}_Y(v)+\mathcal{Q}_Y(0)+\varGamma_Y)$. By the strict separation theorem~\cite{boyd2004convex}, there exists a vector $w \in \eanifold$, a threshold $\kappa\in\mbR$, and a gap $\delta > 0$ such that:
\begin{equation}\label{eq:strict_separation}
    \innerp{w, \zeta} \le \kappa - \delta < \kappa := \inf_{p \in \mathcal{C}_v} \innerp{w, p}.
\end{equation}
Since $\mathcal{Q}_Y(0)$ is a cone (see Lemma~\ref{lem:cone_QY0}), we have $\innerp{w, u} \ge 0$ for all $u \in \mathcal{Q}_Y(0)$. Noticing that the subspace $\ima(\tanL_Y)$ belongs to $\qanifold_Y(0)$, we have $w\perp \ima(\tanL_Y)$. Additionally, since $c\varGamma_Y\subseteq\varGamma_Y$ for any scaling $c\in\mbR$, we have $w\perp \varGamma_Y$. In summary, it holds that
\begin{equation}\label{eq:separation_pro}
    \innerp{w,u_0}\ge0,\ \ \innerp{w,u_1}=0,\ \ \text{and}\ \innerp{w,u_2}=0,
\end{equation}
{for all} $u_0\in\qanifold_Y(0)$, $u_1\in\ima(\tanL_Y)$, and $u_2\in\varGamma_Y$. We then construct the following function,
\[
    f_{\rho}(\tilde{X}) := \innerp{w, \tilde{X}-X} + \frac{1}{2} \innerp{\tilde{X}-X, (H_{\mathrm{shift}} + \rho H) (\tilde{X}-X)},
\]
where $H = I - \projection_{\vecspan\{\eta\}}$ and $H_{\mathrm{shift}} = \frac{-\kappa+\delta/2}{\|\eta\|^2} \projection_{\vecspan\{\eta\}}$. The Euclidean derivatives of $f_\rho$ at $X$ can be computed by $\nabla f_\rho(X)=w$ and $\nabla^2 f_\rho(X)=H_{\mathrm{shift}} + \rho H$.

First, we observe that $X$ is not a second-order stationary point of $f_\rho$. Specifically, along the direction $(\eta, \zeta)$, the second-order condition fails:
\[
    \innerp{\nabla f_\rho(X), \zeta} + \innerp{\eta, \nabla^2 f_\rho(X)[\eta]} \!=\! \innerp{w, \zeta} + \innerp{\eta,H_{\mathrm{shift}}(\eta)}\!\le\!(\kappa - \delta) + (-\kappa + \delta/2) = -\delta/2\!<\!0.
\]

Second, we aim to show that there exists a $\rho>0$ such that $Y\in\bmanifold$ is a second-order stationary point for the objective $\bar{f}_\rho = f_\rho \circ \phi$. Given any direction $d\in\tangent_{\bmanifold}(Y)$, we consider the following function induced by the Hessian of $\bar{f}_\rho$ at $Y$:
\begin{align}
    \mathcal{H}_\rho(d) :=&\ \innerp{\nabla \bar{f}_\rho(Y),u_d} + \innerp{d, \nabla^2 \bar{f}_\rho(Y)[d]}  \nonumber
    \\
    =&\ \underbrace{\innerp{w, \tanQ_{Y,d}(u_d)} + \innerp{\tanL_Y(d), H_{\mathrm{shift}}\circ \tanL_Y(d)}}_{\text{A quadratic form}\ E(d,d)} + \rho \|H(\tanL_Y(d))\|^2,   \label{eq:defEdd}
\end{align}
where~\eqref{eq:defEdd} comes from~\eqref{eq:2order_deri}. Since $w\perp\ima(\tanL_Y)$, the identity~\eqref{eq:imaQ_secfunda} implies that the term $\innerp{w, \tanQ_{Y,d}(u_d)}=\innerp{w,\diff\phi_Y(\secfunda_Y(d,d))+\diff ^2\phi_Y[d,d]}$. Therefore, the value of $\innerp{w, \tanQ_{Y,d}(u_d)}$ is independent of the choice of representative $u_d\in\tangenttwo_{\bmanifold}(Y;d)$, and the function $\hanifold_\rho(d)$ and the quadratic form $E(d,d)$ given in~\eqref{eq:defEdd} are well-defined.

Let $K := \ker(H \circ \tanL_Y) = \tanL_Y^{-1}(\vecspan(\eta))$. We present two properties of $E(d,d)$:
\begin{enumerate}
    \item [1)] Given $d \in \ker(\tanL_Y)$, it holds that $\tanQ_{Y,d}(u_d) \in \qanifold_Y(0)$. By $\tanL_Y(d)=0$ and the separation property~\eqref{eq:separation_pro}, we have $E(d,d) = \innerp{w, \tanQ_{Y,d}(u_d)} \ge 0$.
    \item [2)] Given $d \in K \setminus \kernel({\tanL_Y})$, then $\tanL_Y(d) = c \eta$ with $c \neq 0$. Let $\hat{d} = d/c$. Since $\tanL_Y(\hat{d}) = \eta=\tanL_Y(v)$, we have $\tanQ_{Y,\hat{d}}(u_{\hat{d}}) \in \qanifold_Y(v)$ for $u_{\hat{d}}\in\tangenttwo_{\bmanifold}(Y;\hat{d})$. Thus, we have $\innerp{w, \tanQ_{Y,\hat{d}}(u_{\hat{d}})} \ge \kappa$ by the separation~\eqref{eq:strict_separation}. Consequently, $E(\hat{d},\hat{d}) \ge \kappa + \innerp{\tanL_Y(\hat{d}), H_{\mathrm{shift}}\circ\tanL_Y(\hat{d})} = \kappa + (-\kappa + \delta/2) = \delta/2 > 0$, which indicates that $E(d,d)=c^2E(\hat{d},\hat{d}) > 0$.
\end{enumerate}
In summary, $E(d,d) \ge 0$ for all $d \in K$. Let $K_0 = \{ d \in K \mid E(d,d) = 0 \}$, and the property 2) indicates that $K_0 \subseteq \ker(\tanL_Y)$. More importantly, we find that
\begin{equation}\label{eq:Ed0d_0}
    E(d_0,d) = 0,\ \ \ \text{for all}\ d_0 \in K_0,\ \text{and}\ d \in \tangent_{\bmanifold}(Y),
\end{equation}
the derivation of which is deferred to the end of this proof. Let $K_+$ be the orthogonal complement of $K_0$ in $K$, and $K^\bot$ be the orthogonal complement of $K$ in $\tangent_{\bmanifold}(Y)$. Then, $\tangent_{\bmanifold}(Y)=K_0\oplus K_+\oplus K^\bot$, and there exists $\lambda > 0$ such that $E(d_+,d_+)\ge \lambda \|d_+\|^2$ for $d_+ \in K_+$.

We claim that there exists $\rho>0$ such that $\mathcal{H}_\rho(d) \ge 0$ for all $d\in\tangent_{\bmanifold}(Y)$; otherwise, we can find a sequence $(\rho_k)_{k\in\mbN}$ with $\rho_k\to\infty$ and unit vectors $d_k$ such that 
$\hanifold_{\rho_k}(d_k)<0$ for every $k$. Decompose $d_k = d_{k,0} + d_{k,+} + \ell_k$, where $d_{k,0} \in K_0$, $d_{k,+} \in K_+$, and $\ell_k \in K^\perp$. The penalty term satisfies $\|H(\tanL_Y(d_k))\| = \|H(\tanL_Y(\ell_k))\| \ge \sigma \|\ell_k\|$ for some $\sigma > 0$. Using $E(d_{k,0},\cdot)= 0$, the condition $\mathcal{H}_{\rho_k}(d_k) < 0$ becomes $E(d_{k,+} + \ell_k,d_{k,+} + \ell_k) + \rho_k \sigma^2 \|\ell_k\|^2 < 0$. We then expand and bound the terms to obtain $\lambda \|d_{k,+}\|^2 - 2 \|E\| \|d_{k,+}\| \|\ell_k\| + (\rho_k \sigma^2 - \|E\|) \|\ell_k\|^2 < 0$.
If $\ell_k = 0$, then $\lambda \|d_{k,+}\|^2 < 0 $ leads to a contradiction. If $\ell_k \neq 0$, dividing by $\|\ell_k\|^2$ yields the following quadratic inequality with $t_k = \|d_{k,+}\|/\|\ell_k\|$:
\begin{equation}\label{eq:imposs_ineq}
     \lambda t_k^2 - 2 \|E\| t_k + (\rho_k \sigma^2 - \|E\|) < 0.
\end{equation}
The minimum value of the left side is $\rho_k \sigma^2 - \|E\| - \frac{\|E\|^2}{\lambda}$. As $\rho_k \to \infty$, this minimum becomes positive, making the inequality~\eqref{eq:imposs_ineq} impossible. Thus, there indeed exists a sufficiently large $\rho$ such that $\mathcal{H}_\rho(d) \ge 0$ for all $d\in\tangent_{\bmanifold}(Y)$. Consequently, we have $\innerp{\nabla \bar{f}_\rho(Y),u_d} + \innerp{d,\nabla^2\bar{f}_\rho(Y)(d)} = {\hanifold_\rho(d)} \ge 0$ for all $d\in\tangent_{\bmanifold}(Y)$ and $u_d\in\tangenttwo_{\bmanifold}(Y;d)$. Therefore, $Y$ is second-order stationary for~\eqref{eq:PM}, while $X$ is not second-order stationary for~\eqref{eq:P}. This contradicts the property \twototwo, thereby showing the necessity of the inclusion~\eqref{eq:nece_suff_2to2}.

\noindent\textit{Derivation of~\eqref{eq:Ed0d_0}}\ \ \ In fact, taking into account the definition of $E(d,d)$~\eqref{eq:defEdd}, the property $w\perp \ima(\tanL_Y)$~\eqref{eq:separation_pro}, and the equality~\eqref{eq:imaQ_secfunda}, we have $E(d,d)=\innerp{w,\diff\phi_Y(\secfunda(d,d))+\diff^2\phi_Y[d,d]}+\innerp{\tanL_Y(d), H_{\mathrm{shift}}\circ \tanL_Y(d)}$. Then, $\tanL_Y(d_0)=0$ reveals that
\begin{align}
E(d_0,d)=\innerp{w,\diff\phi_Y(\secfunda(d_0,d))+\diff^2\phi_Y[d_0,d]} \nonumber =\innerp{w,\projection_{\Fnormal_{\manifold}(X)}\kh{\diff\phi_Y(\secfunda(d_0,d))+\diff^2\phi_Y[d_0,d]}}, \nonumber
\end{align}
where the last equality holds from the separation properties $w\in(\ima(\tanL_Y))^\circ=(\tangent_{\manifold}(X))^\circ=\Fnormal_{\manifold}(X)$. Finally, we employ the definition of the set $\varGamma_Y$~\eqref{eq:defGammaY} and the property $w\perp \varGamma_Y$~\eqref{eq:separation_pro} to conclude $E(d_0,d)=0$ for all $d_0\in K_0$ and $d\in \tangent_{\bmanifold}(Y)$. \qed

\medskip

We then provide the proof for Proposition~\ref{pro:composition}.

\smallskip

\noindent\textit{Proof of Proposition~\ref{pro:composition}}\ \ 
(i) Given any objective function $f$, suppose that $Y=\varphi(Z)$ is a second-order stationary point on $\bmanifold$ (with respect to $f\circ\phi$). We note that any curve $\gamma(t)$ on $\zanifold$ passing through $Z$ at $t=0$ satisfies that
\begin{equation}\label{eq:Z_secondorder}
    (f\circ\phi\circ\varphi\circ \gamma)^\prime(0)=0\ \ \ \text{and}\ \ \ (f\circ\phi\circ\varphi\circ \gamma)^{\prime\prime}(0)\ge 0,
\end{equation}
which holds by viewing $\varphi\circ \gamma(t)$ as a curve on $\bmanifold$ and considering the second-order stationarity of $Y$. Therefore, by~\eqref{eq:Z_secondorder}, the point $Z$ is indeed second-order stationary on $\zanifold$, which, together with the \twototwo property at $Z$, reveals the second-order stationarity of $\phi(Y)=\psi(Z)$. 

(ii) Suppose that $Z$ is a second-order stationary point on $\zanifold$ (with respect to $f\circ\phi\circ\varphi$). Given any curve $\beta(t)$ on $\bmanifold$ with $\beta(0)=Y$, since $\varphi$ is a submersion at $Z$, we can apply the local section theorem \cite[Theorem 4.26]{lee2012manifolds} to obtain a (locally defined) mapping $l:\bmanifold\to\zanifold$ such that $(\varphi\circ l\circ \beta)(t)=\beta(t)$. Letting $\gamma(t):= (l\circ \beta)(t)$, the second-order stationarity of $Z$ indicates~\eqref{eq:Z_secondorder}. Substituting $\gamma=l\circ \beta$ and $\varphi\circ l\circ \beta=\beta$ leads to $(f\circ\phi\circ\beta)^\prime(0)=0\ \text{and}\ (f\circ\phi\circ\beta)^{\prime\prime}(0)\ge 0$, the arbitrariness of $\beta(t)$ yielding the second-order stationarity of $Y$. Hence, the property \twototwo\ at $Y$ concludes that $\psi(Z)=\phi(Y)$ is second-order stationary, validating the \twototwo\ at $Z$. \qed


\section{\texorpdfstring{Proofs for section~\ref{sec:PSD_BM}}{Proofs for section \ref{sec:PSD_BM}}}\label{app:PSD_BM}

This appendix collects the technical results and proofs for section~\ref{sec:PSD_BM}, with the notation introduced there.

\subsection{\texorpdfstring{Proofs for section~\ref{sec:PSD_SOC}}{Proofs for section \ref{sec:PSD_SOC}}}\label{app:F1}
We represent every $\eta\in\sym(n)$ as
\begin{equation}\label{eq:Fblock}
    \eta=[U\ U_\bot]\begin{bmatrix}W_1&W_2\\W_2^\top&J\end{bmatrix}[U\ U_\bot]^\top,\ \ W_1\in\sym(s),\ W_2\in\mbR^{s\times(n-s)},\ J\in\sym(n-s),
\end{equation}
and we write $W_2(\eta):=U^\top\eta U_\bot$ and $J(\eta)=U_\bot^\top\eta U_\bot$. Every $R\in\phiBM^{-1}(X)$ reads $R=U\varLambda^{1/2}P^\top$ with $P\in\stiefel(r,s)$. From $\sum_ic_iA_iR=(\sum_ic_iA_iU)\varLambda^{1/2}P^\top$, LICQ at $X$ is equivalent to Assumption~\ref{assu:h_affinesdp}(i) at any $R\in\phiBM^{-1}(X)$, and thus the differential of $\bar h$ is surjective near $\phiBM^{-1}(X)$.

By~\eqref{eq:FtangentF} and Proposition~\ref{pro:tangenttoS+r}, it holds that
\begin{equation}\label{eq:FTX}
    \tangent_{\boundedranksdp(n)\cap\uanifold}(X)=\{\eta\in\kernel(\aanifold)\mid J(\eta)\in\sym^+_{\le r-s}(n-s)\},
\end{equation}
and Proposition~\ref{pro:tangenttoS+r} with $r=s$ gives
\begin{equation}\label{eq:FTs}
    \tangent_{\sdp^+_s(n)}(X)=\{\eta\in\sym(n)\mid J(\eta)=0\}=\{UB^\top+BU^\top\mid B\in\mbR^{n\times s}\}.
\end{equation}
We present the following technical lemma, which is used repeatedly in the rest of this appendix.

\begin{lemma}\label{lem:FetaJ}
    Let $X\in\boundedranksdp(n)\cap\uanifold$ with $\rank(X)=s\le r$ satisfy LICQ.
    \begin{itemize}
        \item[\emph{(i)}] $\aanifold(\tangent_{\sdp^+_s(n)}(X))=\mbR^q$.
        \item[\emph{(ii)}] For every $J\in\sym(n-s)$, there exists $\eta\in\kernel(\aanifold)$ with $J(\eta)=J$.
    \end{itemize}
\end{lemma}
\begin{proof}
    (i) follows from $\innerp{A_i,UB^\top+BU^\top}=2\innerp{A_iU,B}$ and LICQ. (ii) Let $J\in\sym(n-s)$. By Lemma~\ref{lem:FetaJ}(i), there is $\eta_0\in\tangent_{\sdp^+_s(n)}(X)$ with $\aanifold(\eta_0)=-\aanifold(U_\bot JU_\bot^\top)$. Then $\eta:=U_\bot JU_\bot^\top+\eta_0\in\kernel(\aanifold)$ and $J(\eta)=J$, since $J(\eta_0)=0$ by~\eqref{eq:FTs}.
\end{proof}

We then prove Proposition~\ref{pro:Ffirst}.
\begin{proof}
    For $\mu$ and $S$ as in the proposition, $SX=0$ gives $SU=0$, i.e., $S=U_\bot ZU_\bot^\top$ with $Z:=U_\bot^\top SU_\bot$; thus each $\eta\in\kernel(\aanifold)$ satisfies $\innerp{\nabla f(X),\eta}=\innerp{S,\eta}+\innerp{\mu,\aanifold(\eta)}=\innerp{Z,J(\eta)}$, and hence $\innerp{\nabla f(X),\eta}\ge0$ for all $\eta\in\tangent_{\boundedranksdp(n)\cap\uanifold}(X)$ by~\eqref{eq:FTX}, since $Z\succeq0$ and $J(\eta)\succeq0$ if $s<r$, and $J(\eta)=0$ if $s=r$. Conversely, let $X$ be first-order stationary. Since $\tangent_{\sdp^+_s(n)}(X)\cap\kernel(\aanifold)$ is a subspace contained in $\tangent_{\boundedranksdp(n)\cap\uanifold}(X)$, we have $\innerp{\nabla f(X),\eta}=0$ for all $\eta\in\tangent_{\sdp^+_s(n)}(X)\cap\kernel(\aanifold)$, and thus $S=\nabla f(X)-\aanifold^*(\mu)\perp\tangent_{\sdp^+_s(n)}(X)$ for some $\mu$, which is equivalent to $SX=0$; $\mu$ is unique since $(\aanifold|_{\tangent_{\sdp^+_s(n)}(X)})^*$ is injective by Lemma~\ref{lem:FetaJ}(i). If $s<r$, then, for every $w\in\mbR^{n-s}$, Lemma~\ref{lem:FetaJ}(ii) and~\eqref{eq:FTX} give $\eta\in\tangent_{\boundedranksdp(n)\cap\uanifold}(X)$ with $J(\eta)=ww^\top$, and $w^\top Zw=\innerp{\nabla f(X),\eta}\ge0$ yields $Z\succeq0$.
\end{proof}

Let $X$ satisfy LICQ and be first-order stationary, with $S$ as in Proposition~\ref{pro:Ffirst} and $S=U_\bot ZU_\bot^\top$ as in its proof. Every $\eta\in\tangent_{\boundedranksdp(n)\cap\uanifold}(X)$ satisfies $\innerp{\nabla f(X),\eta}=\innerp{Z,J(\eta)}$ by the proof of Proposition~\ref{pro:Ffirst}, and $\innerp{Z,J(\eta)}=0$ if and only if $ZJ(\eta)=0$, since $Z,J(\eta)\succeq0$ if $s<r$ and $J(\eta)=0$ if $s=r$ by~\eqref{eq:FTX}. Thus
\begin{equation}\label{eq:Fcone}
    \tangent_{\boundedranksdp(n)\cap\uanifold}(X)\cap\vecspan\{\nabla f(X)\}^\bot=\big\{\eta\in\kernel(\aanifold)\mid J(\eta)\in\sym^+_{\le r-s}(n-s),\ ZJ(\eta)=0\big\}.
\end{equation}

We next give the proof for Proposition~\ref{pro:Fopt}.
\begin{proof}
    (i) Both conditions in Proposition~\ref{pro:Fopt}(i) include first-order stationarity, which we assume. Fix $\eta$ in the set~\eqref{eq:Fcone} with $\rank(J(\eta))=\ell-s$, and let $U_\eta$, $U^+=[U\ U_\eta]$ and $U_{\eta\bot}$ be as in Proposition~\ref{pro:tangenttoS+r}. By~\eqref{eq:FtangentF} and Proposition~\ref{pro:tangenttoS+r}, every $\zeta\in\tangenttwo_{\boundedranksdp(n)\cap\uanifold}(X;\eta)$ satisfies $\aanifold(\zeta)=0$ and $\zeta=2\eta X^\dagger\eta+[U^+\ U_{\eta\bot}]\big[\begin{smallmatrix}W_1'&W_2'\\W_2'^\top&L\end{smallmatrix}\big][U^+\ U_{\eta\bot}]^\top$ with $L\in\sym^+_{\le r-\ell}(n-\ell)$. The columns of $U_\eta$ lie in $U_\bot\range(J(\eta))\subseteq U_\bot\kernel(Z)$ by~\eqref{eq:Fcone}, and thus $SU^+=0$ and
    \begin{equation}\label{eq:Finf}
        \innerp{\nabla f(X),\zeta}=\innerp{S,\zeta}=2\innerp{S,\eta X^\dagger\eta}+\innerp{U_{\eta\bot}^\top SU_{\eta\bot},L}\ge2\innerp{S,\eta X^\dagger\eta},
    \end{equation}
    where the inequality holds since $S\succeq0$ if $s<r$, and $L=0$ if $s=r$. By Lemma~\ref{lem:FetaJ}(i), there is $\xi\in\tangent_{\sdp^+_s(n)}(X)$ with $\aanifold(\xi)=-2\aanifold(\eta X^\dagger\eta)$; the matrix $\zeta^*:=2\eta X^\dagger\eta+\xi$ belongs to $\tangenttwo_{\boundedranksdp(n)\cap\uanifold}(X;\eta)$, with $L=0$, and attains equality in~\eqref{eq:Finf} because $\innerp{S,\xi}=0$. Hence
    \begin{equation*}
        \min_{\zeta\in\tangenttwo_{\boundedranksdp(n)\cap\uanifold}(X;\eta)}\innerp{\nabla f(X),\zeta}+\innerp{\eta,\nabla^2f(X)[\eta]}=Q_S(\eta),
    \end{equation*}
    and~\eqref{eq:def2optcon} yields (i).

    (ii) Let $\gamma^*:=\frac12\min\{Q_S(\eta)\mid\eta\in\tangent_{\boundedranksdp(n)\cap\uanifold}(X)\cap\vecspan\{\nabla f(X)\}^\bot,\ \|\eta\|=1\}$, with $\min\emptyset:=+\infty$. The minimum is attained, since the unit vectors in~\eqref{eq:Fcone} form a compact set and $Q_S$ is continuous, and thus $\gamma^*>0$ by the hypothesis $Q_S(\eta)>0$ for all nonzero $\eta$ in~\eqref{eq:Fcone}. Let $\gamma\in(0,\gamma^*)$. Suppose, on the contrary, that there exist $Y_i\in(\boundedranksdp(n)\cap\uanifold)\setminus\{X\}$ with $Y_i\to X$ and $f(Y_i)<f(X)+\gamma t_i^2$, where $t_i:=\|Y_i-X\|$. Since $\aanifold(Y_i)=\aanifold(X)=b$, the unit vectors $d_i:=(Y_i-X)/t_i$ satisfy $\aanifold(d_i)=0$, and a subsequence, still denoted by $\{d_i\}$, converges to some $d$ with $\|d\|=1$ and $\aanifold(d)=0$. Write $[U\ U_\bot]^\top Y_i[U\ U_\bot]=\big[\begin{smallmatrix}\varLambda_i&B_i\\B_i^\top&C_i\end{smallmatrix}\big]$, so that $\varLambda_i\to\varLambda$, $B_i=t_iW_2(d_i)$ and $C_i=t_iJ(d_i)$. For large $i$, $\varLambda_i\succ0$ and the Schur complement $\varPsi_i:=C_i-B_i^\top\varLambda_i^{-1}B_i$ satisfies $\varPsi_i\in\sym^+_{\le r-s}(n-s)$, since $Y_i\succeq0$ and $\rank(\varPsi_i)=\rank(Y_i)-s\le r-s$. Since $\aanifold(Y_i-X)=0$, $\innerp{\nabla f(X),Y_i-X}=\innerp{Z,C_i}$, and Taylor's formula gives
    \begin{equation*}
        f(Y_i)-f(X)=\innerp{Z,\varPsi_i}+t_i^2\innerp{Z,W_2(d_i)^\top\varLambda_i^{-1}W_2(d_i)}+\frac{t_i^2}{2}\innerp{d_i,\nabla^2f(X)[d_i]}+o(t_i^2),
    \end{equation*}
    where $\innerp{Z,\varPsi_i}\ge0$ because $Z\succeq0$ if $s<r$ and $\varPsi_i=0$ if $s=r$. Together with $f(Y_i)-f(X)<\gamma t_i^2$, this yields $\innerp{Z,\varPsi_i}=O(t_i^2)$. Moreover, $\varPsi_i/t_i=J(d_i)-t_iW_2(d_i)^\top\varLambda_i^{-1}W_2(d_i)\to J(d)$, and thus $J(d)\in\sym^+_{\le r-s}(n-s)$ and $\innerp{Z,J(d)}=0$. Hence $ZJ(d)=0$, since $Z,J(d)\succeq0$ if $s<r$ and $J(d)=0$ if $s=r$, and $d$ lies in the set~\eqref{eq:Fcone}, which is therefore not $\{0\}$. On the other hand, since $\innerp{Z,\varPsi_i}\ge0$, the expansion above and $f(Y_i)-f(X)<\gamma t_i^2$ give
    \begin{equation*}
        \gamma>\innerp{Z,W_2(d_i)^\top\varLambda_i^{-1}W_2(d_i)}+\frac12\innerp{d_i,\nabla^2f(X)[d_i]}+o(1).
    \end{equation*}
    Letting $i\to\infty$ and using $\innerp{S,dX^\dagger d}=\innerp{Z,W_2(d)^\top\varLambda^{-1}W_2(d)}$, which follows from $U_\bot^\top dU=W_2(d)^\top$ and $X^\dagger=U\varLambda^{-1}U^\top$, we obtain $\gamma\ge\frac12Q_S(d)\ge\gamma^*$, where the last inequality holds since $d$ is a unit vector in~\eqref{eq:Fcone}. This contradicts $\gamma<\gamma^*$.
\end{proof}

The following two lemmas are used in the proof of Proposition~\ref{pro:FBM}.

\begin{lemma}\label{lem:FimL}
    Let $X\in\boundedranksdp(n)\cap\uanifold$ with $\rank(X)=s\le r$ satisfy LICQ, and let $R=U\varLambda^{1/2}P^\top\in\phiBM^{-1}(X)$ with $P\in\stiefel(r,s)$. Then $\tangent_{\buanifold}(R)=\{\dot R\in\mbR^{n\times r}\mid\aanifold(\dot RR^\top+R\dot R^\top)=0\}$, $\tanL_R(\dot R)=\dot RR^\top+R\dot R^\top$ for $\dot R\in\tangent_{\buanifold}(R)$, and
    \begin{equation}\label{eq:FimL}
        \ima(\tanL_R)=\tangent_{\sdp^+_s(n)}(X)\cap\kernel(\aanifold)=\{\eta\in\kernel(\aanifold)\mid J(\eta)=0\}\subseteq\tangent_{\boundedranksdp(n)\cap\uanifold}(X).
    \end{equation}
    Moreover, every $\eta\in\ima(\tanL_R)$ equals $\tanL_R(\dot R)$ for some $\dot R\in\tangent_{\buanifold}(R)$ with $\dot RPP^\top=\dot R$.
\end{lemma}
\begin{proof}
    Since the differential of $\bar h$ is surjective at $R$, $\tangent_{\buanifold}(R)$ is its kernel. Moreover, $\tanL_R(\dot R)=\diff\phiBM(R)[\dot R]=\dot RR^\top+R\dot R^\top$. For $\dot R\in\tangent_{\buanifold}(R)$, $\tanL_R(\dot R)=UB^\top+BU^\top$ with $B=\dot RP\varLambda^{1/2}$, which shows $\ima(\tanL_R)\subseteq\tangent_{\sdp^+_s(n)}(X)\cap\kernel(\aanifold)$. Conversely, for $\eta=UB^\top+BU^\top\in\tangent_{\sdp^+_s(n)}(X)\cap\kernel(\aanifold)$, the matrix $\dot R:=B\varLambda^{-1/2}P^\top$ satisfies $\tanL_R(\dot R)=\eta$, $\dot R\in\tangent_{\buanifold}(R)$, and $\dot RPP^\top=\dot R$. The remaining relations in~\eqref{eq:FimL} follow from~\eqref{eq:FTs} and~\eqref{eq:FTX}.
\end{proof}

\begin{lemma}\label{lem:Fsplit}
    Let $X\in\boundedranksdp(n)\cap\uanifold$ with $\rank(X)=s\le r$ satisfy LICQ, let $R=U\varLambda^{1/2}P^\top\in\phiBM^{-1}(X)$ with $P\in\stiefel(r,s)$, and let $S:=\nabla f(X)-\aanifold^*(\mu)$ satisfy $SR=0$ for some $\mu\in\mbR^q$. Then every $C^\infty$ curve $c$ in $\buanifold$ with $c(0)=R$ satisfies $\bar f(c(\tau))=\bar f(R)+\frac{\tau^2}{2}h_S(c'(0))+o(\tau^2)$, where
    \begin{equation}\label{eq:Fsplit}
        h_S(\dot R):=Q_S(\tanL_R(\dot R))+2\innerp{S,\dot R(I_r-PP^\top)\dot R^\top}\quad\text{for}\ \dot R\in\tangent_{\buanifold}(R).
    \end{equation}
\end{lemma}
\begin{proof}
    The function $\bar f_\mu(R'):=f(R'R'^\top)-\innerp{\mu,\bar h(R')}$ coincides with $\bar f$ on $\buanifold$, and a direct computation gives $\nabla\bar f_\mu(R)=2SR=0$ and $\nabla^2\bar f_\mu(R)[\dot R,\dot R]=\innerp{\tanL_R(\dot R),\nabla^2f(X)[\tanL_R(\dot R)]}+2\innerp{S,\dot R\dot R^\top}$. Let $\eta=\tanL_R(\dot R)$. Since $R^\top S=0$, we have $\eta S\eta=R\dot R^\top S\dot RR^\top$, and $R^\top X^\dagger R=PP^\top$ gives $\innerp{S,\eta X^\dagger\eta}=\trace(\eta S\eta X^\dagger)=\innerp{S,\dot RPP^\top\dot R^\top}$; splitting $\dot R\dot R^\top=\dot RPP^\top\dot R^\top+\dot R(I_r-PP^\top)\dot R^\top$ shows $\nabla^2\bar f_\mu(R)[\dot R,\dot R]=h_S(\dot R)$. Hence the second-order expansion of $\bar f_\mu$ along $c$ gives the expansion of $\bar f(c(\tau))$.
\end{proof}

We then prove Proposition~\ref{pro:FBM}.
\begin{proof}
    (i) Every $\dot R\in\tangent_{\buanifold}(R)$ is the velocity $c'(0)$ of a $C^\infty$ curve $c$ in $\buanifold$, along which $\bar f(c(\tau))-\bar f(R)=2\tau\innerp{\nabla f(X)R,\dot R}+o(\tau)$. Hence $R$ is first-order stationary if and only if $\nabla f(X)R\perp\tangent_{\buanifold}(R)$, i.e., $\nabla f(X)R\in\vecspan\{A_1R,\ldots,A_qR\}$, which reads $SR=0$ for some $\mu$, which is unique by LICQ.

    (ii) By Lemma~\ref{lem:Fsplit}, $R$ is second-order stationary if and only if $h_S\ge0$ on $\tangent_{\buanifold}(R)$. If $s=r$, then $PP^\top=I_r$, and~\eqref{eq:Fsplit} gives $h_S(\dot R)=Q_S(\tanL_R(\dot R))$, which proves the claim. Let $s<r$. If $S\succeq0$ and $Q_S(\eta)\ge0$ for all $\eta\in\ima(\tanL_R)$, both terms in~\eqref{eq:Fsplit} are nonnegative. Conversely, let $h_S\ge0$ on $\tangent_{\buanifold}(R)$. For $u\in\mbR^n$ and $z\in\kernel(R)\setminus\{0\}$, which exists since $s<r$, the matrix $\dot R:=uz^\top$ satisfies $\tanL_R(\dot R)=0$ and $P^\top z=0$, whence $0\le h_S(\dot R)=2\|z\|^2u^\top Su$; thus $S\succeq0$. Finally, for $\eta\in\ima(\tanL_R)$, Lemma~\ref{lem:FimL} provides $\dot R\in\tangent_{\buanifold}(R)$ with $\tanL_R(\dot R)=\eta$ and $\dot RPP^\top=\dot R$, and thus $Q_S(\eta)=h_S(\dot R)\ge0$ by~\eqref{eq:Fsplit}.
\end{proof}

For $r=n$, Ding and Wright~\cite[Theorem~3.1]{ding2025squared} show that $R$ is second-order stationary for~\eqref{eq:BM} if and only if $X$ is first-order stationary and $Q_S(\eta)\ge0$ for all $\eta\in\ima(\tanL_R)$, and \cite[Remark~2.1]{ding2025squared} extends this to $r<n$ with $\rank(R)<r$ for $q=0$. Proposition~\ref{pro:FBM} covers every $r\le n$ and $q\ge0$.

\subsection{\texorpdfstring{Proofs for section~\ref{sec:PSD_4to2}}{Proofs for section \ref{sec:PSD_4to2}}}\label{app:F4}
We prove Proposition~\ref{pro:BM_2to2}.
\begin{proof}
    Since $\buanifold$ is a manifold near $R$ and $\phiBM$ is open (see section~\ref{sec:lowrank_posi}), Theorem~\ref{the:2to2} applies to $\phiBM$ restricted to a neighborhood of $R$. Let $s=\rank(R)=\rank(X)$. The hypothesis of Theorem~\ref{the:2to2} holds, since, for every $\eta\in\tangent_{\boundedranksdp(n)\cap\uanifold}(X)$, the set $\tangenttwo_{\boundedranksdp(n)\cap\uanifold}(X;\eta)$ contains $2\eta X^\dagger\eta+\xi$ for some $\xi\in\tangent_{\sdp^+_s(n)}(X)$ by~\eqref{eq:FtangentF}, Proposition~\ref{pro:tangenttoS+r} and Lemma~\ref{lem:FetaJ}(i). If $s<r$, then, for $w\neq0$, every $\eta\in\kernel(\aanifold)$ with $J(\eta)=ww^\top$, which exists by Lemma~\ref{lem:FetaJ}(ii), lies in $\tangent_{\boundedranksdp(n)\cap\uanifold}(X)\setminus\ima(\tanL_R)$ by~\eqref{eq:FTX} and~\eqref{eq:FimL}; hence $\ima(\tanL_R)\neq\tangent_{\boundedranksdp(n)\cap\uanifold}(X)$, and ``\,$2\!\Rightarrow\!2$'' fails by Theorem~\ref{the:2to2}. Let $s=r$. Then $J(\eta)=0$ for all $\eta\in\tangent_{\boundedranksdp(n)\cap\uanifold}(X)$, and thus $\ima(\tanL_R)=\tangent_{\boundedranksdp(n)\cap\uanifold}(X)$ by~\eqref{eq:FimL}. Let $\dot R\in\tangent_{\buanifold}(R)$, $\eta:=\tanL_R(\dot R)$, and $u\in\tangenttwo_{\buanifold}(R;\dot R)$, which exists by LICQ. Then $\tanQ_{R,\dot R}(u)=uR^\top+Ru^\top+2\dot R\dot R^\top\in\kernel(\aanifold)$, and, writing $R=U\varLambda^{1/2}P^\top$ with $P\in\orth(r)$, we have $R^\top X^\dagger R=PP^\top=I_r$ and $J(\dot R\dot R^\top-\eta X^\dagger\eta)=U_\bot^\top\dot R(I_r-PP^\top)\dot R^\top U_\bot=0$. Hence $\tanQ_{R,\dot R}(u)\in2\eta X^\dagger\eta+\tangent_{\sdp^+_s(n)}(X)$, and Proposition~\ref{pro:tangenttoS+r} with $\ell=s=r$ and~\eqref{eq:FtangentF} give
    \begin{equation*}
        \tangenttwo_{\boundedranksdp(n)\cap\uanifold}(X;\eta)=\big(2\eta X^\dagger\eta+\tangent_{\sdp^+_s(n)}(X)\big)\cap\kernel(\aanifold)=\tanQ_{R,\dot R}(u)+\ima(\tanL_R),
    \end{equation*}
    which implies~\eqref{eq:nece_suff_2to2} as in the proof of Proposition~\ref{pro:LR_2to2}. Theorem~\ref{the:2to2} yields ``\,$2\!\Rightarrow\!2$''.
\end{proof}

We prove Proposition~\ref{pro:Fcompare}.
\begin{proof}
    Let $Z:=U_\bot^\top SU_\bot$. Since $SX=0$ and $S\succeq0$ by Proposition~\ref{pro:Ffirst}, $S=U_\bot ZU_\bot^\top$ with $Z\succeq0$ and $\rank(S)=\rank(Z)$, and thus $\rank(X)+\rank(S)=n$ if and only if $Z\succ0$.

    (i) By~\eqref{eq:FimL} and~\eqref{eq:Fcone}, $\ima(\tanL_R)$ consists of the elements $\eta$ of the set~\eqref{eq:Fcone} with $J(\eta)=0$. If $Z\succ0$, then $ZJ(\eta)=0$ forces $J(\eta)=0$ for every $\eta$ in the set~\eqref{eq:Fcone}, and hence $\ima(\tanL_R)=\tangent_{\boundedranksdp(n)\cap\uanifold}(X)\cap\vecspan\{\nabla f(X)\}^\bot$; if $Zv=0$ for a unit vector $v$, then every $\eta\in\kernel(\aanifold)$ with $J(\eta)=vv^\top$, which exists by Lemma~\ref{lem:FetaJ}(ii), belongs to the set~\eqref{eq:Fcone} but not to $\ima(\tanL_R)$.

    (ii) Combine Propositions~\ref{pro:Ffirst}, \ref{pro:Fopt}(i) and~\ref{pro:FBM}(ii) with Proposition~\ref{pro:Fcompare}(i).

    (iii) Let $Zv=0$ with $\|v\|=1$, let $\hat\eta\in\kernel(\aanifold)$ satisfy $J(\hat\eta)=vv^\top$ (Lemma~\ref{lem:FetaJ}(ii)), which belongs to the set~\eqref{eq:Fcone} since $J(\hat\eta)\in\sym^+_{\le r-s}(n-s)$ and $ZJ(\hat\eta)=0$, and let $\xi:=\hat\eta-\projection_{\ima(\tanL_R)}(\hat\eta)$, which is nonzero since $J(\xi)=vv^\top$. For $\kappa>0$, define $\tilde f(Y):=\innerp{\nabla f(X),Y-X}-\frac{\kappa}{2}\innerp{\xi,Y-X}^2$. Then $\nabla\tilde f(X)=\nabla f(X)$, the multiplier $S$ is unchanged, and the form~\eqref{eq:FQ} of $\tilde f$ reads $\tilde Q_S(\eta)=2\innerp{Z,W_2(\eta)^\top\varLambda^{-1}W_2(\eta)}-\kappa\innerp{\xi,\eta}^2$ (see the proof of Proposition~\ref{pro:Fopt}(ii)). Since $\xi\perp\ima(\tanL_R)$ and $Z\succeq0$, we have $\tilde Q_S\ge0$ on $\ima(\tanL_R)$, and hence $R$ is second-order stationary by Proposition~\ref{pro:FBM}(ii). Since $\innerp{\xi,\hat\eta}=\|\xi\|^2>0$, $\tilde Q_S(\hat\eta)<0$ for large $\kappa$, and thus $X$ is not second-order stationary by Proposition~\ref{pro:Fopt}(i).
\end{proof}

We prove Proposition~\ref{pro:F2to2}.
\begin{proof}
    From $\|R^\top v\|^2=v^\top Xv=0$, every $C^\infty$ curve $c$ in $\buanifold$ with $c(0)=R$ satisfies $\|c(\tau)^\top v\|\le C\tau$ for small $\tau\ge0$, and thus $f_v(c(\tau)c(\tau)^\top)-f_v(X)=-\|c(\tau)^\top v\|^4\ge-C^4\tau^4$; hence $R$ is third-order stationary. Since $\nabla f_v(X)=0$, \eqref{eq:def2optcon} requires $\innerp{\eta,\nabla^2f_v(X)[\eta]}=-2(v^\top\eta v)^2\ge0$ for every $\eta\in\tangent_{\boundedranksdp(n)\cap\uanifold}(X)$, where $\tangenttwo_{\boundedranksdp(n)\cap\uanifold}(X;\eta)\neq\emptyset$ by the proof of Proposition~\ref{pro:Fopt}(i). We show that this requirement fails. With $v=U_\bot w$, let $\hat\eta\in\kernel(\aanifold)$ satisfy $J(\hat\eta)=ww^\top$ (Lemma~\ref{lem:FetaJ}(ii)). Then $\hat\eta\in\tangent_{\boundedranksdp(n)\cap\uanifold}(X)$ by~\eqref{eq:FTX}, and $v^\top\hat\eta v=w^\top J(\hat\eta)w=1$ because $U^\top v=0$; hence $\innerp{\hat\eta,\nabla^2f_v(X)[\hat\eta]}=-2<0$, and thus $X$ is not second-order stationary.
\end{proof}

The following lemma constructs the curves used in section~\ref{sec:PSD_4to2}.

\begin{lemma}\label{lem:Fjet}
    Let $X\in\boundedranksdp(n)\cap\uanifold$ with $\rank(X)=s\le r$ satisfy LICQ, and let $R\in\phiBM^{-1}(X)$. For every $\eta\in\tangent_{\boundedranksdp(n)}(X)\cap\kernel(\aanifold)$ and $\zeta\in\tangenttwo_{\boundedranksdp(n)}(X;\eta)\cap\kernel(\aanifold)$, there exist $\varepsilon>0$ and a $C^\infty$ curve $c:(-\varepsilon,\varepsilon)\to\buanifold$ with $c(0)=R$ and
    \begin{equation}\label{eq:Fjet}
        \phiBM(c(\tau))=X+\tau^2\eta+\tfrac12\tau^4\zeta+O(\tau^6).
    \end{equation}
\end{lemma}
\begin{proof}
    Write $R=U\varLambda^{1/2}P^\top$ with $P\in\stiefel(r,s)$, complete $P$ to $[P\ P_\bot]\in\orth(r)$, and set $\tilde R:=R[P\ P_\bot]=[U\varLambda^{1/2}\ 0]$. Suppose that a $C^\infty$ curve $\tilde c:(-\varepsilon,\varepsilon)\to\buanifold$ satisfies $\tilde c(0)=\tilde R$ and~\eqref{eq:Fjet} with $\tilde c$ in place of $c$. Since $[P\ P_\bot]$ is orthogonal, the curve $c:=\tilde c\,[P\ P_\bot]^\top$ satisfies $c(0)=R$ and $\phiBM(c(\tau))=\phiBM(\tilde c(\tau))$; hence $c$ lies in $\buanifold=\phiBM^{-1}(\uanifold)$ and satisfies~\eqref{eq:Fjet}. It therefore suffices to construct such a curve $\tilde c$. Let $\ell:=s+\rank(J(\eta))\le r$, let $U_\bot J(\eta)U_\bot^\top=U_\eta\varSigma_\eta U_\eta^\top$ with $U_\eta\in\stiefel(n,\ell-s)$ and $\varSigma_\eta\succ0$, choose $U_{\eta\bot}$ with $[U\ U_\eta\ U_{\eta\bot}]\in\orth(n)$, and take $U_\bot=[U_\eta\ U_{\eta\bot}]$ in~\eqref{eq:Fblock}; this choice changes neither $\ell$ nor $U_\bot J(\eta)U_\bot^\top=(I_n-UU^\top)\eta(I_n-UU^\top)$. By Proposition~\ref{pro:tangenttoS+r}, $D:=\zeta-2\eta X^\dagger\eta$ satisfies $U_{\eta\bot}^\top DU_{\eta\bot}=2N_0N_0^\top$ for some $N_0\in\mbR^{(n-\ell)\times(r-\ell)}$; set $D_1:=U_\eta^\top DU_\eta$ and $D_2:=U_\eta^\top DU_{\eta\bot}$. The linear map $\mathscr{B}(Y):=(I_n-\frac12UU^\top)YU\varLambda^{-1/2}$ satisfies $\mathscr{B}(Y)\varLambda^{1/2}U^\top+U\varLambda^{1/2}\mathscr{B}(Y)^\top=Y$ for all $Y\in\tangent_{\sdp^+_s(n)}(X)$, because $Y=UU^\top Y+YUU^\top-UU^\top YUU^\top$ by~\eqref{eq:FTs}. Define $H:=\mathscr{B}(\eta-U_\bot J(\eta)U_\bot^\top)$, $V:=U_\eta\varSigma_\eta^{1/2}$, $V':=\frac14U_\eta D_1\varSigma_\eta^{-1/2}+\frac12U_{\eta\bot}D_2^\top\varSigma_\eta^{-1/2}$, $N:=U_{\eta\bot}N_0$, and
    \begin{equation}\label{eq:Fcurve}
        c_0(\tau):=\big[\,U\varLambda^{1/2}+\tau^2H+\tau^4H'\ \ \ \tau V+\tau^3V'\ \ \ \tau^2N\,\big]\in\mbR^{n\times r},
    \end{equation}
    with column blocks of widths $s$, $\ell-s$ and $r-\ell$, and $H'\in\mbR^{n\times s}$ to be determined. Expanding the product gives
    \begin{equation*}
        \begin{aligned}
            c_0(\tau)c_0(\tau)^\top=&\ X+\tau^2\big(H\varLambda^{1/2}U^\top+U\varLambda^{1/2}H^\top+VV^\top\big)
            \\
            &+\tau^4\big(HH^\top+VV'^\top+V'V^\top+NN^\top+H'\varLambda^{1/2}U^\top+U\varLambda^{1/2}H'^\top\big)
            \\
            &+\tau^6\big(HH'^\top+H'H^\top+V'V'^\top\big)+\tau^8H'H'^\top.
        \end{aligned}
    \end{equation*}
    The coefficient of $\tau^2$ equals $\eta$, since $VV^\top=U_\bot J(\eta)U_\bot^\top$. For the coefficient of $\tau^4$, let $E:=\frac12\zeta-HH^\top-VV'^\top-V'V^\top-NN^\top$. From $U_\bot^\top H=W_2(\eta)^\top\varLambda^{-1/2}$, we have $U_\bot^\top HH^\top U_\bot=W_2(\eta)^\top\varLambda^{-1}W_2(\eta)=U_\bot^\top\eta X^\dagger\eta U_\bot$, and the definitions of $V$, $V'$ and $N$ give
    \begin{equation*}
        U_\bot^\top\big(VV'^\top+V'V^\top+NN^\top\big)U_\bot=\frac12\begin{bmatrix}D_1&D_2\\D_2^\top&2N_0N_0^\top\end{bmatrix}=\frac12U_\bot^\top DU_\bot.
    \end{equation*}
    Hence $U_\bot^\top EU_\bot=0$, i.e., $E\in\tangent_{\sdp^+_s(n)}(X)$, and $H':=\mathscr{B}(E)$ makes the coefficient of $\tau^4$ equal to $\frac12\zeta$. Consequently, $c_0(\tau)c_0(\tau)^\top=X+\tau^2\eta+\frac12\tau^4\zeta+O(\tau^6)$, and $\bar h(c_0(\tau))=O(\tau^6)$ because $\aanifold(\eta)=\aanifold(\zeta)=0$.

    The curve $c_0$ need not lie in $\buanifold$. It remains to bring it into $\buanifold$ by a correction of order $\tau^6$, which does not affect the expansion of $c_0c_0^\top$ up to $O(\tau^6)$. If $q=0$, set $\tilde c:=c_0$. Otherwise, let $\mathcal{D}:=\diff\bar h(\tilde R)$, with adjoint $\mathcal{D}^*\nu=2\sum_{i=1}^q\nu_iA_i\tilde R$, and $\Phi(\tau,\nu):=\bar h(c_0(\tau)+\mathcal{D}^*\nu)$. The map $\Phi$ is polynomial, $\Phi(0,0)=0$, and $\partial_\nu\Phi(0,0)=\mathcal{D}\mathcal{D}^*$ is the Gram matrix $(4\innerp{A_i\tilde R,A_j\tilde R})_{i,j}$, which is invertible by LICQ. The implicit function theorem provides a $C^\infty$ map $\nu$ near $0$ with $\nu(0)=0$ and $\Phi(\tau,\nu(\tau))=0$. Moreover, $0=\Phi(\tau,0)+M(\tau)\nu(\tau)$ with $M(\tau):=\int_0^1\partial_\nu\Phi(\tau,\theta\nu(\tau))\,d\theta\to\mathcal{D}\mathcal{D}^*$, whence $\|\nu(\tau)\|\le C\|\bar h(c_0(\tau))\|=O(\tau^6)$. The curve $\tilde c(\tau):=c_0(\tau)+\mathcal{D}^*\nu(\tau)$ lies in $\buanifold$, passes through $\tilde R$ at $\tau=0$, and satisfies $\tilde c\tilde c^\top=c_0c_0^\top+O(\tau^6)$, which proves~\eqref{eq:Fjet}.
\end{proof}

Consequently, we arrive at the proof of Theorem~\ref{the:F4to2}.
\begin{proof}
    Let $\eta\in\tangent_{\boundedranksdp(n)\cap\uanifold}(X)$ and $\zeta\in\tangenttwo_{\boundedranksdp(n)\cap\uanifold}(X;\eta)$, which contains $2\eta X^\dagger\eta+\xi$ for some $\xi\in\tangent_{\sdp^+_s(n)}(X)$ by~\eqref{eq:FtangentF}, Proposition~\ref{pro:tangenttoS+r} and Lemma~\ref{lem:FetaJ}(i), and let $c$ be the curve from Lemma~\ref{lem:Fjet}. Inserting $\phiBM(c(\tau))-X=\tau^2\eta+\frac12\tau^4\zeta+O(\tau^6)$ into the second-order Taylor expansion of $f$ at $X$ gives
    \begin{equation}\label{eq:Fexpand}
        \bar f(c(\tau))-\bar f(R)=\tau^2\innerp{\nabla f(X),\eta}+\frac{\tau^4}{2}\big(\innerp{\nabla f(X),\zeta}+\innerp{\eta,\nabla^2f(X)[\eta]}\big)+o(\tau^4).
    \end{equation}
    Let $R$ be fourth-order stationary. Then $R$ is second-order stationary, and dividing~\eqref{eq:Fexpand} by $\tau^2$ yields $\innerp{\nabla f(X),\eta}\ge0$ for all $\eta\in\tangent_{\boundedranksdp(n)\cap\uanifold}(X)$, i.e., $X$ is first-order stationary. For every $\eta\in\tangent_{\boundedranksdp(n)\cap\uanifold}(X)$ with $\innerp{\nabla f(X),\eta}=0$ and every $\zeta\in\tangenttwo_{\boundedranksdp(n)\cap\uanifold}(X;\eta)$, dividing~\eqref{eq:Fexpand} by $\tau^4$ yields~\eqref{eq:def2optcon}.
\end{proof}

\section{Proof of Lemma~\ref{lem:fivepart_convergence}}\label{sec:proof_Mor}
We prove Lemma~\ref{lem:fivepart_convergence} in section~\ref{sec:Mordukhovichnormalcone}, borrowing the idea from~\cite[Lemma 4.2]{olikier2022continuity}.

\vspace{1mm}

\noindent\textit{Proof of Lemma~\ref{lem:fivepart_convergence}}\ \ 
Let $\underline{X}_i\in\projection_{\manifold_s}(X_i)$ and $\underline{Y}_i\in\projection_{\manifold_{k-\ell}}(Y_i)$. Throughout the proof, we write $P_i^X:=\underline{X}_i\underline{X}_i^\dagger\in\R^{m\times m}$ and
$\Pi_i^X:=\underline{X}_i^\dagger \underline{X}_i\in\R^{n\times n}$ for the orthogonal projections onto
$\ima \underline{X}_i$ and $\ima \underline{X}_i^\top$, respectively; likewise,
$P_i^Y:=\underline{Y}_i\underline{Y}_i^\dagger$ and $\Pi_i^Y:=\underline{Y}_i^\dagger \underline{Y}_i$ represent the projections onto
$\ima \underline{Y}_i$ and $\ima \underline{Y}_i^\top$, respectively. Furthermore, we define the orthogonalization of a matrix $Z$ with full column rank by $\qf(Z):=Z(Z^\top Z)^{-1/2}$. In the analysis, if needed, we may take a subsequence, still indexed by $i$ for simplicity.

Noticing that $(\underline{X}_i,\underline{X}_i^\dagger)\to(X,X^\dagger)$ as $i\to\infty$, we set $U_i:=\qf(P_i^X\,U)\in\stiefel(m,s)$ and $V_i:=\qf(\Pi_i^X\,V)\in\stiefel(n,s)$ to obtain $U_i\to U$ and $V_i\to V$. Next, complete $U_i$ and $V_i$ to orthonormal bases
of $\ima X_i$ and $\ima X_i^\top$ by selecting
$\tilde U_{i\bot}\in\stiefel(m,\downr-s)$ and
$\tilde V_{i\bot}\in\stiefel(n,\downr-s)$, respectively, such that $\ima[U_i\ \tilde U_{i\bot}]=\ima X_i$ and $\ima[V_i\ \tilde V_{i\bot}]=\ima X_i^\top$. By compactness of the Stiefel manifolds, after taking a subsequence if necessary, there exist $\tilde U_\bot\in\stiefel(m,\downr-s)$ and
$\tilde V_\bot\in\stiefel(n,\downr-s)$ with
$\tilde U_{i\bot}\to \tilde U_\bot$ and $\tilde V_{i\bot}\to \tilde V_\bot$.

In a similar manner, we can set $U_{iY}:=\qf(P_i^Y\,U_Y)\in\stiefel(m,k-\ell)$,
$V_{iY}:=\qf(\Pi_i^Y\,V_Y)\in\stiefel(n,k-\ell)$, and find $\tilde U_{iY\bot}\in\stiefel(m,\ell-\upr)$ and $\tilde V_{iY\bot}\in\stiefel(n,\ell-\upr)$ such that $\ima[U_{iY}\ \tilde U_{iY\bot}]=\ima Y_i$ and $\ima[V_{iY}\ \tilde V_{iY\bot}]=\ima Y_i^\top$. Moreover, we have $[U_{iY}\ \tilde U_{iY\bot}]\to [U_Y\ \tilde{U}_{Y\bot}]$ and $[V_{iY}\ \tilde V_{iY\bot}]\to [V_Y\ \tilde{V}_{Y\bot}]$ for some $\tilde U_{Y\bot}\in\stiefel(m,\ell-\upr)$ and $\tilde V_{Y\bot}\in\stiefel(n,\ell-\upr)$. 

Passing to the limit $i\to\infty$ in the orthogonality $\ima Y_i\perp\ima X_i$ and
$\ima Y_i^\top\perp\ima X_i^\top$ gives the orthogonality of $[U\ \tilde U_\bot\ U_Y\ \tilde U_{Y\bot}]$ and $[V\ \tilde V_\bot\ V_Y\ \tilde V_{Y\bot}]$, respectively. Let $\breve U_\bot\in\stiefel\!\big(m,\,m-k+\upr-\downr\big)$ and
$\breve V_\bot\in\stiefel\!\big(n,\,n-k+\upr-\downr\big)$ be orthonormal bases of $(\ima[U\ \tilde U_\bot\ U_Y\ \tilde U_{Y\bot}])^\perp$ and $(\ima[V\ \tilde V_\bot\ V_Y\ \tilde V_{Y\bot}])^\perp$, respectively. Define
\[
\breve U_{i\bot}:=\qf\!\big((I_m-P_i^X-\tilde{U}_{i\bot}\tilde{U}_{i\bot}^\top-P_i^Y-\tilde{U}_{iY\bot}\tilde{U}_{iY\bot}^\top)\,\breve U_\bot\big)
\in\stiefel\!\big(m,\,m-k+\upr-\downr\big),
\]
\[
\breve V_{i\bot}:=\qf\!\big((I_n-\Pi_i^X-\tilde{V}_{i\bot}\tilde{V}_{i\bot}^\top-\Pi_i^Y-\tilde{V}_{iY\bot}\tilde{V}_{iY\bot}^\top)\,\breve V_\bot\big)
\in\stiefel\!\big(n,\,n-k+\upr-\downr\big).
\]
Using the continuity of the orthogonalization, we obtain $\breve U_{i\bot}\to
\breve U_\bot$ and $\breve V_{i\bot}\to
\breve V_\bot$ as $i\to\infty$. Note that $\breve U_{i\bot}$ are orthogonal to
both $\ima X_i$ and $\ima Y_i$; similarly, $\breve V_{i\bot}$ are orthogonal to
both $\ima X_i^\top$ and $\ima Y_i^\top$. Therefore, we have
\[
[U_i\ \tilde U_{i\bot}\ \breve U_{i\bot}\ \tilde U_{iY\bot}\ U_{iY}]
\in\orth(m),
\ \ \ \text{and}\ \ \ 
[V_i\ \tilde V_{i\bot}\ \breve V_{i\bot}\ \tilde V_{iY\bot}\ V_{iY}]
\in\orth(n)
\]
for each $i$. Letting $i\to\infty$ and collecting the limits of each component, we conclude that the sequences satisfy the required properties. \qed

\vspace{1mm}

\printbibliography

\end{document}